\documentclass[a4paper,12pt,reqno]{amsart}
\usepackage{amssymb}
\usepackage{amsmath}
\usepackage{enumitem}
\usepackage{mathrsfs}
\usepackage[all]{xy}
\usepackage{tikz-cd}
\usepackage{mathtools}
\usepackage{hyperref}
\usepackage{url}
\usepackage{bbm}

\usepackage[OT2,T1]{fontenc}
\DeclareSymbolFont{cyrletters}{OT2}{wncyr}{m}{n}
\usepackage[british]{babel}

\usepackage[margin=1.2in]{geometry}
\numberwithin{equation}{section} \numberwithin{figure}{section}

\DeclareMathOperator{\Pic}{Pic} 
\DeclareMathOperator{\Gal}{Gal} 
\DeclareMathOperator{\Aut}{Aut} 
\DeclareMathOperator{\AutU}{\underline{Aut}}
 
\DeclareMathOperator{\Spec}{Spec}
 \DeclareMathOperator{\rank}{rank}
\DeclareMathOperator{\Hom}{Hom} \DeclareMathOperator{\re}{Re}
\DeclareMathOperator{\im}{Im}   
 
 \DeclareMathOperator{\Val}{Val}
\DeclareMathOperator{\Eff}{Eff}
\DeclareMathOperator{\Br}{Br} \DeclareMathOperator{\Brun}{Br_{un}}

\DeclareMathOperator{\inv}{inv} \DeclareMathOperator{\res}{\partial}
\DeclareMathOperator{\ord}{ord}

 \DeclareMathOperator{\Res}{R}
\DeclareMathOperator{\Norm}{N} \DeclareMathOperator{\Frob}{Frob}
 
\let\H\undefined
\DeclareMathOperator{\H}{H} 

\let\L\undefined
\DeclareMathOperator{\L}{L}

\DeclareMathOperator{\age}{age} 
\DeclareMathOperator{\rad}{rad} 
\DeclareMathOperator{\Ext}{Ext}
\DeclareMathOperator{\ad}{ad}
\DeclareMathOperator{\ind}{ind}
\newcommand{\PicOrb}{\operatorname{Pic}^{\mathrm{orb}}}

\DeclareSymbolFont{cyrletters}{OT2}{wncyr}{m}{n}
\DeclareMathSymbol{\Sha}{\mathalpha}{cyrletters}{"58}
\DeclareMathSymbol{\Be}{\mathalpha}{cyrletters}{"42}

\renewcommand{\subset}{\subseteq}
\renewcommand{\O}{\mathcal{O}}
\newcommand{\GL}{\mathrm{GL}}
\newcommand{\SL}{\mathrm{SL}}

\newcommand{\dual}[1]{\widehat{#1}}
\newcommand{\sep}{\mathrm{sep}}
\newcommand{\orb}{\mathrm{orb}}

\newcommand\F{\mathbb{F}}
\renewcommand\P{\mathbb{P}}
\renewcommand\P{\mathbb{P}}
\newcommand\Z{\mathbb{Z}}
\newcommand\N{\mathbb{N}}
\newcommand\Q{\mathbb{Q}}
\newcommand\R{\mathbb{R}}
\newcommand\C{\mathbb{C}}
\newcommand\GG{\mathbb{G}}

\newcommand\Gm{\GG_\mathrm{m}}
\newcommand{\Adele}{\mathbf{A}}
\newcommand{\br}{\mathscr{B}}

\newcommand{\cycl}{\chi_{\mathrm{cycl}}}

\newtheorem{lemma}{Lemma}
\newtheorem{conjecture}[lemma]{Conjecture}
\newtheorem*{conjecture*}{Conjecture}
\newtheorem{theorem}[lemma]{Theorem}
\newtheorem{proposition}[lemma]{Proposition}
\newtheorem{corollary}[lemma]{Corollary}

\theoremstyle{definition}
\newtheorem{example}[lemma]{Example}
\newtheorem{definition}[lemma]{Definition}
\newtheorem{remark}[lemma]{Remark}

\numberwithin{lemma}{section}

\usepackage{color}

\newcommand{\dan}[1]{{\color{blue} \sf $\clubsuit\clubsuit\clubsuit$ Dan: [#1]}}
\newcommand{\tim}[1]{{\color{magenta} \sf $\spadesuit\spadesuit\spadesuit$ Tim: [#1]}}

\begin{document}

\title{Malle's conjecture and Brauer groups of stacks}

\author{Daniel Loughran}
\address{
Department of Mathematical Sciences \\
University of Bath \\
Claverton Down \
Bath\\ 
BA2 7AY\\
UK.}
\urladdr{https://sites.google.com/site/danielloughran}

\author{Tim Santens}
\address{
University of Cambridge \\ 
DPMMS \\
Centre for Mathematical Sciences\\
Wilberforce Road \\
Cambridge \\
CB3 0WB \\ UK}
	

\begin{abstract}
	We put forward a conjecture for the leading constant in Malle's conjecture on number fields of bounded discriminant, guided by stacky versions of conjectures of Batyrev--Manin, Batyrev--Tschinkel, and Peyre on rational points of bounded height on Fano varieties. A new framework for Brauer groups of stacks plays a key role in our conjecture, and we define a new notion of the unramified Brauer group of an algebraic stack.
\end{abstract}

\maketitle

\thispagestyle{empty}

\tableofcontents

\section{Introduction} \label{sec:intro}

\subsection{Malle's conjecture}
Let $k$ be a number field. In \cite{Mal02,Mal04}, Malle put forward the following.
\begin{conjecture*}[Malle]
Let $G \subseteq S_n$ be a non-trivial transitive subgroup. Then
$$\#\left\{ K/k :\begin{array}{l}
 [K:k] = n, \Gal(\widetilde{K}/k) \cong G, \\
 |\Norm_{k/\Q} \Delta_{K/k}| \leq B
 \end{array}
 \right\}\sim c_{\mathrm{Malle}}(k,G) B^{a(G)} (\log B)^{b(k,G)-1}$$
for some $c_{\mathrm{Malle}}(k,G) > 0, a(G) > 0, b(k,G) \in \N$.
\end{conjecture*}
In the conjecture $\Delta_{K/k}$ denotes the relative discriminant of $K/k$,  one counts isomorphism classes of fields $K/k$ and $\widetilde{K}$ denotes the Galois closure of $K$ with $\Gal(\widetilde{K}/k) \cong G$ being an isomorphism of permutation groups. Malle gave predictions for the invariants $a(G)$ and $b(k,G)$ as follows.
For an element $g \in G$ we denote its \textit{index} by $\ind(g) = n \,-$~the number of orbits of $g$ on $\{1,\dots,n\}$. We let $\mathcal{C}_G^*$ denote the collection of non-identity conjugacy classes of $G$; this comes equipped with a natural action of $\Gamma_k: = \Gal(\bar{k}/k)$ via the anticyclotomic character (see \S\ref{sec:Galois_action_conjugacy_classes}).  We let $\mathcal{M}(G) \subseteq \mathcal{C}_G^*$ denote the collection of conjugacy classes of minimal index. Then 
\begin{equation} \label{def:a_b_Malle}
	a(G) = \min\{ \ind(c) : c \in \mathcal{C}_G^*\}^{-1}, \quad b(k,G) = \#(\mathcal{M}(G)/\Gamma_k).
\end{equation}

One of the first results on counting number fields was the seminal work of Davenport Heilbronn for $S_3$ \cite{DH71}, with subsequent treatments of the case of abelian extensions \cite{Mak85,Mak93,Wri89, Woo10, FLN18}. Since the formulation of Malle's conjecture, the area has attracted considerable interest; for example \cite{CDO02b,ASVW21, Bha05, Bha07, Bha10,KP21,Wan21,Klu05,ShTh22, EV05, EV10, EVW13, OA21, Alb21, AOWW24, Wan24}. 
Despite this, there are two outstanding issues with Malle's conjecture. Firstly, the exponent of $\log B$ is wrong in general, as first observed by Kl\"{u}ners \cite{Klu05}. Secondly, Malle offered no prediction for the leading constant $c_{\mathrm{Malle}}(k,G)$.  Understanding the leading constant is significant open problem; for example recent work of Shankar and Thorne \cite{ShTh22} notes that ``the leading constants appearing in front of Malle’s heuristics are still shrouded with mystery''.



In our paper we solve this problem by giving a corrected version of Malle's conjecture which gives the expected power of $\log B$ providing one removes an explicit exceptional set of field extensions, and moreover we also give a precise prediction for the leading constant. Bhargava \cite{Bha07} has proposed a formula for $S_n$, where it is given by a product of local densities. But in examples $c_{\mathrm{Malle}}(k,G)$ can exhibit pathological properties; for example it need no longer be a product of local densities and subgroups of $G$ can interfere and appear in the leading constant.

To highlight the aims of the paper, we give explicit versions of our conjecture. Firstly, counting $A_4$-quartics of bounded discriminant is a notorious open problem in the Malle's conjecture literature. We make completely explicit what our conjecture says in this case as a challenge to researchers in the community. 

\begin{conjecture} \label{conj:A_4_intro}
	$$2\#\left\{ [K:\Q] = 4 : 
	|\Delta_K|  \leq B, \Gal(\widetilde{K}/\Q) \cong A_4 \right\} \sim c(\Q,A_4,\Delta) B^{1/2} \log B,$$
	where $\widetilde{K}$ denotes the Galois closure of $K$ and
	$$c(\Q,A_4,\Delta) = \frac{35}{648}\prod_{p > 3} \left(1 - \frac{1}{p}\right)^2\left(1 + \frac{2 + \left(\frac{-3}{p}\right)}{p}\right).$$
\end{conjecture}

The factor $2$ on the left has a natural interpretation via a groupoid cardinality. 

\subsection{The Malle--Bhargava heuristics}
Ellenberg and Venkatesh \cite[Ques.~4.3]{EV05} considered analogues of Malle's conjecture where the discriminant is replaced by a different height function. This is quite natural even from the perspective of Malle's conjecture, since for each transitive embedding $G \subset S_n$ one obtains a different counting problem. Moreover, numerous works have considered the problem of counting number fields with local conditions imposed. Bhargava \cite{Bha07} was the first to study this in the case of $S_n$-extensions of degree $n$. The expectation is that the leading constant should change in a simple way exactly corresponding in the local conditions imposed; in the literature this is often informally referred to as the \emph{Malle--Bhargava heuristics}. For Bhargava's original question see \cite[\S8.2]{Bha10}, as well as Wood's \cite[\S 6.1]{Woo19} for a discussion of these heuristics and problems with the existing framework, and \cite[\S 10]{Woo16} which notes that ``An important open question is to even make a good conjecture about when exactly the principle should apply''.

We also give precise predictions for these more general problems. As an example, the next conjecture, despite at first glance looking artificial, is the smallest non-abelian group where the leading constant is given by a finite sum of Euler products. This illustrates the full range of behaviour which we encapsulate in the paper, as it corresponds to a case where there is a transcendental Brauer--Manin obstruction.


\begin{conjecture} \label{conj:A_4_conductor_intro}
	For an $A_4$-quartic field $K$, let $H(K) := |\Delta_{K}|^{\frac{5}{2}}|\Delta_{\tilde{K}}|^{-\frac{1}{2}}.$ Then
	\[
	\begin{split}
		&2\#\left\{ [K:\Q] = 4 :\Gal(\widetilde{K}/\Q) \cong A_4, H(K) \leq B, K \otimes_{\Q} \R \cong \R^4 \right\} \sim c_{\R^4}(\Q,A_4,H) B, \\
		&2\#\left\{ [K:\Q] = 4 :\Gal(\widetilde{K}/\Q) \cong A_4, H(K) \leq B, K \otimes_{\Q} \R \cong \C^2 \right\} \sim c_{\C^2}(\Q,A_4,H) B,
	\end{split}
	\]
	where $\widetilde{K}$ denotes the Galois closure of $K$ and
	\begin{align*}
	c_{\R^4}(\Q,A_4,H) = &\frac{145}{3456}\prod_{p > 3} \left(1 - \frac{1}{p}\right)\left(1 + \frac{1 + \left(\frac{-3}{p}\right)}{p} + \frac{1}{p^2}\right) \\ + \,\, &\frac{319}{10368}\prod_{p > 3} \left(1 - \frac{1}{p}\right)\left(1 + \frac{1 + \left(\frac{-3}{p}\right)}{p}\right) =  0.0594...
	\end{align*}
	\begin{align*}
	c_{\C^2}(\Q,A_4,H) = &\frac{145}{1152}\prod_{p > 3} \left(1 - \frac{1}{p}\right)\left(1 + \frac{1 + \left(\frac{-3}{p}\right)}{p} + \frac{1}{p^2}\right) \\ - \,\, &\frac{319}{3456}\prod_{p > 3} \left(1 - \frac{1}{p}\right)\left(1 + \frac{1 + \left(\frac{-3}{p}\right)}{p}\right) = 0.0347....
	\end{align*}
\end{conjecture} 
From the given numerical values, an interesting prediction of Conjecture \ref{conj:A_4_conductor_intro} is that approximately $63\%$ of $A_4$-quartics are totally real when ordered by $H$. On the other hand, when ordering by discriminant, the Malle--Bhargava heuristics (and experimental verification) predicts that only $25 \%$ of $A_4$-quartics are totally real. This difference in the local behaviour for a different height function comes from the second Euler factor, which itself comes from a transcendental Brauer group element. We give numerical evidence towards this conjecture in \S \ref{sec:A_4-transcendental}.

\subsection{Malle's conjecture via stacks}
Key to our paper is a viewpoint recently put forward in \cite{ESZB,DYTor,DYBM} to study Malle's conjecture via the classifying stack $BG$ of $G$, and interpret Malle's conjecture as a version of the Batyrev--Manin conjecture \cite{FMT89, BM90} for an algebraic stack (see \S \ref{sec:BG} for background on $BG$). The emphasis in these papers was on finding common generalisations of Manin's and Malle's conjecture, whereas the focus in our paper is using the stack theoretic framework to say something \emph{new} about the original Malle's conjecture. (Kedlaya \cite[\S10]{Ked07} appears to have been the first to suggest that one use $BG$ to study Malle's conjecture.) Manin's conjecture  concerns rational points of bounded height on Fano varieties. Here Peyre \cite{Pey95} has put forward a conjectural leading constant in Manin's conjecture, and we take Peyre's approach as our starting point for the conjecture, though serious modifications are required to make it work.

There are two challenges with formulating a conjecture for the leading constant. Firstly, there may be global obstructions: by the Grunwald--Wang theorem there is no $\Z/8\Z$-extension of $\Q$ which realises the degree $8$ unramified extension of $\Q_2$; the leading constant should reflect this obstruction. Secondly, in some examples the leading constant is given as an infinite sum of Euler products indexed by subfields; this happens for example when counting $D_4$ quartic fields \cite{CDO02}.

We  overcome both issues. We control global obstructions using a new definition of a partially unramified Brauer group $\Br_{\mathcal{M}(G)}BG$ for the stack $BG$, depending on the conjugacy classes $\mathcal{M}(G)$ of minimal index (see Definition \ref{def:partially_ramified_Br}). This group has an explicit interpretation via central extensions of $G$ by $\mu_n$. For simplicity we assume here that this group gives no obstruction; in the terminology we introduce later in the paper this means precisely that $BG(\Adele_k)_{\mathcal{M}(G)}^{\Br} = BG(\Adele_k)_{\mathcal{M}(G)}$. This condition holds for example in the setting of Conjecture \ref{conj:A_4_intro}, but not Conjecture~\ref{conj:A_4_conductor_intro}, and in general a more complicated expression involving a  sum of Euler products occurs (see Conjecture \ref{conj:balanced}). It is also convenient to rephrase the problem in terms of counting continuous homomorphisms $\varphi: \Gamma_k \to G$ from the absolute Galois group, rather than fields (for a version in terms of fields, see Conjecture~\ref{conj:discriminant}).

\begin{conjecture} \label{conj:intro}
	Let $G \subseteq S_n$ be a non-trivial transitive subgroup.
	Assume that $\Br_{\mathcal{M}(G)}BG$ gives no Brauer--Manin obstruction.
	Let $\Omega \subset \Hom(\Gamma_k,G)$ be the subset of homomorphisms which are either not surjective or which correspond to field extensions which are not linearly disjoint to $k(\mu_{|\exp(G)|})$, where $\exp(G)$ denotes the exponent of $G$.
	\begin{enumerate}
		\item If the elements of $\mathcal{M}(G)$ generate $G$ then
	$$\frac{1}{|G|}\#\{ \varphi \in \Hom(\Gamma_k,G) : \varphi \notin \Omega, |\Norm_{k/\Q} \Delta_{\varphi}| \leq B\} \sim c(k,G) B^{a(G)} (\log B)^{b(k,G)-1}$$
	where
	\begin{align*}
	c(k,G) &= \frac{a(G)^{b(k,G)-1} \cdot  |\Br_{\mathcal{M}(G)}(BG) / \Br k| \cdot
	\mathrm{Res}_{s=1} \zeta_k(s)^{b(k,G)} \cdot \tau(k,G)}{\#\dual{G}(k) (b(k,G)-1)!},  \\
	\tau(k,G) &= \prod_{\substack{v \in \Val(k) \\ v \mid \infty}}
	\frac{\#\Hom(\Gamma_{k_v}, G)}{|G|}
	 \prod_{\substack{v \in \Val(k) \\ v \nmid \infty}}\frac{(1-1/q_v)^{b(k,G)}}{|G|}\sum_{\varphi_v \in \Hom(\Gamma_{k_v}, G)} 
	 \frac{1}{q_v^{a(G) v(\Delta_{\varphi_v})}}.
	\end{align*}
	Here $\Delta_\varphi$ denotes the relative discriminant
	of the degree $n$ \'etale algebra corresponding to $\varphi$ and $\dual{G} := \Hom(G,\Gm)$ the 
	group scheme of characters of $G$.
	\item If $\mathcal{M}(G)$ does not necessarily generate $G$, consider the quotient map $q:G \to G/\langle \mathcal{M}(G) \rangle$. Then for any $\psi \in \Hom(\Gamma_k,G/\langle \mathcal{M}(G) \rangle)$ the limit
	$$c(k,G,\psi):=\lim_{B \to \infty}\frac{\#\{ \varphi \in \Hom(\Gamma_k,G) : \varphi \notin \Omega, |\Norm_{k/\Q} \Delta_{K/k}| \leq B, q \circ \varphi = \psi \}/|G|}{B^{a(G)} (\log B)^{b(k,G)-1}}$$
	exists. Moreover we have
	$$\frac{1}{|G|}\#\{ \varphi \in \Hom(\Gamma_k,G) : \varphi \notin \Omega, \Norm_{k/\Q} \Delta_{\varphi} \leq B\}\sim c(k,G) B^{a(G)} (\log B)^{b(k,G)-1}$$
	where $$c(k,G) = \sum_{\psi \in \Hom(\Gamma_k,G/\langle \mathcal{M}(G) \rangle)} c(k,G,\psi)$$
	and the sum converges.
	\end{enumerate}
\end{conjecture}

Heuristics and principles abound in the Malle's conjecture literature. We could find very few precise conjectures with no counter-examples; the main one being Bhargava's \cite{Bha07} in the case of $S_n$. Here the minimal index conjugacy classes are the transpositions and the corresponding Brauer group $\Br_{\mathcal{M}(S_n)}(BS_n)$ is constant. In this case Conjecture \ref{conj:intro}  agrees with Bhargava's heuristics (see \S\ref{sec:S_n}).

In Case (1) we say that the discriminant is \emph{balanced}. Here the factors which appear have precise analogues with Peyre's constant \cite[Def.~2.5]{Pey95}; for example $\tau(k,G)$ is a Tamagawa volume and $a(G)^{b(k,G)}/\#\dual{G}(k)$ is a version of Peyre's effective cone constant $\alpha$. However, the natural height to use in Manin's conjecture for Fano varieties is the \textit{anticanonical height}, and Peyre's construction applies in this setting. Part of the reason for the pathological nature of $c_{\mathrm{Malle}}(k,G)$ is that the discriminant is \emph{not} the anticanonical height in general; the role of the anticanonial height is played by what we call the \emph{radical discriminant} $\Norm(\rad \Delta_{K/k})$, namely the norm of the ideal given by the product of ramified primes. As such Peyre's formalism is insufficient in general, and we also require the framework of Batyrev and Tschinkel \cite{BT98}.

For Case (2), where the discriminant need not be balanced, we instead consider the subgroup $M(G):=\langle \mathcal{M}(G) \rangle$ generated by the minimial index elements, which is normal as it is generated by conjugacy classes. We then sort homomorphisms according to the quotient map $G \to G/M(G)$; we make clear in our paper that this is a version of the Iitaka fibration from birational geometry, and this viewpoint is inspired by the work of Batyrev and Tschinkel \cite{BT98}. This sorting corresponds to counting rational points in the fibres of the Iitaka fibration $BG \to B(G/M(G))$ then summing over all fibres. The first key observation is that the fibres of the Iitaka fibration can be identified with $BM(G)_{\varphi}$ where $M(G)_{\varphi}$ is a suitable inner twist of $M(G)$ (viewed as a group scheme). Thus one can interpret the count in the fibres as a version of Malle's conjecture for group schemes. Our next key observation is that the restriction of the discriminant to a fibre now becomes a balanced height function. This allows us to in fact obtain a precise prediction for the constants $c(k,G,\psi)$ which appear in Conjecture \ref{conj:intro}, with an analogous formula to the balanced case. Altogether, this clarifies that the correct generality for Malle's conjecture  should allow $G$ to be a finite \'etale group scheme, so that one considers $G$-torsors instead of fields, and allows arbitrary height functions rather than just the discriminant. Our whole paper is written in this generality, including our most general conjectures in \S\ref{sec:conjectures}. This more general perspective to Malle's conjecture is also taken in the papers \cite{Alb21,OA21,DYTor,DYBM}.

Case (2) occurs for example if $G = D_4$, where one is counting $D_4$-quartics of bounded discriminant.  In this case the Iiataka fibration is given by $D_4 \to C_2$, and the induced map on fields associates to the quartic extension its quadratic resolvent. Here Malle's conjecture is known to hold \cite{CDO02}, and we verify in \S \ref{sec:D_4} that our conjectures agree with these results, including the exact leading constant. Very recent work of Alberts, Lemke Oliver, Wang and Wood \cite{AOWW24} proves many new cases of Conjecture \ref{conj:intro}(2) (in their terminology, they say that $G$ is \emph{concentrated} in $\langle \mathcal{M}(G) \rangle$).


A closely related formula to Case (1) of our Conjecture \ref{conj:intro} appears in Alberts' paper \cite[Thm.~1.2]{Alb23}. His approach and formula is very different and  stated in terms of a random group model for number fields. However the Brauer group factor is missing: if there is a Brauer--Manin obstruction, the leading constant needs to be modified to take into account a potential failure of strong approximation (Conjecture \ref{conj:balanced}).

\subsection{Main results}
Our paper introduces a new framework for Malle's conjecture via stacks. The difference with our work and the papers \cite{ESZB,DYTor,DYBM}, is that the aim of these works was seeking common generalisations of Malle's and Manin's conjecture, whereas our paper is about using stacks to say something \textit{new about the original Malle conjecture}. Our framework allows to explain all outstanding phenomena in the Malle's conjecture literature.  Let us explain our main results.

\subsubsection{Mass formula, Hensel's Lemma, and Tamagawa measures}
In \cite[Thm.~1.1]{Bha07} and \cite[Prop.~5.3]{Ked07}, Bhargava and Kedlaya prove formulae for weighted counts of local fields, which they call \emph{mass formulae}. One of our first results concerns a generalisation of this to finite \'etale group schemes. To explain this we first introduce heights. We simplify slightly the exposition in the introduction, so some of the notation and set up appearing later in the paper is slightly different and more general (see \S\ref{sec:heights} for the theory of heights).

Let $G$ be a finite \'etale tame group scheme over a global field $k$. We say that a place $v$ is \textit{good} with respect to $G$ if $v$ is non-archimedean, $G$ has good reduction modulo $v$, and $q_v$ is coprime to $|G|$. Let $\varphi_v \in Z^1(k_v,G)$ be a $1$-cocycle over $k_v$ with values in $G$. The tame inertia group at $v$ is canonically isomorphic to a group of roots of unity. Thus restricting $\varphi_v$ to the tame inertia we obtain an element of the Tate twist $G(-1):=\Hom(\widehat{\Z}(1), G)$ of $G$ by minus $-1$. We consider this up to the conjugacy action of $G$, and this gives rise to a map
$$\rho_{G,v}: Z^1(k_v,G) \to \mathcal{C}_G, \quad \mathcal{C}_G := G(-1)/\text{conj}$$
which we call the \emph{ramification type} (see \S \ref{sec:ramification_type} for more details). Heights are then defined as follows: Let $w: \mathcal{C}_G \to \Z$ be a Galois equivariant function with $w(e)= 0$ (we call such a function a \emph{weight function}). For good $v$ we define the associated local height function to be 
$$H_v: Z^1(k_v,G) \to \Z, \quad \varphi_v \mapsto q_v^{w(\rho_{G,v}(\varphi_v))}.$$
For bad $v$ we allow a local height to be an arbitrary function. A (global) height is then a product of local heights. Our mass formula is now as follows.

\begin{theorem}[Mass formula] \label{thm:mass_formula_intro}
	Let $v$ be a good place of $G$ and $w$ a weight function with height function
	$H$. 
	Let $f: \mathcal{C}_G \to \C$ be any Galois equivariant function.
	$$
	\frac{1}{|G|}\sum_{\varphi_v \in Z^1(k_v,G)} \frac{f(\rho_{G,v}(\varphi_v))}{ H_v(\varphi_v)}
	= \sum_{c \in \mathcal{C}_G^{\Gamma_{k_v}}} \frac{f(c)}{q_v^{w(c)}}.$$
\end{theorem}

We  interpret this geometrically as a version for $BG$ of Denef's formula \cite[Thm.~3.1]{Den87} for Igusa zeta functions of varieties; see Theorem \ref{thm:Igusa} for details, as well as an equivalent formulation in terms of groupoid cardinalities.

Our proof of Theorem \ref{thm:mass_formula_intro} is stack theoretic and completely different to that of Bhargava and Kedlaya. It uses a new version of Hensel's Lemma for stacks. The traditional version of Hensel's Lemma says that $\mathcal{X}(\O_v) \to \mathcal{X}(\F_v)$ is surjective for a smooth stack $\mathcal{X}$. Our version is very different and gives a description of $BG(k_v)$. It uses the \emph{cyclotomic inertia stack} of Abramovich--Graber--Vistoli \cite[\S3]{AGV08}, which is defined to be $I_{\mu} \mathcal{X} := \coprod_{n} \Hom_{S, \text{rep}}(B\mu_n, \mathcal{X})$. They introduced this stack in the study of Gromov--Witten theory of stacks; our application is completely different and arithmetic in nature. For a uniformising parameter $\pi_v$, we introduce a reduction modulo $\pi_v$ map which can be seen as a categorical avatar of the ramification type, via the identification $I_{\mu} BG \cong [G(-1)/G]$ (Proposition \ref{prop:cyclotomic_inertia_stack_BG}) .

\begin{theorem}[Stacky Hensel's Lemma] \label{thm:Hensel_intro}
	Let $v$ be a good place of $G$. Then the reduction modulo $\pi_v$ map
	$BG(k_v) \to I_{\mu} BG(\F_v)$ is an equivalence of groupoids.
\end{theorem}

In \S \ref{sec:local_Tamagawa} we define a local Tamagawa measure $\tau_{H,v}$ on $BG(k_v)$ modelled on Peyre's definition for Fano varieties \cite[\S 2.2.1]{Pey95}. To obtain a global Tamagawa measure we define convergence factors $\lambda_v$ coming from the local Artin $L$-function factors of the collection $\mathcal{M}(w) =  \{ c \in \mathcal{C}_G : w(c) \neq 0 \text{ is minimal} \}$ of minimal weight conjugacy classes of $G(-1)$. Our mass formula is used to prove that these indeed form convergence factors.

\begin{theorem}[Convergence of global Tamagawa measure] \label{thm:Tamagawa_products_intro}
	The infinite product measure $\prod_v \lambda_v^{-1} \tau_{H,v}$ 
	converges absolutely on $\prod_v BG(k_v)$.
\end{theorem}

With a measure in place, we are in a good position to put forward Conjecture~\ref{conj:balanced} on the leading constant in Malle's conjecture, as well as formalise the Malle--Bhargava heuristics through an equidistribution conjecture on the adelic points of $BG$. This is Conjecture \ref{conj:equi}, as well as the more general problem of imposing infinitely many local conditions, which leads to a property we call \textit{strong equidistribution} (Conjecture \ref{conj:equi_strong}). Crucially however in general one should restrict the measure to those adelic points orthogonal to a suitable Brauer group; this extra factor does not appear in the Malle--Bhargava heuristics in the literature. It is necessary to take into account Grunwald--Wang-type phenomenon in the leading constant. We discuss the corresponding Brauer group in the next section.

\begin{remark}
We give a definition of \emph{balanced height functions}, which is modelled on the notion of fair counting function introduced by Wood for abelian $G$ \cite[\S 2.1]{Woo10}  (Definition~\ref{def:balanced_height}). Balancedness  asks that the minimal weight conjugacy classes generate $G$, whereas Wood asks that they generate the $r$-torsion of $G$ for all $r$.  Wood's notion of fairness has strong implications for the corresponding Brauer groups which arise, so her choice can actually be explained through the Brauer group (see \S\ref{sec:fair_wood} for details). This framework allows us to answer Wood's \cite[\S 10]{Woo16} and Bhargava's \cite[\S8.2]{Bha10} question on when the Malle--Bhargava heuristics should hold: according to Conjecture \ref{conj:equi} and \S\ref{sec:equi_unbalanced}, this should be exactly when counting by a balanced height function and when there is no Brauer--Manin obstruction (providing one removes a possible accumulating collection of fields, as in Conjecture \ref{conj:intro}).
\end{remark}

\subsubsection{Brauer groups}
The Brauer group plays a key role in the leading constant in Manin's conjecture for a Fano variety $X$, where it is essential to the definition that $\Br X /\Br k$ is finite. However the Brauer group $\Br BG$ is huge in general, even in simple cases (e.g.~one has $\Br B(\Z/2\Z)/ \Br k = k^{\times }/k^{\times 2})$. A key observation in our paper is that $\Br BG$ is not the correct Brauer group for the leading constant. When the height is the radical discriminant, one should take the \textit{unramified Brauer group} $\Brun BG$ of $BG$. We define this group in \S\ref{sec:Br} for more general algebraic stacks, since we expect the theory to be of independent interest. If $X$ is a smooth proper variety then $\Br X = \Brun X$, however crucially  $\Br BG \neq \Brun BG$ in general despite $BG$ being smooth and proper for $G$ finite \'etale.

Our main result here is a version of Grothendieck's purity theorem  \cite[Thm.~3.7.1]{Col21} for the unramified Brauer group. In the version for varieties $X$ this is stated in terms of divisors on $X$. However $BG$ has no non-zero divisors. Instead our result is stated in terms of \textit{sectors}; these are the connected components of the cyclotomic inertia stack (see \S \ref{def:sector}). They are not divisors in any sense, but a key aspect of our theory is that they play an analogous role to divisors on stacks.

\begin{theorem}[Purity for the unramified Brauer group]	\label{thm:Br_unramified_intro}
Let $\mathcal{X}$ be a smooth proper tame DM stack over a field $k$ and $b \in \Br \mathcal{X}$. Then $b \in \Brun \mathcal{X}$ if and only if for every sector $\mathcal{S} \in \pi_0(I_{\mu} \mathcal{X})$ with universal map $f_{\mathcal{S}}: (B \mu_n)_{\mathcal{S}} \to \mathcal{X}$, we have $f_{\mathcal{S}}^*(b) \in \Br \mathcal{S}$.
\end{theorem} 

In the special case $\mathcal{X} = BG$ and $k = \C$, we use Theorem \ref{thm:Br_unramified_intro} to recover a famous formula of Bogomologov \cite{Bog87} for the unramified Brauer group of $\mathbb{A}^n/G$ (see Remark \ref{rem:Bogomolov}). In particular, Theorem \ref{thm:Br_unramified_intro} can be viewed as vast generalisation of Bogomolov's formula to more general stacks and non-algebraically closed fields.

We use the unramified Brauer group to study the Brauer--Manin obstruction on $BG$. The following theorem is a version for stacks of a famous theorem of Harari \cite[Thm.~2.1.1]{Har94}, which is part of Harari's formal lemma for varieties \cite[\S 13.4]{Col21}. It shows that our definition of the unramified Brauer group has the correct formal properties from an arithmetic perspective.

\begin{theorem}[Harari's formal lemma] \label{thm:Harari's_formal_lemma_intro}
	Let $\mathcal{X}$ be a smooth finite type DM stack over a number field $k$
	and $b \in \Br \mathcal{X}$. Then $b \in \Brun \mathcal{X}$ 
	if and only if $b$ evaluates trivially
	on 	$\mathcal{X}(k_v)$ for all but finitely many places $v$ of $k$.
\end{theorem}

The unramified Brauer group is the correct object when counting via the radical discriminant. However for general heights it turns out that a different Brauer group is required; this change of Brauer group depending on the choice of height function explains many of the differences observed in the literature regarding counting with different height functions. We call this the \textit{partially unramified Brauer group}. In the case of $BG$ it is defined to be only those elements of $\Br BG$ which are unramified along the sectors determined by a given Galois invariant collection $\mathcal{C} \subset \mathcal{C}_G$ of conjugacy classes, rather than all sectors as in case of the unramified Brauer group.

Calculating this group $\Br_{\mathcal C} BG$ is crucial for calculating the leading constant in our conjecture. Firstly we prove that it is finite modulo $\Br k$ in the balanced case, as required for the leading constant to even be well-defined; given that $\Br BG /\Br k$ is infinite in general, this finiteness is non-obvious. We also obtain an algorithm to calculate it as well as write down elements. We state this informally as a theorem here; see \S \ref{sec:procedure} for a precise description.

\begin{theorem}[Finiteness and computability of the Brauer group] \label{thm:finiteness_partially_unramified_Brauer_group}
	Let $G$ be a finite \'etale tame group scheme scheme over a field $k$ and
	$\mathcal{C} \subset \mathcal{C}_G$ be Galois invariant which generates $G$.
	Then the group $\Br_{\mathcal C} BG/\Br k$ is finite and there exists an effective method
	for calculating it.
\end{theorem}

We achieve Theorem \ref{thm:finiteness_partially_unramified_Brauer_group} through a systematic study of $\Br_{\mathcal C} BG$, and we compute it in terms of central extensions and Kummer theory (Theorem \ref{thm:orbifold_Kummer}), as well as Galois cohomology (Theorem \ref{thm:Br_BG}). This latter theorem in particular describes the algebraic part of the Brauer group in terms of the Galois cohomology of what we call the \emph{orbifold Picard group} of $BG$; it can be viewed as an enhancement of the well-known isomorphism $\Br_1 \mathcal{X}/\Br k \cong \H^1(k, \Pic \mathcal{X}_{\bar k})$ \cite[Prop.~5.4.2]{Col21} for any algebraic stack $\mathcal{X}$ over $k$ with a rational point. We also obtain a version of Theorem~\ref{thm:Harari's_formal_lemma_intro} for the partially unramified Brauer group; this is Theorem \ref{thm:Harari_formal_partially_unramified} and is crucial to know that the Brauer--Manin pairing is well-defined on what we call the associated \emph{partial adelic space}.

Our description of the transcendental Brauer group of $BG$ involves $\H^2(G,\C)$, which is essentially the Schur multiplier of $G$. Ellenberg and Venkatesh \cite[\S2.4]{EV10} first postulated the appearance of the Schur multiplier in the leading constant in Malle's conjecture, so we explain this via transcendental Brauer groups (Remark \ref{rem:Schur_multiplier}). We also relate the Brauer group to the ``algebraic lifting invariants'' of Ellenberg--Venkatesh--Westerland \cite[\S7.4]{EVW13} and Wood \cite{Woo21}. These lifting invariants are closely related to the components of Hurwitz spaces and feature prominently in the study of Malle's conjecture and the Cohen--Lenstra heuristics over function fields. They fit into our persepective as they can be interpreted via the Brauer group of $BG$.

\begin{remark}
Despite our investigation being based on Manin's conjecture, it also allows us to feedback into Manin's conjecture for Fano varieties. Namely the leading constant in Manin's conjecture for a Fano variety $X$ involves the factor $\#\H^1(k, \Pic \bar{X})$, which equals $|\Br_1 X / \Br k|$ where $\Br_1 X$ denotes the algebraic Brauer group of $X$. However we give evidence in \S\ref{sec:A_4-transcendental} that for Malle's conjecture, it should be the cardinality of the \textit{full} Brauer group which appears, i.e.~one should also include the transcendental Brauer group. This is exactly the counting problem in Conjecture~\ref{conj:A_4_conductor_intro}. This suggests that in the leading constant in Manin's conjecture $\#\H^1(k, \Pic \bar{X})$ should actually be replaced by $|\Br X / \Br k|$.
\end{remark}

\subsubsection{The exceptional set and the total count}
Kl\"{u}ners \cite{Klu05} was the first to come up with a counter-example to Malle's conjecture; this has the property that the exponent of $\log B$ is too large. Koymans and Pagano \cite{KP23} have given counter-examples to Ellenberg and Venkatesh's version  \cite[Ques.~4.3]{EV05} of Malle's conjecture for the radical discriminant; we give new counter-examples ourselves involving dihedral extensions in \S \ref{sec:D_n}. These counter-examples  come from non-trivial interactions with cyclotomic fields, and are formally similar to the requirement in Manin's conjecture to remove an accumulating thin set to obtain the correct asymptotic formula (e.g.~lines in cubic surfaces).

We formalise all these counter-examples in \S \ref{sec:breaking}, including for finite \'etale group schemes $G$, using what we \emph{breaking cocycles}. All known counter-examples come from breaking cocycles, and moreover we show in Lemma \ref{lem:breaking_b} that in order to obtain a Malle-type conjecture for $BG$ which is consistent with respect to change of $G$, it suffices to remove the breaking cocycles. Our main result on these cocycles is that they form a thin set and must come from cyclotomic extensions (see Theorem~\ref{thm:breaking} for a more general version for finite \'etale group schemes).

\begin{theorem}[Breaking cocycles are thin] \label{thm:breaking_intro}
	Let $G$ be a finite group of order coprime to the characteristic of $k$.
	Then the collection of breaking homomorphisms
	$\Gamma_k \to G$ is thin and lies in the set
	$$\left\{ \varphi: \Gamma_k \to G :
	\begin{array}{ll}
	 	\varphi \text{ is not surjective or } \\
	   k_{\varphi} \text{ is not linearly disjoint to } k(\mu_{\exp(G)})
	   \end{array} \right\}.$$
\end{theorem}

The thinness of the breaking rational points in Manin's conjecture \cite[Thm.~1.4]{LST22} is a difficult result proved used deep techniques in birational geometry. Our proof is substantially easier, and has the advantage of giving an explicit computable description of the breaking thin set, something which is currently lacking in the setting of Manin's conjecture.

Our conjectures state that after removing the breaking cocycles, one should obtain the correct power of $\log B$ as predicted by Malle (see Remark \ref{rem:breaking}). More than this however, one should remove these to obtain our predicted leading constant and a version of the Malle--Bhargava heuristics (equidistribution in our terminology).

Despite needing to remove a thin set to get the correct leading constant, our general conjectures imply a formula for the total count in Malle's conjecture by strafying the collection of all homomorphisms according to whether they lift suitable cyclotomic fields. See \S \ref{sec:total_count} for details.

\subsection{Structure of the paper}
We have tried to make the paper self-contained where possible for the benefit of readers unfamiliar with the theory of stacks. For example, we begin the paper by recalling all the properties of $BG$ we will need. The exceptions are \S \ref{sec:Hensel} and \S \ref{sec:Br} which require more background in the theory of algebraic stacks. Throughout we include numerous examples to illustrate our results and conjectures, so as to create a new toolkit for other researchers.

In \S\ref{sec:BG} we introduce the key object of study in the paper, namely the stack $BG$. We give various presentations for this stack and discuss its basic properties, including the roles of inner twists and groupoid cardinalities.

In \S \ref{sec:orbifold_Picard} we consider the \emph{orbifold effective cone}. This was introduced in \cite{DYBM} for general algebraic stacks, and replaces the role of the effective cone of divisors in  Manin's conjecture. We specialise the theory to the case of interest, namely $BG$. Here we develop the theory further, and obtain a new definition for the \emph{orbifold Picard group}, which is crucial when we come to defining Peyre's effective cone constant in our setting. We introduce the associated Iitaka fibration and also define breaking cocycles: the exceptional set  which must be removed in Malle's conjecture to obtain the correct asymptotic formula. 

The next three sections form the technical heart of the paper. In \S \ref{sec:Hensel} we introduce some of the key technical tools which we will use in our stacky framework, namely roots stacks, the cyclotomic inertia stack, and sectors. These are used to prove a new stacky version of Hensel's lemma (Theorem \ref{thm:Hensel}). There already exist versions of Hensel's lemma in the literature, which for a smooth stack $\mathcal{X}$ over a complete DVR $\mathcal{O}$ says that the map $\mathcal{X}(\O) \to \mathcal{X}(\F)$ to the residue field $\F$ is surjective, similarly to schemes (see e.g.~\cite[Lem.~4.2]{LS23}). Our new version is completely different and has no analogue in the world of schemes. It gives a description of $\mathcal{X}(K)$ where $K$ is the fraction field of $\O$, and is based upon the modified valuative criterion for properness of stacks from \cite{BV24}.

In \S \ref{sec:Br} we create a new theory of Brauer groups for algebraic stacks, including a new definition of the unramified Brauer group of an algebraic stack and partially unramified Brauer group, which will appear in our conjectures. Our main result (Theorem \ref{thm:equivalent_conditions_unramified}) gives a stacky version of Grothendieck's purity theorem in terms of sectors. We also define the Brauer--Manin obstruction for a stack and prove a version of Harari's formal lemma \cite[Thm.~2.1.1]{Har94} for stacks.

In \S \ref{sec:Br_BG} we specialise the preceding theory to the case of $BG$, and compute the transcendental and algebraic part of the Brauer group of $BG$, via central extensions (Theorem \ref{thm:orbifold_Kummer}) and Galois cohomology (Theorem \ref{thm:Br_BG}). This leads to a computable procedure, explained in \S \ref{sec:procedure}, to calculate the partially unramified Brauer group of $BG$. This group appears in our main conjectures, hence it is crucial for applications that one calculate it.

In \S \ref{sec:BMO_BG} we study the Brauer--Manin obstruction for $BG$. We include the important definition of the partial adelic space (Definition \ref{def:adelic_space}), which is the natural space upon which the partially unramified Brauer--Manin pairing is defined (Theorem \ref{thm:Harari_formal_partially_unramified}). We also give explicit examples of the Brauer--Manin obstruction in this case, through Stickelberger's Theorem  and the Grunwald--Wang Theorem, as well as a gerbe which fails the Hasse principle, and a discussion of differences over global function fields.

In \S \ref{sec:heights_Tamagawa} we introduce heights on $BG$ following \cite{DYBM}. We then proceed to construct a new Tamagawa measure on $BG$. We prove a mass formula for our local measures and use these to obtain convergence factors for a global measure. We also give a more advanced mass formula to  calculate the Tamagawa measure of the Brauer--Manin set (Theorem \ref{thm:local_invariant_integral}).

In \S \ref{sec:conjectures} we state our conjectures. These concern general height functions on $BG$ for a finite \'etale group scheme $G$, and not just the case of the discriminant considered in Conjecture \ref{conj:intro}. We give a conjecture regarding equidistribution, which formalises the Malle--Bhargava heuristics, and also a conjecture we call \textit{strong equidistribution}, which covers for example the problem of counting number fields of squarefree discriminant.

We finish in \S \ref{sec:examples} by studying a range of examples, including providing evidence for our conjectures from existing results in the literature. This includes $S_n$-extensions, $D_4$-extensions, and abelian extensions. We also show how Kl\"{u}ners's counter-example is compatible with our conjecture, and give a new counter-example to Ellenberg and Venkatesh's generalisation of Malle's conjecture for the radical discriminant. We  demonstrate what our conjecture says for counting $A_4$-quartic extensions, and consider its compatibility with an existing conjecture in the literature \cite[\S2.7]{CDO02b}.

\subsection{Shortcut to the Conjecture:} The reader looking for a quick overview to understand the conjecture in the balanced case should study \S \ref{sec:groupoid} and \S \ref{sec:BS_n} to get to grips with the definition of $BG$. One should then look at \S \ref{sec:Galois_action_conjugacy_classes} to understand the Galois action on conjugacy classes and \S\ref{sec:orbifold_effective_cone} for the definition of the Fujita invariant and the minimal weight conjugacy classes.

The (partially) unramified Brauer group is the most subtle notion in our conjecture; indeed defining and calculating this group is one of the main new technical innovations in the paper. It should arguably be skipped on a first reading. If one wants to know the precise definition one should read Definitions \ref{def:sector} and \ref{def:partially_ramified_Br}. Tools for calculating it in the case of $BG$ are contained in \S\ref{sec:Br_BG}.

To understand the definition of the heights one should read Definition \ref{def:orbifold_line_bundle},  \S\ref{sec:ramification_type} and \S \ref{sec:heights}. The corresponding local Tamagawa measures are defined in \S\ref{sec:local_Tamagawa} and computed at all but finitely many places in Corollary \ref{cor:mass_formula}. This computation is used to show that the global Tamagawa measure defined in \S\ref{sec:global_Tamagawa} is well-defined.

One can now understand the conjectures in \S\ref{sec:conjectures} regarding balanced heights. A useful statement is Lemma \ref{lem:sum_Euler_products}, which shows that the leading constant in the conjecture is equal to an explicit finite sum of Euler products.

After this one should look at the examples in \S\ref{sec:examples} to get a feeling for what the conjecture looks like in concrete situations, and to see some explicit Brauer group computations.

\subsection{Notation and conventions} \label{sec:notation}
For a finite \'etale group scheme $G$ over a field $k$, we denote by $\dual{G} = \Hom(G,\Gm)$ its group scheme of $1$-dimensional characters, by $G(-1)$ its Tate twist by $-1$, and by $|G|$ its degree. We call $G$ \emph{tame} if $|G|$ is not divisible by the characteristic of $k$. We denote by $Z^1(k,G)$ the set of $1$-cocycles with values in $G(k^{\sep})$. We recall that the \emph{normaliser} of a subgroup scheme $H \subset G$ is the subgroup scheme given by the collection of $g \in G$ such that $g H g^{-1} = H$.

All algebraic stacks are assumed to be quasi-separated.
Following \cite[Thm.~3.2]{AOV08}, we call an algebraic stack \emph{tame} if all its stabilisers are geometrically reductive. A  finite \'etale group scheme $G$ is tame if and only if $BG$ is tame. (This can fail for finite flat group schemes, e.g.~$\mu_p$ is not tame in characteristic $p$, but $B\mu_p$ is).

We call a groupoid $X$ \emph{finite} if the automorphism group of each object is finite and there are only finitely many isomorphism classes of objects. We denote by $[X]$ the set of isomorphism classes of classes of $X$, and by $\#X = \sum_{x \in [X]} \frac{1}{|\Aut x|}$ the groupoid cardinality of $X$ (provided $X$ is finite). Similarly, for an algebraic stack $\mathcal{X}$ over a ring $R$, we denote by $\mathcal{X}(R)$ the groupoid of $R$-points of $\mathcal{X}$ and by $\mathcal{X}[R]$ the set of isomorphism classes of objects in $\mathcal{X}(R)$.

In the Malle's conjecture literature, functions which assign a value to a number field go by various names, e.g.~``generalised discriminant'', ``counting function'', or ``invariant''. We use the term \emph{height function} to unify existing terminology regarding heights on varieties, and to avoid confusion with the local invariant $\inv_v: \Br k_v \to \Q/\Z$ from class field theory. A counting function for us is the function which gives the cardinality of the number of elements of bounded height.

To avoid possible confusion, the notation $c(k,G,H)$ is always used to refer to the leading constant for $BG$ over $k$ when counting $G$-torsors of bounded height $H$. This has the consequence that various counting functions in the paper may have different normalisations, given by multiplying by a suitable rational number, to make sure that the exact leading constant $c(k,G,H)$ is obtained. Whilst this may look awkward, it is the correct way to normalise counts via the groupoid cardinality and leads to more uniform formulae.

All cohomology is fppf cohomology.

\subsection*{Acknowledgements}
This project started after Arul Shankar asked us if we could explain the factor $1/2$ which appears
in Bhargava's heuristic~\cite[Conj.~1.2]{Bha07}. We thank Brandon Alberts, Tim Browning, Peter Koymans, Gunter Malle, Ross Paterson,
and Julie Tavernier for useful comments.
Daniel Loughran
was supported by UKRI Future Leaders Fellowship
\texttt{MR/V021362/1}. Tim Santens was supported by FWO-Vlaanderen (Research Foundation-Flanders) via grant number 11I0621N.

\section{The stack $BG$} \label{sec:BG}
\subsection{The groupoid $BG(k)$} \label{sec:groupoid}
Let $G$ be an \'etale group scheme over a field $k$. The groupoid $BG(k)$ has multiple equivalent presentations. We follow the conventions for torsors as in \cite[\S 2]{Sko01}.

\begin{lemma}  \label{lem:BG(k)}
The following categories are equivalent.
\begin{enumerate}
	\item The groupoid $BG(k)$ of right $G$-torsors over $\Spec k$ with isomorphisms as $G$-equivariant morphisms.
	
	\item The groupoid of left $G$-\'etale algebras over $k$, i.e.~\'etale algebras $A/k$ equipped with a faithful left action of $G$ such that $A^G = k$.
		
	\item The groupoid which has as objects $1$-cocycles $\varphi: \Gamma_k \to G(k^{\sep})$ and morphisms $\varphi \xrightarrow{g} g \varphi g^{-1}$ for each cocycle $\varphi$ and $g \in G(k^{\sep})$ with the obvious composition law. The cocyle $g \varphi g^{-1}$ sends $\sigma \in \Gamma_k$ to $g \varphi(\sigma) \sigma(g)^{-1}$.
	
	\item [(4)] If $G$ is \emph{constant}, then the groupoid of (continuous) homomorphisms $\Gamma_k \to G$,
	with morphisms given by conjugation in $G$.
\end{enumerate}
\end{lemma}
\begin{proof}

	(1) $\iff$ (2): If $A$ is such an \'etale algebra then $\Spec A \to \Spec k$ is a right $G$-torsor. Conversely if $T \to \Spec k$ is a right $G$-torsor then $\H^0(T, \mathcal{O}_T)$ is a left $G$-\'etale algebra.

	(1) $\iff$ (3):	The equivalence between this groupoid and $BG(k)$ follows from the fact that a cocycle is the same as a descent datum and the fact that descent for $G$-torsors is effective \cite[\href{https://stacks.math.columbia.edu/tag/0245}{Tag 0245}]{stacks-project}.
	
	(3) $\iff$ (4): As $G$ is constant, a $1$-cocyle is the same as a homomorphism.
\end{proof}

In particular the set $BG[k]$ of isomorphism classes of elements of $BG(k)$ is simply $\H^1(k,G)$; but it is crucial throughout to also keep track of the automorphism groups of the elements, which is exactly what $BG$ does. Our standard choice of presentation will usually be in terms of cocycles. We denote by $e \in BG(k)$ the rational point corresponding to the identity cocycle, i.e.~the cocycle with constant value equal to the identity element of $G$.

The category $BG(k)$ contains many elements, some of which are not relevant to Malle's conjecture (e.g.~the identity cocycle). We classify general elements of $BG(k)$ as follows.

\begin{lemma} \label{lem:connected}
	The equivalences in Lemma \ref{lem:BG(k)} induce equivalences between
	the following subgroupoids.
	\begin{enumerate}
	\item The groupoid of connected right $G$-torsors over $\Spec k$.
	
	\item The groupoid of left $G$-fields over $k$.
		
	\item The groupoid of surjective $1$-cocycles $\varphi: \Gamma_k \to G(k^{\sep})$.
	
	\item [(4)] If $G$ is \emph{constant}, then the groupoid of 
	surjective homomorphisms $\Gamma_k \to G$.
	\end{enumerate}
\end{lemma}
\begin{proof}

	(1) $\iff$ (2): This is clear from the definition.
	
	(1) $\iff$ (3): Let $\varphi: \Gamma_k \to G$ be a cocycle and $T$ the corresponding right $G$-torsor. The $\Gamma_k$-set $T(k^{\sep})$ can be identified with $G(k^{\sep})$ equipped with the twisted action $\sigma \cdot^{\varphi} g := \varphi(\sigma) \cdot \sigma(g)$. The scheme $T$ is connected if and only if this $\Gamma_k$-action is transitive. However the orbit of the identity $e \in G(k)$ under this action is equal to the image of $\varphi$, so the orbit is transitive if and only if the image of $\varphi$ is $G(k^{\sep})$.
	
	(3) $\iff$ (4): Clear.
\end{proof}

We call any element of $BG(k)$ satisfying one of the equivalent conditions in Lemma \ref{lem:connected} \textit{surjective}.

\begin{remark}[Outer automorphisms] \label{rem:outer}
	Assume that $G$ is constant. 
	Let $\psi: G \to G$ be an outer automorphism and $\varphi: \Gamma_k \to G$ a surjective
	homomorphism. Then $\varphi$ and $\psi \circ \varphi$ are non-isomorphic in $BG$,
	despite having the same kernel and so defining isomorphic Galois extensions with Galois
	group $G$.
	
	This convention is important for the theory for non-Galois extensions;
	for example if $G = S_6$ then $\varphi$
	and $\psi \circ \varphi$ determine non-isomorphic degree $6$ extensions of $k$ 
	(see Lemma \ref{lem:S_n} below).
	
	However outer automorphisms do occur in a somewhat subtle way in the theory.
	They arise in the following situation: given a subgroup scheme
	$H \subseteq G$ with normaliser $N \subseteq G$, any element of $N$ induces an automorphism
	of $H$ by conjugation, which may be an outer automorphism of $H$. Galois twists
	of $H$ by such automorphisms	are key to our work (Lemma \ref{lem:fibration_normal_quotient}),
	and moreover keeping track of such outer automorphisms is important to get the 
	correct groupoid cardinalities (see Lemma \ref{lem:outer_automorphism}). This ultimately
	leads to additional factors in Conjecture~\ref{conj:discriminant}.
\end{remark}

The natural way to count in a groupoid is via the groupoid cardinality, i.e.~the number of isomorphism classes of objects in the groupoid weighted by the inverse of the cardinality of the automorphism group of each object (see \S \ref{sec:notation} for our conventions regarding sums over groupoids). To get to more elementary looking statements (as in Conjecture~\ref{conj:intro}), we use the following lemma. We denote by $Z^1(k,G)$ the set of $1$-cocycles with values in $G$.

\begin{lemma} \label{lem:groupoid_count}
	Let $F:Z^1(k,G) \to BG[k]$ be the natural map and $f:BG[k] \to \C$ any function. Then for any
	finite subset $W \subseteq BG[k]$ cwe have
	$$\sum_{\varphi \in W} \frac{f(\varphi)}{|\Aut \varphi|} = \frac{1}{|G|}\sum_{ \psi \in F^{-1}(W)} f(F(\psi)).$$
\end{lemma}
\begin{proof}
	We may assume $|W| =1$. Here it is 
	the 	orbit-stabiliser theorem.
\end{proof}

\subsection{Thin sets} \label{sec:thin}
Thin sets were introduced by Serre \cite[\S 3.1]{Ser08} in his study of Hilbert's irreducibility theorem. Peyre \cite[\S 8]{Pey03} was the first to suggest that one should remove a thin set in Manin's conjecture to avoid counter-examples. This requirement is now part of all modern formulations of Manin's conjecture. We will take a similar approach and remove a thin set in Malle's conjecture to avoid counter-examples. 
Our definition of thin sets on $BG$ is as follows.

\begin{definition}
We say that a subset of $BG[k]$ is \emph{thin} if it is contained in the finite union of images $f(X(k))$ where $f:X \to BG$ is a finite morphism which admits no generic section. 
\end{definition}

We classify these in terms of inner twists (see \S \ref{sec:inner_twists} for terminology).

\begin{lemma} \label{lem:thin_classification}
	Any thin subset of $BG[k]$ is contained in the finite union of the images $f(BH(k))$ where
	$H$ proper subgroup scheme of an inner twist of $G$ and $f:BH \to BG$ 
	denotes the induced map. Conversely such subsets are thin.
\end{lemma}
\begin{proof}
	Let $f:X \to BG$ be a finite morphism which admits no generic section. By considering each
	component of $X$ separately, we may assume that $X$ is connected. We may also assume that $X$
	is reduced, hence $f$ is \'etale. Moreover we can assume that $X(k) \neq \emptyset$.
	The first part now follows from Lemma \ref{lem:subgroup_inner_twists}. The converse is 
	immediate from the definition.
\end{proof}

In the notation of Lemma \ref{lem:thin_classification} we have $\deg f = |G|/|H|$. Some examples of thin sets are as follows.

\begin{lemma} \label{lem:thin_base_change}
	Let $K/k$ a finite extension and $H \subset G_K$ a subgroup scheme. Let $\Omega \subset BG[k]$ be the subset consisting of those cocyles $\varphi \in BG(k)$ such that $\varphi_K$ lies in the image of $BH(K) \to BG(K)$. Then $\Omega$ is thin.
	
	In particular, the collection of non-surjective elements of $BG(k)$, 
	in the sense of Lemma \ref{lem:connected}, is thin.
\end{lemma}
\begin{proof}
	Consider the Cartesian square
		\[
\xymatrix{ X  \ar[r] \ar[d] & B \Res_{K/k} H  \ar[d] \\ 
BG \ar[r] & {B \Res_{K/k} G_K},}
\]
	where $\Res_{K/k}$ denotes the Weil restriction.
	A diagram chase shows that cocycles $\varphi$ with $\varphi_K \in \im(BH(K) \to BG(K))$ are exactly
	those in $\im(X(k) \to BG(k))$.
	 We have $\Res_{K/k} G \times_k k^{\sep} \cong G^n$, where $n := [K: k]$. The map $X_{k^{\sep}} = BG_{k^{\sep}} \times_{B G_{k^{\sep}}^n} BH_{k^{\sep}}^n \to BG_{k^{\sep}}$ is a disjoint union of maps of the form $BH_{k^{\sep}} \to BG_{k^{\sep}}$, each of which is finite and of degree $\geq 2$, so the image of $X$ is thin as desired.
	 
	 For the last part, let $K/k$ be a finite field extension over which $G$ becomes constant.
	 Then any non-surjective cocycle $\varphi \in BG(k)$ becomes a non-surjective homomorphism
	 over $K$. In particular its image is a proper subgroup of $G(K)$. As there are only finitely many
	 subgroups, the result follows from the first part.
\end{proof}


\subsection{Inner twists} \label{sec:inner_twists}
We now study inner twists and the automorphism groups of elements of $BG(k)$. Some of the results in this section are rephrasing in stacky terms properties of non-abelian Galois cohomology, as can be found in \cite[\S5]{Ser02}.

\begin{definition} \label{def:inner_twist}
 Given a cocycle $\varphi \in Z^1(k,G)$ we let $G_{\varphi}$ be the corresponding inner twist, i.e.~the \'etale group scheme whose underlying group is $G(k^{\sep})$ and on which $\Gamma_k$ acts via 
\begin{equation} \label{eqn:inner_twist}
(\sigma, g) \in \Gamma_k \times G(k^{\sep}) \to \varphi(\sigma) \sigma(g) \varphi(\sigma)^{-1}.
\end{equation}
From a cohomological perspective, this is the twist of $G$ corresponding to the cohomology class $\varphi \in \H^1(k,G)$ where $G$ acts on itself by conjugation (see \cite[Examples, \S5.3]{Ser02} for more details).
\end{definition}

In the statement $Z(G)$ denotes the centre of $G$, viewed as a group scheme, and $\AutU(\varphi)$ denotes the automorphism group scheme of $\varphi$.

\begin{lemma}\label{lem:automorphism_group_is_inner_twist}
	Let $\varphi \in BG(k)$ be a cocycle. 
	\begin{enumerate}
		\item We have $\AutU(\varphi) \cong G_{\varphi}$.
		\item If $G$ is constant, then $\AutU(\varphi)(k)$ is isomorphic to the centraliser
		of the image of $\varphi$. In particular, if $\varphi$ is surjective then 
		$\AutU(\varphi)(k) \cong Z(G)$.
		\item In general, there exists a thin subset
		$\Omega \subset BG[k]$ such that 
		for all 	$\varphi \in BG[k] \setminus \Omega$ we have 
		$\AutU(\varphi)(k) \cong Z(G)(k)$.
	\end{enumerate}
\end{lemma}
\begin{proof}
	Let $K/k$ be a field extension and let
	$g: \varphi \to \varphi$ be an automorphism defined over $K$. Then $g \varphi(\sigma) \sigma(g)^{-1} =  \varphi(\sigma)$ for all $\sigma \in \Gamma_K \subset \Gamma_k$. We can rewrite this equality as $g = \varphi(\sigma) \sigma(g) \varphi(\sigma)^{-1}$, which from \eqref{eqn:inner_twist} exactly means that $g \in G_{\varphi}(K)$.
	As $G$ and $\underline{\Aut}(\varphi)$ are both finite \'etale over $k$,
	this proves (1).
	
	For (2), we use the action from 
	\eqref{eqn:inner_twist}. Let $g \in G_\varphi(k^{\sep})$.
	Then $g$ is defined over $k$ if and only
	if
	$$g = \varphi(\sigma) \sigma(g) \varphi(\sigma)^{-1} \quad \text{ for all }
	\sigma \in \Gamma_k.$$
	As $\sigma(g) = g$, this is equivalent to $g$ lying in the centraliser of the image
	of $\varphi$.
	
	For (3), the inclusion $Z(G)(k) \subset G_\varphi(k)$ is clear since the centre is preserved
	under inner twists. So assume that $Z(G)(k) \neq G_\varphi(k)$.
	Take $K/k$ to be a splitting field of $G$. Then we have $Z(G_K) \neq G_{K,\varphi_K}(K)$.
	By $(2)$ we deduce that $\varphi_K$ is not surjective. The collection of such $\varphi$
	is thin by Lemma \ref{lem:thin_base_change}.
\end{proof}

Note that if $G$ is abelian, then one may even take $\Omega = \emptyset$ in Lemma \ref{lem:automorphism_group_is_inner_twist}, since every inner twist is trivial.

We next observe that inner twists do not change $BG$ (compare with \cite[Prop.~35 bis]{Ser02}).

\begin{lemma} \label{lem:inner_twist_BG}
	Let $\varphi \in Z^1(k,G)$. Then $BG_\varphi \cong BG$.
\end{lemma}
\begin{proof}
	For all 	$\varphi \in BG(k)$ we have $B \AutU\varphi \cong BG$ 
	\cite[Tag 06QG]{stacks-project}.
	However $\AutU \varphi \cong G_\varphi$ 
	by Lemma \ref{lem:automorphism_group_is_inner_twist}.
\end{proof}

Inner twists of $G$ naturally arise when considering subgroups.

\begin{lemma} \label{lem:subgroup_inner_twists}
	Let $f:X \to BG$ be a finite \'etale morphism with $X$ connected 
	and $X(k) \neq \emptyset$.
	Then $X \cong BH$ for some subgroup scheme $H$ of an inner twist of $G$.
\end{lemma}
\begin{proof}
	As $X$ is connected and $X(k) \neq \emptyset$ we see that $X$ must be a neutral gerbe,
	hence $X \cong BH$ for some finite \'etale group scheme $H$ over $k$. 
	Let $e \in BH(k)$ be the identity cocycle. As $f:BH \to BG$ is finite it is representable, 
	so the map  $H = \AutU(e) \to \AutU f(e)$ is injective. However $\AutU f(e)$ is an inner twist of $G$
	by Lemma \ref{lem:automorphism_group_is_inner_twist}.
\end{proof}

Therefore it useful to understand the subgroup schemes of inner twists of $G$. Here is one way to construct examples.

\begin{definition} \label{def:inner_twist_normaliser}
	Let $H \subseteq G$ be a subgroup scheme. 
	Let $N \subseteq G$ denote the normaliser of $H$ and let $\varphi \in Z^1(k,N)$.
	Denote by $H_\varphi$ the twist of $H$ obtained from $\varphi$ where $N$ acts
	on $H$ by conjugation, i.e.~$H_\varphi$ is the \'etale group scheme with
	underlying group $H(k^{\sep})$ and on which $\Gamma_k$ acts via 
	\[
(\sigma, h) \in \Gamma_k \times H(k^{\sep}) \to \varphi(\sigma) \sigma(h) \varphi(\sigma)^{-1}.
\]
\end{definition}

Definition \ref{def:inner_twist} is a special case of Definition \ref{def:inner_twist_normaliser},
as such we call $H_\varphi$ \textit{the inner twist of $H$ by $\varphi$}. But note that $H_\varphi$ is \textit{not} an inner twist of $H$ in the sense of Definition \ref{def:inner_twist} in general as, despite the action given by conjugation, these come from $N$ hence can induce \emph{outer automorphisms} of $H$. So in general we will have  $BH \not \cong BH_{\varphi}$, despite Lemma \ref{lem:inner_twist_BG}.

We next classify fibres of quotient maps with a rational point
(compare with Corollary 2 in \cite[\S5.5]{Ser02}).

\begin{lemma} \label{lem:fibration_normal_quotient}
	Let $H \subseteq G$ be a normal subgroup scheme and $f:BG \to B(G/H)$ the induced 
	map.
	Let $\varphi \in Z^1(k,G)$ be a cocycle. Then
	$f^{-1}(f(\varphi)) \cong BH_\varphi$ where $H_\varphi$ is the inner twist
	of $H$ by $\varphi$.
\end{lemma}
\begin{proof}
	Let $\varphi'$ denote the image of $\varphi$ in the gerbe
	$f^{-1}(f(\varphi))$. Then 
	$\AutU \varphi' = \ker( \AutU \varphi \to \AutU f(\varphi))
	= (\ker G_\varphi \to G_\varphi/H_\varphi) = H_\varphi$.
	Thus by \cite[Tag 06QG]{stacks-project} we have
	$f^{-1}(f(\varphi)) \cong B \AutU \varphi' \cong B H_\varphi$.	
\end{proof}

Lemma \ref{lem:fibration_normal_quotient} applies to the following explicit problem. 
Let $\psi \in Z^1(k,G/H)$. Then the collection of $\varphi \in Z^1(k,G)$ which lift $\psi$ has the structure of a stack isomorphic to $BH_\varphi$ for some choice of $\varphi$, providing a lift exists (a lift is exactly a solution to the corresponding embedding problem). This in particular allows one to phrase the perspective from \cite{Alb21} in terms of stacks. In this situation we sometimes abuse notation and denote the corresponding stack by $BH_\psi$; this implicitly depends on the choice of lift $\varphi$ and is well-defined up to isomorphism.

It will also be useful to compare the groupoid cardinalities of $BH$ and $BG$, where $H \subseteq G$ is a subgroup scheme. The functor $BH(k) \to BG(k)$ is not full in general; this means that the induced map on isomorphism classes of elements need \emph{not} be injective in general. This has consequences for groupoid cardinalities, which we summarise in the following lemma.

\begin{lemma} \label{lem:outer_automorphism}
	Let $H \subseteq G$ be a subgroup scheme and $F: BH(k) \to BG(k)$ the induced map.
	Let $f:BG[k] \to \C$ be any function.
	Then for any finite collection of objects $W \subseteq B H(k)$ with $F^{-1}(F(W)) = W$
	 we have
	$$\sum_{\varphi \in [W]} \frac{F(f(\varphi))}{|(G_{F(\varphi)}/H_{\varphi})(k)| \cdot |\Aut \varphi|}
	 = 
	\sum_{\psi \in [F(W)]} \frac{f(\psi)}{|\Aut \psi|}.$$
	Moreover:
	\begin{enumerate}
		\item 
		If $H \subseteq G$ is normal, then 
		$G_{F(\varphi)}/H_{\varphi} = G/H$ for all $\varphi \in  BH(k)$.
		\item If $G$ is constant and $\varphi$ is surjective, then 
		$|(G_{F(\varphi)}/H_{\varphi})(k)| = |N/H|$ where $N$ denotes
		the normaliser of $H$ in $G$.
	\end{enumerate}
\end{lemma}
\begin{proof}
	Assume without loss of generality that $F(W)$ consists of a single equivalence class. The statement then reduces to the following: if $\varphi \in BH(k)$ then
	\[
	\sum_{\varphi' \in BH[k], F(\varphi') = F(\varphi)}  \frac{1}{|G_{F(\varphi)}/H_{\varphi'}(k)| \cdot |\Aut(\varphi')|} = \frac{1}{\Aut(F(\varphi))}.
	\]
	By \cite[Cor.~5.5:2]{Ser02} the sum in the left-hand side is equivalent to a sum over the orbits of $G_{F(\varphi)}(k)$ acting on $G_{F(\varphi)}/H_{\varphi}(k)$. Moreover, by \cite[Prop.~36]{Ser02}, the stabiliser of this action for the orbit corresponding to $\varphi'$ is $H_{\varphi'}(k)$. Lemma \ref{lem:inner_twist_BG} implies that $\Aut(\varphi') = H_{\varphi'}(k)$ and $\Aut(F(\varphi)) = G_{F(\varphi)}(k)$. The desired statement then follows from the orbit-stabiliser theorem.

For the last parts, note that $(G_{F(\varphi)}/H_{\varphi})$ by construction is equal to the twist of the finite \'etale scheme $G/H$ by the cocycle $\varphi$, where $H$ acts on $G/H$ by conjugation, i.e. multiplication on the left. For (1), if $H \subseteq G$ is normal then $H$ acts trivially on $G/H$ by multiplication on the left. So the twist of $G/H$ by $F(\varphi)$ is trivial, i.e. $G_{F(\varphi)}/H_{\varphi} = G/H$.

For (2), if $\varphi$ is surjective then $(G_{F(\varphi)}/H_{\varphi})(k)$ consists of those elements of $G/H$ on which $H$ acts trivially. If $g H \in G/H$ is such a coset then this means that for all $h \in H$ there exists a $h' \in H$ such that $h g = g h'$, i.e.~$g \in N$. The number of such cosets is thus $|N/H|$.
\end{proof}

\subsection{Extensions with given Galois closure} \label{sec:BS_n}
In the special case of $S_n$, we have the following equivalent descriptions for the groupoid $BS_n(k)$.

\begin{lemma}  \label{lem:S_n} 
The following categories are equivalent.
\begin{enumerate}
\item The groupoid $BS_n(k)$ of right $S_n$-torsors over $\Spec k$ with isomorphisms as $S_n$-equivariant morphisms.
\item Homomorphisms $\Gamma_k \to S_n$ with isomorphisms given by conjugation in $S_n$.
\item $\Gamma_k$-actions on the set $\{1,\cdots, n\}$ with isomorphisms as conjugation by an element of $S_n$ acting on $\{1, \cdots, n\}$.
\item \'Etale $k$-algebras of degree $n$ with isomorphisms given by $k$-algebra isomorphisms.
\end{enumerate}
\end{lemma}
\begin{proof}
	(1) $\iff$ (2): Immediate from Lemma \ref{lem:BG(k)}. 
	
	(2) $\iff$ (3):  Clear.
	
	(2) $\iff$ (4): This is well-known and follows from the arguments in 
	\cite[Lem.~3.1, 3.2]{Ked07}
	(strictly speaking, the category in (4) is canonically equivalent to the category of 
	$\Gamma_k$-actions on sets of cardinality $n$; this  is equivalent to the category
	in (3)).
\end{proof}

In Malle's conjecture one is traditionally interested in field extensions with given Galois closure $G$. To study this in our setting, let $G \subseteq S_n$ be a transitive subgroup. Then by Lemma \ref{lem:S_n}, the image of the functor $BG(k) \to BS_n(k)$ exactly corresponds to \'etale algebras of degree $n$ with Galois closure $G$ acting via the inclusion into $S_n$. One can calculate the correct groupoid cardinalities in this case using Lemma~\ref{lem:outer_automorphism}. The formulae which appear this way can be quite complicated; this all points to the fact that the most natural counting problem is counting elements of $BG(k)$, which boils down to counting homomorphisms $\Gamma_k \to G$, rather than counting field extensions with Galois group $G$. In any case we record the corresponding groupoid cardinalities here for surjective homomorphisms.

\begin{lemma} \label{lem:G_S_n}
	Let $G \subseteq S_n$ be a subgroup and $F: BG(k) \to BS_n(k)$ the induced map.
	Let $f:BS_n[k] \to \C$ be any function.
	Then for any finite collection of \textbf{surjective} objects $W \subseteq B G[k]$ with $F^{-1}(F(W)) = W$
	 we have
	$$\frac{1}{|Z(G)|}\sum_{\varphi \in W} F(f(\varphi))
	 = \frac{|N|}{|C| \cdot |G|}\sum_{\psi \in [F(W)]} f(\psi),$$
	where $N$ and $C$ denote the normaliser and centraliser of $G$ in $S_n$, respectively.
\end{lemma}
\begin{proof}
	First note that $\Aut \varphi = Z(G)$ and $\Aut \psi = C$ by 
	Lemma~\ref{lem:automorphism_group_is_inner_twist}(2). The result then follows from
	Lemma~\ref{lem:outer_automorphism}(2).
\end{proof}

\section{Orbifold Picard group} \label{sec:orbifold_Picard}
In this section we begin by setting-up our notation. Some of our definitions are inspired by those appearing in \cite{DYBM}, though there are some differences (as we explain in Remark \ref{def:DY_diff}).

\subsection{Galois action on conjugacy classes} \label{sec:Galois_action_conjugacy_classes}
Let $k$ be a field of characteristic $p$ (possibly $0$) with absolute Galois group $\Gamma_k := \Gal(k^{\sep}/k)$ and let $G$ be a tame finite \'etale group scheme over $k$, e.g.~$G$ is a finite group of order not divisible by the  characteristic of $k$.

Fundamental to Malle's conjecture is the Tate twist of $G$ by $-1$.

\begin{definition}[Tate twist] \label{def:G(-1)}
	Let $\widehat{\Z}(1) := \varprojlim_n \mu_n$. We define
\begin{equation*} 
	G(-1):=\Hom(\widehat{\Z}(1), G).
\end{equation*}
\end{definition}

This is a finite \'etale scheme. If $G$ is non-abelian then $G(-1)$ has no natural group scheme structure.  Any element of $G(-1)$ may be represented by a homomorphism $\mu_{\exp(G)} \to G$ where $\exp(G)$ denotes the exponent of $G$.
There is a very closely related scheme which appears more commonly in the Malle's conjecture literature. Recall that the set $G(k^{\sep})$ admits an action of $\hat{\Z}^\times$ via exponentiation (if $G$ is non-abelian then this need not preserve the group structure).

\begin{definition}[Anticyclotomic twist]\label{def:G(cycl)}
Let $\cycl: \Gamma_k \to \widehat{\Z}^\times$ be the cyclotomic character. Let $G(\cycl^{-1})$ denote the finite \'etale scheme whose geometric points are $G(k^{\sep})$, but with Galois action twisted by $\cycl^{-1}$:
\begin{equation*} 
\Gamma_k \times G(\cycl^{-1})(k^{\sep}) \to G(\cycl^{-1})(k^{\sep}), \quad (\sigma,g) \mapsto \sigma(g)^{\chi(\sigma)^{-1}}.
\end{equation*}
\end{definition}

Warning: not all authors take the anticyclotomic character $\cycl^{-1}$, which is especially important for non-constant $G$. These two schemes are isomorphic, though non-canonically.

\begin{lemma}\label{lem:Galois_action_on_G(-1)}
Choose a primitive $\exp(G)$-th root of unity $\zeta$. Then the map
$$G(-1) \to G(\cycl^{-1}), \quad (\gamma: \mu_{|\exp(G)|} \to G) \mapsto \gamma(\zeta)$$
is an isomorphism of schemes.
\end{lemma}
\begin{proof}
	It is clearly bijective. Thus it suffices to show that it preserves the Galois
	action. For this we have
	\[(\sigma \cdot \gamma)(\zeta) := \sigma(\gamma(\sigma^{-1}(\zeta))) = \sigma(\gamma(\zeta^{\chi(\sigma^{-1})})) 
	= \sigma(\gamma(\zeta))^{\chi(\sigma)^{-1}}. \qedhere \]
\end{proof}

\noindent 
We define 
\begin{equation} \label{def:conjugacy_classes}
	\mathcal{C}_G := G(-1)(k^{\sep})/\mathrm{conj}
\end{equation}
where the quotient is by the conjugacy action of $G$ on $G(-1)$. Let $\mathcal{C}_G^* := \mathcal{C}_G \setminus \{e\}$ where $e$ denotes the identity of $G(-1)$ (by which we mean the homomorphism with constant value $e$). We call $\{e\}$ the \emph{identity conjugacy class}.

The scheme $G(-1)$ arises in the following way. Heights will be determined by the restriction of a cocycle $\varphi: \Gamma_k \to G$ to the tame inertia group. The tame inertia is well-defined up to conjugacy, and is canonically isomorphic to a group of roots of unity \cite[Tag 09EE]{stacks-project}. Hence we canonically obtain an element of $G(-1)$, not of $G(\cycl^{-1})$ (see \S\ref{sec:ramification_type} and \S\ref{sec:heights} for details). To help differentiate between $G(-1)$ and $G(\cycl^{-1})$, we will write $\gamma$ for an element of $G(-1)$ and $g$ for an element of $G(\cycl^{-1})$.

Nevertheless, the scheme $G(\cycl^{-1})$ is often easier to work with in examples due it being easier to write down elements. Though we shall only do this when in the given setting, the choice of isomorphism in Lemma \ref{lem:Galois_action_on_G(-1)} is irrelevant (see Definition~\ref{def:special} and Remark~\ref{rem:special} for example). In this case we refer to the situation as $\widehat{\Z}^\times$\emph{-invariant}.

\begin{definition}[Subgroup generated by conjugacy classes] \label{def:subroup_generated}
	Let  $\mathcal{C} \subset \mathcal{C}_G$
	be Galois invariant. We let $\langle \mathcal{C} \rangle \subset G$ be the subgroup
	scheme generated by the image of all $\gamma \in G(-1)$ whose conjugacy class
	lies in $\mathcal{C}$. This is a normal subgroup scheme as it is generated by conjugacy classes.
\end{definition}

\begin{remark}[Cyclotomic twist]
Define the scheme $G(\cycl)$ in an analogous way to Definition \ref{def:G(cycl)}. In general $G(\cycl) \not \cong G(\cycl^{-1})$, nonetheless if $G$ is constant then an element $g \in G(k^{\sep})$ has the same Galois orbit in both $G(\cycl)$ and $G(\cycl^{-1})$, which explains why $G(\cycl)$ also appears in the Malle's conjecture literature. 

For constant $G$ one can  determine the Galois action on $G(\cycl)(k^{\sep})/\mathrm{conj}$ through the character table of $G$, see  \cite[Ch.~12, Thm.~25]{Ser77}. In particular, the Galois action on $\mathcal{C}_G$ is trivial if and only if the character table of $G$ takes values in $k$.
\end{remark}

\subsection{Orbifold Picard group of $BG$}
Let $\dual{G} = \Hom(G,\Gm)$ be the group scheme of $1$-dimensional characters $G$. For any field extension $k \subset L$, we will identify the Picard group $\Pic (BG)_{L}$ of $BG$ over $L$ with $\dual{G}(L)$, the group of $1$-dimensional characters defined over $L$ (see e.g.~Lemma \ref{lem:coh_sep_closed}).

\begin{definition} \label{def:age}
	We define the \emph{age pairing}\footnote{Darda--Yasuda \cite[Def.~2.21]{DYBM} use a slightly different definition where 
	they choose the canonical representative of an element of $\Q/ \Z$ in $\Q \cap [0,1)$.}
	as follows:
	$$\age:  \dual{G}(k^{\sep}) \times G(-1)(k^{\sep})
	\to \Q/\Z.$$
	For $(\gamma: \mu_n \to G) \in G(-1)(k^{\sep})$ we let $\age(\chi,\gamma)$ be
	the unique element of $\Z/n\Z \subset \Q/\Z$ such that the map
	$\chi \circ \gamma: \mu_n \to \Gm$ has the form $
	\zeta \mapsto \zeta^{\age(\chi,\gamma)}.$
\end{definition}


\begin{lemma} \label{lem:age}
\hfill
	\begin{enumerate}
		\item $n\age(\cdot,\gamma) = \age(\cdot,\gamma^n)$ for all $n$ coprime to $|G|$.
		\item The $\age$ pairing is bilinear and non-degenerate on the left. 
		\item $\age(\cdot,\gamma)$ only depends the image of $\gamma$ in $\mathcal{C}_G$.		
		\item $\age(\sigma(\chi), \sigma(\gamma)) = \age(\chi,\gamma)$ for all $\sigma \in \Gamma_k$.
	\end{enumerate}
\end{lemma}
\begin{proof}
	Part (1) is clear from the definition.  Therefore 
	applying Lemma \ref{lem:Galois_action_on_G(-1)},
	as abstract groups this is just the usual bilinear pairing
\begin{equation*}
	\Hom(G, \Q/\Z) \times G \to \Q/\Z: (\chi, g) \to \chi(g).
	\label{Properties of age pairing}
\end{equation*}
Part $(2)$ follows immediately. Claim $(3)$ holds because the value of a $1$-dimensional character at an element of $G$ only depends on the conjugacy class of $g$.

The last statement holds because for all $\zeta \in \mu_n$ and $\sigma \in \Gamma_{k}$ we have
\[
\sigma(\zeta)^{\age(\sigma(\chi), \sigma(\gamma))}  = \sigma(\chi)\left(\sigma(\gamma)\left(\sigma(\zeta)\right)\right) = \sigma(\chi)\left(\sigma\left(\gamma(\zeta)\right)\right) = \sigma\left(\chi(\gamma(\zeta))\right) = \sigma(\zeta)^{\age(\chi, \gamma)}.
\]
\end{proof}

Our main innovation over Darda-Yasuda \cite{DYBM}, which will allow us to compute the effective cone constant, is an integral structure on the orbifold Picard group.
\begin{definition} \label{def:orbifold_line_bundle}
A \emph{geometric orbifold line bundle} is a pair $(\chi, w)$ where $\chi \in \dual{G}(k^{\sep})$ and $w: G(-1)(k^{\sep}) \to \Q$ is a conjugacy invariant function such that 
\begin{enumerate}
	\item $w(e) = 0$,
	\item for all $\gamma \in G(-1)(k^{\sep}) $ we have $w(\gamma) \bmod \Z = \age(\chi, \gamma).$
\end{enumerate} 
Here $e$ denotes the identity of $G(-1)$. We call the function $w$ a \emph{weight function}.
\end{definition}

Condition (2) says that $\zeta^{w(\gamma)} = \chi(\gamma(\zeta))$
for all $\gamma \in G(-1)(k^{\sep})$ and all $\zeta \in \mu_n$ with $n$ not divisible by the characteristic of $k$; via a choice of isomorphism from Lemma \ref{lem:Galois_action_on_G(-1)} this becomes $\zeta^{w(g)} = \chi(g)$ for all $g \in G(k^{\sep})$ and all such $n$. In particular $w$ takes integer values if and only if $\chi$ is trivial.

We usually consider $w$ as a function $w: \mathcal{C}^*_G \to \Q$. We can add orbifold line bundles componentwise, i.e.~$(\chi, w) + (\chi', w') = (\chi \chi', w + w')$; this is well defined since $\age$ is b e.g.~ilinear.  We denote the group of geometric orbifold line bundles as $\PicOrb (BG)_{k^{\sep}}$ and call it the \emph{geometric orbifold Picard group}. 

\begin{definition} \label{def:orbifold_line_bundle_over_k}
An \emph{orbifold line bundle} is a geometric orbifold line bundle $(\chi, w)$ such that $\chi \in \dual{G}(k)$ and $w: \mathcal{C}_G \to \Q$ is $\Gamma_k$-equivariant. We denote the group of orbifold line bundles by $\PicOrb BG$ and call it the \emph{orbifold Picard group}. We clearly have $\PicOrb BG = (\PicOrb (BG)_{k^{\sep}})^{\Gamma_k}$.
\end{definition}

\begin{definition} \label{def:special}
	We say that an orbifold line bundle is $\widehat{\Z}^\times$\textit{-invariant}
	if the weight function $w$ satisfies
	$w(\gamma) = w(\gamma^n)$ for all $\gamma \in G(-1)(k^{\sep})$ and all 
	$n$ coprime to $|G|$.
\end{definition}

\begin{remark} \label{rem:special}
If $w$ is $\widehat{\Z}^\times$-invariant, we may view $w$ as function on $G(k^{\sep})$ in a canonical way; namely the value of $w$ via the induced map  $G(-1)(k^{\sep}) \to G(k^{\sep})$ from Lemma~\ref{lem:Galois_action_on_G(-1)} is independent of the choice of primitive root of unity.  This always holds if $k$ contains no non-trivial roots of unity and is imposed in the literature in many cases, for example \cite[\S2.1]{Woo10} and \cite[\S4.2]{EV05}. It gives simpler formulae for height functions, but it is not relevant to our theory in general.

If $w$ is not $\widehat{\Z}^\times$-invariant however, then there is no canonical way to view $w$ as a function on $G(k^{\sep})$. 
\end{remark}

\begin{example}\label{ex: Picorb of B mu_n}
	If $G = \mu_n$ then $\mathcal{C}_G = \Hom(\mu_n, \mu_n) = \Z/n \Z$ and $\Pic B \mu_n = \Z /n \Z$. The age pairing $\Z/ n \Z \times \Z / n \Z \to \Z/ n \Z \subset \Q/\Z$ is given by multiplication.
	
	It follows that $\PicOrb B \mu_n = \PicOrb (B \mu_n)_{k^{\sep}}$ is the group of pairs $(k, r_0, \cdots, r_{n-1}) \in \Z / n \Z \times \Q^n$ such that $r_0 = 0$ and $r_i = \frac{ik}{n} \mod \Z$ for all $i$. This is a free group with $n-1$ generators $(1, 0, \frac{1}{n}, \frac{2}{n}, \cdots, \frac{n-1}{n}), (0,0,0,1,0, \ldots, 0), \ldots, (0,0, \ldots,0, 1)$. Forgetting the first two coordinates embeds this into $\Q^{n-1}$ as a sublattice generated by $(\frac{1}{n}, \frac{2}{n}, \cdots, \frac{n-1}{n}), (0, 1,0, \ldots, 0), \ldots, (0, \ldots, 0,1)$.
\end{example}

In this example, the orbifold Picard group is a sublattice of a $\Q$-vector space. The following shows that this is always the case.

\begin{lemma} \label{lem:Picorb_torsion_free}
The map $\PicOrb (BG)_{k^{\sep}} \to \Hom(\mathcal{C}^*_G, \Q): (\chi, w) \to w$ is injective and its image contains $\Hom(\mathcal{C}^*_G, \Z)$. In particular $\PicOrb (BG)_{k^{\sep}}$ is torsion free.
\end{lemma} 
\begin{proof}
	If $(\chi, w)$ is a geometric orbifold line bundle in the kernel of this map then $w = 0$. This implies that $\text{age}(\chi, c) = 0$ for all $c \in \mathcal{C}^*_G$. Lemma \ref{Properties of age pairing} implies that $\chi = 1$.
	
	If $w \in \Hom(\mathcal{C}^*_G, \Z)$ then $(1, w)$ is a geometric orbifold line bundle.
	Torsion freeness follows from the fact that it embeds in a $\Q$-vector space.
\end{proof}

This lemma yields an isomorphism $\PicOrb (BG)_{k^{\sep}} \otimes_{\Z} \R \cong \Hom(\mathcal{C}^*_G, \R)$, which shows that our definition of $\PicOrb (BG)_{k^{\sep}} \otimes_{\Z} \R$ agrees with that of Darda-Yasuda \cite[Def.~8.1]{DYBM}. Another way to see this is because we have an exact sequence.

\begin{lemma} \label{lem:orb_Pic_exact}
The sequence
\begin{equation*}
	0 \to \Hom(\mathcal{C}^*_G, \Z) \to \PicOrb (BG)_{k^{\sep}} \to \dual{G}(k^{\sep}) \to 0
\end{equation*}
is exact. This induces an exact sequence
\begin{equation*}
	0 \to \Hom(\mathcal{C}^*_G, \Z)^{\Gamma_k} \to \PicOrb (BG) \to \dual{G}(k) \to 0.
\end{equation*}
\end{lemma}
\begin{proof}
	The morphism $\PicOrb (BG)_{k^{\sep}}  \to \dual{G}(k^{\sep})$ sending $(\chi, w)$ to $\chi$ is surjective; indeed for $\chi \in \dual{G}(k^{\sep})$ one may define $w(c)$ to be any lift of $\age(\chi,c)$ from $\Q/\Z$ to $\Q$. Its kernel consists of pairs $(1, w)$ where $w: \mathcal{C}_G \to \Q$ is a weight function such that $w(c) \in \Z$ for all $c \in \mathcal{C}_G$. The group of such pairs is isomorphic to $\Hom(\mathcal{C}^*_G, \Z) \to \PicOrb (BG)_{k^{\sep}}$. This proves the exactness of the first sequence.
	
	For the second part we apply Galois cohomology and note that
	$$\H^1(k,\Hom(\mathcal{C}^*_G, \Z)) = 0$$
	by Shapiro's Lemma. 
\end{proof}

\begin{definition} \label{def:canonical_class}
The role of the canonical divisor  is played by the orbifold line bundle $K_{BG}^{\text{orb}} := -(1, \mathbf{1})$ (see \cite[Def.~9.1]{DYBM}). Here $1$ denotes the trivial character and $\mathbf{1}: \mathcal{C}^*_G \to \{1\} $ the constant function with value $1$.
\end{definition}

\begin{remark} \label{def:DY_diff}
	Our definition of orbifold line bundle is modelled on Darda and Yasuda's
	notion of raised line bundles in \cite[Def~4.2]{DYBM}; there is a crucial difference
	however, in that we impose the condition (2) which requires compatibility
	with the age pairing. It is this which allows us to define the orbifold Picard group, which is a
	finitely generated abelian group, whereas Darda and Yasuda are only able to define an orbifold
	N\'{e}ron--Severi space in \cite[Def~8.1]{DYBM}, which is a real vector space and the tensor
	of our orbifold Picard group with $\R$. This new integral structure is crucial to be able
	to define the correct measure on the orbifold N\'{e}ron--Severi space and hence define the
	effective cone constant.
	
	What we call a \emph{weight function}, Darda and Yasuda refer to as a \emph{raising function}.
	We use the former terminology, since crucial to our framework will be the non-trivial
	conjugacy classes which minimise $w$, which can conveniently be referred to as 
	the \textit{minimal weight conjugacy classes}.
\end{remark}

\begin{remark} \label{rem:intuition}
	The sequences in Lemma \ref{lem:orb_Pic_exact} have a geometric interpretation.
	Let $U$ be a smooth variety over $k$ with $\bar{k}^\times[U]
	= \bar{k}^\times$ and $X$ a smooth compactification of $U$. 
	Write $X \setminus U = D_1 \cup \dots \cup D_r$
	for irreducible divisors $D_i$. Then we have
	\begin{equation} \label{seq:Pic_intuition}
		0 \to \bigoplus_{i=1}^r \Z[D_i] \to \Pic X \to \Pic U \to 0.
	\end{equation}
	Therefore one gains a useful intuition in pretending
	that there is a smooth compactification of $BG$ with boundary divisors given by the disjoint
	union of (the dual of) the non-identity conjugacy classes of $G(-1)$.
	
	This analogy is not perfect however, there are at least two differences.
	Firstly $\Z/2\Z$ has $1$ non-trivial
	conjugacy class, but $\Z/2\Z \times \Z/2\Z$ has $3$ non-identity conjugacy classes,
	whereas the product of two varieties each with a single boundary divisor has $2$ boundary
	divisors. In particular $\PicOrb (BG_1 \times BG_2) \neq \PicOrb BG_1 \times \PicOrb BG_2$
	in general. 
	
	Secondly, for many arithmetic applications one chooses a smooth proper model
	$\mathcal{X}_v$ over $\O_v$ and studies the intersection multiplicity
	of elements of $\mathcal{X}_v(\O_v)$ with the boundary divisors. In our situation
	of $BG$, the intersection multiplicity is always either $0$ or $1$. This will help
	to simplify numerous formulae which appear.
\end{remark}

\subsection{Partial orbifold Picard group}

Motivated by Remark \ref{rem:intuition}, we consider a modified version of the orbifold Picard group, which will appear in our work when dealing with general line bundles. This offers a more flexible approach which is easier to work with (e.g.~it respects products, see Lemma \ref{lem:Picorb_products}).

\begin{definition} \label{def:partial_orbifold_line_bundle}
Let $\mathcal{C} \subseteq \mathcal{C}_G^*$ be Galois invariant. A \emph{partial geometric orbifold line bundle (with respect to $\mathcal{C}$)} is a pair $(\chi, w)$ where $\chi \in \dual{G}(k^{\sep})$ and $w: \mathcal{C} \cup \{e\} \to \Q$ is such that 
\begin{enumerate}
	\item $w(e) = 0$,
	\item for all $c \in \mathcal{C}$ we have $w(c) \bmod \Z = \age(\chi, c).$
\end{enumerate} 
\end{definition}
Taking $\mathcal{C} = \mathcal{C}_G^*$ one recovers Definition \ref{def:orbifold_line_bundle}. Analogously to the above theory, one defines a \emph{partial orbifold line bundle} to be one which is Galois invariant. We denote by $\PicOrb_{\mathcal{C}} BG$ the group of isomorphism classes of such orbifold line bundles. Regarding Remark \ref{rem:intuition}, one should view $\PicOrb_{\mathcal{C}} BG$ as the playing the role of the Picard group of the open subset given by removing the ``divisors'' corresponding to the elements of $\mathcal{C}_G^*\setminus \mathcal{C}$. In particular, in analogy with \eqref{seq:Pic_intuition} one easily sees that we have the exact sequence
\begin{equation} \label{seq:Pic(BG,C)}
	0 \to \Hom(\mathcal{C}_G^*\setminus \mathcal{C}, \Z)^{\Gamma_k} \to \PicOrb (BG) \to \PicOrb_{\mathcal{C}} BG \to 0.
\end{equation}
One is tempted to construct a section of the latter map by simply extending the weight function to take value $0$ outside of $\mathcal{C}$; this cannot be done in general since one will lose compatibility with the age pairing required in (2) of Definition~\ref{def:orbifold_line_bundle}.

\begin{lemma} \label{lem:Picorb_Galois_action}
	If $\mathcal{C}$ generates $G$ then the map 
	$$\PicOrb_{\mathcal{C}} BG_{k^{\sep}} \to \Hom(\mathcal{C}, \Q): (\chi, w) \to w$$ is injective and its image contains $\Hom(\mathcal{C}, \Z)$. In particular $\PicOrb_{\mathcal{C}} BG_{k^{\sep}}$ is torsion free.
\end{lemma}
\begin{proof}
	Analogous to Lemma \ref{lem:Picorb_torsion_free}.
\end{proof}

\begin{lemma} \label{lem:Picorb(BG,C)}
	The sequence
	\begin{equation}\label{eq:exact_sequence_Picorb(BG,L)}
		0 \to \Hom(\mathcal{C}, \Z)^{\Gamma_k} \to \PicOrb_{\mathcal{C}} BG \to \dual{G}(k) \to 0
	\end{equation}
	is exact.
\end{lemma}
\begin{proof}
	Analogous to Lemma \ref{lem:orb_Pic_exact}.
\end{proof}

\begin{lemma} \label{lem:Picorb_products}
	Let $G_1,G_2$ be a finite \'etale group schemes over $k$ and $\mathcal{C}_i \subset \mathcal{C}_{G_i}^*$ Galois invariant. Then the natural map
	$$\PicOrb(G_1 \times G_2, \mathcal{C}_1 \times\{e\} \, \cup \, \{e\} \times \mathcal{C}_2) \to \PicOrb(G_1,\mathcal{C}_1) \times \PicOrb(G_2, \mathcal{C}_2)$$
	is an isomorphism.
\end{lemma}
\begin{proof}
	The map $\dual{(G_1 \times G_2)} \to \dual{G_1} \times \dual{G_2}$ is an isomorphism as $\Hom(\cdot, \Gm)$ commutes with finite products. Note also that for $c \in  \mathcal{C}_1$, resp. $c \in \mathcal{C}_2$ we have $\age((\chi_1, \chi_2),c \times \{e\}) = \age(\chi_1, c)$, resp. $\age((\chi_1, \chi_2),\{e\} \times c) = \age(\chi_2, c)$.

	The group of functions $w: \mathcal{C}_1 \times\{e\} \, \cup \, \{e\} \times \mathcal{C}_2 \to \Q$ can be identified with the group of pairs $(w_1, w_2)$ where $w_i: \mathcal{C}_i \to \Q$. Moreover, this identification is clearly compatible with the Galois action and it is compatible with the age pairing by the previous paragraph. The lemma then immediately follows from the definitions.
\end{proof}

\subsection{Orbifold effective cone}\label{sec:orbifold_effective_cone}

\begin{definition}
We say that an orbifold line bundle $(\chi, w)$ is \emph{effective} if $w \in \Hom(\mathcal{C}_G, \Q_{\geq 0})$. Equivalently, $(\chi, w)$ is contained in the \emph{orbifold effective cone} $\Eff_G := \Hom(\mathcal{C}^*_G, \R_{\geq 0})^{\Gamma_k} \subset \Hom(\mathcal{C}^*_G, \R)^{\Gamma_k} \cong \PicOrb (BG) \otimes_{\Z} \R$. An orbifold line bundle is effective if it is effective as a geometric orbifold line bundle.
\end{definition}

This definition agrees with the conventions of \cite{DYBM} (see \cite[Cor.~9.22]{DYBM}).

\begin{definition}
An orbifold line bundle $(\chi, w)$ is \emph{big} if $(\chi, w) \in \Eff_G^{\circ}$. Equivalently we have $w(c) > 0$ for all $c \in \mathcal{C}^*_G$.
\end{definition}

\begin{definition}
Let $L \in  \Eff_G^{\circ}$ be a big orbifold line bundle. Define the \emph{Fujita invariant} of $L$ to be $a(L) := \inf\{a \in \R: K_{BG}^{\text{orb}} + a L \in \Eff_G \}$. We call $\ad(L) := K_{BG}^{\text{orb}} + a(L) L$ the \emph{adjoint} of $L$, which  lies in a minimal face of the cone $\Eff_G$. We define $b(k,L)$ to be the codimension of the minimal face containing $\ad(L)$.
\end{definition}

\begin{definition} \label{def:C(L)}
Let $L = (\chi,w) \in  \Eff_G$ be an effective orbifold line bundle. We denote by $\mathcal{A}(L) := \{ c \in \mathcal{C}_G^* : w(c) = 0\}$
its non-identity zero locus.
\end{definition}

We call $\mathcal{M}(L):=\mathcal{A}(\ad(L))$ the collection of \textit{minimal (weight) conjugacy classes} of $G(-1)$, as the following lemma clarifies.

\begin{lemma} \label{lem:Fujita_minimal}
	Let $L$ be a big orbifold line bundle. Then
	\begin{align*}
	a(L) & = (\min_{c \in \mathcal{C}_G^*}w(c))^{-1},   \quad
	\mathcal{M}(L) = \{ c \in \mathcal{C}_G^* : w(c) = a(L)^{-1}\}, \\
	 b(k,L) & = \#\mathcal{M}(L)/\Gamma_k.
	 \end{align*}
\end{lemma}
\begin{proof}
	Recalling $K_{BG}^{\text{orb}} := -(1, \mathbf{1})$, 
	this is immediate from the definitions.	
\end{proof}

\subsection{Iitaka fibration} \label{sec:Iitaka}
We now consider the birational geometry of $BG$, and define  a version of the Iitaka fibration in our setting (see \cite[\S 2.1.C]{Lar04} for the case of varieties).
This is  key to our paper for understanding when the leading constant is a sum of terms.

\begin{definition}[Iitaka fibration]
Let $A(L) \subset G$ be the subgroup scheme generated by the conjugacy classes $\mathcal{A}(L)$ (see Definition \ref{def:subroup_generated}). It is normal since it is generated by conjugacy classes. We call $I(L) = G/A(L)$ the \emph{Iitaka group} and the \emph{Iitaka fibration} associated to $L$ is the map $B G \to B I(L)$.

We define the \textit{Iitaka dimension} of $L$ to be $\kappa(L) := \deg I(L) -1$.
\end{definition}

\begin{example}
If $L$ is big then $\mathcal{A}(L)$ is empty since the orbifold effective cone is simplicial generated by the conjugacy classes of $G$. Thus here the Iitaka fibration is simply the identity map $BG \to BG$, in particular it is birational as expected.
\end{example}



\begin{remark}
	It is possible that there is a more convenient normalisation of the Iitaka dimension in general. We will be mostly concerned with whether the Iitaka dimension is $0$.
\end{remark}

\begin{definition} \label{def:rigid}
An effective orbifold line bundle $L$ is called \emph{rigid} if $\kappa(L) = 0$, equivalently the conjugacy classes $\mathcal{A}(L)$ generate the entire group $G$. We call $L$ \emph{adjoint rigid}  (or \emph{balanced}) if $\ad(L)$ is rigid; equivalently if the (non-identity) minimal weight conjugacy classes $\mathcal{M}(L):=\mathcal{A}(\ad(L))$ generate $G$ (see Definition \ref{def:subroup_generated}).
\end{definition}

The terminology \emph{balanced} is used to conjure the image that there are enough minimal weight elements to balance a plate on top of; as opposed to the concentrated case from \cite{AOWW24} where the minimal weight elements are too clumped together to balance a plate on. We use the terminology balanced line bundle in a different way to the paper \cite{LTT18}; however this terminology has subsequently been replaced in the literature by the ``breaking'' terminology from \cite{LST22}.

For example $-K_{BG}^{\text{orb}}$ is adjoint rigid (as expected in analogy with Manin's conjecture), since its adjoint line bundle is trivial.
The following is an analogue of \cite[Thm.~2.1.33]{Lar04} in our setting.
\begin{lemma} \label{lem:Iitaka_rigid_fibres}
	The restriction of $L$ to every fibre of the (adjoint) Iitaka fibration is (adjoint) rigid.
\end{lemma}
\begin{proof}
	Rigidity can be checked after base change to $k^{\sep}$. 
	Over a separably closed field $B I(L)$ has a single point and the fibre over this point is 
	$B A(L)$ (cf.~Lemma~\ref{lem:fibration_normal_quotient}).
	The result follows from the definition of $A(L)$. The case of adjoint rigidity is similar.
\end{proof}

\begin{remark}
	Wood defined the notion of fair heights in \cite[\S 2.1]{Woo10} in the case that $G$ is abelian. Our definition of balanced (or adjoint rigidity)
	is strictly weaker than Wood's in the abelian case (see \S\ref{sec:fair_wood}), and is a property of the underlying orbifold line bundle and not the choice of height.
	
	Our condition also appears in Alberts' paper \cite[\S7.6]{Alb23}, however Alberts does
	not call such height functions fair since an example is included in the abelian case where
	the leading constant in Malle's conjecture need not satisfy the Malle--Bhargava heuristics.
	However this example is not as pathological as one might first think,
	as it  can be explained by a Brauer--Manin
	obstruction; again see \S\ref{sec:fair_wood} for more details. Thus we believe
	that our definition of balancedness is the correct one to encapsulate a
	measure-theoretic interpretation of the leading constant in Malle's conjecture.
\end{remark}

Throughout the remainder of the paper, we will be primarily interested in a big orbifold line bundle $L$,
and the corresponding adjoint Iitaka fibration induced by the minimal weight conjugacy
classes $\mathcal{M}(L)$.

\subsection{Effective cone constant}
We now have everything in place to define an analogue of Peyre's effective cone constant in our setting.

Let $L$ be a big orbifold line bundle on $BG$ and let $\mathcal{M}(L)$ be the corresponding minimal weight conjugacy classes (see Definition \ref{def:C(L)}). We define
\begin{equation*}
	\PicOrb_{L} BG := \PicOrb_{\mathcal{M}(L)} BG.
\end{equation*}
(This definition is inspired by \cite[Def.~2.3.11]{BT98}.)

From now on to define the effective cone constant, we assume that $L$ is balanced. We let $\Eff_G(L)$ be the image of $\Eff_G$ inside the real vector space $\PicOrb_{L} BG \otimes \R$.

\begin{definition}
Take the Haar measure $\mathrm{d}v$ on the vector space $(\PicOrb_{L} BG \otimes \R)^\wedge$ such that the dual lattice $\Hom(\PicOrb_{L} BG,\Z)$ has covolume $1$.  We define the effective cone constant in our case to be
\begin{equation}
	\alpha^*(BG,L) = \int_{ v \in \Eff_G(L)^\wedge} e^{-\langle L, v\rangle} \mathrm{d}v.
	\label{def:theta}
\end{equation}
\end{definition}
This definition is inspired by \cite[Def.~2.3.14, 2.3.16]{BT98} (one should replace $-K_X$ by $\mathcal{L}$ in \cite[Def.~2.3.16]{BT98}). In the integral we abuse notation and take the inner product with the image of $L$ in $\PicOrb_{L} BG$. Note that what \cite{BT98} calls ``$\mathcal{L}$-primitive'' is now called ``adjoint rigid'' in the Manin's conjecture literature.

We calculate this by relating it to the Lebesgue measure. The exact sequence \eqref{eq:exact_sequence_Picorb(BG,L)} implies that the map
$$\Hom(\mathcal{M}(L), \R)^{\Gamma_k} \to \PicOrb_{L} BG \otimes \R$$
is an isomorphism. We have $\Hom(\mathcal{M}(L), \R)^{\Gamma_k} = \R^{\Gamma_k \backslash \mathcal{M}(L)}$. Taking duals we obtain $\R^{\Gamma_k \backslash \mathcal{M}(L)} \cong (\PicOrb_{L} BG \otimes \R)^{\wedge}$. This isomorphism induces a measure on $(\PicOrb_{L} BG \otimes \R)^{\wedge}$, which we call the \textit{Lebesgue measure} on $(\PicOrb_{L} BG \otimes \R)^{\wedge}$.

\begin{lemma} \label{lem:Lebesgue}
	The measure $\mathrm{d}v$ is $1/\#\dual{G}(k)$ times the Lebesgue measure.
\end{lemma}
\begin{proof}
	We have an inclusion of lattices $\Hom(\PicOrb_{L} BG, \Z) \subset \Z^{\Gamma_k \backslash \mathcal{M}(L)}$ of index $\#\dual{G}(k)$ by the exact sequence \eqref{eq:exact_sequence_Picorb(BG,L)}. The lattice $\Hom(\PicOrb_{L} BG, \Z)$ thus has covolume $\#\dual{G}(k)$ with respect to the Lebesgue measure.
\end{proof}

\begin{lemma} \label{lem:effective_cone_calc}
	Let $L$ be a balanced orbifold line bundle. Then 
	$$\alpha^*(BG, L) = a(L)^{b(k,L)}/\#\dual{G}(k).$$
\end{lemma}
\begin{proof}
	By Lemma \ref{lem:Fujita_minimal} the image of $L$ in $\PicOrb_{L} BG \otimes \R \cong \Hom(\mathcal{M}(L), \R)^{\Gamma_k}$ is the constant function $c \to a(L)^{-1}$.
	If we now apply Lemma~\ref{lem:Lebesgue}, we see that the integral \eqref{def:theta} equals
	$$\frac{1}{\#\dual{G}(k)}\prod_{c \in \mathcal{M}(L)/\Gamma_k}\int_{r \geq 0}e^{-a(L)^{-1}r} \mathrm{d} r.$$
	 Thus it suffices to note that $\int_{r=0}^\infty e^{-a(L)^{-1} r}\mathrm{d} r  = a(L)$ and $b(k,L) = |\mathcal{M}(L)/\Gamma_k|$.
\end{proof}

\subsection{Breaking cocycles} \label{sec:breaking}

In Manin's conjecture, in order to obtain an asymptotic formula with the correct leading constant, one is required to remove some thin subset of rational points. In the literature one often says that these rational points are \textit{accumulating} or \textit{breaking} (or that they come from a \textit{breaking closed subset/cover}). 
In this subsection we will identify the elements of $BG(k)$ that could potentially be contained in a breaking thin set. To do so it suffices to see how the $a$- and $b$-invariants change along finite morphisms. Throughout this section we use the theory of inner twists from \S\ref{sec:inner_twists}.

It is clear from Lemma \ref{lem:Fujita_minimal} that the $a$-invariant is non-decreasing when pulling-back orbifold line bundles along $BH \to BG$ as in Lemma \ref{lem:thin_classification} (since the minimum can only increase when passing to a subgroup scheme). So we may focus on the $b$-invariant, which is Galois theoretic in nature. We will show that it suffices to remove the following elements.

\begin{definition}
	Let $c \subset G(-1)(k^{\sep})$ be a conjugacy invariant subset. Let $\Gamma_c \subset \Gamma_k$ be the subgroup which leaves $c \subset G(-1)(k^{\sep})$ invariant as a subset. Note that $\Gamma_c$ stays the same after any inner twist of $G$.
	
	A cocycle $\varphi \in BG(k)$ \emph{breaks} $c$ if the action of $\Gamma_c$ on $c \subset  G_\varphi(-1)(k^{\sep})$ is not transitive, where $G_\varphi$ denotes the inner 	twist of $G$ by $\varphi$ (Definition \ref{def:inner_twist}).
	
	We say that a cocycle $\varphi \in BG(k)$ is \emph{breaking} if it breaks some conjugacy class.
\end{definition}

Before studying the formal properties of this definition, we give an example. Note that $c$ is breaking if and only if $c^n$ is breaking for any $n$ coprime to $|G|$. Thus being breaking is $\hat{\Z}^\times$-invariant and we can use the more explicit Galois action from $G(\cycl^{-1})$ (see Definition \ref{def:G(cycl)}).

\begin{example} \label{ex:S_3_breaking}
	Consider the case $G = S_3$. Let $\varphi: \Gamma_k \to S_3$ be a morphism. The group $G$ has two non-identity conjugacy classes.
	
	Consider $c = \{(1,2), (1,3), (2,3)\}$. We have $\Gamma_c = \Gamma_k$ and the action of $\Gamma_k$ on $c \subset G(\cycl^{-1}))$ is trivial.
	The group $\Gamma_k$ thus acts on $c \subset G_{\varphi}(\cycl^{-1})$ by factoring $\varphi$ through $S_3$ and letting $S_3$ act on $c$ by conjugation. The morphism $\varphi$ thus breaks $c$ if and only if the image of $\varphi$ is not transitive, in other words if the cubic \'etale algebra corresponding to $\varphi$ is not a field.
	
	Let now $c = \{(1,2,3), (1,3,2)\}$. We have $\Gamma_c = \Gamma_k$ and the action of $\Gamma_k$ on $c \subset G(\cycl^{-1})$ is the action on a $2$-element set which corresponds to the character $\chi: \Gamma_k \to \Z/2\Z$ defining $k(\zeta_3)/k$. The group $\Gamma_k$ thus acts on $c \subset G_{\varphi}(\cycl^{-1})$ via the character $\psi: \Gamma_k \to \Z/2\Z$ given by the sum of $\chi$ and the composition $\Gamma_k \xrightarrow{\varphi} S_3 \to \Z/ 2\Z$. The morphism $\varphi$ breaks $c$ if and only if $\psi$ is trivial. This only happens if $\chi$ is equal to the composition $\Gamma_k  \xrightarrow{\varphi} S_3 \to \Z/ 2\Z$, i.e.~if the quadratic resolvent of the cubic \'etale algebra corresponding to $\varphi$ is $k(\zeta_3)/k$, which for fields means exactly that the extension is a pure cubic extension. The need to consider pure cubic extensions separately when counting cubic extensions has been observed by Shankar and Thorne in \cite[Thm.~4]{ShTh22}.
\end{example}

This example illustrates the two ways of being breaking, at least for constant $G$: either the homomorphism is not surjective, or the corresponding field non-trivially intersects a cyclotomic field. This will be the main result of this section (Theorem~\ref{thm:breaking}).

We first demonstrate the relevance of breaking cocycles to Malle's conjecture, by showing that only breaking cocycles can be contained in accumulating thin subsets. It is stated in terms of the minimal weight conjugacy classes $\mathcal{M}(L)$ from Definition~\ref{def:C(L)}.

\begin{lemma} \label{lem:breaking_b}
	Let $L$ be a big orbifold line bundle on $BG$. Let $f: BH \to BG$ be a finite map with $a(L) = a(f^* L)$. Let $\varphi \in BH(k)$.
	\begin{enumerate}
		\item If $b(k,f^*L) > b(k,L)$ then 
		$f(\varphi)$ breaks a minimal weight conjugacy class.
		\item If $b(k,f^*L) = b(k,L)$ then either 		
		$f(\varphi)$ breaks a minimal weight conjugacy class or
		$\langle \mathcal{M}(L) \rangle \subseteq H$. 
	 \end{enumerate}
\end{lemma}
\begin{proof}
	Assume that $b(k,f^*L) \geq b(k,L)$.
	Consider the map $f_*: \mathcal{M}(f^*(L))/\Gamma_k \to 
	\mathcal{M}(L)/ \Gamma_k$. If $f_*$ is injective, then since 
	$b(k,f^*L) \geq b(k,L)$, we deduce that $f_*$ is a bijection and hence
	$b(k,f^*L) = b(k,L)$. The surjectivity of $f_*$ implies that $\mathcal{M}(L) \subset H(-1)$
	as in the statement and shows we must be in case (2).
	
	Now assume that $f_*$ is not injective.
	Let $\varphi \in BH(k)$ and let $H_{\varphi} \to G_{f(\varphi)}$ be the induced map of inner twists. The map $f$ is representable so $H_{\varphi} \to G_{f(\varphi)}$ is an inclusion. As $f_*$ is not injective, there exists a conjugacy class $c \in \mathcal{M}(L)$ such that $c \cap H_{\varphi}(-1)(k^{\sep})$ breaks into a collection of conjugacy classes upon which $\Gamma_c$ does not transitively. In particular $\Gamma_c$ does not act transitively on $c \subset  G_\varphi(-1)(k^{\sep}))$, so $f(\varphi)$ breaks $c$.
\end{proof}

In the case that $L$ is \textit{balanced}, the condition $\langle \mathcal{M}(L) \rangle \subseteq H$ implies that $BH \cong BG$, which is a trivial case. For unbalanced $L$ however, this condition covers more cases: for example any fibre of the Iitaka fibration will satisfy $b(k,f^*L) \geq b(k,L)$. Not all such cocycles should be removed (indeed, they cover $BG(k)$), but only those which give a strict inequality should be removed.

We next show that the collection of breaking cocycles is indeed thin.
\begin{lemma}\label{lem:Breaking_implies_reducible}
	Let $K/k$ be a finite extension such that $G_K$ is constant. Let $\varphi \in BG(k)$ be breaking. There exists a proper subgroup $H \subsetneq G_K$ such that $\varphi_{K(\mu_{\exp(G)})}$ lies in the image of $B H(K(\mu_{\exp(G)})) \to B G(K(\mu_{\exp(G)}))$, where $\exp(G)$ denotes the exponent of $G$.
	\end{lemma}
\begin{proof}
	Let $c \subset G_{\varphi}(-1)(k^{\sep})$ be the conjugacy class that $\varphi$ breaks and $\gamma \in c$. Note that $\Gamma_{K(\mu_{\exp(G)})} \subset \Gamma_K \cap \Gamma_c$ because $\Gamma_{K(\mu_{\exp(G)})}$ acts trivially on $G(-1)$. The action of $\sigma \in \Gamma_{K(\mu_{\exp(G)})} \subset \Gamma_K \cap \Gamma_c$ on $c \subset G(-1)$ is $\gamma \mapsto \varphi(\sigma) \gamma \varphi(\sigma)^{-1}$. It follows from the definition of breaking that the morphism $\varphi: \Gamma_{K(\mu_{\exp(G)})} \to G(k^{\sep})$ is not surjective, i.e.~factors through a proper subgroup $H \subset G(k^{\sep})$.
\end{proof}


The following summarises the results  of this section and is a generalisation of Theorem \ref{thm:breaking_intro}. Theorem \ref{thm:breaking}(1) should be viewed as an analogue of \cite[Thm.~1.4]{LST22} in the setting of $BG$. More than this however, we are able to give an explicit description of the breaking thin set, something which is currently lacking the setting of Manin's conjecture.

\begin{theorem} \label{thm:breaking}\hfill
\begin{enumerate}
	\item The collection of breaking cocycles is thin.
	\item If $G$ is constant, then the breaking homomorphisms in 
	$\Hom(\Gamma_k,G)$ are contained in the thin subset
	$$\Omega_{G,k}:= \left\{ \varphi: \Gamma_k \to G :
	\begin{array}{ll}
	 	\varphi \text{ is not surjective or } \\
	   k_{\varphi} \text{ is not linearly disjoint to } k(\mu_{\exp(G)})
	   \end{array} \right\}.$$
	\item In general, let $K/k$ be a splitting field of $G$. If $\varphi$ is breaking, 
	then $\varphi_K \in \Omega_{G_K,K}$, as in $(2)$.
\end{enumerate}
\end{theorem}
\begin{proof}
	The first statement follows by combining Lemmas \ref{lem:Breaking_implies_reducible} and \ref{lem:thin_base_change}. For the second,
	let $\varphi: \Gamma_k \to G$ be a breaking homomorphism. Then Lemma~\ref{lem:Breaking_implies_reducible} implies that $\varphi$ is no longer surjective when restricted to $\Gamma_{k(\mu_{\exp(G)})}$. So
	either $\varphi$ is not surjective to begin with, or it defines a field extension $k_{\varphi}$ which is not linearly disjoint to $k(\mu_{\exp(G)})$. The last part follows from the (2) and the fact
	that if $\varphi$ is breaking then so is $\varphi_K$.
\end{proof}

The set in part (2) of Theorem \ref{thm:breaking} is wasteful, in the sense that it may contain some non-breaking homomorphisms. Nonetheless, when counting, there is no problem	with removing a larger thin set (see Theorem \ref{thm:thin_negligable}).

\section{Stacky Hensel's Lemma} \label{sec:Hensel}
The main obstruction to extending arithmetic constructions such as heights to stacks is that proper stacks fail the valuative criterion of properness, i.e.~if $\mathcal{X}$ is a proper stack over $\Z_p$ then $\mathcal{X}(\Z_p)$ is not necessarily equal to $\mathcal{X}(\Q_p)$. In \cite[\S2.1]{ESZB} this issue is dealt with via the role of the ``tuning stack'', a certain global construction. We will instead use an (essentially equivalent) local construction, the \emph{arithmetic valuative criterion of properness} \cite[Thm.~3.2]{BV24}. We remark that this valuative criterion only holds for tame stacks, while the tuning stack exists in general. But the non-tame case is pathological in many ways.

\subsection{The stack $\Spec \mathcal{O}_{\sqrt[n]{v}}$} \label{sec:def_O_root_n}
The following construction plays a crucial role in the arithmetic of stacks.
\begin{definition}
	Let $\mathcal{O}$ be a henselian DVR with valuation $v$ and $n \in \N$. Define $\Spec \mathcal{O}_{\sqrt[n]{v}}$ as the root stack given by taking the $n$th root \cite[Def.~2.2.4]{Cad07} of the Cartier divisor defined by the maximal ideal of $\O$.
	
	We will use the following abbreviations: $\mathcal{X}(\mathcal{O}_{\sqrt[n]{v}}) := \mathcal{X}(\Spec \mathcal{O}_{\sqrt[n]{v}})$ for a stack $\mathcal{X}$ and $\H^p( \mathcal{O}_{\sqrt[n]{v}}, \mathcal{F}) := \H^p(\Spec \mathcal{O}_{\sqrt[n]{v}}, \mathcal{F})$ for an \'etale sheaf $\mathcal{F}$.
\end{definition}
Note that $\Spec \mathcal{O}_{\sqrt[n]{v}}$ is not equal to the spectrum of any ring, so the use of $\Spec$ is an abuse of notation. Intuitively one may think of this stack as a ramified extension of $K$ with field of fractions $K$ and ramification degree $n$.

We can now state the stacky arithmetic valuative criterion of properness \cite[Thm.~3.2]{BV24}, we note that this theorem seems to have been discovered independently by Darda and Yasuda \cite[Lem.~2.16]{DYBM}.

\begin{lemma}\label{lem:arithmetic_valuative_criterion}
	Let $\mathcal{X}$ be a tame proper stack over $\mathcal{O}$. For every point $x \in \mathcal{X}(K)$ there exists a unique $n$ and a representable map $\Spec \mathcal{O}_{\sqrt[n]{v}} \to \mathcal{X}$ lifting $x$. This lift is unique up to a unique $2$-isomorphism.
\end{lemma}

Let $\F$ denote the residue field of $\mathcal{O}$.
By the definition of root stacks, the stack $\mathcal{O}_{\sqrt[n]{v}}$ contains a Cartier divisor which is an $n$-th root of the Cartier divisor defined by the maximal ideal of $\mathcal{O}$. We will call this Cartier divisor the special point and denote it by $\Spec \F_{\sqrt[n]{v}}$. There exists a non-canonical isomorphism $\Spec \F_{\sqrt[n]{v}} \cong (B \mu_n)_{\F}$. Let us be more explicit.

Let $\pi \in \mathcal{O}$ be a uniformiser. By \cite[Ex.~2.4.1]{Cad07} we have $\Spec \mathcal{O}_{\sqrt[n]{v}} \cong [\Spec \mathcal{O}[s^n - \pi]/ \mu_n]$ where $\mu_n$ acts by multiplication on $s$. The stack $\Spec \F_{\sqrt[n]{v}}$ is given by the zero locus of $s$ under this isomorphism, which is equal to $[\Spec \F/ \mu_n] = (B\mu_n)_{\F}$. We denote the composition $(B\mu_n)_{\F} \to  [\Spec \mathcal{O}[s^n - \pi]/ \mu_n] \cong \Spec \mathcal{O}_{\sqrt[n]{v}}$ by $i_{\pi}$.

The following lemma describes the dependence of $i_{\pi}$ on the choice of $\pi$.
\begin{lemma}
	For $u \in \mathcal{O}^{\times}$ consider the class $\overline{u} \in \F^{\times}/\F^{\times n} \cong \H^1(\F, \mu_n)$. Let $\emph{tr}_{\overline{u}}: (B \mu_n)_{\F} \to (B \mu_n)_{\F}$ denote the translation map which sends the identity cocycle $e$ to $\overline{u}$. Then there exists a $2$-commutative triangle
	\[\begin{tikzcd}
		{(B\mu_n)_{\F}} & {\mathcal{O}_{\sqrt[n]{v}}} \\
		{(B\mu_n)_{\F}}
		\arrow["{i_{u\pi}}", from=1-1, to=1-2]
		\arrow["{\emph{tr}_{\overline{u}}}", from=1-1, to=2-1]
		\arrow["{i_{\pi}}"', from=2-1, to=1-2]
	\end{tikzcd}\]
\end{lemma}
\begin{proof}
	Unfolding the definition of $i_{\pi}$ we can describe it as follows. Let $S \to \Spec \F$ be an $\F$-scheme. A point of $B\mu_n(S)$ corresponds to a $\mu_n$-torsor, which by Kummer theory \cite[\href{https://stacks.math.columbia.edu/tag/03PK}{Tag 03PK}]{stacks-project} is the same as a pair $(\mathcal{L}, \eta)$, where $\mathcal{L}$ is a line bundle on $S$ and $\eta: \mathcal{L}^n \cong \mathcal{O}_S$ is an isomorphism. The map $i_{\pi}$ sends this to the pair $(\mathcal{L}, \pi \cdot \eta)$ where $\pi \cdot \eta: \mathcal{L}^n \cong \mathcal{O}_S \xrightarrow{s \to \pi \cdot s} \pi \cdot \mathcal{O}_S$ is an isomorphism, (that this defines an element of $\Spec  {\mathcal{O}_{\sqrt[n]{v}}}$ follows from the definition.)
	
	On the other hand, the map $\text{tr}_{\overline{u}}$ sends $(\mathcal{L}, \eta)$ to $(\mathcal{L}, \overline{u} \cdot \eta)$. The commutativity is then clear.
\end{proof}
\subsection{Quotients and twisting}
We will encounter quotient stacks in the next section. We will require a different description of quotient stacks in terms of twists torsors; as far as we are aware, this description is new.

\subsubsection{Twisting torsors}
We first recall some facts on  twisting torsors. The following is based upon \cite[p.~20--22]{Sko01}.
Let $S$ be a scheme, $G$ a group scheme over $S$ acting on the left on an $S$-scheme $X$. Let $a:P \to S$ be a right $G$-torsor. It is also a left-torsor under the inner twist $G_a$ \cite[p.~20 Ex.~1]{Sko01} of $G$ by $a$.
	
The \emph{twist} of $X$ by $a$ is the quotient stack $X_a := [P \times X/ G]$ where $g \in G$ acts on the left by $g \cdot (p,x) = (p \cdot g^{-1}, g \cdot x)$. This is an algebraic space as stabilisers are trivial. The left $G_a$-action on $P$ induces a left $G_a$-action on $X_a$. Twisting is functorial. Note also that if $a':P' \to S$ is a $G$-torsor and $P \to P'$ is a a map then we get an induced map $X_a \to X_{a'}$.

The torsor $P$ also has a left $G$-action given by $g \cdot p := p \cdot g^{-1}$ and we can thus consider $P_a$. The diagonal map $P \to P \times P$ induces a section $S \cong [P/G] \to [P \times P/G] = P_a$ which leads to a canonical isomorphism $P_a \cong G_a$. 

The torsor $P$ has a left $G_a$-action and thus a right $G_a$-action given by $p \cdot g = g^{-1} p$. This makes $P \to S$ a right $G_a$-torsor, which we will call the \emph{inverse torsor} and denote by $a^{-1}$. As above we have a canonical isomorphism $P_{a^{-1}} \cong G$ which leads to a canonical isomorphism $(X_a)_{a^{-1}} \cong X$. In particular $(P_a)_{a^{-1}} \cong P$.

\subsubsection{Quotient stacks}
Recall \cite[\href{https://stacks.math.columbia.edu/tag/04UV}{Tag 04UV}]{stacks-project} that for any scheme $T \to S$ the groupoid $[X/G](T)$ consists of a pair $(a,f)$ of a right $G$-torsor $a: P \to T$ equipped with a map $f: P \to X$ such that $f(p \cdot g^{-1}) = g \cdot f(p)$. Morphisms $(a,f) \to (a', f')$ are given by a commutative triangle 
\[\begin{tikzcd}
	P & X \\
	P'
	\arrow["f", from=1-1, to=1-2]
	\arrow[from=1-1, to=2-1]
	\arrow["f'"', from=2-1, to=1-2].
\end{tikzcd}\]
Consider the composition $T \to P_a \xrightarrow{f_a} X_a$. This procedure defines a function of groupoids natural in $T$
\begin{equation}\label{eq:points_quotient_stacks}
	[X/G](T) \to \{(a, x): P \to T \text{ right $G$-torsor} \text{ and } x \in X_a(T) \}
\end{equation}
A morphism $(a, x) \to (a', x')$ in the groupoid on the right is a map of $G$-torsors $p: P \to P'$ such that the image of $x$ via the induced map $X_a \to X_{a'}$ is $x'$.

\begin{lemma}\label{lem:points_quotient_stacks}
	The map \eqref{eq:points_quotient_stacks} is an equivalence of groupoids for all $T \to S$.
\end{lemma}
\begin{proof}
		An inverse of this map is given by untwisting, i.e. twisting by $a^{-1}$.
\end{proof}
This alternative description of $[X/G](T)$ will be quite useful but we were surprisingly unable to find it in the literature. 

\subsubsection{Maps of quotient stacks}\label{sec:map_quotient_stack}
We will now consider a construction of maps of quotient stacks. Let $G$ and $G'$ be group schemes over $S$ acting on $S$-schemes $X$ and  $X'$ respectively. Consider a commutative diagram
\[\begin{tikzcd}
	{G \times X} & X \\
	{G' \times X'} & {X'}.
	\arrow[from=1-1, to=1-2]
	\arrow[from=1-1, to=2-1]
	\arrow[from=1-2, to=2-2]
	\arrow[from=2-1, to=2-2]
\end{tikzcd}\]
This defines a map from the groupoid scheme $G \times X \rightrightarrows X$ to $G' \times X' \rightrightarrows X'$ in the sense of \cite[\href{https://stacks.math.columbia.edu/tag/0230}{Tag 0230}]{stacks-project} and thus induces a map on quotient stacks $[X/G] \to [X'/G']$.
\subsection{Sectors and the cyclotomic inertia stack}\label{sec:cylotomic_inertia}
The arithmetic valuative criterion of properness will allow us to consider the ``reduction modulo $\pi$'' of an element $x \in \mathcal{X}(K)$. In this section we will describe the space where this reduction lives. This stack is known as the \emph{cyclotomic inertia stack} and was first described in \cite[\S3]{AGV08}, Darda and Yasuda \cite[Def.~2.3]{DYBM} call this the stack of twisted $0$-jets. Since we will use integral models we extend the definition to general base schemes.

\begin{definition} \label{def:sector}
	Let $S$ be a Noetherian base scheme and $\mathcal{X} \to S$ a finite type tame DM stack over $S$. The \emph{cyclotomic inertia stack} is the disjoint union $I_{\mu} \mathcal{X} := \coprod_{n} \Hom_{S, \text{rep}}(B\mu_n, \mathcal{X})$, where $\Hom_{S, \text{rep}}(B\mu_n, \mathcal{X})$ is the stack parametrizing representable morphisms $(B \mu_n)_S \to \mathcal{X}$ over $S$.
	
	A \emph{sector} $\mathcal{S} \in \pi_0(I_{\mu} \mathcal{X})$ of $\mathcal{X}$ is a connected component of $I_{\mu} \mathcal{X}$. 
	
	The \emph{order} of a sector $\mathcal{S}$ is the number $n$ such that $\mathcal{S} \subset \Hom_{S, \text{rep}}(B\mu_n, \mathcal{X})$.
	
	A \emph{trivial sector} is a sector of order $1$. If $\mathcal{X}$ is connected,
	then there is a unique trivial sector, namely $\mathcal{X} \cong \Hom_{S, \text{rep}}(B\mu_1, \mathcal{X}) \subset I_{\mu} \mathcal{X}$.
	
	If $k$ is a field and $S = \Spec k$ then a \emph{geometric sector} is a sector of $\mathcal{X}_{k^{\sep}} \to \Spec k^{\sep}$.
\end{definition}

If $\mathcal{X}$ is smooth over $S$ then so is $\mathcal{I}_{\mu}(\mathcal{X})$ \cite[Cor.~3.1.4]{AGV08}. Let $G$ be a tame finite \'etale group scheme over $S$. We will now construct a map
\[
I_{\mu}BG \to [G(-1)/G]
\]
where the left action of $G$ on $G(-1)$ is given by conjugation.

For each scheme $T \to S$ an element of $I_{\mu}BG(T)$ is represented by a representable map $f: (B\mu_n)_T \to (BG)_T$ of $T$-stacks. When we compose $f$ with the map $e: T \to (B \mu_n)_T$ we get a map $\varphi := f \circ e: T \to (BG)_T$. 

The map $f$ is representable so defines an injective map on automorphism schemes $f_*: (\mu_n)_T =  \underline{\Aut}_{B\mu_n}(e) \to \underline{\Aut}_{BG}(f) \cong (G_{T})_{\varphi}$, where the last isomorphism is Lemma \ref{lem:automorphism_group_is_inner_twist} (technically Lemma \ref{lem:automorphism_group_is_inner_twist} only applies when $T$ is the spectrum of a field, but the general case is the same).

Twisting commutes with $\Hom$-schemes, as both products and quotients do, so we find that
\[
f_* \in \Hom(\mu_n, G_{\varphi})(T) = \Hom(\mu_n, G)_{{\varphi}}(T) \subset G(-1)_{\varphi}(T).
\]
The pair $(\varphi, f_*)$ defines an element of $[G(-1)/G](T)$ by Lemma \ref{lem:points_quotient_stacks}. The above construction is functorial in $T$ and thus defines a map $I_{\mu} BG \to [G(-1)/G]$.

\begin{proposition}\label{prop:cyclotomic_inertia_stack_BG}
	The map $I_{\mu}(BG) \to [G(-1)/G]$ constructed above is an isomorphism.
\end{proposition}
\begin{proof}
	We will construct an explicit inverse.
	
	Let $T \to S$ be a morphism. We may assume that $T$ is connected by descent. By Lemma \ref{lem:points_quotient_stacks}, an element of $[G(-1)/G](T)$ is  given by a pair of a torsor $\varphi: P \to T$ and an element $\gamma: T \to G(-1)_{T, \varphi}$. As twisting commutes with Hom schemes the map $\gamma$ is represented by a map $(\mu_n)_T \to G_{T,\varphi}$; by the connectivity of $T$ we may assume that this map is injective. Let $(B \mu_n)_T \to BG_{T,\varphi}$ be the induced morphism, which representable by injectivity.
	
	Lemma \ref{lem:automorphism_group_is_inner_twist} induced an isomorphism $BG_{T,\varphi} \cong BG$ which sends $e$ to $\varphi$. The composition $(B \mu_n)_T \to BG_{T,\varphi} \cong BG$ defines an element of $I_\mu BG$. As everything was functorial in $T$ this leads to a map $[G(-1)/G] \to I_{\mu} BG$.
	
	That this map is an inverse to the map $I_{\mu} BG \to [G(-1)/G]$ is clear after unfolding all the definitions.
%
%
\end{proof}

The following corollary, which was shown directly by Darda-Yasuda \cite[Ex~2.15]{DYBM}, provides a geometric explanation for the role of $\mathcal{C}_G$ in Malle's conjecture.
\begin{corollary}\label{cor:components_cyclotomic_inertia_stack}
	Let $S = \Spec k$ be the spectrum of a field. The map in Proposition~\ref{prop:cyclotomic_inertia_stack_BG} induces an isomorphism $\pi_0(I_{\mu} BG_{k^{\sep}} ) \cong \mathcal{C}_G$ of $\Gamma_k$-sets.
\end{corollary}
\begin{proof}
	We have natural identifications of sets \[\pi_0([G(-1)/G]_{k^{\sep}}) = \pi_0(G(-1)_{k^{\sep}})/G(k^{\sep}) = G(-1)(k^{\sep})/G(k^{\sep}) = \mathcal{C}_G.\] These identifications are $\Gamma_k$-equivariant because they are natural.
\end{proof}
\begin{remark}
	The stack $I_{\mu} \mathcal{X}$ should really be called the \emph{anticyclotomic} inertia stack
	since, as we have seen, it gives a Tate twist by $(-1)$ of the inertia stack. We have 
	kept the terminology from \cite[\S3]{AGV08} for brevity and consistency.
\end{remark}

The following lemma gives a complete description of the sectors of $BG$ over a normal base $S$. 
Let $S$ be an integral normal scheme with fraction field $k$, important cases are when $S$ is the spectrum of a field or the spectrum of the ring of integers of a local or global field. Note that there is a surjection $\Gamma_k \to \pi_{1}(S)$ by normality. Let $G$ be a finite \'etale tame group scheme over $S$; the $\Gamma_k$-action on $G(k^{\sep})$ must then factor through $\pi_1(S)$. The same holds on $G(-1)(k^{\sep})$ since $G$ is tame. The following lemma describes the sectors of $BG$ explicitly.
\begin{lemma}\label{lem:sectors_BG}
	\hfill
	\begin{enumerate} 
		\item The $\Gamma_k$-action on $\mathcal{C}_G$ factors through $\pi_{1}(S)$ and the set $\pi_0(I_{\mu} BG)$ is naturally identified with the quotient $\mathcal{C}_G/\pi_1(S)$.
		\item Let $c \in \mathcal{C}_G/\pi_1(S)$. The elements of $G(-1)(k^{\sep})$ which lie in a conjugacy class of the orbit $c$ are $\pi_1(S)$-stable and thus define a closed subscheme $\mathcal{C}_c \subset G(-1)$. 
		
		The sector corresponding to $c$ is $\mathcal{S}_c = [\mathcal{C}_c/G] \subset [G(-1)/G]$. The order $n$ of $\mathcal{S}_c$ is the order of any element of $\mathcal{C}_c(k^{\sep})$.
		\item
		The universal map $(B \mu_n)_{\mathcal{S}_c} \cong [\mathcal{C}_c/G \times \mu_n]\to BG$ is induced by the following commutative diagram, as in \S\ref{sec:map_quotient_stack}
		\begin{equation}\label{eq:universal_map_sector}
			\begin{tikzcd}
				{G \times_S \mathcal{C}_c \times \mu_n} & {\mathcal{C}_c} \\
				G & S
				\arrow[from=1-1, to=1-2]
				\arrow["{(g, h, \zeta) \to g h(\zeta)}"', from=1-1, to=2-1]
				\arrow[from=1-2, to=2-2]
				\arrow[from=2-1, to=2-2].
			\end{tikzcd}
		\end{equation}
	\end{enumerate}
\end{lemma}
\begin{proof}
	We will use the identification in Proposition \ref{prop:cyclotomic_inertia_stack_BG} and the Galois correspondence between finite \'etale schemes over $S$ and finite $\pi_1(S)$-sets freely in this proof.
	
	(1) The connected components of $[G(-1)/G]$ are given by $G$-orbits of the connected components of $G(-1)$, which by the Galois correspondence are given by orbits $G(-1)(k^{\sep})/\pi_1(S)$. The set of $G$-orbits is then equal to $\mathcal{C}_G/\pi_1(\mathcal{S})$, which is (1).
	
	(2) The same argument also shows that the pre-image of the connected component $\mathcal{S}_c$ corresponding to $c \in \mathcal{C}_G/\pi_1(\mathcal{S})$ under the map $G(-1) \to [G(-1)/G]$ is exactly $\mathcal{C}_c$, which implies that $\mathcal{S}_c \cong [\mathcal{C}_c/G]$
	
	Note that by construction the map $I_{\mu} BG \to [G(-1)/G]$ sends representable maps $B \mu_n \to BG$ to the closed substack $[G(-1)_n/G]$, where $G(-1)_n$ corresponds to the elements of order $n$ in $G(-1)(k^{\sep}) = G(k^{\sep})$. This implies (2).
	
	(3) We will show that the map induced by the commutative square is the same as the map given by applying the inverse constructed in Proposition \ref{prop:cyclotomic_inertia_stack_BG} to the element of $[G(-1)/G](\mathcal{S}_c)$ given by $[\mathcal{C}_c/G] \subset [G(-1)/G]$.
	
	The torsor corresponding to this element is $\varphi: \mathcal{C}_c \to[\mathcal{C}_c/G]$. The twist $G_{\mathcal{S}_c(-1), \varphi}$ is thus given by $[\mathcal{C}_c \times_S G(-1)/ G]$ where the action is given by $g \cdot(\gamma, \xi) = (g \gamma g^{-1}, g \xi g^{-1})$. The element $\mathcal{S}_c \to G(-1)_{\mathcal{S}_c, \varphi}$ is then given by the diagonal map $[\mathcal{C}_c/G] \subset [\mathcal{C}_c \times_S G(-1)/ G]$.
	
	This diagonal map corresponds to the injective map $(\mu_n)_{\mathcal{S}_c} \to G_{\mathcal{S}_c, \varphi}$ given by 
	$[\mathcal{C}_c \times \mu_n/G] \subset [\mathcal{C}_c \times G/ G]$ induced by $\mathcal{C}_c \times \mu_n \to \mathcal{C}_c \times_S G: (\gamma, \zeta) \to (\gamma, \gamma(\zeta))$.
	
	This induces the map $(B\mu_n)_{\mathcal{S}_c} \to BG_{\mathcal{S}_c, \varphi}$ given by the commutative square
	\begin{equation}\label{eq:universal_map_sector_proof_1}
		\begin{tikzcd}
			{G \times_S \mathcal{C}_c \times \mu_n} & {\mathcal{C}_c} \\
			{G \times_S  \mathcal{C}_c \times_S G} & {\mathcal{C}_c}
			\arrow[from=1-1, to=1-2]
			\arrow["{(g, \gamma, \zeta) \to (g, \gamma ,\gamma(\zeta))}"', from=1-1, to=2-1]
			\arrow[from=1-2, to=2-2]
			\arrow[from=2-1, to=2-2]
		\end{tikzcd}
	\end{equation}
	On the other hand, the isomorphism $BG_{\mathcal{S}_c, \varphi} \cong BG_{\mathcal{S}_c}$ is induced by the commutative diagram 
	\begin{equation}\label{eq:universal_map_sector_proof_2}
		\begin{tikzcd}
			{G \times_S  \mathcal{C}_c \times_S G} & {\mathcal{C}_c} \\
			G  & S
			\arrow[from=1-1, to=1-2]
			\arrow["{(g, \gamma, h) \to g h}"', from=1-1, to=2-1]
			\arrow[from=1-2, to=2-2]
			\arrow[from=2-1, to=2-2]
		\end{tikzcd}
	\end{equation}
	Composing \eqref{eq:universal_map_sector_proof_1} and \eqref{eq:universal_map_sector_proof_2} gives \eqref{eq:universal_map_sector}, which is what we had to show.
\end{proof}

\subsection{Reduction modulo $\pi$}\label{sec:modulo_pi}
Let $\mathcal{O}$ be a DVR with valuation $v$, uniformiser $\pi$, fraction field $K$, and residue field $\F$ which has characteristic $p$ and $\mathcal{X} \to \Spec \mathcal{O}$ a tame proper stack. We can now describe the operation of taking elements $x \in \mathcal{X}(K)$ modulo $\pi$.

Lemma \ref{lem:arithmetic_valuative_criterion} implies that any point $x \in X(K)$ extends uniquely (up to a unique $2$-isomorphism) to a representable map $x: \Spec \mathcal{O}_{\sqrt[n]{v}} \to \mathcal{X}$ for some $n$. Composing with the map $i_{\pi}$ (see \S\ref{sec:def_O_root_n}) defines a representable map $(B \mu_n)_{\F} \to \mathcal{X}$, i.e.~an $\F$-point of $\Hom_{\F, \text{rep}}(B\mu_n, \mathcal{X}_{\F})(\F) \subset I_{\mu} \mathcal{X}(\F)$. This procedure defines a map of groupoids 
\[
\pmod{\pi}: X(K) \to I_{\mu} \mathcal{X}(\F)
\]
which we call the \emph{reduction modulo $\pi$}. 
\begin{remark}
	If $\mathcal{X}$ is a scheme then $I_{\mu} \mathcal{X} = \mathcal{X}$ and the above map is just the usual reduction modulo $\pi$.
	
	The map $\pmod{\pi}$ in general depends on the choice of uniformiser $\pi$, unlike the case of schemes.
\end{remark}
\subsection{The case of $BG$}
Let $\mathcal{G}$ be a finite \'etale group scheme over $\mathcal{O}$ of order coprime to $p$ and let $G$ be its generic fibre. We can give an explicit description of the reduction modulo $\pi$ map on $BG$.

Let us first describe the image $I_{\mu}BG(\F) \cong [G(-1)/G](\F)$ explicitly.
\begin{lemma}\label{lem:groupoid_twisted_inertia_stack}
	Let $G$ be a finite \'etale tame group scheme over $\F$. The groupoid $[G(-1)/G](\F)$ is equivalent to the following groupoid.
	
	The objects are pairs $(\gamma, \varphi)$ where $\gamma \in G_{\varphi}(-1)(\F)$ and $\varphi: \Gamma_{\F} \to G(\F^{\sep})$ is a cocycle. For each $g \in G$ there is a morphism $(\gamma, \varphi) \to (g \gamma g^{-1}, g \varphi g^{-1})$ and we consider the obvious composition law.
\end{lemma}
\begin{proof}
	This follows immediately from Lemma \ref{lem:points_quotient_stacks}.
\end{proof}
Let us now recall the structure of $\Gamma_K$, cf.~\cite[Tag 09E3]{stacks-project}. Let $\Gamma_k^{\text{tame}}$ be Galois group of the maximal tamely ramified extension of $K$. The kernel of $\Gamma_K \to \Gamma_k^{\text{tame}}$ is a pro-$p$-group.
Let $I_K^{\text{tame}} \subset \Gamma_k^{\text{tame}}$ be the tame inertia subgroup. We have a short exact sequence
\[
1 \to I_K^{\text{tame}} \to \Gamma_K^{\text{tame}} \to \Gamma_{\F} \to 1.
\]
Let $\pi$ be a uniformiser. All tamely ramified extensions of the field $\bigcup_{p \nmid n} K[\sqrt[n]{\pi}]$ are unramified so it defines a splitting $s_{\pi}: \Gamma_{\F} \cong \Gamma_{\bigcup_{p \nmid n} K[\sqrt[n]{\pi}]}^{\text{tame}} \subset \Gamma_K^{\text{tame}}$.
 
Moreover, by \cite[Tag 09EE]{stacks-project} there is a canonical isomorphism 
\begin{equation}
I_k^{\text{tame}} \cong \widehat{\Z}(1)(p') := \varprojlim\limits_{n, p \nmid n} \mu_n. \label{eqn:tame_inertia}
\end{equation}

\begin{lemma}\label{lem:residue_map_BG}
	Let $G$ be a finite \'etale group scheme over $\mathcal{O}$ of order coprime to the characteristic $p$ of $\F$. Let $\varphi: \Gamma_k \to G(k^{\sep})$ be a cocycle.
	
	The fact that $G$ is finite \'etale over $\mathcal{O}$ implies that the action of $\Gamma_{K}$ on $G(k^{\sep})$ factors through $\Gamma_{\F}$. The restriction to $I_k^{\emph{tame}} \cong \hat{\Z}(1)(p')$ thus defines an element $I_{\varphi} \in \Hom(\widehat{\Z}(1)(p'), G(k^{\sep})) = G(-1)(k^{\sep})$. 
	
	We then have that $\varphi \pmod{\pi} = (I_{\varphi}, \varphi \circ s_{\pi})$, where we use the description of $I_{\mu} BG(\F)$ from Lemma \ref{lem:groupoid_twisted_inertia_stack}.
\end{lemma}
\begin{proof}
	Let $f: \Spec \O_{\sqrt[n]{v}} \to BG$ be a representable map for some $n \in \N$. Since $G$ has order coprime to $p$ we must have $p \nmid n$.
	
	Let $\mathcal{O}^{\text{unr}}$ be the maximal unramified extension of $\mathcal{O}$ and $K^{\text{unr}}$ its fraction field. The map $\Spec \mathcal{O}^{\text{unr}}[\sqrt[n]{\pi}] \to \Spec \mathcal{O}_{\sqrt[n]{v}}$ is a limit of finite \'etale maps and $\O^{\text{unr}}[\sqrt[n]{\pi}]$ is strictly henselian. This implies that
	 \[
	 \pi_1(\Spec \O_{\sqrt[n]{v}}) \cong \Aut(\Spec \O^{\text{unr}}[\sqrt[n]{\pi}]/ \Spec \O_{\sqrt[n]{v}}) = \Gal(K^{\text{unr}}[\sqrt[n]{\pi}]/K).
	 \]
	 Let $\varphi$ be the cocycle $\pi_1(\Spec \mathcal{O}_{\sqrt[n]{v}}) \to G(\mathcal{O}^{\text{unr}})$ representing $f$. The composition 
	 \[
	 \Gamma_K \to \Gal(K^{\text{unr}}[\sqrt[n]{\pi}]/K) \cong \pi_1(\Spec \mathcal{O}_{\sqrt[n]{v}}) \xrightarrow{\varphi} G(\mathcal{O}^{\text{unr}}) = G(k^{\sep})
	 \]
	 is the cocycle representing the generic fibre $f_K$ of $f$.
	 
	 Consider the following diagram
	 \begin{equation}\label{eq:torsor_special_fibre_diagram}
	 	\begin{tikzcd}
	 		{\Spec \F} & {\Spec \O[\sqrt[n]{\pi}]} & {\Spec \O[\sqrt[n]{\pi}]} \\
	 		{(B\mu_n)_{\F}} & {[\Spec \O[\sqrt[n]{\pi}]/\mu_n]} & {\Spec \O_{\sqrt[n]{v}}}
	 		\arrow[from=1-1, to=1-2]
	 		\arrow["e", from=1-1, to=2-1]
	 		\arrow["{=}"{description}, draw=none, from=1-2, to=1-3]
	 		\arrow[from=1-2, to=2-2]
	 		\arrow[from=1-3, to=2-3]
	 		\arrow[from=2-1, to=2-2].
	 		\arrow["\cong"{description}, draw=none, from=2-2, to=2-3]
	 	\end{tikzcd}
	 \end{equation}
	 We claim that the horizontal maps in this diagram are maps of $\mu_n$-torsors, and particular this means that this diagram commutes. Indeed, this is clear for both the right and left squares.
	 
	 Consider $(g, \overline{\varphi}) := \varphi \pmod{\pi}$. The cocycle  $\overline{\varphi} \in BG(\Spec \F)$ is given by the composition of 
	 \[
	 \Spec \F \xrightarrow{e} (B\mu_n)_{\F} \subset [\Spec \O[\sqrt[n]{\pi}]/\mu_n] \cong \Spec \O_{\sqrt[n]{v}} \xrightarrow{f} BG.
	 \] 
	 The commutativity of the diagram implies that the cocycle $\overline{\varphi}$ is given by the composition
	 \[
	 \Gamma_{\F} \cong \Aut(\O^{\text{unr}}[\sqrt[n]{\pi}]/ \O[\sqrt[n]{\pi}]) \subset \Gal(K^{\text{unr}}[\sqrt[n]{\pi}]/K) \xrightarrow{\varphi} G(k^{\sep}).
	 \]
	 This is equal to $\varphi \circ s_{\pi}$ by definition.
	 
	 The element $g \in G(-1)_{\overline{\varphi}}$ is given by the induced map $\mu_n \to \underline{\Aut}(\overline{\varphi}) \cong G_{\overline{\varphi}}$ on automorphism sheaves. To compute this we may base change diagram \eqref{eq:torsor_special_fibre_diagram} to $\mathcal{O}^{\text{unr}}$. In this case we see by the diagram that $g$ comes from the $\mu_n$-torsor $\Spec \O^{\text{unr}}[\sqrt[n]{\pi}] \to \Spec \O^{\text{unr}}_{\sqrt[n]{v}}$. This $\mu_n$-torsor at the generic point is $K^{\text{unr}}[\sqrt[n]{\pi}]/K^{\text{unr}}$. The map $g$ is thus induced by the composition 
	 \[
	 \mu_n = \Gal(K^{\text{unr}}[\sqrt[n]{\pi}]/K^{\text{unr}}) \cong  \Gal(K^{\text{unr}}[\sqrt[n]{\pi}]/K)  \xrightarrow{\varphi} G(k^{\sep}).
	 \]
	 This is equal to $I_{\varphi}$ by definition.
\end{proof}

\subsection{Stacky Hensel's Lemma}
Let $\mathcal{X}$ be a smooth proper stack over a henselian DVR $\O$. Then the usual version of Hensel's lemma says that $\mathcal{X}(\O) \to \mathcal{X}(\F)$ is surjective (see e.g.~\cite[Lem.~4.2]{LS23} for the case where $\O$ is complete). If $\mathcal{X}$ is a scheme then this automatically applies to $\mathcal{X}(K) = \mathcal{X}(\O)$. However in general we have $\mathcal{X}(K) \neq \mathcal{X}(\O)$ due to the failure of the usual valuative criterion for properness for stacks. The following can be viewed as generalisation of Hensel's lemma for proper $0$-dimensional stacks in light of this failure. It says that $\mathcal{X}(K)$ can still be related to the $\F$-points of a stack, not $\mathcal{X}$ itself, but rather the cyclotomic inertia stack. It shows that the cyclotomic inertia stacks plays the role of a smooth proper compactification of $\mathcal{X}$ (even though $\mathcal{X}$ itself is already smooth and proper!). This result highlights the arithmetic significance of the cyclotomic inertia stack and will be used in the proof of our generalisation of Bhargava's mass formula (Corollary~\ref{cor:mass_formula}). The following is a slightly more general version of Theorem \ref{thm:Hensel_intro}.

\begin{theorem} \label{thm:Hensel}
	Let $\mathcal{X} \to \Spec \mathcal{O}$ be proper, tame, \'etale DM stack. Then 
	the map $\pmod{\pi}: \mathcal{X}(K) \to I_{\mu} \mathcal{X}(\F)$ is an equivalence of groupoids.
\end{theorem}
\begin{proof}
	We may assume without loss of generality that $\mathcal{X}$ is connected. If $\mathcal{O}$ is strictly henselian then $\mathcal{X} \cong BG$ where $G$ a finite constant group. The stack $\mathcal{X}$ is tame so the order of $G$ has to be coprime to the characteristic $p$ of $\F$.
	
	By \eqref{eqn:tame_inertia}, the prime to $p$ part of $\Gamma_F$ is isomorphic to $\widehat{\Z}(1)(p')$ so $BG(K)$ is the groupoid of maps $\varphi: \widehat{\Z}(1) \to G$ up to conjugacy. On the other hand, by Proposition~\ref{prop:cyclotomic_inertia_stack_BG} the stack $I_{\mu} BG(\F)$ is equal to the action groupoid $[G(-1)/G]$, where $G$ acts on $G(-1)$ by conjugation, since $\F$ is algebraically closed. The map $\pmod{\pi}$ is given by the identification $\Hom(\widehat{\Z}(1), G) = G(-1)$ by Lemma \ref{lem:residue_map_BG} and this is clearly an equivalence of groupoids.
	
	In general let $\mathcal{O}^{\text{sh}}$ be the strict henselization with fraction field $\F^{\sep}$ and fraction field $K^{\text{sh}}$. The map of groupoids $\mathcal{X}(K^{\text{sh}}) \to I_{\mu} \mathcal{X}(\F^{\sep})$ is an equivalence and preserves the $\Gamma_{\F} = \Gal(K^{\text{sh}}/K)$-action, so the map $\mathcal{X}(K) \to I_{\mu} \mathcal{X}(\F)$ is an equivalence of groupoids by descent.
\end{proof}
\begin{remark}
	If $G$ is constant and $K$ is a non-archimedean local field with residue field $\F_q$, then Theorem \ref{thm:Hensel} recovers the standard description of $BG(K)$ used in e.g.~\cite[Proof of Prop.~5.3]{Ked07}. Indeed, after making explicit the Galois action on the inner twist in Lemma \ref{lem:groupoid_twisted_inertia_stack} we see that $[G(-1)/G](\F_q)$ is the groupoid of pairs $(\gamma, g)$ with $\gamma^q = g^{-1} \gamma g$ up to conjugation. Recall that $\Gamma_K^{\text{tame}}$ is topologically generated by two elements $x, y$ subject to the relation $x^q = y^{-1} x y$, where $y$ is a lift of Frobenius and $x$ generates the inertia group. Then it follows from Lemma \ref{lem:residue_map_BG} that the map $\pmod{\pi}: BG(K) \to [G(-1)/G](\F_q)$ sends a morphism $\varphi: \Gamma_K \to G$ to the pair $(\varphi(x), \varphi(y))$.
\end{remark}

\section{Unramified Brauer groups of stacks} \label{sec:Br}
The \emph{(cohomological) Brauer group of an algebraic stack} $\mathcal{X}$ is $\Br \mathcal{X} := \H^2(\mathcal{X}, \Gm)$. The aim of this section is to define the unramified Brauer group $\Br_{\text{un}} \mathcal{X} \subset \Br \mathcal{X}$. Even if $\mathcal{X}$ is smooth and proper, this may be different to the Brauer group of $\mathcal{X}$. It is the unramified Brauer group of $\mathcal{X}$ which plays a role in the Brauer--Manin obstruction to weak approximation and the Hasse principle. A version of this appears in the leading constant for Malle's conjecture.

\subsection{Purity for stacks}
We will first prove a version of Grothendieck purity \cite[Thm.~3.7.1]{Col21} for stacks, for which we were unable to find a reference.

\begin{lemma}\label{lem:purity_codimenion_2}
	Let $\mathcal{X}$ be a regular algebraic stack.
	Let $\mathcal{Z} \subset \mathcal{X}$ be a closed substack of codimension $\geq 2$. Then $\H^q(\mathcal{X}, \Gm) = \H^q(\mathcal{X} \setminus \mathcal{Z}, \Gm)$ for $q = 0,1,2$.
\end{lemma}
\begin{proof}
	Let $i: \mathcal{X} \setminus \mathcal{Z} \to \mathcal{X}$ be the inclusion. We claim that $i_* \Gm = \Gm, \Res^1i_* \Gm = \Res^2 i_* \Gm = 0$. Indeed, this is local on $\mathcal{X}$ so we may assume that $\mathcal{X}$ is the spectrum of a strictly henselian ring where this follows from Grothendieck purity \cite[Thm.~3.7.6]{Col21}. The first part then follows from the Leray spectral sequence.
\end{proof}

\begin{proposition}\label{prop:Grothendieck_purity}
	Let $\mathcal{X}$ be a regular algebraic stack.
	Let $j: \mathcal{D} \to \mathcal{X}$ be a smooth divisor and let $i: \mathcal{X} \setminus  \mathcal{D} \to \mathcal{X}$ be the inclusion of the complement. Consider the map $i_*\Gm \to j_* \Z$ of sheaves on $\mathcal{X}$ which is smooth locally on $f: U \to \mathcal{X}$ given by sending $t \in \Gm(U \setminus f^{-1}(\mathcal{D}))$ to its valuation at the smooth divisor $f^{-1}(\mathcal{D}) \subset U$. We then have the following
	\begin{enumerate}
		\item The sequence of \'etale sheaves 
		\begin{equation}\label{eq:residue_sheaves}
			0 \to \Gm \to i_* \Gm \to j_* \Z \to 0
		\end{equation}
		is exact.
		\item We have $\Res^1 i_* \Gm = 0$. Moreover, the sheaf $\Res^2 i_* \Gm$ is torsion, supported on $\mathcal{D}$ and if a prime $\ell$ is invertible on $\mathcal{D}$ then $\Res^2 i_* \Gm[\ell^{\infty}] = 0$.
		\item The sequence \eqref{eq:residue_sheaves} induces the following long exact sequence on cohomology
		\[
		0 \to \Br \mathcal{X} \to \ker(\Br (\mathcal{X} \setminus \mathcal{D})) \to \H^0(\mathcal{X}, \Res^2i_* \Gm)) \to \H^2(\mathcal{D}, \Z) = \H^1(\mathcal{D}, \Q/\Z).
		\]
		\item Let $\ell$ be a prime invertible on $\mathcal{D}$. Then the $\ell$-primary part of this exact sequence is 
		\begin{equation} \label{eqn:Br_residue}
		0 \to \Br \mathcal{X}[\ell^{\infty}] \to \Br (\mathcal{X} \setminus \mathcal{D}))[\ell^{\infty}] \to \H^1(\mathcal{D}, \Q_{\ell}/\Z_{\ell}).
		\end{equation}
	\end{enumerate} 
\end{proposition}
\begin{proof}
	The exactness of \eqref{eq:residue_sheaves} is local and well-known for schemes. For the computation of $\Res^1i_* \Gm$ and $\Res^2i_* \Gm$ we may assume that $\mathcal{X}$ is the spectrum of a regular strictly henselian ring. Then $\mathcal{D}$ is also the spectrum of a strictly henselian ring. 
	
	Note that $\Pic \mathcal{X} = \Br \mathcal{X} = 0$ since $\mathcal{X}$ is the spectrum of a strictly henselian ring. We thus have that $\Res^1 i_* \Gm = \Pic \mathcal{X} \setminus \mathcal{D} = 0$ because $\Pic \mathcal{X} = 0$ surjects onto it. Moreover, we have $\Res^2 i_* \Gm = \Br(\mathcal{X} \setminus \mathcal{D})$. This group is torsion and the $\ell$-primary part injects into $\H^1(\mathcal{D}, \Q_{\ell}/\Z_{\ell})$ by Grothendieck purity \cite[Thm.~3.7.1]{Col21}. But $\H^1(\mathcal{D}, \Q_{\ell}/\Z_{\ell}) = 0$ because $\mathcal{D}$ is the spectrum of a strictly henselian ring.
	
	We have $\H^2(\mathcal{X}, i_*\Gm) = \ker(\Br (\mathcal{X} \setminus \mathcal{D}) \to \H^0(\mathcal{X}, \Res^2i_* \Gm))$ by the Leray spectral sequence. The map $j$ is a closed immersion so $\H^p(\mathcal{X}, j_* \Z) = \H^p(\mathcal{D}, \Z)$ for all $p$. The exact sequence $0 \to \Z \to \Q \to \Q/\Z \to 0$ and the fact that $ \Q$ has trivial cohomology implies that $\H^1(\mathcal{D}, \Z) = 0$ and that $\H^2(\mathcal{D},\Z) = \H^1(\mathcal{D}, \Q/\Z)$. The exact sequence in (3) is thus the long exact sequence coming from $0 \to \Gm \to i_* \Gm \to j_* \Z \to 0$. The second exact sequence follows by taking $\ell$-primary torsion because $\Res^2 i_* \Gm[\ell^{\infty}] = 0$.
\end{proof}
As a corollary we get the Grothendieck purity theorem.
\begin{theorem}[Purity for the Brauer group]\label{thm:Grothendieck_purity}
\hfill \\ 
	Let $\mathcal{X}$ be a smooth integral finite type algebraic stack over a field $k$ of characteristic $0$ and $\mathcal{Z} \subset \mathcal{X}$ a proper closed substack. Let $\mathcal{D}_1, \cdots, \mathcal{D}_n \subset \mathcal{X}$ be the irreducible divisors contained in $\mathcal{Z}$ and $\mathcal{D}_1^{\emph{sm}}, \cdots, \mathcal{D}_n^{\emph{sm}}$ be the smooth loci of these divisors. Then the following sequence is exact
	\[
	0 \to \Br \mathcal{X} \to \Br (\mathcal{X} \setminus \mathcal{Z}) \to \bigoplus_{i = 1}^n \H^1(\mathcal{D}_i^{\emph{sm}}, \Q/\Z).
	\]
\end{theorem}
\begin{proof}
	Let $\mathcal{W} = \mathcal{Z} \setminus (\bigcup_{i = 1}^n \mathcal{D}_i^{\text{sm}})$. This has codimension at least $2$ in $\mathcal{X}$ so by Lemma \ref{lem:purity_codimenion_2} we have $ \Br \mathcal{X} \setminus \mathcal{W} = \Br \mathcal{X}$. After replacing $\mathcal{X}$ by $\mathcal{X} \setminus \mathcal{W}$ we may thus assume that the $\mathcal{D}_i$ are smooth and that $\mathcal{Z} = (\bigcup_{i = 1}^n \mathcal{D}_i)$.
	The theorem then follows by applying Proposition \ref{prop:Grothendieck_purity}(4) iteratively to each $\mathcal{D}_i$.
\end{proof}

An important special case of Proposition \ref{prop:Grothendieck_purity} is when $\mathcal{X}$ is the spectrum of a DVR. In this case the map to $\H^1(\mathcal{D},\Q/\Z)$ is known as the \emph{Witt residue}. We refer to \cite[\S 1.4.3]{Col21} for an account of this map.
\begin{definition}\label{def:Witt_residue}
	Let $\mathcal{O}$ be a DVR with fraction field $K$ and residue field $\F$. Let $\mathcal{O}^{\text{sh}}$ be the strict henselization of $\mathcal{O}$ and $K^{\text{sh}}$ its fraction field. Let $j: \Spec K \to \Spec \mathcal{O}$ be the generic point and $i: \Spec \F \to \Spec \O$ the special point.
	
	In this case $\Res^2i_* \Gm$ is concentrated at the special point and takes the value $\Br K^{\text{sh}}$ there, by a computation at the stalks. It follows that the sequence in Proposition~\ref{prop:Grothendieck_purity} takes the form
	\begin{equation}\label{eq:purity_DVR}
		1 \to \Br \mathcal{O} \to \ker(\Br K \to \Br K^{\text{sh}}) \to \H^1(\F, \Q/\Z).
	\end{equation}
	We define the \emph{Witt residue} to be the map $\res_{\mathcal{O}}: \ker(\Br K \to \Br K^{\text{sh}}) \to  \H^1(\F, \Q/ \Z)$ in \eqref{eq:purity_DVR}.
	If $\F$ is perfect then $\Br K^{\text{sh}} = 0$ by Tsen's theorem so the domain of $\res_{\mathcal{O}}$ is $\Br K$.
\end{definition}
\subsection{Points on stacks}

We record some basic lemmas which will be used later.

\begin{lemma}\label{lem:torsor_field_points}
	Let $\mathcal{X}$ be a normal algebraic stack over a field $k$ and $\mathcal{Y} \to \mathcal{X}$ a non-split finite \'etale morphism. There exists a field $K/k$ such that the morphism $\mathcal{Y}(K) \to \mathcal{X}(K)$ is not surjective. In other words, there exists a point $x \in \mathcal{X}(K)$ such that $\mathcal{Y}_x$ is non-split.
\end{lemma}
\begin{proof}
	We may assume that $\mathcal{X}$ and $\mathcal{Y}$ are connected since we can glue sections on the components. They are then both integral since they are normal. We may also always replace $\mathcal{X}$ by an open substack by normality.
	
	If $\mathcal{X}$ is an algebraic space then it has a dense open subscheme \cite[\href{https://stacks.math.columbia.edu/tag/06NH}{Tag 06NH}]{stacks-project} so we take $K$ the function field of $\mathcal{X}$.
	
	In general apply \cite[Thm.~6.5]{Lau00} to choose, after replacing $\mathcal{X}$ by a dense open substack, a smooth atlas $U \to \mathcal{X}$ with geometrically connected fibres, where $U$ is an algebraic space.
	
	The stack $\mathcal{X}$, resp. $\mathcal{Y}$, is connected and $U \to \mathcal{X}$, resp. $U_{\mathcal{Y}} \to \mathcal{Y}$, has geometrically connected fibres so the algebraic space $U$, resp. $U_{\mathcal{Y}}$, is connected. If $K$ is the fraction field of $U$ then $\mathcal{Y}_K$ is the spectrum of the fraction field of $U_{\mathcal{Y}}$ which is non-split.
\end{proof}

\begin{lemma} \label{lem:Chebotarev}
	Let $k$ be a global field and $S$ a finite set of places containing the archimedean places.
	Let $\mathcal{X}$ be a connected normal finite type stack over $\Spec \mathcal{O}_{k, S}$ and 
	$\chi \neq 0 \in \H^1(\mathcal{X}, \Z/n \Z)$. 
	There then exists a place $v \not \in S$ and a point $x \in \mathcal{X}(\F_v)$ such that 
	$\chi_v(x) \neq 0 \in \H^1(\F_v, \Z/n\Z)$.
\end{lemma}
\begin{proof}
	By Lemma \ref{lem:torsor_field_points} and the correspondence between cohomology and torsors there exists a field $K/k$ and a point $x \in X(K)$ such that $\chi(x) \neq 0 \in \H^1(K, \Z/n\Z)$. By spreading out we may replace $K$ by an integral finite type $\mathcal{O}_{k, S}$ scheme $U$ and $x$ by a map $f: U \to \mathcal{X}$ such that $f^*(\chi) \neq 0 \in \H^1(U, \Z/n \Z)$.
	
	Let $m \mid n$ be minimal such that $f^*(\chi) \in \H^1(U, \Z/m \Z) \subset \H^1(U, \Z/n \Z)$, in particular $m \neq 1$. This implies that the $\Z/m\Z$-torsor $\pi: V \to U$ defined by $f^*(\chi)$ has the property that $V$ is connected.
	
	It now suffices to find a prime $v \not \in S$ and a point $u \in U(\F_v)$ such that $\pi^{-1}(u)$ is connected. Indeed in this case we have  that $f^*(\chi)(u) \neq 0 \in \H^1(\F_v, \Z/m \Z)$. We may then choose $x = f(u)$ since $\chi(x) = f^*(\chi)(u)  \neq 0 \in \H^1(\F_v, \Z/m \Z) \subset \H^1(\F_v, \Z/n\Z)$
	
	By Chebotarev's density theorem for schemes \cite[Thm.~B.9.]{Pink97}, the set of closed points $u \in U$ such that $\pi^{-1}(u)$ is connected has positive Dirichlet density \cite[Def.~B.6]{Pink97}. Moreover, the set of closed points $u \in U$ whose image $v \in \Spec \mathcal{O}_{k, S}$ is such that $\F(u) = \F_v$ has Dirichlet density $1$ by \cite[Prop.~B.8]{Pink97}. It follows that we may choose $u$ satisfying both properties. The second property implies that $u \in U(\F_v)$ and the first property is exactly what we needed to conclude the proof.
\end{proof}

We require suitably nice atlases of our stacks.

\begin{lemma} \label{lem:image_k_points}
	Let $\mathcal{X}$ be a finite type algebraic stack over a noetherian scheme. Then there exists
	a sequence of smooth morphisms $U_n \to \mathcal{X}$ from affine schemes such that $\coprod_n U_n(K) \to \mathcal{X}(K)$
	is essentially surjective for any field $K$.	
	
	Moreover, if $X$ is DM then we may choose $U_1 \to \mathcal{X}$ such that $U_1(K) \to \mathcal{X}(K)$ is already essentially surjective. 
\end{lemma}

Many versions of Lemma \ref{lem:image_k_points} can be found in the literature. In \cite[Thm.~6.3]{Lau00} one finds the weaker claim that a given $k$-point of $\mathcal{X}$ is in the image of some smooth morphism; this  is often sufficient in proofs. In \cite[Thm.~6.5]{Lau00} one finds the first statement where the $U_n$ are algebraic spaces. A more general statement for certain classes of rings is \cite[Prop.~7.0.8]{Chr}, which implies Lemma \ref{lem:image_k_points}. For an example which demonstrates the need for the sequence $U_n$ for non-DM stacks in general, see \cite[Rem. 6.5.1]{Lau00}.

\subsubsection{Topologies of points on stacks}\label{sec:topologies_points_stacks}
Let $k$ be a topological field which is \emph{essentially analytic} \cite[Def.~5.0.8]{Chr} (this essentially means that the inverse function theorem holds). Let $\mathcal{X}$ be a finite type algebraic stack over $k$. Christensen has defined in \cite[\S5]{Chr} a topology on the set $\mathcal{X}[k]$ of isomorphism classes of $k$-points, extending the usual topology on the set of $k$-points for schemes. It has the
following properties:
\begin{enumerate}
	\item Any morphism of stacks over $k$ induces a continuous
	map on $k$-points \cite[Thm.~9.0.3]{Chr}.
	\item Any smooth morphism of stacks over $k$ induces an open
	map on $k$-points \cite[Thm.~11.0.4]{Chr}.
\end{enumerate}
These properties uniquely determine the topology, by Lemma \ref{lem:image_k_points}.

We will consider the following examples of essentially analytic fields: $\R$, $\C$ or $K$ which is the fraction field of a henselian DVR $\mathcal{O}$ equipped with its standard topology, i.e.~the coarsest topology for which $\mathcal{O} \subset K$ is open.

We will need the following two lemmas.
\begin{lemma}\label{lem:dense_implies_dense_points}
	Let $K$ be either the field of fractions of a henselian DVR, the field $\R$ or the field $\C$. Let $\mathcal{X}$ be a finite type algebraic stack over $K$ and $\mathcal{V} \subset \mathcal{X}$ a dense open. The inclusion $\mathcal{V}[K] \subset \mathcal{X}[K]$ is then also dense.
\end{lemma}
\begin{proof}
	Let $f_n:U_n \to \mathcal{X}$ be a sequence as in Lemma \ref{lem:image_k_points}. The maps $V_n := f_n^{-1}(\mathcal{V}) \to \mathcal{V}$ also satisfy Lemma \ref{lem:image_k_points} as the functor $\mathcal{X} \to \mathcal{X}(K)$ commutes with pullbacks. The inclusion $V_n \subset U_n$ is dense so the inclusion $V_n(K) \subset U_n(K)$ is dense by \cite[Thm~10.5.1]{Col21}. This implies the lemma.
\end{proof}
\begin{lemma}\label{lem:evaluation_continuous}
	Let $K$ be either the field of fractions of a henselian DVR, the field $\R$ or the field $\C$. Let $\mathcal{X}$ be a finite type algebraic stack over $K$. The following pairing
	\[
	\mathcal{X}[K] \times \Br \mathcal{X} \to \Br K: (x, b) \to b(x).
	\]
	is continuous, where $\Br \mathcal{X}$ and $\Br K$ are equipped with the discrete topologies.
\end{lemma}
\begin{proof}
	Let $f_n:U_n \to \mathcal{X}$ be a sequence as in Lemma \ref{lem:image_k_points}. The fact that $U_n(K) \to \mathcal{X}[K]$ is open reduces the lemma to the following statement: for all $b \in \Br \mathcal{X}$ the map $U_n(K) \to \Br K: u \to b(f(u))$ is continuous.
	
	But $b(f(u)) = (f^*b)(u)$ so this follows from \cite[Prop.~10.5.2]{Col21}.
\end{proof}
\subsection{Unramified Brauer group}

We now define the unramified Brauer group of a stack. It can differ from the actual Brauer group, even if the stack is smooth and proper. 

\begin{definition} \label{def:Br_unramified}
Let $\mathcal{X}$ be a smooth finite type algebraic stack over a field $k$.
	Let $b \in \Br \mathcal{X}$. 
\begin{enumerate} 
	\item	Let  $\mathcal{O}$ be a DVR with fraction field $K$ over $k$
	and $x \in \mathcal{X}(K)$ a map. We say that $b$ is \textit{unramified} at the triple $(\mathcal{O}, K, x)$ if $b(x) \in \Br  \mathcal{O}$.
	\item  The element $b$ is \emph{unramified} if it is unramified at all possible $(\mathcal{O}, K, x)$; otherwise it is called \emph{ramified}.
	\item	The \emph{unramified Brauer group} is
	$\Brun \mathcal{X} := \{b \in \Br \mathcal{X}: b \text{ is unramified}\}.$
\end{enumerate}	
\end{definition}

\begin{remark} \label{rem:functorial}
Let $g: \mathcal{X} \to \mathcal{Y}$ be a map of smooth integral finite type stacks, $b \in \Br \mathcal{Y}$ and $(\mathcal{O}, K, x \in \mathcal{X}(K))$ a triple. The definition immediately implies that $b$ is unramified at $g(x)$ if and only if $g^*b$ is unramified at $x$. Pullling back thus defines a map $g^*: \Br_{\text{un}} \mathcal{Y} \to \Br_{\text{un}} \mathcal{X}$.
\end{remark}

It suffices to check whether a Brauer class is unramified for either henselian or complete DVRs.
\begin{lemma}\label{lem:unramified_henselian} 
	Let $b \in \Br \mathcal{X}$.
	\begin{enumerate}
	\item If $b$ is unramified at all triples $(\mathcal{O}, K, x)$ with $\mathcal{O}$ henselian, then $b$ is unramified.
	\item If $b$ is unramified at all triples of the form $(\F[[t]], \F((t)), x)$ with $\F$ a  field and $k \subset \F[[t]]$ an inclusion \footnote{The inclusion $k \subset \F[[t]]$ does not necessarily factor through the constant power series $\F$.}, then $b$ is unramified.
	\end{enumerate}
\end{lemma}
\begin{proof}
	Let us first assume that $b$ is unramified for all henselian triples. Let $(\mathcal{O}, K, x)$ be any triple. We will show that $b(x) \in \Br \mathcal{O}$.
	
	Let $\F$ be the residue field of $\mathcal{O}$, $\mathcal{O} \subset \mathcal{O}^{\text{h}}$ a henselization of $\mathcal{O}$ with fraction field $K^{\text{h}}$ and $\mathcal{O}^{\text{h}} \subset \mathcal{O}^{\text{sh}}$ a strict henselization with fraction field $K^{\text{sh}}$. Consider the following commutative diagram
	\[\begin{tikzcd}
		0 & {\Br \mathcal{O}} & {\ker(\Br K \to \Br K^{\text{sh}})} & {\H^1(\F, \Q/\Z)} \\
		0 & {\Br \mathcal{O}^{\text{h}}} & {\ker(\Br K^{\text{h}} \to \Br K^{\text{sh}})} & {\H^1(\F, \Q/\Z)}
		\arrow[from=1-1, to=1-2]
		\arrow[from=1-2, to=1-3]
		\arrow[from=1-2, to=2-2]
		\arrow[from=1-3, to=1-4]
		\arrow[from=1-3, to=2-3]
		\arrow["{=}"{marking, allow upside down}, draw=none, from=1-4, to=2-4]
		\arrow[from=2-1, to=2-2]
		\arrow[from=2-2, to=2-3]
		\arrow[from=2-3, to=2-4].
	\end{tikzcd}\]
	The rows of this diagram are special cases of the exact sequence in Proposition \ref{prop:Grothendieck_purity}(3) (see Definition \ref{def:Witt_residue}). We know that $b(x) \in \Br \mathcal{O}^{\text{h}}$ by assumption and a simple diagram chase then shows that $b(x) \in \Br \mathcal{O}$. 
	
	If $\mathcal{O}$ is a henselian DVR with fraction field $K$ and $\hat{\mathcal{O}}$ is the completion with fraction field $\hat{K}$ then $\Br \mathcal{O} = \Br \hat{\mathcal{O}}$ and $\Br K = \Br \hat{K}$ by \cite[Cor.~3.4.3, Prop.~7.1.8]{Col21}. But by Cohen's structure's theorem \cite[\href{https://stacks.math.columbia.edu/tag/0C0S}{Tag 0C0S}]{stacks-project} we have $\hat{\mathcal{O}} \cong \F[[t]]$ where $\F$ is the residue field of $\mathcal{O}$. This implies the second part of the lemma.
\end{proof}

We next show that for varieties, our definition of the unramified Brauer group agrees with the usual definition \cite[Def.~6.2.1]{Col21}. The subtle point is that we allow DVRs with arbitrary fraction fields, but \cite[Def.~6.2.1]{Col21} allows only DVRs with fraction field equal to $k(X)$. The proof of the following lemma is inspired by \cite[Thm.~10.5.12]{Col21}, but works in arbitrary characteristic unlike \textit{loc.~cit.}
\begin{lemma}\label{lem:unramified_other_definition}
	Let $X$ be a smooth geometrically integral variety over a field $k$. Then
	$\Brun X = \Brun(k(X)/k).$
\end{lemma}
\begin{proof}
	The inclusion $\Br X \subset \Br k(X)$ and the definitions immediately imply that $\Brun X \subset \Brun(k(X)/k)$. The other inclusion is more difficult.
	
	We have the inclusion $\Brun(k(X)/k) \subset \Br X$ by \cite[Thm.~3.7.7]{Col21}. Given $b \in \Brun(k(X)/k) \subset \Br X$ we thus have to show that it is unramified in our sense. By Lemma \ref{lem:unramified_henselian} we have to show that for all fields $\F$, inclusions $k \subset \F[[t]]$ and points $x \in X(\F((t)))$ that $b(x) \in \Br \F[[t]]$. Moreover, by \cite[Lem.~10.5.11]{Col21} we may assume that $x: \Spec \F((t)) \to X$ is a dominant morphism (the stated lemma assumes that $X_{\F[[t]]}$ arises by base change from $X_{\F}$, but the proof does not use this assumption). But \cite[Prop.~6.2.3]{Col21} implies that $b(x) \in \Brun(\F((t))/k)$. Now note that $\Brun(\F((t))/k) \subset \Br(\F[[t]])$ by definition, which finishes the proof.
\end{proof}

\subsection{Birational invariance}
\begin{lemma}
	Let $\mathcal{U} \subset \mathcal{X}$ be a dense open immersion of smooth finite type algebraic stacks over a field. Then $\Brun \mathcal{U} = \Brun \mathcal{X}$
\end{lemma}
\begin{proof}
	We have $\Br \mathcal{X} \subset \Br \mathcal{U}$ by Lemma \ref{lem:purity_codimenion_2} and Proposition \ref{prop:Grothendieck_purity}. The inclusion $\Brun \mathcal{X} \subset \Brun \mathcal{U}$ is then immediate from the definition.
	
	We will next show that $\Brun \mathcal{U} \cap \Br \mathcal{X} \subseteq \Brun \mathcal{X}$. Let $b \in \Brun \mathcal{U} \cap \Br \mathcal{X}$ and let $(\mathcal{O}, K, x)$ be a triple. We may assume that $\mathcal{O}$ is henselian by Lemma \ref{lem:unramified_henselian}. The map $x \to b(x)$ is locally constant for $x \in \mathcal{X}[K]$ by Lemma \ref{lem:evaluation_continuous} and $\mathcal{U}[K] \subset \mathcal{X}[K]$ is dense by Lemma \ref{lem:dense_implies_dense_points}. There thus exists an $u \in \mathcal{U}[K]$ such that $b(x) = b(u)$. But $b \in \Brun \mathcal{U}$ so $b(u) \in \Br \mathcal{O}$. This implies that $b \in \Brun \mathcal{X}$.
	
	The last step is to show that $\Brun \mathcal{U} \subset \Br \mathcal{X}$. If $\mathcal{X} \setminus \mathcal{U}$ has codimension $2$ then $\Br \mathcal{U} = \Br \mathcal{X}$ by Lemma \ref{lem:purity_codimenion_2} so we may remove a codimension $2$ subset and assume without loss of generality that $\mathcal{D} := \mathcal{X} \setminus \mathcal{U}$ is a smooth divisor.
	
	Let $i: \mathcal{U} \to \mathcal{X}$ and $j: \mathcal{D} \to \mathcal{X}$ be the inclusions. Let $b \in \Brun \mathcal{U}$. We will first show that the image of $b$ in $\H^0(\mathcal{X}, \Res^2i_* \Gm)$ is trivial. This is a local question so we may assume that $\mathcal{X}$ is the spectrum of a strictly henselian ring $R$ with fraction field $K$. We can then think of $b \in \Br K$.
	The assumption that $b \in \Brun \mathcal{U}$ implies in particular that $b$ is unramified at all triples $(\mathcal{O}_{\mathcal{X}, x}, K, \Spec K \to \mathcal{X})$ where $x \in X$ ranges over all codimension $1$ points of $\mathcal{X}$ and $ \Spec K \to \mathcal{X}$ is the generic point. This implies that $b \in \Br \mathcal{X} = 0$ by \cite[Thm.~1.2]{Ces19}, which shows the claim.
	
	Thus by Proposition \ref{prop:Grothendieck_purity} it remains to show that $\partial_{\mathcal{D}}(b) \in \H^1(\mathcal{D}, \Q/\Z)$ is trivial. Assume for the sake of contradiction that this is false.
	
	By Lemma \ref{lem:torsor_field_points} and the correspondence between $\H^1$ and torsors there exists a field $\F$ and a point $x \in \mathcal{X}(\F)$ such that $\partial_{\mathcal{D}}(b)(x) \neq 0 \in \H^1(\F, \Q/\Z)$. By Lemma \ref{lem:image_k_points}  there exists an atlas $f: X \to \mathcal{X}$ and a $v \in X(\F)$ with $f(v) = x$.
	
	Let  $D = f^{-1}(\mathcal{D})$ and $i': U \subset X$ the pullback of $i$. The sequence \eqref{eq:residue_sheaves} is functorial in $f$ so the following diagram is commutative and has exact rows by Proposition~\ref{prop:Grothendieck_purity}
	\[\begin{tikzcd}
		0 & {\Br \mathcal{X}} & {\ker({\Br \mathcal{U}} \to \H^0(\mathcal{X}, \Res^2i_* \Gm))} & {\H^1(\mathcal{D}, \Q/\Z)} \\
		0 & {\Br X} & {\ker({\Br U} \to \H^0(X, \Res^2i'_* \Gm))} & {\H^1(D, \Q/\Z)}.
		\arrow[from=1-1, to=1-2]
		\arrow[from=1-2, to=1-3]
		\arrow["{f^*}"', from=1-2, to=2-2]
		\arrow["{\res_{\mathcal{D}}}", from=1-3, to=1-4]
		\arrow["{f^*}"', from=1-3, to=2-3]
		\arrow["{f^*}"', from=1-4, to=2-4]
		\arrow[from=2-1, to=2-2]
		\arrow[from=2-2, to=2-3]
		\arrow["{\res_D}", from=2-3, to=2-4]
	\end{tikzcd}\]
	We then have the equality $\partial_{\mathcal{D}}(b)(x) = f^*(\partial_{\mathcal{D}}(b))(v) = \partial_{D}(f^*b)(v)$ by the commutativity of the above diagram. This implies that $\partial_{D}(f^*b) \neq 0$ and thus that $f^*b \not \in \Br X$.
	
	However $b \in \Brun \mathcal{U}$ implies that $f^*b \in \Brun U$. But from Lemma \ref{lem:unramified_other_definition} we deduce that $\Brun U = \Brun(k(U)/k) = \Brun X \subset \Br X$ which provides the desired contradiction.
\end{proof}

We now show that the unramified Brauer group is a stable birational invariant. For this we need the following lemma; by a projective bundle over $\mathcal{X}$ we mean an algebraic stack over $\mathcal{X}$ isomorphic to $\P(\mathcal{V})$ for some vector bundle $\mathcal{V}$ over $\mathcal{X}$. Note that a vector bundle on an algebraic stack need only be locally trivial for the smooth topology, and not necessarily the Zariski topology. We start with some basic lemmas.


\begin{lemma}
	Let $\mathcal{X}$ be a smooth algebraic stack over a field and $f:\mathcal{Y} \to \mathcal{X}$ a projective bundle. Then $f_* \Gm = \Gm$, we have $\Res^1 f_* \Gm = \Z$ generated by $\mathcal{O}_{\mathcal{Y}}(1)$, and $\Res^2 f_* \Gm = 0$.
	\label{Higher pushforwards of projective bundle}
\end{lemma}
\begin{proof}
	It suffices to consider the case $\mathcal{X} = \Spec R$ where $R$ is a regular strictly henselian local ring. In this case we have $\mathcal{Y} \cong \P_{R}^n$ for some $n$.
	
	The lemma follows from the fact that $\mathcal{O}_{\P_{R}^n}(\P_{R}^n)^{\times} = R^{\times}$, $\Pic \P_{R}^n \cong \Z$, generated by $\mathcal{O}_{\P_{R}^n}(1)$, and $\Br \P_{R}^n \cong \Br R = 0$ by \cite[Cor.~6.1.4]{Col21}.
\end{proof}

\begin{lemma}
	Let $\mathcal{X}$ be  smooth finite type algebraic stack over a field $k$ and $f:\mathcal{Y} \to \mathcal{X}$ a projective bundle. The map $f^*: \Br_{\emph{un}}\mathcal{X} \to \Br_{\emph{un}} \mathcal{Y}$ is an isomorphism.
\end{lemma}
\begin{proof}
	By Lemma \ref{Higher pushforwards of projective bundle} and the Leray spectral sequence we have an exact sequence 
	\begin{equation*}
		\Pic \mathcal{Y} \to \H^0(\mathcal{X}, \Z) \to \Br \mathcal{X} \to \Br \mathcal{Y} \to \H^1(\mathcal{X}, \Z).
	\end{equation*}
	We also deduce from Lemma \ref{Higher pushforwards of projective bundle} that the image of $\mathcal{O}_{\mathcal{Y}}(1) \in \Pic \mathcal{Y}$ generates $\H^0(\mathcal{X}, \Z) = \Z$. Moreover, we have $\H^1(\mathcal{X}, \Z) = 0$ since $\mathcal{X}$ is smooth. We deduce that $\Br \mathcal{Y} = \Br \mathcal{X}$.
	
	Let $R$ be a DVR with fraction field $K$. Any point $x \in \mathcal{X}(K)$ can be lifted to a point $y \in \mathcal{Y}(K)$; indeed the base change $\mathcal{Y}_x$ is a projective bundle over $K$ and thus isomorphic to $\P^n_K$ for some $n$ and this variety always has a $K$-point.
	
	It follows immediately from the definition 
	that $\Br_{\text{un}}\mathcal{X} = \Br_{\text{un}} \mathcal{Y}$.
\end{proof}
\begin{remark}
	It also follows from the above lemma that if $\mathcal{V} \to \mathcal{X}$ is a vector bundle then $\Br_{\text{un}} \mathcal{V} = \Br_{\text{un}} \mathcal{X}$ since the unramified Brauer group is a birational invariant and $\mathcal{V} \subset \P(\mathcal{V} \oplus \mathcal{O}_{\mathcal{X}})$ is a dense open immersion.
\end{remark}
\begin{example}\label{rem:Bogomolov_multiplier}
	Let $G$ be a finite group scheme over a field $k$ of characteristic $0$ and $G \to  \GL_n$ a faithful representation. The map $[\mathbb{A}_{k}^n/G] \to BG$ is a vector bundle so by the above remark we have $\Br_{\text{un}} BG_{k} = \Br_{\text{un}} [\mathbb{A}_{k}^n/G]$.
	
	The action of $G$ on $\mathbb{A}_{k}^n$ is faithful so the stack $[\mathbb{A}_{k}^n/G]$ has trivial generic stabiliser and is thus birational to its coarse moduli space $\mathbb{A}_{k}^n/G$. Let $X$ be a proper smooth variety birational to $\mathbb{A}_{k}^n/G$. We deduce that $\Br_{\text{un}} BG = \Br_{\text{un}} X = \Br X$. 	This provides a new proof of the fact that $\Br X$ is independent of the choice of faithful representation $G \to \GL_n$ (the usual proof passes by the so-called \textit{no-name lemma}, see \cite[\S  3]{Bog87}.)
	
	When $k = \C$ the group $\Br X$ is a known invariant of the group $G$, the \emph{Bogomolov multiplier}, introduced by Bogomolov \cite{Bog87}.  We revisit this group in \S \ref{sec:transcendental_Brauer_group}.
\end{example}

\subsection{Purity}
We now obtain a version of the purity theorem for the unramified Brauer group (stated as Theorem \ref{thm:Br_unramified_intro} in the introduction). It is a version of Theorem \ref{thm:Grothendieck_purity} but with sectors (see \S \ref{sec:cylotomic_inertia}) in place of divisors.

\begin{theorem}[Purity for the unramified Brauer group]	\label{thm:equivalent_conditions_unramified}\hfill \\ 	Let $\mathcal{X}$ be a smooth proper tame DM stack over a field $k$. The following are equivalent for $b \in \Br \mathcal{X}$:
	\begin{enumerate}
		\item $b$ is unramified.
		\item For every field $K/k$, $n \in \N$ and representable map $f: (B \mu_n)_K \to \mathcal{X}$ we have $f^* b \in \Br K$.
		\item For every algebraic stack $S$, $n \in \N$ and representable map $f: (B \mu_n)_S \to \mathcal{X}_S$ we have $f^* b \in \Br S$.
		\item For every sector $\mathcal{S} \in \pi_0(I_{\mu} \mathcal{X})$ of order $n$ let $f_{\mathcal{S}}: (B \mu_n)_{\mathcal{S}} \to \mathcal{X}$ be the universal map. We have $f_{\mathcal{S}}^*(b) \in \Br \mathcal{S}$.
	\end{enumerate}
\end{theorem}
\begin{remark}\label{rem:tame_DM_representable_map}
	The fact that $\mathcal{X}$ is tame DM implies that the stabiliser groups are tame finite \'etale groups. This implies that every representable map $(B \mu_n)_S \to \mathcal{X}_S$ in the above theorem is such that $n$ is coprime to the characteristic of $k$.
\end{remark}
The proof of this theorem is quite technical so we postpone it to  \S \ref{sec:mu_n-gerbe}, \ref{sec:root-stacks}, \ref{sec:proof_sector_purity}. Theorem \ref{thm:equivalent_conditions_unramified} naturally leads to the following definition.

\begin{definition} \label{def:partially_ramified_Br}
	Let $\mathcal{X}$ be a finite type tame DM stack over a field. Let $\mathcal{S} \in \pi_0(I_{\mu} \mathcal{X})$ be a sector and $f_{\mathcal{S}}: (B \mu_n)_{\mathcal{S}} \to \mathcal{X}$ the universal map. 
	We say that $b \in \mathcal{X}$ is \emph{unramified} along $\mathcal{S}$ if $f_{\mathcal{S}}^*(b) \in  \Br \mathcal{S}$.
	
	Let $\mathcal{C} \subset \pi_0(I_{\mu} BG)$ be a subset of sectors. We define the \emph{Brauer group partially unramified with respect to $\mathcal{C}$} as $$\Br_{\mathcal{C}}(\mathcal{X}) := \{b \in \Br \mathcal{X}: b \text{ is unramified along $\mathcal{S}$ for all } \mathcal{S} \in \mathcal{C}\}.$$
\end{definition}

This group satisfies the following functoriality. For a map $f: \mathcal{X}_1 \to \mathcal{X}_2$ and a collection of sectors $\mathcal{C}_2$ of $\mathcal{X}_2$, denote by $\mathcal{C}_1$ the sectors of $\mathcal{X}_1$ above $\mathcal{C}_2$. Then $f^*: \Br \mathcal{X}_2 \to \Br \mathcal{X}_1$ restricts to a map
\begin{equation} \label{eq:Br_C_functorial}
f^*: \Br_{\mathcal{C}_2} \mathcal{X}_2 \to \Br_{\mathcal{C}_1} \mathcal{X}_1
\end{equation}
It is this partially unramified Brauer group which will appear in the leading constant in  Malle's conjecture. This definition naturally begs the question whether it is possible to define the residue of a Brauer group element along a sector; this is indeed possible and we achieve this in Definition \ref{def:residue_along_sector} below.

We now turn to arithmetic applications.

\subsection{Brauer--Manin obstruction} \label{sec:BM}
We now define the Brauer--Manin obstruction on a stack. There are actually two kinds of Brauer--Manin obstruction on a stack; both contain useful information, however it seems that it is the unramified Brauer--Manin pairing which is more fundamental. For background on the Brauer--Manin obstruction for varieties, see \cite[\S13.3]{Col21}. We denote by $\inv_v: \Br k_v \to \Q/\Z$ the local invariant from class field theory. We present some examples of Brauer--Manin obstructions on algebraic stacks in \S \ref{sec:Br_BG}.

\subsubsection{Brauer--Manin pairing}
Let $\mathcal{X}$ be a smooth integral stack over a global field $k$. We have the \textit{Brauer--Manin pairing}
$$\langle \cdot, \cdot \rangle_{\text{BM}}: \mathcal{X}(\Adele_k) \times \Br \mathcal{X} \to \Q/\Z, \quad ( (x_v), b) \mapsto \sum_v \inv_v b(x_v).$$
We define $\mathcal{X}(\Adele_k)^{\Br}$ to be the left kernel of this pairing; the usual argument shows that $\mathcal{X}(k) \subset \mathcal{X}(\Adele_k)^{\Br}$. This pairing can be used for example to prove Stickelberger's theorem that the discriminant of a number field is  $0$ or $1$ modulo $4$ (see \S \ref{sec:Stickelberger}). 

We equip $\mathcal{X}[k_v]$ with the topology from \S\ref{sec:topologies_points_stacks}. The space $\mathcal{X}[\Adele_k]$ is then equipped with the restricted product topology as in \cite[\S13]{Chr}.

In what follows we use the following abuse of notation: Let $f:\mathcal{X}(k_v) \to Y$ be a map to a topological space $Y$ which is constant on isomorphism classes of $k_v$-points. Then we say that $f$ is \emph{continuous} if the induced map from $\mathcal{X}[k_v]$ is continuous (similarly for adelic points).

\begin{lemma}\label{lem:BM_pairing}
	The Brauer--Manin pairing on $\mathcal{X}(\Adele_k)$ is well-defined and continuous on the left.
\end{lemma}
\begin{proof}
	Let $b \in \Br \mathcal{X}$. 
	It suffices to show continuity when restricted to the basic open 
	$\prod_{v \notin S} \mathcal{X}(k_v) \prod_{v\in S} \mathcal{X}(\O_v)$ for a finite
	set of places $S$. However the usual argument \cite[Prop.~13.3.1]{Col21} shows that 
	$b$ evaluates trivially on $\mathcal{X}(\O_v)$ providing $S$ is sufficiently large;
	this shows that the pairing is well-defined.
	Continuity follows from Lemma~\ref{lem:evaluation_continuous}.
\end{proof}

Thus the Brauer--Manin pairing can be used to obtain obstructions to \emph{strong approximation} (here we say that $\mathcal{X}$ satisfies strong approximation if the image of $\mathcal{X}(k)$ in $\mathcal{X}[\Adele_k]$ is dense in the collection of connected components).

\subsubsection{Unramified Brauer--Manin pairing}

The Brauer--Manin pairing can however be quite unwieldy in general since $\Br X/\Br k$ can be infinite, even for $X = BG$. So we  define the  \emph{unramified Brauer--Manin pairing} to be
$$\langle \cdot, \cdot \rangle_{\text{BM}}: \prod_v X(k_v) \times \Brun X \to \Q/\Z, \quad ( (x_v), b) \mapsto \sum_v \inv_v b(x_v).$$
We similarly define $(\prod_v X(k_v))^{\Brun}$ to be the left kernel of this pairing and have $X(k) \subset (\prod_v X(k_v))^{\Brun}$. This pairing seems to be more fundamental to the theory and can be used  to give obstructions to \emph{weak approximation}. (Here we say that $\mathcal{X}$ satisfies weak approximation if the image of $\mathcal{X}[k]$ in $\prod_v \mathcal{X}[k_v]$ is dense; some background on this for stacks can be found in \cite[\S4]{LS23}.) For example, it allows one to recover Wang's counter-example to Grunwald's theorem as a Brauer--Manin obstruction (see \S \ref{sec:Grunwald}). To verify that this pairing is well-defined we require the following result, which is a minor generalistion of Theorem \ref{thm:Harari's_formal_lemma_intro}. For varieties this is well-known and part of Harari's formal lemma \cite[Thm.~13.4.1]{Col21}. We generalise this result to the case of the partially unramified Brauer group of $BG$ in Theorem~\ref{thm:Harari_formal_partially_unramified}.

\begin{theorem}\label{thm:Harari's_formal_lemma}
	Let $\mathcal{X}$ be a smooth finite type DM stack over a global field $k$
	and let $b \in \Br \mathcal{X}$. If $b \in \Brun \mathcal{X}$ then $b$ evaluates trivially
	on 	$\mathcal{X}(k_v)$ for all but finitely many places $v$ of $k$.
	The converse holds if $k$ is a number field.
\end{theorem}
\begin{proof}
	Suppose that $b$ is unramified. Let $f:U \to \mathcal{X}$ be an atlas as in Lemma \ref{lem:image_k_points}. The pullback $f^* b \in \Br U$ is unramified by Remark \ref{rem:functorial}. We can thus reduce the lemma to the case of schemes where it is well-known.
	
	Now assume that $b$ is ramified and $k$ is a number field. Let $\mathcal{O}$ be a DVR over $k$ with fraction field $K$ and $f: \Spec K \to \mathcal{X}$ a map such that $f^*(b) \not \in \Br \mathcal{O}$. By spreading out we may replace $\Spec \mathcal{O}$ by a smooth finite type scheme $Y$ over $k$, $K$ by an open subscheme $V \subset Y$ and the map $f$ by a map $f: V \to \mathcal{X}$ such that $f^*(b) \not \in \Br Y$. We may assume without loss of generality that $Y$ is connected. It follows from Harari's formal lemma \cite[Thm.~13.4.1]{Col21} that there exists infinitely many places $v$ such that $f^*(b)$ evaluates non-trivially on $V(k_v)$. This implies that $b$ evaluates non-trivially on $f(V(k_v)) \subset \mathcal{X}(k_v)$.
\end{proof}

\subsubsection{Partially unramified Brauer---Manin pairing}
For the partially unramified Brauer group (Definition~\ref{def:partially_ramified_Br}),
it makes sense to modify the adelic
space to obtain a different topological space between 
$\mathcal{X}(\Adele_k)$ and $\prod_v\mathcal{X}(k_v)$ and upon which the Brauer--Manin pairing
is well-defined. This can be done in general, though to simplify the
exposition we only do this in the case of $BG$ (see \S \ref{sec:BMO_BG}.)

\subsection{Brauer groups of $\mu_n$-gerbes} \label{sec:mu_n-gerbe}
We now prepare for the proof of Theorem~\ref{thm:equivalent_conditions_unramified}. 
The proof will rely on the computation of the Brauer groups of neutral $\mu_n$-gerbes and root stacks. 

The inclusion $\mu_n \subset \Gm$ induces a map $B \mu_n \to B \Gm$ defining a canonical element $\mathcal{O}(\frac{1}{n}) \in \Pic B \mu_n$.
\begin{lemma} \label{lem:pushforward_mu_n}
	Let $\mathcal{X}$ be an algebraic stack, $n \in \N$ an invertible number on $\mathcal{X}$ and $f: (B \mu_n)_{\mathcal{X}} \to \mathcal{X}$ the structure map. Then $f_* \Gm = \Gm$, $\Res^1 f_* \Gm \cong \Z/ n \Z$, generated by $\mathcal{O}(\frac{1}{n})$, and $\Res^2 f_*\Gm = 0$.
\end{lemma}
\begin{proof}
	We may assume that $\mathcal{X} = \Spec R$ for $R$ a strictly henselian local ring. Let $\pi: \Spec R \to (B \mu_n)_R$ be the corresponding $\mu_n$-torsor. The scheme $\Spec R$ has trivial cohomology so the Hochschild-Serre spectral sequence associated to this torsor implies that for all $i \in \N$ we have $\H^i((B \mu_n)_R, \Gm) \cong \H^i(\mu_n , R^{\times})$. Where $\mu_n$ acts trivially on $R^{\times}$.
	
	 We have $\H^0(\mu_n, R^{\times}) = R^{\times}$, $\H^1(\mu_n, R^{\times}) = \Hom(\mu_n, R^{\times}) \cong \Z / n\Z$ with a generator given by the inclusion $\mu_n \subset R^{\times}$. A cocycle computation shows that this inclusion corresponds to $\mathcal{O}(\frac{1}{n}) \in \H^1((B \mu_n)_R, \Gm)$. The group $\mu_n$ is cyclic so $\H^2(\mu_n, R^{\times}) = R^{\times}/R^{\times n} = 0$ since $R$ is strictly henselian and $n \in R^{\times}$.
\end{proof}
The map $f$ has a section so the Leray spectral sequence $\H^p(\mathcal{X}, \Res^q f_* \Gm) \implies \H^{p + q}((B \mu_n)_{\mathcal{X}}, \Gm)$ splits. We deduce from Lemma \ref{lem:pushforward_mu_n} the existence of a residue map $\res_{\mathcal{X}, n}: \Br((B \mu_n)_{\mathcal{X}}) \to \H^1(\mathcal{X}, \Z/n\Z)$. This yields a split exact sequence
\begin{equation} \label{eqn:residue_map}
		0 \to \Br \mathcal{X} \to \Br((B \mu_n)_{\mathcal{X}}) \xrightarrow{\res_{\mathcal{X}, n}} \H^1(\mathcal{X}, \Z/n\Z) \to 0.
\end{equation}
The  sequence \eqref{eqn:residue_map} suggest the following definition of the residue along a sector.

\begin{definition}\label{def:residue_along_sector}
	Let $\mathcal{X}$ be a finite type tame DM stack over a Noetherian base $S$.
	Let $\mathcal{S} \subset I_{\mu} \mathcal{X}$ be a sector of order $n$. Consider the universal map $(B \mu_n)_{\mathcal{S}} \to \mathcal{X}$. The \emph{residue along $\mathcal{S}$} is the composition
	\[
	\res_{\mathcal{S}}:\Br \mathcal{X} \to \Br (B \mu_n)_{\mathcal{S}} \xrightarrow{\res_{\mathcal{S}, n}} \H^1(\mathcal{S}, \Z/n \Z) \subset \H^1(\mathcal{S}, \Q/\Z).
	\]
\end{definition}	

We verify that the residue has the expected property with respect to the partially unramified
Brauer group from Definition \ref{def:partially_ramified_Br}.

\begin{lemma} \label{lem:Br_partially_unramified_purity}
	Let $\mathcal{X}$ be a finite type tame DM stack over a field $k$ and
	$\mathcal{C} \subset \pi_0(I_{\mu} BG)$ a subset of sectors. Then the sequence
	$$0 \to \Br_{\mathcal{C}} \mathcal{X} \to \Br \mathcal{X} \to \bigoplus_{\mathcal{S} \in \mathcal{C}} \H^1(\mathcal{S}, \Q/\Z) $$
	is exact.
\end{lemma}
\begin{proof}
	Immediate from \eqref{eqn:residue_map} and Definition \ref{def:partially_ramified_Br}.
\end{proof}

For cup products there is an explicit description of $\res_{\mathcal{X}}$; the following argument is inspired by the proof of \cite[Thm.~4.1.1]{Sko01}.
\begin{lemma} \label{lem:residue_map_cup}
	Let $T$ be a group of multiplicative type over an algebraic stack $\mathcal{X}$ and $\dual{T} := \Hom(T, \Gm)$ the sheaf of characters. Let $\pi: \mathcal{Y} \to (B \mu_n)_{\mathcal{X}}$ be a $T$-torsor corresponding to a map $(B \mu_n)_{\mathcal{X}} \to (B T)_{\mathcal{X}}$. Let $p: \mu_n \to T$ be the induced map on automorphism sheaves and $\dual{p}: \dual{T} \to \Z/n\Z$ the Cartier dual.
	
	Let $[\pi] \in \H^1((B \mu_n)_{\mathcal{X}}, T)$ be the class of $\pi$. Then for all $\alpha \in \H^1(\mathcal{X}, \hat{T})$ one has $\res_{\mathcal{X} , n}(\alpha \cup [\pi]) = \dual{p}_*(\alpha)$, where the cup product is induced by the bilinear pairing $\hat{T} \times T \to \Gm$.
\end{lemma}
\begin{proof}
	The map $\res_{\mathcal{X} , n}$ is a boundary map of the Leray spectral sequence of $f:(B \mu_n)_{\mathcal{X}} \to \mathcal{X}$. The Leray spectral sequence has cup products \cite[Cor.~8.8]{Swa99} which implies that $\res_{\mathcal{X} , n}(\alpha \cup [\pi])$ is the cup product of $\alpha \in  \H^1(\mathcal{X}, \hat{T})$ and the image of $[\pi]$ under the boundary map $d:\H^1((B \mu_n)_{\mathcal{X}}, T) \to \H^0(\mathcal{X}, \Res^1 T)$. The same computation as in Lemma \ref{lem:pushforward_mu_n} shows that $\Res^1f_* T \cong \Hom(\mu_n, T)$.
	
	Consider the bilinear pairing $\hat{T} \times \Res^{1}f_* T \cong \hat{T} \times \Hom(\mu_n, T) \to \Res^1f_*\Gm = \Hom(\mu_n, \Gm)$ induced by the cup product. The functoriality of the isomorphism $\Res^1 f_* T \cong \Hom(\mu_n, T)$  implies that this bilinear pairing is given by the formula $(\chi: T \to \Gm, t: \mu_n \to T) \to \chi(t) \in \Hom(\mu_n, \Gm)$. We find that $\res_{\mathcal{X} , n}(\alpha \cup [\pi]) = \alpha \cup d([\pi])$. It suffices to show that $d([\pi]) = p$ since $\alpha \cup p = \dual{p}_*(\alpha)$.
	
	This is a local statement so we may reduce to the case that $\mathcal{X}$ is the spectrum of a strictly henselian ring. In this case $d$ is the identity and $[\pi] = p$ by definition.
\end{proof}

\subsection{Brauer groups of root stacks} \label{sec:root-stacks}
There is an analogue of Lemma \ref{lem:pushforward_mu_n} for root stacks.

\begin{lemma}
	Let $\mathcal{X}$ be an algebraic stack, $n \in \N$ invertible on $X$ and $\mathcal{D} \subset \mathcal{X}$ an effective Cartier divisor. Let $i: \mathcal{D} \to \mathcal{X}$ be the inclusion. Let $p: \mathcal{X}[\sqrt[n]{\mathcal{D}}] \to \mathcal{X}$ be the forgetful map and $q: \mathcal{G} \to \mathcal{D}$ be the $\mu_n$-gerbe over $\mathcal{D}$.
	
	Then $p_* \Gm = \Gm$, the map $\Res^1 p_* \Gm \to i_* \Res^1 q_* \Gm \cong i_* \Z / n \Z$ is an isomorphism and $\Res^2 p_* \Gm = 0$. 
\end{lemma}
\begin{proof}
	We may assume that $\mathcal{X} = \Spec R$ for $R$ a strictly henselian local ring. Let $f \in R$ be the section defining $\mathcal{D}$ and $S = R[x]/(x^n - f)$, it is finite over $R$ and thus strictly henselian. Let $\Spec S \to \mathcal{X}[\sqrt[n]{\mathcal{D}}]$ be the defining $\mu_n$-torsor and $\Spec R \to \mathcal{G}$ the defining $\mu_n$-torsor of $\mathcal{G}$.
	
	The map $\H^1(\mathcal{X}[\sqrt[n]{\mathcal{D}}], \Gm) \to \H^1(\mathcal{G}, \Gm) \cong \Z/ n\Z$ is an isomorphism by the computation of the Picard group of root stacks \cite[Cor. 3.1.2]{Cad07}.
	
	The cohomology of $S$ and $R$ is trivial so by the Hochschild-Serre spectral sequence we find that $\H^0(\mathcal{X}[\sqrt[n]{\mathcal{D}}], \Gm) = \H^0(\mu_n, S^{\times}) = (S^{\times})^{\mu_n} = R^{\times}$ and $\H^2(\mathcal{X}[\sqrt[n]{\mathcal{D}}], \Gm) = \H^2(\mu_n, S^{\times})$. The group $\mu_n$ is cyclic so $\H^2(\mu_n, S^{\times})$ is isomorphic to $R^{\times}$ modulo norms. The norms of $R^{\times} \subset S^{\times}$ are $R^{\times n} = R^{\times}$ so $\H^2(\mu_n, S^{\times}) = 0$.
\end{proof}

Let $\mathcal{O}$ be a DVR with fraction field $K$, valuation $v$ and residue field $\F$. Let $n \in \N$ be invertible on $\O$. We now consider the special root stack $\Spec \mathcal{O}_{\sqrt[n]{v}}$ from \S \ref{sec:def_O_root_n}. 

We next deduce the existence of a residue map from the Leray spectral sequence. Fix an isomorphism $\Spec \F_{\sqrt[n]{v}} \cong (B\mu_n)_{\F}$ and let $\res_{ \mathcal{O}, n}: \Br \mathcal{O}_{\sqrt[n]{v}} \to \H^1(\F, \Z/ n\Z)$ be the composition of the map $\Br \mathcal{O}_{\sqrt[n]{v}} \to \Br \Spec \F_{\sqrt[n]{v}} \cong \Br (B \mu_n)_{\F}$ with $\res_{\F, n}$. 

The Leray spectral sequence implies that the following sequence is exact (compare with \eqref{eqn:residue_map}).
\begin{equation}\label{eqn:residue_map_DVR}
	0 \to \Br \mathcal{O} \to \Br \mathcal{O}_{\sqrt[n]{v}} \xrightarrow{\res_{\O, \sqrt[n]{v}}}  \H^1(\F, \Z/ n\Z)
\end{equation}
If $\mathcal{O}$ is henselian then this exact sequence splits after a choice of uniformiser. The following lemma also shows that the map $\res_{\O, \sqrt[n]{v}}$ is independent of the choice of isomorphism $\Spec \F_{\sqrt[n]{v}} \cong (B\mu_n)_{\F}$.
\begin{lemma} 	\label{lem:Cup products defines a section of residue map for DVR}
	Let $\pi \in \mathcal{O}$ be a uniformiser. Consider the $\mu_n$-torsor $p: \Spec \mathcal{O}[\pi^{\frac{1}{n}}] \to \Spec \mathcal{O}_{\sqrt[n]{v}}$. For all $\alpha \in \H^1(\mathcal{O}, \Z/ n \Z)$ we have $\res_{\O, \sqrt[n]{v}}(\alpha \cup [p]) = \alpha|_\F$.
	
	If $\mathcal{O}$ is henselian then $\H^1(\mathcal{O}, \Z/n \Z) \cong \H^1(\F, \Z/ n\Z)$ so $\cdot \cup [p]$ defines a section of $\res_{\O, \sqrt[n]{v}}$. Moreover, the map $\Br \mathcal{O}_{\sqrt[n]{v}} \to \Br (B \mu_n)_{\F}$ is an isomorphism.
\end{lemma}
\begin{proof}
	The fibre $\Spec \F \cong \pi^{-1}(\Spec \F_{\sqrt[n]{v}}) \to \Spec \F_{\sqrt[n]{v}} \cong (B \mu_n)_{\F}$ is a $\mu_n$-torsor such that the induced map $(B \mu_n)_{\F} \to (B \mu_n)_{\F}$ is the identity on automorphism groups (but the induced map is not necessarily equal to the identity). The lemma then follows from functoriality of the cup product and Lemma \ref{lem:residue_map_cup}.
	
	If $\mathcal{O}$ is henselian then $\H^1(\mathcal{O}, \Z/n \Z) \cong \H^1(\F, \Z/ n\Z)$. Comparing \eqref{eqn:residue_map} and \eqref{eqn:residue_map_DVR} shows that $\Br \mathcal{O}_{\sqrt[n]{v}} \to \Br (B \mu_n)_{\F}$ is an isomorphism
\end{proof}

We can relate the residue $\res_{\O, \sqrt[n]{v}}$ to the Witt residue $\res_{\mathcal{O}}$.
\begin{lemma}\label{lem:Witt_residue_equal_root_residue}
	Let $v$ be the valuation of $\mathcal{O}$ and $n \in \N$ coprime to the characteristic of $\F$. The group $\Br \mathcal{O}_{\sqrt[n]{v}}$ is contained in $\ker(\Br K \to \Br K^{\emph{sh}})$. Moreover, the following square commutes
	\[\begin{tikzcd}
		{\Br \mathcal{O}_{\sqrt[n]{v}}} & {\ker(\Br K \to \Br K^{\emph{sh}})} \\
		{\H^1(\F, \Z/n\Z)} & {\H^1(\F, \Q/\Z).}
		\arrow["\subset"{description}, draw=none, from=1-1, to=1-2]
		\arrow["{\res_{\mathcal{O}}}", from=1-2, to=2-2]
		\arrow["\subset"{description}, draw=none, from=2-1, to=2-2]
		\arrow["{\res_{\O, \sqrt[n]{v}}}"', from=1-1, to=2-1]
	\end{tikzcd}\]
\end{lemma}
This is a minor generalization of \cite[Thm.~1]{Oes18}; at first sight there seems to be a sign difference, but this is due to different sign conventions for the Leray spectral sequence.
\begin{proof}
	Both $\res_{\mathcal{O}}$ and $\res_{\O, \sqrt[n]{v}}$ are compatible with unramified base change so we may assume that $\mathcal{O}$ is henselian. Let $\pi$ be a uniformiser and consider the $\mu_n$-torsor $p: \Spec \mathcal{O}[\pi^{\frac{1}{n}}] \to \Spec \mathcal{O}_{\sqrt[n]{v}}$.	
	Let $A \in \Br \mathcal{O}_{\sqrt[n]{v}}$. By Lemma \ref{lem:Cup products defines a section of residue map for DVR} and the exactness of \eqref{eqn:residue_map_DVR} there exists an $A_0 \in \Br \mathcal{O}$ such that $A = A_0 + \res_{\O, \sqrt[n]{v}}(A) \cup [p]$. We have $\res_{\mathcal{O}}(A_0) = 0$. As $A_0 \in \Br \O$ and $\res_{\O, \sqrt[n]{v}}(A) \in \H^1(\F,\Q/\Z)$, these both become trivial after base change to $K^{\text{sh}}$. So $A \in \ker(\Br K \to \Br K^{\text{sh}})$. 
	
	The class $[p] \in \H^1(K, \mu_n) \cong K^{\times}/K^{\times n}$ is equal to $\pi$ by Kummer theory. So $\res_{\mathcal{O}}(A)= \res_{\mathcal{O}}(\res_{\O, \sqrt[n]{v}}(A) \cup \pi) + \res_{\mathcal{O}}(A_0) = \res_{\O, \sqrt[n]{v}}(A)$, where the last equality follows from \cite[(1.29) p.~36]{Col21} and the fact that $\res_{\mathcal{O}}(A_0) = 0$ by \eqref{eq:purity_DVR}.
\end{proof}

\subsection{Proof of Theorem \ref{thm:equivalent_conditions_unramified}} \label{sec:proof_sector_purity}
	Recall from Remark \ref{rem:tame_DM_representable_map} that all the $n \in \N$ which will appear are coprime to the characteristic of $k$.
	
	That $(3)$ implies $(2)$ is immediate. We also have that $(4)$ implies $(3)$. Indeed, we may assume that $S$ is connected and in that case any representable map $(B \mu_n)_S \to \mathcal{X}$ factors through the universal map $f_{\mathcal{S}}: (B \mu_n)_{\mathcal{S}} \to \mathcal{X}$ for some map $S \to \mathcal{S}$ and a sector $\mathcal{S} \in \pi_0(I_{\mu} \mathcal{X})$.
	
	Let us now show that $(2)$ implies $(4)$. Assume for the sake of contradiction that $f_{\mathcal{S}}^{*} b \not \in \Br \mathcal{S}$. The exactness of \eqref{eqn:residue_map} implies that $\res(f_{\mathcal{S}}^ b) \in \H^1(\mathcal{S}, \Z/n\Z)$ is non-zero. By Lemma \ref{lem:torsor_field_points} there exists a field $K/k$ and a point $x \in \mathcal{S}(K)$ such that $\res(f^* b)(x) \neq 0$. The point $x$ corresponds to a map $f_{\mathcal{S}}(x): (B \mu_n)_K \to \mathcal{X}$ such that $\res(f_{\mathcal{S}}(x)^* b) = \res(f_{\mathcal{S}}^* b)(x) \neq 0$. The exactness of \eqref{eqn:residue_map} implies that $f_{\mathcal{S}}^* b \not \in \Br K$ which contradicts $(2)$.
	
	We now show that $(1)$ is equivalent to $(2)$.	
	Let $\O$ be a DVR over $k$ with fraction field $K$, residue field $\F$, valuation $v$ and uniformiser $\pi$. Let $f: K \to \mathcal{X}$. By the arithmetic valuative criterion for properness (Lemma \ref{lem:arithmetic_valuative_criterion}) there exists an $n \in \N$ and a representable map $\underline{f}: \O_{\sqrt[n]{v}} \to \mathcal{X}$ whose generic fibre is $f$. Let ${\bar f} := \underline{f} \circ i_{\pi}: (B \mu_n)_{\F} \to \mathcal{X}$ be the restriction to the special point. 
	
	By definition of $\res$ we can combine \eqref{eqn:residue_map} and \eqref{eqn:residue_map_DVR} into the following commutative diagram with exact rows.
	\[
	\begin{tikzcd}
		0 \arrow[r] & \Br \mathcal{O} \arrow[d] \arrow[r] & \Br \mathcal{O}_{\sqrt[n]{v}} \arrow[r, "\res_{\O, \sqrt[n]{v}}"] \arrow[d, "i_{\pi}^*"] & \H^1(\F, \Z/n \Z) \arrow[d, equal] \\
		0 \arrow[r] & \Br \F \arrow[r] & \Br (B \mu_n)_{\F} \arrow[r, "\res_{\F, n}"] & \H^1(\F, \Z/n \Z)
	\end{tikzcd}
	\] 
	A diagram chase then shows that for $\beta \in \Br \mathcal{O}_{\sqrt[n]{v}}$ we have $\beta \in \Br \mathcal{O}$ if and only if $i_{\pi}^* \beta \in \Br \F$. We will apply this to the case $\beta = f^* b = \underline{f}^* b$, where $i_{\pi}^* \beta = i_{\pi}^* \underline{f}^* b = \bar{f}^* b$.
	 
	 If (2) holds then let $\mathcal{O}$ and $f$ be arbitrary. We then have that $\bar{f}^{*} b \in \Br \F$ by (2). This implies that $f^* b \in \Br \mathcal{O}$ and thus that (1) holds.
	
	Let $\F/k$ be a field and $\bar{f} :(B \mu_n)_{\F} \to \mathcal{X}$ a map. Let $\Tilde{f}: \Spec \F[[t]]_{\sqrt[n]{t}} \to \mathcal{X}$ be the composition with the map $\Spec \F[[t]]_{\sqrt[n]{t}} \to (B \mu_n)_{\F}$ defined by the composition $\F[[t]]_{\sqrt[n]{t}} \cong [\Spec \F[[t^{\frac{1}{n}}]]/\mu_n] \to (B\mu_n)_{\F}$, where $\mu_n$ acts by multiplication on $t^{\frac{1}{n}}$, to get a map $\Tilde{f}: \Spec \F[[t]]_{\sqrt[n]{t}} \to \mathcal{X}$. We have $\bar{f} = \tilde{f} \circ i_{t}$ by the definition of $i_t$ so the above argument in reverse shows that (1) implies (2).
\qed

\begin{remark}
	Some of the tools in this section also appear in recent work of Achenjang \cite{Ach24},
	though ours aims and focus are quite different, in particular Achenjang does not
	consider the unramified Brauer group.
	For example a version of our Proposition \ref{prop:Grothendieck_purity}(3) be found
	in \cite[Prop.~8.5]{Ach24} for stacks over a $1$-dimensional base,
	with the Grothendieck purity theorem for stacky curves in 
	\cite[Prop.~8.15]{Ach24}. A version of our sequence \eqref{eqn:residue_map}
	and Lemma \ref{lem:residue_map_cup} appear in 
	\cite[Prop.~3.9, Prop.~3.15]{Ach24}, but Achenjang allows $\mu_n$ for arbitrary $n$ 
	(see also \cite[Prop.~4.1.4]{Lieb11} for a version over fields).
	A version of our sequence \eqref{eqn:residue_map_DVR} appears in 
	\cite[Prop.~6.11]{Ach24}, though again Achenjang allows arbitrary $n$.
\end{remark}

\subsection{A geometric interpretation of the residue}
We finish with a description of the residue for certain algebraic Brauer group elements, which we will require when studying the Brauer group of $BG$. The construction works in the following generality. Let $\mathcal{X}$ be a stack and $G$ an \'etale group of multiplicative type over $\mathcal{X}$. Consider the map $f: B G \to \mathcal{X}$. Let $T$ be another group of multiplicative type. A similar computation as in Lemma \ref{lem:pushforward_mu_n} shows that $f_* T = T$ and $R^1 f_* T = \underline{\Hom}(A, T)$, where $\underline{\Hom}$ denotes the Hom sheaf. The Leray spectral sequence leads to a boundary map $r:\H_1^2(BG,T) := \ker(\H^2(BG, T) \to \H^0(\mathcal{X}, R^2f_* T)) \to \H^1(\mathcal{X}, \underline{\Hom}(G, T))$. The goal of this section is give a geometric description of this map. In practice we will mostly care about the case $T = \Gm$ and $G = \mu_m$ in which case $r = \res_{\mathcal{X}, n}$.

An element of $[\pi] \in \H^2(BG, T)$ corresponds to a $T$-gerbe $\pi: \mathcal{T} \to BG$ as in \cite[\S12.2.7]{Ols}. The condition that $[\pi] \in \H_1^2(BG, T)$ means that the gerbe is locally neutral on $\mathcal{X}$ (rather than just locally neutral on $BG$), i.e.~locally on $\mathcal{X}$ of the form $B T \times G$. In particular, $\mathcal{T}$ is a gerbe with banded by an abelian group scheme $A$ over $\mathcal{X}$. The map $\pi: \mathcal{T} \to BG$ induces a map $A \to G$ on bands. Moreover, this map is surjective and has kernel $T$, both of these facts are checked locally on $\mathcal{X}$. In particular, $A$ defines an element of $\Ext^1_{\mathcal{X}}(G, T)$. Moreover, this extensions splits locally on $\mathcal{X}$ and by the Leray spectral sequence $\H^p(\mathcal{X}, \underline{\Ext}^q(G, T)) \implies \Ext^{p + q}(G, T)$ thus lives as  a class $[A] \in \H^1(\mathcal{X}, \underline{\Hom}(G, T))$.

This gives a geometric description of $r$.
\begin{proposition}\label{prop:geometric_description_residue}
	We have the equality $r([\pi]) = [A]$.
\end{proposition}
To prove this we will use the language of derived categories. We remark that this can be proven up to a sign by essentially only using functoriality. But something more is needed if one wants to understand the sign.

For any algebraic stack $\mathcal{X}$ let $D(\mathcal{X})$ denote the derived category of sheaves on the (small) fppf site of $\mathcal{X}$. Given a map $\mathcal{X} \to \mathcal{Y}$ we let $Rf_*: D(\mathcal{X}) \to D(\mathcal{Y})$ be the derived pushforward. We let $\tau_{\leq n}: D(\mathcal{X}) \to D(\mathcal{X})$ be the functor which truncates complexes to cohomological degree $\leq n$. The following lemma provides a geometric interpretation of the class $[\pi]$.

\begin{lemma}\label{lem:derived_interpretation_gerbe}
	Let $\pi: \mathcal{Y} \to \mathcal{X}$ be a $T$-gerbe and $T'$ a sheaf of abelian groups over $\mathcal{X}$. We have a canonical isomorphisms $ \pi_* T' \cong T'$ and $R^1 \pi_*T' \cong \underline{\Hom}(T, T')$.
	
	Moreover the map $d$ in the following distinguished triangle in $D(\mathcal{X})$
	\[
	T'[1] \to \tau_{\leq 0} R f_* T'[1] \to \underline{\Hom}(T, T') \xrightarrow{d} T'[2]
	\]
	sends $h \in \underline{\Hom}(T, T')$ to $h_*[\pi] \in T'[2]$.
\end{lemma}
\begin{proof}
	The first computation is local and can thus be done in the case $\mathcal{X}$ is the spectrum of a strictly henselian ring $R$, in which case $\mathcal{Y} \cong BT$ and the computation is completely analogous to the proof of Lemma \ref{lem:pushforward_mu_n}. An important point is that the identification $R^1 \pi_* T' \cong \underline{\Hom}(T, T')$ is independent of the choice of isomorphism $\mathcal{Y} \cong BT$ as $T$-gerbes (which is unique up to non-unique isomorphism.)
	
	For the second statement we are reduced by functoriality to the case that $T' = T$ and $h = \text{id}_T$. We thus have to study the map $\Hom(T, T) \to \H^2(\mathcal{X}, T)$. This is known as the transgression map and the statement in the lemma then follows from unfolding \cite[Prop.~V.3.2.1]{Gir71}. 
%
\end{proof}

\begin{proof}[Proof of Proposition \ref{prop:geometric_description_residue}]
	By Lemma \ref{lem:derived_interpretation_gerbe} we have that $[\pi]$ is the image of $\text{id}_T$ in the distinguished triangle in $D(BG)$.
	\[
	T[1] \to \tau_{\leq 0} R\pi_* T[1] \to \underline{\Hom}(T, T) \to T[2]
	\]
	Applying the functor $\tau_{\leq 0} Rf_*$ and using the fact that $\tau_{\leq 0} Rf_* \tau_{\leq 0} R\pi_* T[1] \cong \tau_{\leq 0} R(f \circ \pi) T[1]$ by \cite[\href{https://stacks.math.columbia.edu/tag/015M}{Tag 015M}]{stacks-project} we get a triangle
	\begin{equation}\label{eq:important_triangle}
		\tau_{\leq 0} Rf_* T[1] \to \tau_{\leq 0} R(f \circ \pi)_* T[1] \to \underline{\Hom}(T, T) \to \tau_{\leq 1} Rf_* T[2].
	\end{equation}
	We claim that the triangle \eqref{eq:important_triangle} is distinguished. Indeed, it suffices to show that the corresponding long exact sequence of cohomology is exact. To see this we apply Lemma \ref{lem:derived_interpretation_gerbe} to $Rf_*$ and $R(f \circ \pi)_*$ to see that the corresponding long exact sequence is 
	\[
	0 \to T \to T \to 0  \to \underline{\Hom}(G, T) \to \underline{\Hom}(A, T)  \to \underline{\Hom}(T, T) \to 0
	\]
	The exact sequence $0 \to T \to A \to G \to 0$ splits locally on $\mathcal{X}$ which implies that this long exact sequence is exact.
	
	The description of spectral sequences in terms of derived functors (as explained in e.g.~\cite[pp.~67--68]{Sko01}) shows that $r([\pi])$ is equal to the image of $\text{id}_T$ under the composition of the maps $\underline{\Hom}(T, T) \to \tau_{\leq 1} Rf_* T[2] \to R^1 f_* T[1] \cong \underline{\Hom(G, T)}[1]$.
	
	The locally split exact  sequence $0 \to T \to A \to G \to 0$ leads to a distinguished triangle 
	\[
	\underline{\Hom}(G, T) \to \underline{\Hom}(A, T)  \to \underline{\Hom}(T, T) \to \underline{\Hom}(G, T)[1]
	\]
	where the last map sends $\text{id}_T$ to $[A]$. It thus suffices to show that the following diagram commutes, which is clear by the functoriality of Lemma \ref{lem:derived_interpretation_gerbe}.
	\[\begin{tikzcd}
		{	\tau_{\leq 0} Rf_* T[1]} & {\tau_{\leq 0} R(f \circ \pi)_* T[1]} & {\underline{\Hom}(T, T)} & {	\tau_{\leq 1} Rf_* T[2]} \\
		{\underline{\Hom}(G, T)} & {\underline{\Hom}(G, T)} & {\underline{\Hom}(T, T)} & {\underline{\Hom}(G, T)[1]}.
		\arrow[from=1-1, to=1-2]
		\arrow[from=1-1, to=2-1]
		\arrow[from=1-2, to=1-3]
		\arrow[from=1-2, to=2-2]
		\arrow[from=1-3, to=1-4]
		\arrow["{ = }", from=1-3, to=2-3]
		\arrow[from=1-4, to=2-4]
		\arrow[from=2-1, to=2-2]
		\arrow[from=2-2, to=2-3]
		\arrow[from=2-3, to=2-4]
	\end{tikzcd}\]
\end{proof}

\section{Brauer group of $BG$} \label{sec:Br_BG}

In this section we specialise the general theory from \S \ref{sec:Br} to the case of $BG$, and use it to calculate the partially unramified Brauer group of $BG$. The ultimate goal is to give an effective method to compute it, which we achieve using central extensions and Kummer theory (Theorem \ref{thm:orbifold_Kummer}), as well as Galois cohomology (Theorem \ref{thm:Br_BG}). This method is described in \S \ref{sec:procedure}.  Let $k$ be a field of characteristic $p$ and $G$ a finite \'etale group scheme over $k$.

\subsection{Brauer group of $BG$} \label{sec:representation_Brauer_group_elements}

 For a stack $X$ over a field $k$, we denote by $\Br_1 X = \ker(\Br X \to \Br X_{k^{\sep}})$ its \emph{algebraic Brauer group}. An element which is not algebraic is called \emph{transcendental}.

\begin{lemma}\label{lem:coh_sep_closed}
	For all $p > 0$ there exist functorial isomorphisms $\H^p(BG_{k^{\mathrm{sep}}}, \Gm) \cong \H^p(G(k^{\mathrm{sep}}), k^{\mathrm{sep}, \times})$ , where the latter cohomology is group cohomology. In particular, $\Pic BG_{k^{\mathrm{sep}}} = \Hom(G(k^{\mathrm{sep}}), k^{\mathrm{sep},\times})$ and $\Br BG_{k^{\mathrm{sep}}} = \H^2(G(k^{\mathrm{sep}}), k^{\mathrm{sep}, \times})$.
\end{lemma}
\begin{proof}
	Follows from applying the Hochschild-Serre spectral sequence to the $G$-torsor $\Spec k^{\mathrm{sep}} \to BG_{k^{\mathrm{sep}}}$, since $\Spec k^{\mathrm{sep}}$ has trivial cohomology.
\end{proof}

Note in particular that this lemma describes the Galois module structure of the finite group $\Br BG_{k^{\mathrm{sep}}}$. Moreover it allows us to identify the Picard scheme of $BG$ with $\dual{G}$. Inspired by Sansuc  \cite[p.~42]{San81}, we define $\Br_{e} BG = \{\beta \in \Br_1 BG: \beta(e) = 0\}$ where $e$ is the identity cocycle. 

\begin{lemma} \label{lem:H1_Pic}
	The map $r: \Br_1 BG \to \H^1(k , \dual{G})$ from the Hochshild--Serre spectral sequence
	induces an isomorphism $\Br_1 BG/ \Br k \cong \H^1(k , \dual{G})$.
	
	The map $r$ is an isomorphism when restricted to $\Br_{e} BG$.
\end{lemma}
\begin{proof}
	The first part is well-known in the case of smooth proper varieties with a rational point
	(see e.g.~\cite[Prop.~5.4.2]{Col21}); the same proof works here.
	
	For the second part the map $\beta \to \beta - \beta(e)$ defines an inverse to the morphism $\Br_{e} BG \to \Br_1 BG/ \Br k$.
\end{proof}

A version of Lemma \ref{lem:H1_Pic} for different base schemes can be found in \cite[Cor.~5.15]{Ach24}, under the additional assumption that the transcendental Brauer group is trivial.

\begin{corollary} \label{cor:no_p_torsion}
	The group $\Br BG/\Br k$ is $|\exp(G)|^2$-torsion and $\Br BG_{k^{\mathrm{sep}}}$ is finite,
	where $\exp(G)$ denotes the exponent of $G$.
\end{corollary}

Lemma \ref{lem:H1_Pic} shows that the natural map $\Br_{1} BG^{\mathrm{ab}} \to \Br_{1} BG$ is an isomorphism. This need not be true for the partially unramified algebraic Brauer group from Definition \ref{def:partially_ramified_Br}, but see Lemma \ref{lem:Brun_ab} for a partial result. For algebraic Brauer group elements, the Brauer  pairing admits an interpretation via the cup product.

\begin{lemma} \label{lem:cup_products}
	Assume that $G$ is tame. The diagram 
	\[
\xymatrix{ \Br_{e} BG \times BG(k)  \ar[r] \ar[d] & \Br k \ar[d] \\ 
\H^1(k, \dual{G}) \times \H^1(k,G^{\emph{ab}}) \ar[r] & \H^2(k, \Gm) }
\]
	commutes, where the bottom row is induced by the cup product.
\end{lemma}
\begin{proof}
	Consider the fibre of the map $BG \to B G^{\text{ab}}$ over $e$. This is the $G^{\text{ab}}$-torsor $\pi: B[G, G] \to BG$. The type, as defined in \cite[Def.~2.3.2]{Sko01}, of this $G^{\text{ab}}$-torsor is a map $\Hom(G^{\text{ab}}, \Gm) = \dual{G} \to \Pic BG = \dual{G}$. It is the identity map, which can be seen by unfolding the definitions.

	Consider the map $\cdot \cup [\pi]: \H^1(k, \dual{G}) \to \Br BG$ given by the cup product with $[\pi] \in \H^1(BG, G^{\text{ab}})$. This map factors through $\Br_1 BG$ since elements of $\H^1(k, \dual{G})$ become trivial after base change to $k^{\sep}$. Moreover, it factors through $\Br_{e} BG$ since $[\pi](e) = [\pi(e)]$ is the identity cocycle.
	
	It follows from \cite[Thm.~4.1.1]{Sko01} that $\cdot \cup [\pi]:  \H^1(k, \dual{G}) \to \Br_{e} BG$ is an inverse to $r$ (this result is only stated for varieties, but the proof for stacks is analogous. To finish the lemma it suffices to notice that the map $BG(k) \to \H^1(k, G^{\text{ab}}), \chi \to [\pi](\chi)$ is by construction the map induced by $BG \to BG^{\text{ab}}$.
\end{proof}

\subsection{Representation by central extensions} \label{sec:central_extension}
We next explain how to write down elements of $\Br BG$ for $G$ a finite \'etale group scheme over a field $k$.

Recall that for an abelian group scheme $N$, the isomorphism classes of $N$-gerbes over a stack $\mathcal{X}$ are classified by the cohomology group $\H^2(\mathcal{X},N)$ \cite[Thm.~12.2.8]{Ols}.

\begin{definition} \label{def:construction_Brauer_group_elements}
Let $$1 \to \mu_n \to E \to G \to 1$$ be a central extension of group schemes. Then $BE \to BG$ is a $\mu_n$-gerbe. The image of this gerbe under the induced map $\H^2(BG,\mu_n) \to \H^2(BG,\Gm) = \Br BG$ is a Brauer group element, which we denote by $b_E$.

We call such an extension \emph{tame} if $E$ is a tame group scheme.
\end{definition}

This element has the property that $b_E(e) = 0$ for $e$ is the identity of $G$; we next show that every such element arises this way. (This is not a major restriction, since every Brauer group element satisfies this up to translation by an element of $\Br k$.) We begin with some Kummer theory.

\begin{lemma}\label{lem:boundary_map_Kummer_theory}
	The Kummer sequence gives rise to the short exact sequence
	$$0 \to \Pic BG/n\Pic BG \to \H^2(BG,\mu_n) \to \H^2(BG,\Gm)[n] \to 0.$$
	This map sends a character $\chi \in \Pic BG = \dual{G}(k)$
	to the central extension
	$$G_{\chi, n} := \{ (g, \zeta) \in G \times \Gm : \chi(g) = \zeta^n\} \to G.$$
\end{lemma}
\begin{proof}
	The sequence is immediate from Kummer theory.
	The corresponding gerbe is described as the fiber product of the following diagram $BG \xrightarrow{f} B \Gm \xleftarrow{\cdot n} B \Gm$ by the construction in \cite[\S12.2.5]{Ols}. A simple computation shows that this is $BG_{\chi}$.
\end{proof}

\begin{lemma} \label{lem:central_gerbe}
	Every element $b \in \Br BG$ with $b(e) = 0$ arises by the construction in Definition \ref{def:construction_Brauer_group_elements}.
\end{lemma}
\begin{proof}
	Let $\mathcal{X} \to BG$ be a gerbe and $P \in \mathcal{X}(k)$ a lift of $e$, let $E := \underline{\Aut}(P)$. The stack $\mathcal{X}$ is a gerbe over $k$ and has a rational point $P$ so $\mathcal{X} \cong BE$. This procedure defines an equivalence of categories between the category of gerbes over $BG$ equipped with a lift $P \in \mathcal{X}(k)$ and the category of surjective maps of group schemes $E \to G$, see \cite[Thm.~7.2.5]{Gir71}. Moreover, if $A$ is an abelian group scheme and $\mathcal{X} \to BG$ is an $A$-gerbe then the analysis of \S7.3 of \textit{loc.~cit.} shows that the corresponding extension $E \to G$ is a central $A$-extension (see also \S7.1 of \textit{loc.~cit.})
	
By Lemma \ref{lem:boundary_map_Kummer_theory} any $b \in \Br BG$ may be represented by a $\mu_n$-gerbe $\mathcal{X}$ over $BG$ for	some $n$. The fact that $b(e) = 0 \in \H^2(k, \Gm)$ and Hilbert Theorem 90 implies that the fibre of $\mathcal{X}$ over $e$ has a rational point. By the above reasoning this gerbe is given by $BE \to BG$ for some central $\mu_n$-extension $E \to G$ and functoriality implies that $b = b_E$.
\end{proof}

\begin{remark}
	Another way of seeing the relation between elements of $\Br BG$ and central extensions is via Hochschild cohomology, see \cite[\S2.1]{Ach24}. 
\end{remark}

\subsection{Central extensions and the residue}
From now on we assume that $G$ is tame.
We begin to study the residue of a Brauer group element in terms of the corresponding central extension. In the next lemma we use the following notation. Let $\mathcal{C} \subset \mathcal{C}_G^*$ and let $c \in \mathcal{C}$. We denote by $\mathcal{S}_{c}$ be the sector corresponding to the Galois orbit of $c$ and $\mathcal{C}_c \subset G(-1)$ the subscheme corresponding to the Galois orbit of $c$, as in Lemma \ref{lem:sectors_BG}. Then we write $\Br_{\mathcal{C}} BG$ for the partially unramified Brauer group with respect to these sectors.

\begin{lemma}\label{lem:residue_central_extension}
	Let $1 \to \mu_n \to E \xrightarrow{f} G \to 1$ be a tame central extension with corresponding Brauer group element $b_E$ and $c \in \mathcal{C}^*_G/\Gamma_{k}$ a Galois orbit of order $n$. Note that the $G$-action on $\mathcal{C}_c$ lifts to $f^{-1}(\mathcal{C}_c) \subset E(-1)$ because $f$ is a central extension. Moreover, translation by $\Z/n \Z = \mu_n(-1) \subset E(-1)$ defines a $\Z/n \Z$-action on $f^{-1}(\mathcal{C}_c)$ which commutes with the $G$-action and such that $f^{-1}(\mathcal{C}_c) \to \mathcal{C}_c$ is a $\Z/n \Z$-torsor. We thus get a $\Z/n \Z$-torsor
	\begin{equation}\label{eq:residue_central_extension}
	 	[f^{-1}(\mathcal{C}_c)/G] \to [\mathcal{C}_c/G] \cong \mathcal{S}_c
	\end{equation}
	The residue $\res_{\mathcal{S}_c}(b_E) \in \H^1(\mathcal{S}_c, \Z/n)$ is represented by the torsor \eqref{eq:residue_central_extension}.
\end{lemma}
\begin{proof}
	Let $d$ be the order of $c$ and $\mathcal{X} \to (B\mu_d)_{\mathcal{S}_c}$ the pullback of $BE \to BG$. This is an gerbe with abelian automorphism group over $\mathcal{S}_c$ and by Proposition \ref{prop:geometric_description_residue} we have to understand the automorphism group scheme $A$.
	
	We can directly compute $\mathcal{X}_{\mathcal{C}_c}$ using Lemma \ref{lem:sectors_BG} and find that $A_{\mathcal{C}_c} = \{(\zeta, e, h) \in \mu_d \times E \times \mathcal{C}_c: f(e) = h(\zeta)\}$. Moreover, $G$ acts on $A$ by acting by conjugation on $E$ and $\mathcal{C}_c$. One checks using this description that $A$ is obtained by twisting $\mu_n \times \mu_d$ by the torsor \eqref{eq:residue_central_extension}, i.e. the class $[A] \in \H^0(\mathcal{S}_c, \Hom(\mu_n, \mu_d))$ is equal to \eqref{eq:residue_central_extension}, as desired.
\end{proof}
We deduce the following criterion for being partially unramified. This will be crucial for calculating the partially ramified Brauer group.

\begin{lemma}\label{lem:central_extension_unramified}
	Let $1 \to \mu_n \to E \xrightarrow{f} G \to 1$ be a tame central 
	extension with Brauer group element $b_E \in \Br BG$.
	Let $C \subset \mathcal{C}_G^*$ be Galois invariant and 
	let $\mathcal{C} \subset G(-1)(k^{\sep})$ be the corresponding collection of elements.
	
	We have $b_E \in \Br_{\mathcal{C}} BG$ if and only if there exists a Galois and conjugacy invariant
	subset $\mathcal{E} \subset E(-1)(k^{\mathrm{sep}})$ such that $f$ induces a (Galois equivariant) bijection between 
	$\mathcal{C}$ and $\mathcal{E}$.
\end{lemma}
\begin{proof}
	By Lemma \ref{lem:Br_partially_unramified_purity} we have to show that the given condition is equivalent to the vanishing of all residues. Without loss of generality, we may assume that $C$ consists of a single Galois orbit of conjugacy classes. Let $c \in C$ and $d$ be its order.
	
	By Lemma \ref{lem:residue_central_extension} we have to show that the torsor \eqref{eq:residue_central_extension} splits if and only if the conditions of the lemma hold. 
	
	We claim that $\eqref{eq:residue_central_extension}$ splits if and only if there exists a $G$-equivariant section of $f^{-1}(\mathcal{C}_c) \to \mathcal{C}_c$. Indeed, a section $[\mathcal{C}_c/G] \to 	[f^{-1}(\mathcal{C}_c)/G]$ defines by pullback a $G$-equivariant section $\mathcal{C}_c \to f^{-1}(\mathcal{C}_c)$. Conversely, a $G$-equivariant section $\mathcal{C}_c \to f^{-1}(\mathcal{C}_c)$ defines a section $[\mathcal{C}_c/G] \to 	[f^{-1}(\mathcal{C}_v)/G]$ after quotienting by $G$.
	
	It now remains to notice that the condition in the lemma is equivalent to the existence of a $G$-equivariant section of $f^{-1}(\mathcal{C}_c) \to \mathcal{C}_c$. If such a section exists then $\mathcal{E}$ is equal to the $k^{\sep}$ points of the image. On the other hand given a $\mathcal{E}$ as in the lemma the subscheme of $E(-1)$ defined by it is $G$-equivariantly isomorphic to $\mathcal{C}_c$ due to the fact that $\mathcal{E} \to \mathcal{C}_c(k^{\sep})$ is a bijection.
\end{proof}

\subsection{Marked central extensions} \label{sec:marked_central_extensions}
Lemma \ref{lem:central_extension_unramified} naturally leads us to the following definition. This definition is also inspired by \cite[\S 7.4]{EVW13}, but we allow $G$ and $E$ to be group schemes instead of just groups and consider a marking to be a subscheme of $E(-1)$ instead of $E$. We remark that the definition of $E(-1)$ makes sense when $E$ is a profinite \'etale group scheme of order coprime to $p$.
\begin{definition} \label{def:marked_central_extension}
	Let $\mathcal{C} \subset G(-1)(k^{\sep})$ be a conjugacy and Galois invariant subset. A $\mathcal{C}$-\emph{marked central extension} of $G$ is a central extension $1 \to A \to E \to G \to 1$, where $A$ and $E$ are profinite \'etale group schemes, equipped with a conjugacy and Galois invariant subset $\mathcal{E} \subset E(-1)(k^{\sep})$ such that $\mathcal{E} \to \mathcal{C}$ is a (Galois equivariant) bijection. We call $\mathcal{E}$ the \emph{marking}.
	
	The \emph{trivial marked extension (with kernel $A$)} is the central extension $G \times A \to G$ equipped with the \emph{trivial marking} $\mathcal{C} \times \{0\} \subset G(-1) \times A(-1)$.
	
	A morphism of marked central extensions is a morphism of extensions which preserves the marking.
	
	We will denote the set of isomorphism classes of central marked extensions with kernel $A$ by $\H^{2, \text{orb}}_{\mathcal{C}}(BG, A)$. As the notations suggests, it should be viewed as some version of orbifold cohomology.
\end{definition}

\begin{definition}
	The \emph{Baer sum} of two marked central extensions $1 \to A \to E_1 \to G \to 1$, $1 \to A \to E_2 \to G \to 1$ with respective markings $\mathcal{E}_i \subset E_i(-1)$ is the Baer sum extension \cite[\href{https://stacks.math.columbia.edu/tag/010I}{Tag 010I}]{stacks-project} $1 \to A \to E_1 + E_2 \to G \to 1$. Note that $E_1 + E_2$ is a subquotient of $E_1 \times_G E_2$ and the marking $(\mathcal{E}_1, \mathcal{E}_2) \subset (E_1 \times_G E_2)(-1)$ defines a marking of $\mathcal{E}_1 + \mathcal{E}_2 \subset (E_1 + E_2)(-1)$.
\end{definition}
\begin{lemma}
	Baer sum defines a commutative group structure on $\H^{2, \orb}_{\mathcal{C}}(BG, A)$.
\end{lemma}
\begin{proof}
	The Baer sum of group extensions defines a group structure. It remains to check that the isomorphisms witnessing commutativity, associativity and the existence of an inverse, all preserve the marking. This is a straightforward but tedious calculation, which is omitted.
\end{proof}

The following allows one to change the kernel of a marked central extension.
\begin{lemma}\label{lem:marked_central_extension_kernel}
	Let $1 \to A \to E \xrightarrow{f} G \to 1$ a central extension with marking $\mathcal{E} \subset E$. For any profinite abelian group scheme $A'$ and map $A \to A'$ there exists a unique marked central extension $1 \to A' \to E' \to G \to 1$ and map of marked central extensions $E \to E'$ which restricts to $A \to A'$ on the kernel.
	
	Moreover, if $\mathcal{C}$ generates $G$ then $E'$ is unique up to unique isomorphism.
\end{lemma}
\begin{proof}
	Existence is given by letting $E' := E \coprod_{A} A'$ be the pushout and letting $\mathcal{E}'$ be the image of $\mathcal{E}$. Uniqueness follows from the universal property of pushout.
	
	To see that it is unique up to unique isomorphism note that the automorphism group of $E'$ which fixes $A'$ is isomorphic to $\Hom(G, A')$, where $\chi \in \Hom(G, A')$ acts by sending $\epsilon \in E$ to $\epsilon + \chi(f(\epsilon))$. An automorphism which preserves the marking thus has the property that $\chi(g) = 0$ for all $g \in \mathcal{C}$. But $\mathcal{C}$ generates $G$ so this implies that $\chi = 0$.
\end{proof}

\subsubsection{The universal central marked extension}
In \cite[\S7.4]{EVW13} a specific marked central extension is identified which is initial with respect to all marked central extension. This extension was also studied in \cite{Woo21}. Its kernel is closely related to the components of Hurwitz spaces.

There are two main differences with our approach and the construction in \cite{EVW13, Woo21}. The first that we allow group schemes over fields, the second is that we since we define a marking as living in $E(-1)$ the universal marked central extension will be a non-constant group scheme even if $G$ is constant.

For the rest of this section to keep notation light, we will use the non-standard definition $\widehat{\Z} := \varprojlim_{p \nmid n} \Z/n \Z$ instead of $\widehat{\Z}(p')$. We also consider the $\widehat{\Z}^{\times}$-torsor $\widehat{\Z}(1)^{\times} = \varprojlim_{p \nmid n} \mu_n^{\times}$, where $\mu_n^{\times} \subset \mu_n$ denotes the subscheme of primitive roots of unity.

\begin{definition} \label{def:universal_marked_central_extension}
	Let $\mathcal{C} \subset \mathcal{C}_G$ be a Galois invariant subset of conjugacy classes which generate $G$.
	
	Define $\hat{U}(G, \mathcal{C})$ to be the profinite group scheme such that $\hat{U}(G, \mathcal{C})(k^{\sep})$ is the profinite group generated by symbols of the form $[\gamma]^{\zeta}$ for all $\gamma \in c \in \mathcal{C}$ and all $\zeta \in \widehat{\Z}(1)^{\times}$. These generators are subject to the relations $[\gamma]^{\lambda \zeta} = ([\gamma]^{\zeta})^{\lambda}$ for all $\lambda \in \hat{\Z}^{\times}$ and $[\xi]^{\eta} [\gamma]^{\zeta} [\xi]^{-\eta} = [\xi(\eta) \gamma \xi(\eta)^{-1}]^{\zeta}$ for all $\xi \in c' \in  \mathcal{C}$ and $\eta \in \widehat{\Z}(1)^\times$.
	
	The $\Gamma_k$-action on $\hat{U}(G, \mathcal{C})$ is given by the formula $\sigma([\gamma]^{\zeta}) = [\sigma(\gamma)]^{\sigma(\zeta)}$ where $\sigma \in \Gamma_k$ and $\gamma \in G(-1)(k^{\sep})$.
\end{definition}
The presence of $\zeta$ in the generators is to give the correct definition of the Galois action and to be able to define the morphism $\hat{U}(G, \mathcal{C}) \to G:[\gamma]^{\zeta} \to \gamma(\zeta) $. 
This morphism preserves the relations and is equivariant with respect to the Galois action and is thus well-defined. It is a central extension by the same argument as in \cite[Lem.~2.1]{Woo21}.

We now define orbifold homology.

\begin{definition}
	We denote by $\H_{2, \orb}^{\mathcal{C}}(G, \widehat{\Z}) :=\ker(\hat{U}(G, \mathcal{C}) \to G)$. 
	We consider $\hat{U}(G, \mathcal{C}) \to G$ as a marked central extension by giving it the marking consisting of the functions $\widehat{\Z}(1)^{\times} \to \hat{U}(G, \mathcal{C}): \zeta \to [g]^{\zeta}$ for all $g \in c \in \mathcal{C}$.
\end{definition}
We then have the following analogue of \cite[\S8.1.6]{EVW13}.
\begin{proposition}\label{prop:universal_marked_extensions}
	If $\mathcal{C}$ generates $G$ then $\hat{U}(G, \mathcal{C}) \to G$ is an initial object in the category of $\mathcal{C}$-marked central extensions.
\end{proposition}
\begin{proof}
	Let $E \to G$ be a marked central extension and for each $\gamma \in c \in \mathcal{C}$ let $e_\gamma: \widehat{\Z}(1)^{\times} \to E$ be the corresponding marked element. The only possible map of central extensions $\hat{U}(G, \mathcal{C})  \to E$ which preserves the markings has to send $[\gamma]^{\zeta} \mapsto e_\gamma(\zeta)$. Conversely, we note that the formula $[\gamma]^{\zeta} \to e_\gamma(\zeta)$ defines a map of extensions since it is compatible with the relations by the definition of a marking.
\end{proof}
Following this proposition, providing $\mathcal{C}$ generates $G$,  we will call $\hat{U}(G, \mathcal{C}) \to G$ the \emph{universal central $\mathcal{C}$-marked extension}. Note that since initial objects are unique up to unique isomorphism this property defines $\hat{U}(G, \mathcal{C})$.

An immediate consequence of Proposition \ref{prop:universal_marked_extensions} and Lemma \ref{lem:marked_central_extension_kernel} is the following formula for $\H^{2, \text{orb}}_{\mathcal{C}}(BG, A)$, which is analogous to the universal coefficient theorem.

\begin{proposition} \label{prop:hom_cohom}
	Assume that $\mathcal{C}$ generates $G$. For all profinite tame abelian group schemes $A$
	the map $\Hom(\H_{2, \orb}^{\mathcal{C}}(G, \widehat{\Z}), A) \to \H^{2, \orb}_{\mathcal{C}}(BG, A)$ is an isomorphism of groups.
\end{proposition}
\begin{remark}
	If $G$ is constant then $\hat{U}(G, \mathcal{C})(k^{\sep})$ agrees with what Wood \cite[p. 7]{Woo21} denotes $\hat{U}_k(G, \mathcal{C})$ equipped with the Galois action given by $\sigma(u) = \chi_{\text{cycl}}(\sigma) \cdot u$, where $\sigma \in \Gamma_k$, $u \in \hat{U}_k(G, \mathcal{C})$, $\chi_{\text{cycl}}: \Gamma_k \to \widehat{\Z}^{\times}$ denotes the cyclotomic character and $\cdot$ is defined in \textit{loc.~cit.} page 7.
\end{remark}
It is immediate from the construction that universal marked central extensions are compatible with products.
\begin{proposition}\label{prop:universal_marked_central_product}
	Let $G_1, G_2$ finite \'etale tame group schemes over $k$. Let $\mathcal{C}_i \subset \mathcal{C}_{G_i}$ generate $G_i$ for $i = 1,2$. Let $\mathcal{C} = \mathcal{C}_1 \times \{e\} \cup \{e\} \times \mathcal{C}_2 \subset \mathcal{C}_{G_1 \times G_2}$. Then the canonical map $\hat{U}(G_1 \times G_2, \mathcal{C}) \to \hat{U}(G_1, \mathcal{C}_1) \times \hat{U}(G_2, \mathcal{C}_2)$ is an isomorphism.
\end{proposition}
\begin{proof}
	Consider the definition of $\hat{U}(G_i, \mathcal{C}_i)$ in terms of generators and relations. This leads to a construction of the product $\hat{U}(G_1, \mathcal{C}_1) \times \hat{U}(G_2, \mathcal{C}_2)$ in terms of generators and relations which is exactly the same as the definition of $\hat{U}(G_1 \times G_2, \mathcal{C})$. The induced isomorphism is clearly the canonical map.
\end{proof}
\subsection{Orbifold Kummer exact sequence}
Lemma \ref{lem:central_extension_unramified} implies that base change from $\mu_n$ to $\Gm$ defines a morphism $\H^{2, \text{orb}}_{\mathcal{C}}(G, \mu_n) \to \Br_{\mathcal{C}} BG[n]$. We will now describe the kernel of this morphism. To do this we first construct a morphism $\PicOrb_{\mathcal{C}}(BG) \to \H^{2, \text{orb}}_{\mathcal{C}}(G, \mu_n)$.


For $\chi \in \dual{G}(k)$ consider the central extension $1 \to \mu_n \to G_{\chi, n } \to G \to 1$ defined in Lemma \ref{lem:boundary_map_Kummer_theory}. Let $n \mid m \in \N$ be such that $G_{\chi, n} \subset G \times \mu_m \subset G \times \Gm$. We then have $(\gamma, r) \in G_{\chi, n}(-1) \subset G(-1) \times \Z/m\Z \subset G(-1) \times \Q/\Z$ if and only if for all $\zeta \in \mu_m$ we have $\chi(\gamma(\zeta)) = \zeta^{n r}$, or in other words $\age(\chi, \gamma) = n r$. 

For $(\chi, w) \in \PicOrb_{\mathcal{C}}(BG)$ consider the marking 
\begin{equation} \label{eqn:C_chi_w_n}
	\mathcal{C}_{(\chi, w), n} \subset G_{\chi, n}(-1)
\end{equation} given by all pairs $(\gamma, \frac{w(C)}{n} \bmod \Z) \in G_{\chi, n}(-1)$, where $C \in \mathcal{C}$, $\gamma \in C$ and $\frac{w(C)}{n} \in \Q$. This is well-defined because $w(C) = \age(\chi, \gamma)$.

Then the map $\PicOrb_{\mathcal{C}}(BG) \to \H^{2, \text{orb}}_{\mathcal{C}}(G, \mu_n)$ sending $(\chi, w)$ to the central extension $G_{\chi, n}$ equipped with the marking $\mathcal{C}_{\chi,w}$ is a group homomorphism. We then have the following analogue for the partially unramified Brauer group of the Kummer exact sequence (Lemma \ref{lem:boundary_map_Kummer_theory}).

\begin{theorem}\label{thm:orbifold_Kummer}
	Assume that $\mathcal{C}$ generates $G$.
	The following sequence is exact for all tame $n$
	\begin{equation*}
		\PicOrb_{\mathcal{C}}(BG) \xrightarrow{n} \PicOrb_{\mathcal{C}}(BG)  \to \H^{2, \orb}_{\mathcal{C}}(BG, \mu_n) \to \Br_{\mathcal{C}} BG/\Br k \xrightarrow{n} \Br_{\mathcal{C}} BG/ \Br k.
	\end{equation*}
\end{theorem}
\begin{proof}
	Exactness at $\Br_{\mathcal{C}} BG/\Br k$ is due to Lemmas \ref{lem:central_gerbe}
	and \ref{lem:central_extension_unramified}.
	
	Consider now exactness at $\H^{2, \text{orb}}_{\mathcal{C}}(G, \mu_n)$. An element maps to $0$ in $\Br_{\mathcal{C}} BG$ if and only if it is a marking on the extension $G_{\chi,n} \to G$ for some $\chi \in \dual{G}(k)$, by Lemma~\ref{lem:boundary_map_Kummer_theory}. Note that in $G_{\chi,n}(-1) \subset G(-1) \times \Q/\Z$ one has that $(\gamma, r)$ lies in the same $G$-orbit as $(\gamma',r')$ if and only if $\gamma$ is conjugate to $\gamma'$ and $r = r'$. It follows that all markings consist of the elements $(\gamma, r(C))$ where $\gamma \in C \in \mathcal{C}$ and $r$ is a function $\mathcal{C} \to \Q/\Z$ such that $\age(\chi, g) = n r(C)$.
	
	 Let $w: \mathcal{C} \to \Q$ be a $\Gamma_k$-equivariant lift of  $nr: \mathcal{C} \to \Q/ n \Z$. We then have $(\chi, w) \in \PicOrb_{\mathcal{C}} BG$ since $\age(\chi, C) = n r(C) = w(C) \bmod \Z$ for all $C \in \mathcal{C}$
	 and the marking is equal to $\mathcal{C}_{(\chi, w), n}$ from \eqref{eqn:C_chi_w_n}.
	 
	 It remains to prove exactness at $\PicOrb_{\mathcal{C}}(BG)$. Let $(\chi, w) \in \PicOrb_{\mathcal{C}}(BG)$ be such that that there exists an isomorphism $f: G \times \mu_n \to G_{\chi,n}$ which sends the trivial marking to the marking $\mathcal{C}_{(\chi, w), n}$.
	 
	 As $f$ is a map of extensions which preserves $\mu_n$ we have $f(g, \zeta) = (g, \psi(g)\zeta)$ where $\psi: G \to \Gm$ is a morphism such that $\chi(g) = \psi(g)^n$ for all $g \in G$. The condition that $f$ preservers the markings is then equivalent to $\age(\psi, g) = w(C)/n \mod \Z$ for all $g \in C \in \mathcal{C}$. This means that $(\psi, w/n) \in \PicOrb BG$ which implies that $(\chi, w)  = n \cdot (\psi, w/n)\in n \PicOrb_{\mathcal{C}} BG$.
	 
	 Conversely, if $(\chi, w) \in n \PicOrb_{\mathcal{C}} BG$ then there exists $\psi \in \dual{G}(k)$ such that we have $(\psi, w/n) \in \PicOrb_{\mathcal{C}} BG$. The same reasoning as above then implies that the map $G \times \mu_n \to G_{\chi,n}: (g, \zeta) \to (g, \psi(g) \zeta)$ is an isomorphism of central extensions which sends the trivial marking to $\mathcal{C}_{(\chi, w), n}$. This implies that the image of $(\chi, w)$ in $\H^{2, \text{orb}}_{\mathcal{C}}(G, \mu_n)$ is $0$.
\end{proof}
Unfolding the definitions we see that for all $n \mid m$ we have the commutative diagram  
\[
\xymatrix{
	0 \ar[r]  & \PicOrb_{\mathcal{C}}(BG)/n \PicOrb_{\mathcal{C}}(BG) \ar[r] \ar[d]^{\cdot m/n} & \H^{2, \text{orb}}_{\mathcal{C}}(BG, \mu_n) \ar[r] \ar[d] &  \Br_{\mathcal{C}} BG[n] \ar[r] \ar[d] & 0\\
	0  \ar[r] & \PicOrb_{\mathcal{C}}(BG)/m \PicOrb_{\mathcal{C}}(BG) \ar[r] & \H^{2, \text{orb}}_{\mathcal{C}}(BG, \mu_m) \ar[r] & \Br_{\mathcal{C}} BG[m] \ar[r] & 0
}
\]
This diagram has exact rows by Theorem \ref{thm:orbifold_Kummer}. Define $\H^{2, \orb}_{\mathcal{C}}(BG, \Q/\Z(1)) := \lim_{\rightarrow{n}} \H^{2, \text{orb}}_{\mathcal{C}}(BG, \mu_n)$. Taking the colimit of the above diagram over $n$ we thus find the following.
\begin{corollary} \label{cor:orbifold_Kummer}
	The following sequence is exact
	\[
	0 \to \PicOrb_{\mathcal{C}}(BG) \otimes \Q/\Z  \to \H^{2, \orb}_{\mathcal{C}}(BG, \Q/\Z(1)) \to \Br_{\mathcal{C}} BG/\Br k \to 0
	\]
	\end{corollary}
\begin{remark}
	One can interpret $\H^{2, \text{orb}}_{\mathcal{C}}(BG, \Q/\Z(1))$ as being the group of all marked central extensions with kernel $\Q/\Z(1)$. The group $\PicOrb_{\mathcal{C}}(BG) \otimes \Q/\Z \cong \Hom(\mathcal{C}, \Q/\Z)$ can then be thought of as corresponding to different markings.
\end{remark}

An important consequence of this is that even though the Brauer group itself is not compatible with products, it turns out that the partially unramified Brauer group is, providing one is in a balanced case.
\begin{proposition}\label{prop:product_Brauer_group}
	Let $G_1, G_2$ finite \'etale tame group schemes over $k$. Let $\mathcal{C}_i \subset \mathcal{C}_{G_i}$ and assume that $\mathcal{C}_i$ generates $G_i$ for $i = 1,2$. Let $\mathcal{C} = \mathcal{C}_1 \times \{e\} \cup \{e\} \times \mathcal{C}_2 \subset \mathcal{C}_{G_1 \times G_2}$. Then the map
	\begin{equation*}
		\Br_{\mathcal{C}} B(G_1 \times G_2)/\Br k \to \Br_{\mathcal{C}_1} B G_1/\Br k \times \Br_{\mathcal{C}_2} B G_2/\Br k
	\end{equation*}
	induced by $BG_i \to B(G_1 \times G_2)$
	is an isomorphism.
\end{proposition}
\begin{proof}
	Combining Propositions \ref{prop:hom_cohom} and \ref{prop:universal_marked_central_product} and we find that 
	\[
	\H^{2, \orb}_{\mathcal{C}}(B(G_1 \times G_2), \Q/\Z(1)) \cong \H^{2, \orb}_{\mathcal{C}_1}(BG_1, \Q/\Z(1)) \times \H^{2, \orb}_{\mathcal{C}_2}(BG_2, \Q/\Z(1)).
	\]
	We then conclude using Corollary \ref{cor:orbifold_Kummer} and Lemma \ref{lem:Picorb_products}.
\end{proof}
\subsection{Algebraic Brauer group}
Let $k$ be a field and $G$ a finite \'etale tame group scheme over $k$ with $\mathcal{C} \subset \mathcal{C}_G$ Galois invariant. We now calculate the partially unramified algebraic Brauer group  $\Br_{\mathcal{C},1} BG$. We call a central extension \emph{algebraic} if the corresponding Brauer group element is algebraic. One can construct such extensions as Galois twists of the trivial central extension $\mu_n \times G$.

\begin{definition} \label{def:algebraic_central_extension}
	Let $n$ be coprime to $p$ and $\alpha \in \H^1(k, \Hom(G, \mu_n))$. We define the central extension
	$1 \to \mu_n \to E_{\alpha} \to G \to 1$ by defining $E_{\alpha}(k^{\sep}) = \mu_n(k^{\sep}) \times G(k^{\sep})$ equipped with the twisted Galois action $\sigma(\zeta, g) = (\sigma(\zeta) \cdot \alpha(\sigma)(g), \sigma(g))$ for $\sigma \in \Gamma_k$ and $(\zeta, g) \in \mu_n(k^{\sep}) \times G(k^{\sep})$.
\end{definition}

We next show that they all arise this way.
\begin{lemma}\label{lem:algebraic_central_extension}
	We have $b_{E_{\alpha}} \in \Br_1 BG$. The image of $b_{E_{\alpha}}$ under $\Br_1 BG \to \H^1(k, \Pic BG)$ is the image of $\alpha$ along the map $\H^1(k, \Hom(G, \mu_n)) \to \H^1(k, \Pic BG)$. 
	Hence every element of $\Br_e BG$ arises via the construction in Definition \ref{def:algebraic_central_extension} for some $n \mid \exp(G)$.
\end{lemma}
\begin{proof}
	By definition $E_{\alpha}$ becomes isomorphic to the trivial central extension $\mu_n \times G$ 
	after a finite field extension, hence $b_{E_{\alpha}} \in \Br_1 BG$. For the second part,
	we have $\Br_1 BG \cong \Br_1 BG^{\text{ab}}$ so by functoriality we may assume that $G$ is abelian. This then follows after unfolding Proposition \ref{prop:geometric_description_residue} for the case $\mathcal{X} = \Spec k$.
\end{proof}

We next study in detail the residue for algebraic Brauer group elements, and give various different interpretations.
\subsubsection{Stacky residue}
We make explicit Definition \ref{def:residue_along_sector} for algebraic Brauer group elements. Let $\mathcal{S}_{\mathcal{C}}$ be the sector corresponding to $\mathcal{C}$ from Lemma \ref{lem:sectors_BG}. Consider the composition induced from Lemma \ref{lem:H1_Pic}:
\begin{equation}\label{eqn:residue_map_alg_Br_stacky}
	\H^1(k , \dual{G}) \cong \Br_e BG \subset \Br BG \xrightarrow{\partial_{\mathcal{S}_{\mathcal{C}}}} \H^1(\mathcal{S}_{\mathcal{C}}, \Q/ \Z).
\end{equation}
Elements of $ \Br_e BG$ become trivial after base change to $k^{\sep}$ so the image of this map is contained in the kernel of the map $\H^1(\mathcal{S}_{\mathcal{C}}, \Q/ \Z) \to \H^1(\mathcal{S}_{\mathcal{C}, k^{\sep}}, \Q/ \Z)$. This is equal to $\H^1(k, \H^0(\mathcal{S}_{\mathcal{C}, k^{\sep}}, \Q/\Z))$ by the Hochschild-Serre spectral sequence. The geometrically connected components of $\mathcal{S}_{\mathcal{C}, k^{\sep}}$ are in a $\Gamma_k$-equivariant bijection with $\mathcal{C}$ by Lemma \ref{lem:sectors_BG}. We thus have $\H^0(\mathcal{S}_{\mathcal{C}, k^{\sep}}, \Q/\Z) = \Hom(\mathcal{C}, \Q/\Z)$ as a $\Gamma_{k}$-module.

The map in \eqref{eqn:residue_map_alg_Br_stacky} thus factors through a map $\H^1(k , \dual{G}) \to \H^1(k, \Hom(\mathcal{C}, \Q/\Z))$ which we will call the \emph{stacky residue} $\res_{\mathcal{C}, \text{Stacky}}$.

\subsubsection{Picard residue}
 By Lemma~\ref{lem:Picorb(BG,C)} and Shapiro's Lemma we have for all subsets $\mathcal{C} \subset \mathcal{C}^{*}_G$ an exact sequence 
\begin{equation}\label{eqn:exact_seq_H1_Picorb_H1_Pic}
	0 \to \H^1(k , \PicOrb_{\mathcal{C}} BG_{k^{\sep}}) \to \H^1(k , \dual{G}) \to \H^2(k, \Hom(\mathcal{C}, \Z)).
\end{equation}
Note that $\H^p(k, \Hom(\mathcal{C}, \Q)) = 0$ for $p > 0$ because $\Hom(\mathcal{C}, \Q)$ is divisible. The short exact sequence 
\[
0 \to \Hom(\mathcal{C}, \Z) \to \Hom(\mathcal{C}, \Q) \to \Hom(\mathcal{C}, \Q/\Z) \to 0
\]
then induces an isomorphism $\H^2(k, \Hom(\mathcal{C}, \Z)) \cong \H^1(k, \Hom(\mathcal{C}, \Q/\Z))$.

We define the \emph{Picard residue} via the composition
\begin{equation}\label{eqn:residue_map_alg_Br_les}
	\res_{\mathcal{C}, \Pic}: \H^1(k , \dual{G}) \to \H^2(k, \Hom(\mathcal{C}, \Z)) \cong \H^1(k, \Hom(\mathcal{C}, \Q/\Z)).
\end{equation}

\subsubsection{Age residue}\label{sec:age_residue}
Recall the age pairing from Definition \ref{def:age}. Consider the $\Gamma_k$-equivariant map
\[
\age(\cdot, \mathcal{C}): \dual{G}(k^{\sep}) \to \Hom(\mathcal{C}, \Q/\Z): \chi \to (c \in \mathcal{C} \to \age(\chi, c)).
\]
Let $\res_{\mathcal{C}, \age}: \H^1(k , \dual{G}) \to \H^1(k, \Hom(\mathcal{C}, \Q/\Z))$ be the induced map on cohomology; we call it the \emph{age residue}.

The following lemma gives an alternative description of the age residue. 
\begin{lemma}\label{lem:age_residue_restriction}
	Let $C \in \mathcal{C}$ and $k(C)$ its field of definition. The age residue is equal to the composition
	\[
		\H^1(k , \dual{G}) \to \H^1(k(C) , \dual{G}) \to \H^1(k(C), \Q/\Z) \cong \H^1(k, \Hom(\mathcal{C}, \Q/\Z)).
	\]
	The first map is the restriction with respect to the field extension $k \subset k(C)$,
	the second is the map induced by $\age(\cdot, c): \dual{G}(k^{\sep}) \to \Q/\Z$ on $\H^1$
	and the last isomorphism is due to Shapiro's lemma.
\end{lemma}
\begin{proof}
	The isomorphism $\H^1(k, \Hom(\mathcal{C}, \Q/\Z)) \cong \H^1(k(C), \Q/\Z)$ given by Shapiro's Lemma has the following description. It is the composition 
	\[
	\H^1(k, \Hom(\mathcal{C}, \Q/\Z)) \to \H^1(k(C), \Hom(\mathcal{C}, \Q/\Z)) \to \H^1(k(C), \Q/\Z).
	\]
	The first map is restriction and the second map is given by $\Hom(\mathcal{C}, \Q/\Z) \to \Q/\Z: f \to f(C)$. The lemma then follows the commutativity of the following diagram
	\[
	\xymatrix{
		\H^{1}(k, \hat{G}(k^{\sep})) \ar[rr] \ar[d] & & \H^1(k, \Hom(\mathcal{C}, \Q/\Z)) \ar[d] \\
		\H^1(k(C), \hat{G}(k^{\sep})) \ar[r] & \H^1(k(C), \Q/\Z) & \H^1(k(), \Hom(\mathcal{C}, \Q/\Z)).   \ar[l] 
	} 
	\]
\end{proof}

This version of the age residue seems to be the most amenable for explicit computations.

\subsubsection{Comparison of residues}
\begin{lemma}\label{lem:residue_maps_equal}
	The maps $\res_{\mathcal{C}, \Pic}, \res_{\mathcal{C}, \emph{Stacky}}$ and $\res_{\mathcal{C}, \age}$ are equal for all Galois orbits $\mathcal{C} \subset \mathcal{C}^*_G$.
\end{lemma}
\begin{proof}
	We will first show that $\res_{\mathcal{C}, \Pic} = \res_{\mathcal{C}, \age}$. For this consider the following diagram with exact rows
	\begin{equation*}
		\begin{tikzcd}
		0 & {\Hom(\mathcal{C}, \Z)} & {\PicOrb_\mathcal{C} BG_{k^{\sep}}} & {\dual{G}(k^{\sep})} & 0 \\
		0 & {\Hom(\mathcal{C}, \Z)} & {\Hom(\mathcal{C}, \Q)} & {\Hom(\mathcal{C}, \Q/\Z)} & 0
		\arrow[from=1-1, to=1-2]
		\arrow[from=1-2, to=1-3]
		\arrow["{=}"{marking, allow upside down}, draw=none, from=1-2, to=2-2]
		\arrow[from=1-3, to=1-4]
		\arrow["{(\chi, w) \to w}", from=1-3, to=2-3]
		\arrow[from=1-4, to=1-5]
		\arrow["{\age(\cdot, \mathcal{C})}", from=1-4, to=2-4]
		\arrow[from=2-1, to=2-2]
		\arrow[from=2-2, to=2-3]
		\arrow[from=2-3, to=2-4]
		\arrow[from=2-4, to=2-5].
		\end{tikzcd}
	\end{equation*}
	This diagram commutates because for all $(\chi, w) \in \PicOrb_{\mathcal{C}} BG_{k^{\sep}}$ we have $\age(\chi, \mathcal{C}) \equiv w \bmod \Z$ by Definition \ref{def:partial_orbifold_line_bundle}.
	This diagram induces the following commutative diagram in cohomology
	\[
	\begin{tikzcd}
	& {\H^1(k, \dual{G}(k^{\sep}))} & {\H^2(k, \Hom(\mathcal{C}, \Z))} \\
	0 & {\H^1(k, \Hom(\mathcal{C}, \Q/\Z))} & {\H^2(k, \Hom(\mathcal{C}, \Z))} & 0
	\arrow[from=1-2, to=1-3]
	\arrow["{\res_{\mathcal{C}, \age}}", from=1-2, to=2-2]
	\arrow["{=}"{marking, allow upside down}, draw=none, from=1-3, to=2-3]
	\arrow[from=2-1, to=2-2]
	\arrow["\cong"{description}, draw=none, from=2-2, to=2-3]
	\arrow[from=2-3, to=2-4].
	\end{tikzcd}
	\]
	The commutativity of this diagram shows that $\res_{\mathcal{C}, \Pic} = \res_{\mathcal{C}, \age}$.
	
	We will now compare $\res_{\mathcal{C}, \age}$ and $\res_{\mathcal{C}, \text{Stacky}}$. Let $b \in \Br_1 BG$, let $\alpha = r(b) \in \H^1(k, \Pic BG_{k^{\sep}})$, and fix $n$ coprime to $p$ such that $\alpha \in \H^1(k, \Hom(G, \mu_n))$. Let $E_\alpha$ be as in Definition \ref{def:algebraic_central_extension}.
	We may then assume that $b = b_{E_{\alpha}}$ by Lemma \ref{lem:algebraic_central_extension}.

	Let $\mathcal{C}_c \subset G(-1)$ denote the subscheme of $G(-1)$ corresponding to $c \in \mathcal{C}$, as in Lemma \ref{lem:residue_central_extension}. Let $\mathcal{E}_c$ be the pre-image of $\mathcal{C}_c \subset G(-1)$ under the map $E_{\alpha}(-1) \to G(-1)$. By the definition of $E_{\alpha}$ we have $\mathcal{E}_c(k^{\sep}) = \Z/n \Z \times \mathcal{C}$ equipped with the twisted Galois action $\sigma(r, \gamma) \to (r + \age(\alpha(\sigma), \gamma), \sigma(\gamma))$, where $\sigma \in \Gamma_k$ and $(r, \gamma) \in E_{\alpha}(k^{\sep})$.
	
	We have that $\res_{\mathcal{C}, \text{Stacky}}(\alpha)$ is equal to the cocycle representing the $\Z/n \Z$-torsor $[\mathcal{E}_c/G] \to [\mathcal{C}_c/G]$ by Lemma \ref{lem:residue_central_extension}. This cocycle is computed to be $\sigma \in \Gamma_k \to \age(\alpha(\sigma), \gamma)$. This is exactly the definition of $\res_{\mathcal{C}, \age}(\alpha)$.
\end{proof}

This leads to the following description of the algebraic Brauer group. Part (2) in particular shows that the partial orbifold Picard group plays the expected role in the classical isomorphism $\H^1(k,\Pic X_{k^{\mathrm{sep}}}) \cong \Br_1 X/ \Br k$ for any variety $X$ with $k^{\mathrm{sep}}[X]^\times =k^{\mathrm{sep}, \times}$  \cite[Prop.~5.4.2]{Col21}.

\begin{theorem} \label{thm:Br_BG}
	Let $G$ be a finite \'etale tame group scheme over a field $k$
	and $\mathcal{C} \subset \mathcal{C}_G^*$ a Galois invariant collection of conjugacy classes.
	The following subgroups of $\H^1(k , \dual{G})$ are equal:
	\begin{enumerate}
		\item The image of $\Br_{\mathcal{C},1} BG$ under the map
		$r: \Br_1 BG \to \H^1(k , \dual{G})$.
		\item The subgroup $\H^1(k , \PicOrb_{\mathcal{C}} BG_{k^{\sep}})$
		embedded via \eqref{eqn:exact_seq_H1_Picorb_H1_Pic}.
		\item The kernel of the residue map $\H^1(k , \dual{G}) \to \H^1(k, \Hom(\mathcal{C},\Q/\Z))$.
	\end{enumerate}
	In particular $\Br_{\mathcal{C},1} BG / \Br k \cong \H^1(k , \PicOrb_{\mathcal{C}} BG_{k^{\sep}})$. If $\mathcal{C}$ generates $G$ then this group is finite.
\end{theorem}
\begin{proof}
	Immediate from Lemma \ref{lem:Br_partially_unramified_purity}, Lemma \ref{lem:residue_maps_equal} 
	and the exact sequence \eqref{eqn:exact_seq_H1_Picorb_H1_Pic}. Finiteness then follows from the fact
	that $\PicOrb_{\mathcal{C}} BG_{k^{\sep}}$ is a finitely generated free abelian
	group (Lemma \ref{lem:Picorb_Galois_action}).
\end{proof}

We can now show the finiteness from Theorem \ref{thm:finiteness_partially_unramified_Brauer_group}.

\begin{corollary} \label{cor:Br_finite}
	If $\mathcal{C}$ generates $G$, then $\Br_{\mathcal{C}} BG / \Br k$ is finite.
\end{corollary}
\begin{proof}
	Immediate from Theorem \ref{thm:Br_BG} and Corollary \ref{cor:no_p_torsion}.
\end{proof}

The following lemmas assist with computing the partially unramified Brauer group.
The stated property in the first lemma holds for example if $G \to G^{\mathrm{ab}}$ admits
a section.

\begin{lemma} \label{lem:Brun_ab}
	Let $f:G  \to G^{\mathrm{ab}}$ be the abelianisation and
	$\mathcal{C} \subset \mathcal{C}_{G^{\mathrm{ab}}}^*$ Galois invariant.
	Assume that for all Galois orbits $\mathcal{C}' \subset \mathcal{C}$ there
	is a Galois orbit $\mathcal{T} \subset f^{-1}(\mathcal{C}')$ such 
	that $k(\mathcal{C}') = k(\mathcal{T})$.
	Then the map $\Br_{\mathcal{C},1} BG^{\mathrm{ab}} \to \Br_{f^{-1}(\mathcal{C}),1} BG$ 
	is an isomorphism.
\end{lemma}
\begin{proof}
	We may assume that $\mathcal{C}$ consists of a single Galois orbit.
	The map $\Br_{1} BG^{\mathrm{ab}} \to \Br_{1} BG$ is an isomorphism
	by Lemma \ref{lem:H1_Pic}. With respect to this we have an embedding 
	$\Br_{\mathcal{C},1} BG^{\mathrm{ab}} \subseteq 
	\Br_{f^{-1}(\mathcal{C}),1} BG$.
	Choose a Galois orbit $\mathcal{T} \subset f^{-1}(\mathcal{C})$ 
	as in the statement. Then any $b \in \Br_{f^{-1}(\mathcal{C}),1} BG$
	is unramified along $\mathcal{T}$. 
	
	The assumption $k(\mathcal{C}') = k(\mathcal{T})$ implies that $\age(\cdot, \mathcal{C}') = \age(\cdot, \mathcal{T})$ (as maps on $\Pic BG = \Pic BG^{\text{ab}}$). We then have that $\res_{\age, \mathcal{C}'} = \res_{\age, \mathcal{T}}$, which implies the statement by Lemma \ref{lem:Br_partially_unramified_purity}. 
\end{proof}

\begin{lemma} \label{lem:Br_1_trivial}
	Assume that $\mathcal{C}$ generates $G$ and that the Galois action on $\mathcal{C}$ is trivial.
	Then $\Br_{\mathcal{C},1} BG = \Br k$.
\end{lemma}
\begin{proof}
	By Lemmas \ref{lem:Picorb(BG,C)} and \ref{lem:Picorb_Galois_action}, we know that 
	$\PicOrb_{\mathcal{C}} BG_{k^{\mathrm{sep}}}$
	is a free $\Z$-module with trivial Galois action. Its first Galois cohomology group
	is therefore trivial. The result now follows from Theorem \ref{thm:Br_BG}.
\end{proof}

\begin{remark}
	Lemma \ref{lem:H1_Pic} shows that $\Br_1 BG/ \Br k \cong \H^1(k , \Pic (BG)_{k^{\sep}})$. However this group is huge in general; for example for 
	$G = \Z/n\Z$ we have $\dual{(\Z/n\Z)} = \mu_n$, thus $\Br_1 B(\Z/n\Z) / \Br k \cong k^\times /k^{\times n}$. The unramified Brauer group is much smaller, as Theorem \ref{thm:Br_BG} shows.
\end{remark}

\subsection{Procedure for calculating $\Br_{\mathcal{C}} BG/ \Br k$} \label{sec:procedure}
We summarise our theoretical investigations and explain how it leads to a procedure to calculate $\Br_{\mathcal{C}} BG/ \Br k$ as a group, as well as write down elements, when $\mathcal{C}$ generates $G$. We use the exact sequence
\begin{equation} \label{seq:Br_extension}
	0 \to \Br_{\mathcal{C},1} BG \to \Br_{\mathcal{C}} BG \to (\Br_{\mathcal{C}} BG_{k^{\sep}})^{\Gamma_k}.
\end{equation}
The first step is to calculate $\PicOrb_{\mathcal{C}}(BG_{k^\sep})$. This is a finitely generated free abelian group with Galois action. If one is very lucky there is no transcendental Brauer group (see Lemma \ref{lem:unramified_geometric_Brauer_group_computation} for example computations) and one only has to calculate the algebraic Brauer group. By Theorem~\ref{thm:Br_BG} it suffices to compute $\H^1(k , \PicOrb_{\mathcal{C}} BG_{k^{\sep}})$, for which standard computer algebra systems can be applied. This gives a description of the Brauer group as an abstract group. To actually write down elements, one can use either cup products using Lemma \ref{lem:cup_products}, or represent elements via marked central extensions using Lemma \ref{lem:algebraic_central_extension}.

If there are transcendental elements then by the orbifold Kummer sequence (Corollary \ref{cor:orbifold_Kummer}) it suffices to calculate the orbifold cohomology $\H^{2, \orb}_{\mathcal{C}}(BG, \Q/\Z(1))$. One achieves this by calculating the universal marked central extension (Definition~\ref{def:universal_marked_central_extension}); if $G$ is constant then this is computable by  \cite[Thm.~2.5]{Woo21}. Then the orbifold cohomology is calculated using Proposition \ref{prop:hom_cohom}.

An approach which works for non-constant $G$ but is less efficient, is as follows. Once the algebraic Brauer group has been calculated, by Corollary \ref{cor:no_p_torsion} one enumerates all relevant central extensions of $G$ by $\mu_{|\exp(G)|^2}$. Given such a central extension $E$, one needs to determine whether there is a Galois action on $E(k^{\sep})$ which preserves both a marking induced by $\mathcal{C}$ and the kernel $\mu_{|\exp(G)|^2}$; if a Galois action exists then it is unique as $\mathcal{C}$ generates $G$. The relevant Galois action must factor through $\Gal(L/k)$, where $L$ is a splitting field of $G$ which contains $\mu_{|\exp(E)|}$. This is a finite group so it is a finite computation. Altogether this procedure allows one to compute $\Br_{\mathcal{C}} BG / \Br k$ as a group and write down explicit representatives in terms of cup products and marked central extensions.

Naturally in special cases there are all kinds of tricks which can be used to simplify various steps of the process. In the next section we give some methods to calculate $\Br_{\mathcal{C}} BG_{k^{\sep}}$. This additional information can help to speed up the above algorithm. But a description of $\Br_{\mathcal{C}} BG_{k^{\sep}}$ alone is insufficient in general, as a Galois invariant element need not descend to $k$, and the sequence \eqref{seq:Br_extension} need not split.

\subsection{Transcendental Brauer group} \label{sec:transcendental_Brauer_group}

Let $k$ be a separably closed field, $G$ be a finite group of order coprime to the characteristic of $k$ and $\mathcal{C} \subset \mathcal{C}_G$ a collection of conjugacy classes.
We finish with a method to calculate the partially unramified transcendental Brauer group. Recall from \S \ref{sec:central_extension} that elements of $\beta \in \Br BG$ correspond to central extensions $1 \to \Q/\Z(1) \to G_{\beta} \to G \to 1$

\begin{proposition}\label{prop:partially_unramified_Brauer_groups_alg_clos}
	Let $\beta \in \Br BG \cong \H^2(G, \Q/\Z(1))$ correspond to the central extension $1 \to \Q/\Z(1) \to G_{\beta} \to G \to 1$. Then $\beta \in \Br_{\mathcal{C}} BG$ if and only if for all $\gamma \in c \in \mathcal{C}$ and $g \in G$ such that $g \gamma g^{-1} = \gamma$ there exist lifts $\hat{\gamma} \in G_{\beta}(-1)$ and $\hat{g} \in G_{\beta}$ of $\gamma,g$ respectively such that $\hat{g} \hat{\gamma} \hat{g}^{-1} = \hat{\gamma}$.
\end{proposition}
\begin{proof}
	We use Lemma \ref{lem:central_extension_unramified}.	Assume that there exists a marking $\mathcal{E} \subset G_{\beta}(-1)(k^{\sep})$. Let $\gamma \in c \in \mathcal{C}$ and $g \in G$ such that $g \gamma g^{-1} = \gamma$. Let $\hat{\gamma} \in \mathcal{E}$ be the unique lift of $\gamma$ and $\hat{g} \in G_{\beta}$ any lift of $g$. Since $\hat{\gamma}$ is unique we must have $\hat{g} \hat{\gamma} \hat{g}^{-1} = \hat{\gamma}$ as desired.
	
	For the converse we may assume that $\mathcal{C}$ consists of single conjugacy class $c$. Fix one element $\gamma \in c$ and choose any lift $\hat{\gamma} \in G_{\beta}(-1)$. Consider the map of centralizers $C_{G_{\beta}}(\hat{\gamma}) \to C_{G}(\gamma)$. This map is surjective if the condition in the statement is satisfied. Moreover $\Q/\Z(1) \subset C_{G_{\beta}}(\gamma)$ because $G_{\beta} \to G$ is a central extension. It follows that the map of quotients $G_{\beta}/C_{G_{\beta}}(\hat{\gamma}) \to G/C_{G}(\gamma)$ is a bijection. This implies that the conjugacy class of $\hat{\gamma}$ bijects onto the conjugacy class of $\gamma$. The conjugacy class of $\hat{\gamma}$ thus defines a marking.
\end{proof}

\begin{remark} \label{rem:Bogomolov}
	Fix a system of primitive roots of unity to identify $G(-1)$ with $G$ and $\Q/\Z(1)$ with $\Q/\Z$. Then the condition in Proposition \ref{prop:partially_unramified_Brauer_groups_alg_clos} translates to asking that all bicyclic groups $A \subset G$ such that one of the generators is contained in some $c \in \mathcal{C}$ can be lifted to a bicyclic group of $G_{\beta}$, i.e.~$\beta_A = 0 \in \H^2(A, \Q/\Z)$ for all such bicyclic groups $A \subset G$. In the case that $\mathcal{C} = \mathcal{C}_G^*$ this recovers Bogomolov's formula \cite[Thm.~3.1]{Bog87} for $\Brun \mathbb{A}^n/G$, hence gives a new stacky proof of Bogomolov's formula (see Example
	\ref{rem:Bogomolov_multiplier}).
	
	For general $\mathcal{C}$, one should interpret 
	Proposition \ref{prop:partially_unramified_Brauer_groups_alg_clos} 
	as giving a description for the Brauer
	group of certain open subsets of a smooth proper model of $\mathbb{A}^n/G$.
\end{remark}

There exists a (non-unique) central extension
\[1 \to \H_2(G, \Z) \to \tilde{G} \to G \to 1\]
such that for each $\beta \in \H^2(G, \Q/\Z(1)) = \Hom(\H_2(G, \Z), \Q/\Z(1))$ (the isomorphism being due to the universal coefficient theorem) the extension $G_{\beta} \to G$ is given by the pushout of $\tilde{G}$ along the map $\H_2(G, \Z) \to \Q/\Z(1)$. The group $\tilde{G}$ is known as a \emph{Schur covering group}.

\begin{lemma}\label{lem:Schur_covering_group}
	For each pair of commuting elements $g, h \in G$ let $\tilde{g}, \tilde{h} \in \tilde{G}$ be lifts of these two elements. The commutator $[\tilde{g}, \tilde{h}] = \tilde{g} \tilde{h} \tilde{g}^{-1} \tilde{h}^{-1} \in H_2(G, \Z)$ is independent of the choice of $\tilde{G}$ and of the choices of lifts $\tilde{g}, \tilde{h}$. Then the subgroup $\Br_{\mathcal{C}} BG \subset \H^2(G, \Q/\Z(1)) \cong \Hom(\H_2(G, \Z), \Q/\Z(1))$ of Proposition \ref{prop:partially_unramified_Brauer_groups_alg_clos} is given by
	$$\left\{\beta \in \Hom(\H_2(G, \Z), \Q/\Z(1)):
	\begin{array}{l}
	[\tilde{\gamma}(\zeta), \tilde{h}](\beta) = 0 \text{ for all } 
	 \gamma \in c \in \mathcal{C}, \zeta \in \widehat{\Z}(1) \\ \text{and }h \in G \text{ such that } g \gamma g^{-1} = \gamma.
	\end{array}\right\}.$$
\end{lemma}

\begin{proof}
	Let $n, m$ be the orders of $\gamma, h$ and $A =\langle \gamma(\zeta), h\rangle$. It is well-known (e.g.~it can be deduced from \cite[Lem.~2.9.1, Thm.~2.9.3]{Kar87}) that $\H^2(A,\Q/\Z) = 0$ if $A$ is cyclic and $\H^2(A, \Q/\Z) \cong \Z/\gcd(n,m) \Z$ otherwise. In the second case $\frac{k}{\gcd(n,m)} \in \Z/\gcd(n,m) \Z$ is represented by the extension
	\[
		1 \to \Q/\Z \to A_{\frac{k}{\gcd(n,m)}} \to A \to 1
	\]
	The group $A_{\frac{k}{\gcd(n,m)}}$ is generated by $\Q/\Z$ and two elements $\hat{g}, \hat{h}$. It is defined by the relations $\hat{g}^n = \hat{h}^m = 1$ and $[\hat{g}, \hat{h}] = \frac{k}{\gcd(n,m)}$. 
	
	We claim that if $\beta \in \H^2(G, \Q/\Z(1)) = \Hom(\H_2(G, \Z), \Q/\Z)$ then $\beta_A \in \H^2(A, \Q/\Z)$ is equal to $\beta([\tilde{g}, \tilde{h}])$. The lemma then follows from Proposition \ref{prop:partially_unramified_Brauer_groups_alg_clos}.
	
	The claim is clear if $A$ is cyclic so assume otherwise. Since $G_{\beta}$ is the pushout of $\tilde{G}$ we have that $\beta([\tilde{\gamma}(\zeta), \tilde{h}])$ is equal to $[\tilde{\gamma}(\zeta)_{\beta}, \tilde{h}_{\beta}] \in \Q/\Z(1)$ where $\tilde{\gamma}(\zeta)_{\beta}, \tilde{h}_{\beta}$ are lifts of $g, h$ along $G_{\beta} \to G$. The claim is functorial in $G$ so we may assume that $G = A$, this case follows immediately from the explicit description of $A_{\frac{k}{\gcd(n,m)}}$.
\end{proof}

\begin{remark} \label{rem:Schur_multiplier}
	The group $\H^2(G,\Q/\Z(1))$ from Lemma~\ref{lem:coh_sep_closed}
	is the (dual of the) Schur multiplier of $G$. 
	The appearance of the Schur multiplier in Malle's conjecture was proposed in \cite[\S2.4]{EV10}.
	We have $\H_2(G, \Z) = \H^2(G, \Q/\Z)^{\sim}$ by the universal coefficient theorem. Lemma \ref{lem:Schur_covering_group} shows that $\Br_{\mathcal{C}} BG \cong \H_2(G,\mathcal{C}; \Z)^{\sim}$ where the later group is defined by Ellenberg, Venkatesh and Westerland in the retracted preprint \cite[Def.~7.3]{EVW13}. We refer to \cite{Woo21} for a published account of this group.

	We explain these observations via the transcendental Brauer group of $BG$. In particular, Ellenberg-Venkatesh consider the (dual of) $\Br_{\mathcal{C}} BG_{\overline{k}}$ in \cite[(4)]{EV10} in the case that $\mathcal{C}$ consist of a single conjugacy class. They believe that the size of this group should play a role as long as the number field $k$ contains sufficiently many roots of unity, but the relevant factor could be smaller otherwise. This is explained by the fact that in general the transcendental Brauer group will be a subgroup of $(\Br_{\mathcal{C}} (BG)_{\overline{k}})^{\Gamma_k}$, and its cardinality appears in Conjecture \ref{conj:balanced}.
\end{remark}

The following will be used in \S \ref{sec:examples}. In this proof we fix a system of compatible roots of unity to identify $G(-1)$ with $G$ and $\Q/\Z(1)$ with $\Q/\Z$, via Lemma \ref{lem:Galois_action_on_G(-1)}.

\begin{lemma}\label{lem:unramified_geometric_Brauer_group_computation}
Let $k$ be a separably closed field, $G$ be a finite group of order coprime to the characteristic of $k$ and $\mathcal{C} \subset \mathcal{C}_G$ a collection of conjugacy classes.
In the following situations we have $\Br_{\mathcal{C}} BG = 0$.
	\begin{enumerate}
		\item $G$ is abelian and $\mathcal{C}$ generates $G$.
		\item $G = S_n$ and $\mathcal{C}$ contains a transposition.
		\item $G = D_n$ and $\mathcal{C}$ contains a reflection (i.e.~an element whose image under the map $D_n \to C_2$ is non-trivial).
		\item $G = A_4$ and $\mathcal{C}$ contains an element of order $2$.
	\end{enumerate}
\end{lemma}	
\begin{proof}
	We will use Lemma \ref{lem:Schur_covering_group} to compute $\Br_{\mathcal{C}} BG$.
	
	(1) If $G = A$ is abelian then $\H^2(A, \Q/\Z)^{\sim} =\wedge^2 A$ and any Schur covering group $0 \to \wedge^2 A \to \Tilde{A} \to A \to 0$ has the property that $[\tilde{a}, \tilde{b}] = a \wedge b$. 
		
	The elements of $\mathcal{C} \subset \mathcal{C}_A = A(-1)$ generate $A$ so every pure wedge $a \wedge b$ with $a,b$ arbitrary can be written as a sum of pure wedges with $a \in \mathcal{C}, b \in A$. We deduce that $\Br_{\mathcal{C}} BA  = 0$ since the pure wedges generate $\wedge^2 A$.
	
	(2a) If $G = S_3$ then $\H^2(G, \Q/\Z) = 0$ by \cite[Thm.~2.12.3]{Kar87}.
	
	(2b) If $G = S_n$ and $n \geq 4$ then by \cite[Thm.~2.12.3]{Kar87} one Schur covering group is the group $2 \cdot S_n^{-}$ which has generators $z$ and $s_i$ for $1 \leq i \leq n-1$ subject to the relations $z^2 = 1$, $s_i^2 = z, [s_i, z] = 1$ for $1 \leq i \leq n-1$, $(s_{i} s_{i + 1})^3 = z$ for $1 \leq i \leq n-2$ and $[s_i, s_j] = z$ for $|i - j| \geq 2$. The map $\Tilde{G} \to S_n$ is given by sending $z$ to the identity and $s_i$ to $(i, i + 1)$.
		
		The central subgroup $\H^2(S_n, \Q/\Z) \subset 2 \cdot S_n^{-}$ is the copy of $\Z/2\Z$ generated by $z$. We may then choose $\tilde{g} = s_1$ and $\tilde{h} = s_3$.
	
	(3a) If $G = D_n$ and $n$ is odd then $\H^2(G, \Q/\Z) = 0$ by \cite[Prop.~2.11.4]{Kar87}.
	
	(3b) If $G = D_{2n}$ then by the proof of \cite[Prop.~2.11.4]{Kar87} a Schur covering group is given by $D_{4n} \to D_{2n}$. This map has kernel $\Z/2\Z$ generated by $(1, 2n + 1)(2, 2n+ 2)\cdots(2n, 4n)$. A choice for $\tilde{g}$ is then any lift of the reflection and $\tilde{h} = (1, n+ 1, 2n + 1, 3n + 1)(2, n + 2, 2n + 2, 3n + 2)\cdots(n,2n,3n,4n)$. The images of these elements commute since $h = (1, n+1)(2, n+2)\cdots(n,2n)$ lies in the center of $D_{2n}$.
		
	(4) If $G = A_4$ then its Schur covering group is unique by \cite[Thm.~2.12.5]{Kar87} and it is the subgroup $2 \cdot A_4 \subset 2 \cdot S_4^{-}$ which is the pre-image of $A_4 \subset S_4$. We may assume without loss of generality that $g = (1,2)(3,4) \in c \in \mathcal{C}$. A possible $\tilde{g}$ is then $s_1 s_3$. If we pick $\tilde{h} = s_2 s_4$ then $[\tilde{g}, \tilde{h}] = z$. 
\end{proof}

The earliest examples of groups with \textit{non-trivial} unramified transcendental Brauer group (via Example \ref{rem:Bogomolov_multiplier}) are due to Saltman \cite{Sal84}. He constructed such groups of order $p^9$ for any prime $p$. We give a new example which forms the basis of Conjecture \ref{conj:A_4_conductor_intro}, namely in \S \ref{sec:A_4-quartics} we show that $A_4$ has a non-trivial partially unramified transcendental Brauer group for certain choices of generating conjugacy classes.

\section{Brauer--Manin obstruction on $BG$} \label{sec:BMO_BG}
We now study the Brauer--Manin pairing on $BG$ for a finite \'etale tame group scheme $G$, as defined in \S \ref{sec:BM}. We begin by defining the ramification type of a cocycle.

\subsection{Ramification type} \label{sec:ramification_type}
Let $\O$ be a DVR with fraction field $k$, residue field $\F$ of characteristic $p$, and $\pi$ a uniformizer. Let $G$ be a finite \'etale tame group scheme over $k$ with good reduction.

\begin{definition}
 We define the \emph{ramification type} map $\rho_G: BG(k) \to \mathcal{C}_G^{\Gamma_k}$ via taking $\rho_G = \pi_0 \circ \pmod{\pi}$ where $\pmod{\pi}$ is the modulo $\pi$ map from \S\ref{sec:modulo_pi} and we use the identification $\pi_0(I_\mu BG(k)) = \mathcal{C}_G^{\Gamma_k}$ from Lemma \ref{lem:sectors_BG}.
\end{definition}
This has the following explicit description. An element $\varphi \in BG(k)$ is represented by a cocycle $\Gamma^{\text{tame}}_k \to G(k^{\sep})$. Moreover being continuous, it factors through some finite tamely ramified Galois extension $K/k$ of some ramification degree $e$. Enlarging $K$ if necessary, we assume that $K$ contains $\mu_e$. Let $\varpi$ be a uniformiser of $K/k$ and denote by $I_{K/k} \subset \Gal(K/k)$ the inertia group. By \cite[Tag 09EE]{stacks-project} the map
\begin{equation} \label{eqn:tame_inertia_finite}
\mu_e \to I_{K/k}, \quad \zeta \mapsto (\sigma_{\zeta}: \varpi \mapsto \zeta \varpi),
\end{equation}
is an isomorphism, where $\sigma_\zeta$ is defined to act trivially on the residue field of $K$. 
Hence composing with $\varphi$ we obtain a cocycle
\begin{equation} \label{def:inertia}
	\mu_e \to I_{K/k} \overset{\varphi}{\to} G.
\end{equation}
However, as $G$ has good reduction the action of $I_k$ on $G(k^{\sep})$ is trivial, thus this is actually a homomorphism, hence yields an element of $G(-1)$. 
Changing the cocycle $\varphi$ changes the homomorphism by conjugation in $G(k^{\sep})$. This map thus defines an element of $\mathcal{C}_G$, which is checked to be invariant under $\Gamma_k$ and equals the ramification type $\rho_{G}(\varphi)$ by Lemma~\ref{lem:residue_map_BG}.

Thus the ramification type of $\varphi$ is the induced map on the inertia subgroup, viewed up to conjugacy (this is more data than simply the image of the inertia subgroup).

\subsection{Partial adelic space}

We next define a partial adelic space for $BG$. This is important as it the natural space on which the Brauer--Manin pairing is well-defined for partially unramified Brauer group elements, and moreover the measure of this space will appear in our conjecture. For the rest of this section we let $G$ be a finite \'etale tame group scheme over a global field $k$.

\begin{definition} \label{def:adelic_space}
	Let $\mathcal{C}\subseteq \mathcal{C}_G^*$ be Galois invariant.
	Then for a non-archimedean place $v$ of good tame reduction we define 
	\begin{align*}
	BG(\O_v)_{\mathcal{C}} &= \{ \varphi_v \in BG(k_v) : \rho_{G,v}(\varphi_v) \in \mathcal{C} \cup \{e\}\}, \\
	BG(\Adele_k)_{\mathcal{C}} & = \lim_{S} \prod_{v \in S} BG(k_v) 
	\prod_{v \notin S} BG(\O_v)_{\mathcal{C}},
	\end{align*}
	where the limit is over all finite sets of places $S$, i.e.~$BG(\Adele_k)_{\mathcal{C}}$
	is the restricted direct product of the $BG(k_v)$ with respect to 	$BG(\O_k)_{\mathcal{C}}$.
	We call this the \emph{partial adelic space} with respect to $\mathcal{C}$
	and  an element of $BG(\O_v)_{\mathcal{C}}$ a \textit{partial $v$-adic integral point} with respect to 
	$\mathcal{C}$.
\end{definition}

The space $BG(\Adele_k)_{\mathcal{C}}$ is locally compact. It is compact if and only if
$\mathcal{C} = \mathcal{C}_G^*$ (this is compatible with Remark \ref{rem:intuition}).

If $G$ is constant, then an element of $BG(\O_v)_{\mathcal{C}}$ exactly corresponds to a continuous homorphism $\varphi: \Gamma_{k_v} \to G$ whose ramification type (from \S\ref{sec:ramification_type}) is either trivial or lies in $\mathcal{C}$. With respect to the boundary divisor analogy (Remark \ref{rem:intuition}), the set $BG(\O_v)_{\mathcal{C}}$ should be viewed as exactly the collection of $\O_v$-points of the space given by removing the divisors not in $\mathcal{C}$.

\subsection{Brauer--Manin pairing}

We now return to the partially unramified Brauer group from Definition \ref{def:partially_ramified_Br}. We first give a description of the local evaluation map for Brauer elements in terms of the residue maps $\partial_{S}$. The following should be viewed as an analogue of \cite[Thm.~3.7.5]{Col21} in our setting. (In \emph{loc.~cit~} a valuation appears, but the corresponding valuation always equals $1$ in our case, c.f.~Remark~\ref{rem:intuition}.)

\begin{lemma} \label{lem:local_evaluation_residue_map}
	Let $\mathcal{O}$ be a DVR with fraction field $K$, valuation $v$, residue field $\F$ and fix a uniformizer $\pi$. Let $\mathcal{G}$ be a finite \'etale tame group scheme over $\mathcal{O}$ with generic fibre $G$.
	Let $\varphi \in BG(K)$ and $\mathcal{S} := \rho_G(\varphi) \in \pi_0(I_{\mu} B \mathcal{G})$ its ramification type, viewed as a sector. Let $n$ be the order of $\mathcal{S}$ and $\overline{\varphi} := \varphi \pmod{\pi} \in \mathcal{S}(\F)$.
	
	 The following square is then well-defined and commutes
	 \[\begin{tikzcd}
	 {\Br B \mathcal{G}} & {\H^1(\mathcal{S}, \Z/n\Z)} \\
	 {\ker(\Br K \to \Br K^{\emph{sh}})} & {\H^1(\F, \Q/\Z)}
	 \arrow["{\partial_{\mathcal{S}}}", from=1-1, to=1-2]
	 \arrow["{b \to b(\varphi)}"', from=1-1, to=2-1]
	 \arrow["{f \to f(\overline{\varphi})}", from=1-2, to=2-2]
	 \arrow["{\partial_{\mathcal{O}}}"', from=2-1, to=2-2].
	 \end{tikzcd}\]
\end{lemma}
\begin{proof}
	Let $\varphi: \Spec \mathcal{O}_{\sqrt[n]{v}} \to B \mathcal{G}$ be the extension of $\varphi$ using the arithmetic valuative criterion of properness (Lemma \ref{lem:arithmetic_valuative_criterion}). By the definition of $\varphi \pmod{\pi}$ and $\rho_G$ we get the following commutative diagram
	\[\begin{tikzcd}
		{\Spec k_v} & {\Spec \mathcal{O}_{\sqrt[n]{v}}} & {(B \mu_n)_{\F}} \\
		BG & {B\mathcal{G}} & {(B \mu_n)_{\mathcal{S}_{\F}}}
		\arrow[from=1-1, to=1-2]
		\arrow["{\varphi}", from=1-1, to=2-1]
		\arrow["{\varphi}", from=1-2, to=2-2]
		\arrow["{i_{\pi}}", from=1-3, to=1-2]
		\arrow["{(B \mu_n)_{\overline{\varphi}}}", from=1-3, to=2-3]
		\arrow[from=2-1, to=2-2]
		\arrow[from=2-3, to=2-2].
	\end{tikzcd}\] 
The functoriality of $\Br$ and $\partial_{\cdot, n}$ applied to this diagram and Lemma \ref{lem:Witt_residue_equal_root_residue} shows that the following diagram commutes
\[\begin{tikzcd}
	{\H^1(\F, \Q/\Z)} & {\H^1(\F, \Z/n \Z)} & {\H^1(\F, \Z/n \Z)} \\
	{\ker(\Br K \to \Br K^{\text{sh}})} & {\Br \mathcal{O}_{\sqrt[n]{v}}} & {\Br (B\mu_n)_{\F}} \\
	& {\Br B\mathcal{G}} & {\Br (B\mu_n)_{\mathcal{S}_{\F}}} & {\H^1(\mathcal{S}_{\F}, \Z/n\Z)}
	\arrow[from=1-2, to=1-1]
	\arrow["{=}"{description}, draw=none, from=1-2, to=1-3]
	\arrow["{\partial_{\mathcal{O}}}", from=2-1, to=1-1]
	\arrow["{\partial_{\mathcal{O}, n}}", from=2-2, to=1-2]
	\arrow[from=2-2, to=2-1]
	\arrow[from=2-2, to=2-3]
	\arrow["{\partial_{\F, n}}", from=2-3, to=1-3]
	\arrow["{b \to b(\varphi)}", from=3-2, to=2-2]
	\arrow[from=3-2, to=3-3]
	\arrow["{b \to b(\overline{\varphi})}", from=3-3, to=2-3]
	\arrow["{\partial_{\mathcal{S}_{\F}, n}}", from=3-3, to=3-4]
	\arrow[from=3-4, to=1-3].
\end{tikzcd}\]
The lemma follows as the composition of the bottom maps is by definition $\res_{\mathcal{S}}$ composed with the restriction map $\H^1(\mathcal{S}, \Z/n \Z) \to \H^1(\mathcal{S}_{\F}, \Z/n\Z)$.
\end{proof}

Motivated by Harari's formal lemma (Theorem \ref{thm:Harari's_formal_lemma}) we now prove the following. It demonstrates that $BG(\O_v)_{\mathcal{C}}$ and $\Br_{\mathcal{C}} BG$ interact in the expected way, i.e.~as if they were the $\O_v$-points and the Brauer group of some open subset of $BG$. The second part of this proof is inspired by Harari's formal lemma for varieties \cite[Thm.~13.4.1]{Col21}.
\begin{theorem}\label{thm:Harari_formal_partially_unramified}
	Let $k$ be a global field and $G$ a finite \'etale tame group scheme over $k$, $\mathcal{C} \subset \mathcal{C}_G$ Galois invariant, and $b \in \Br BG$. Then $b \in \Br_{\mathcal{C}} BG$ if and only if $b$ evaluates trivially on $BG(\O_v)_{\mathcal{C}}$ for all but finitely many $v$.
	
	More precisely, let $v$ be a place of good reduction coprime to the order of $G$ and $\mathcal{G}_v$ a group scheme model of $G$ over $\mathcal{O}_v$. Then any $b \in \Br B \mathcal{G}_v \cap \Br_{\mathcal{C}} BG$ evaluates trivially on $BG(\O_v)_{\mathcal{C}}$.
\end{theorem}
\begin{proof}
	Let $b \in \Br BG$. Fix $S$ a finite set of places containing the archimedean ones such that $G$ has a finite \'etale tame model $\mathcal{G}$ over $\mathcal{O}_{k, S}$. 
	In which case $B \mathcal{G}$ is a tame proper \'etale DM stack over $\mathcal{O}_{k,S}$.
	By enlarging $S$ we may assume that $b \in \Br B\mathcal{G}$.
	
	Let $v \not \in S$, let $\varphi_v \in B \mathcal{G}(\mathcal{O}_v)_{\mathcal{C}}$ and
	$\mathcal{S} = \rho_{G,v}(\varphi_v) \in \mathcal{C}$. If $b \in \Br_{\mathcal{C}} BG_{k_v}$ then $\partial_{\mathcal{S}}(b) = 0 \in \H^1(\mathcal{S}_{\mathcal{O}_v}, \Q/\Z) \subset \H^1(\mathcal{S}_{k_v}, \Q/\Z)$. Lemma \ref{lem:local_evaluation_residue_map} then implies that $\partial_{\mathcal{O}_v}(b(\varphi_v)) = 0$. The exactness of \eqref{eq:purity_DVR} implies that $b(\varphi) \in \Br \mathcal{O}_v = 0$, as required.
	
	Assume now that $b \not \in {\Br_{\mathcal{C}} BG}$. Then there exists a sector $\mathcal{S}$ of order $n$ such that $\partial_{\mathcal{S}}(b) \neq 0 \in \H^1(\mathcal{S}, \Z/n\Z)$.
	By Lemma \ref{lem:Chebotarev} there exists $v \not \in S$ and a point $\overline{x} \in \mathcal{S}(\F_v)$ such that $\partial_{\mathcal{S}}(b)(x) \neq 0 \in \H^1(\F_v, \Z/n\Z)$. By Theorem \ref{thm:Hensel} there exists an $x \in BG(K_v)$ such that $x \pmod{\pi} = \overline{x}$. It then follows from Lemma~\ref{lem:local_evaluation_residue_map} that $\res_{\mathcal{O}}(b(x)) \neq 0$ from which we deduce that $b(x) \neq 0$, as required.
\end{proof}

\begin{lemma} \label{lem:modified_adeles_Brauer-Manin}
	Let $\mathcal{C}\subseteq \mathcal{C}_G^*$ be Galois invariant.
	Then the Brauer--Manin pairing yields a well-defined continuous pairing
	$$ \Br_{\mathcal{C}} BG \times BG(\Adele_k)_{\mathcal{C}} \to \Q/\Z.$$
\end{lemma}
\begin{proof}
	By Lemma \ref{lem:evaluation_continuous}, Theorem \ref{thm:Harari_formal_partially_unramified},
	and the proof of Lemma \ref{lem:BM_pairing}.
\end{proof}

\subsubsection{Via central extensions} \label{sec:BM_central_extension}
Let 
$$1 \to \mu_n \to E \to G \to 1$$
be a tame central extension with corresponding Brauer group element $b_E \in \Br BG$
(see \S \ref{sec:central_extension}).
Let $k \subset L$ be a field extension.
Applying Galois cohomology, by \cite[Prop.~5.7.43]{Ser02}  we obtain the exact sequence of pointed sets
$$\H^1(L,E) \to \H^1(L,G) \to \H^2(L,\mu_n) = \Br L[n].$$
Given a cocycle $\varphi \in BG(L)$, the image of $\varphi$ under this map is exactly the evaluation of $b_E(\varphi) \in \Br L$. This description can be used to calculate the Brauer--Manin obstruction in terms of $E$. For example, from exactness we see that $b_E(\varphi) = 0$ if and only if $\varphi$ lifts to $E$, i.e.~if and only if the corresponding embedding problem has a solution. We use this approach to calculate the Brauer--Manin obstruction for Conjecture \ref{conj:A_4_conductor_intro}.

\subsection{Stickelberger's Theorem as a Brauer-Manin obstruction} \label{sec:Stickelberger}
Consider the stack $BS_n$ classifying $S_n$-extensions, or equivalently degree $n$ extensions (Lemma \ref{lem:S_n}). It has a $\Z/2\Z$-torsor given by $B A_n \to B S_n$. Consider the image of the cup product $\alpha:=-1 \cup [B A_n \to B S_n] \in \H^2(B S_n, \mu_n)$ in the Brauer group $\Br BS_n$. We can use the Brauer--Manin obstruction coming from this Brauer element to reprove Stickelberger''s theorem.
\begin{theorem} \label{thm:Stickelberger}
	Let $k/\Q$ be a number field. Then $\Delta_{k/\Q} \equiv 0, 1 \bmod{4}$. This is explained by a Brauer--Manin obstruction to strong approximation on $BS_n$.
\end{theorem}
\begin{proof}
	Let $n \geq 2$ and  $(k_v)_{v} \in BS_n(\Adele_\Q)^\alpha$. It suffices to show that
	$\prod_v \Delta_{k_v/\Q_v}  \equiv 0,1 \bmod 4$.  By Lemma \ref{lem:BG(k)}, 
	each $k_v$ corresponds to a homomorphism 
	$\chi_v: \Gamma_{\Q_v} \to S_n$. 
	The $\Z/2\Z$-torsor $[B A_n \to B S_n](\chi_v)$ is equal to the composition $\Gamma_{\Q_v} \to S_n \to \Z/2\Z$, which corresponds to the quadratic resolvent of $k_v/\Q_v$. Since $\Delta_{k_v/\Q_v}$ is equal to a square times the discriminant of the quadratic resolvent we may assume that $n = 2$.

	If $k_2/\Q_2$ is ramified then $4 \mid \Delta_{k_2/\Q_2}$ as there are no extensions of $\Q_2$ with discriminant $2$. So we may assume that $k_2/\Q_2$ is unramified.
	We compute $\inv_v(-1, \chi_v)$ for all places $v$ of $\Q$.
	
	\begin{enumerate}
		\item $\inv_2(-1, \chi_2) = 0$ since $k_2/\Q_2$ is unramified.
		\item $\inv_\infty(-1, \chi)_\infty = 1/2$ if and only if $k_\infty \cong \C$, in which case $\Delta_{k_\infty/\R} = -1 \equiv 3 \bmod 4$. Otherwise $\Delta_{k_\infty/\R} = 1$.
		\item If $p \equiv 1 \bmod 4$ then $-1$ is a square in $\Q_p$ so $\inv_p(-1, \chi_p) = 0$. In this case $\Delta_{k_p/\Q_p}$ is a power of $p$ so $\Delta_{k_p/\Q_p} \equiv 1 \bmod 4$.
		\item If $p \equiv 3 \bmod 4$ then $-1$ is not a square in $\Q_p$ so $\inv_p(-1, \chi_p) = 1/2$ if and only if $k_p/\Q_p$ is ramified. In this case $\Delta_{k_p/\Q_p} = p \equiv 3 \bmod{4}$. Otherwise we have $\Delta_{k_p/\Q_p} = 1$.
	\end{enumerate}
	The relation $\sum_v \inv_v (-1,\chi_v) = 0$ thus shows that the number of places $v$ for which $\Delta_{k_v/\Q_v} \equiv 3 \bmod 4$ is even. Stickelberger's theorem follows.
\end{proof}

\begin{remark}
	The Brauer group element $\alpha$ we use in the proof is ramified, i.e.~lies in $\Br S_n$ but not in $\Brun S_n$. To see this we note that $\inv_p(-1, \chi_p) \neq 0$ for all $p \equiv 3 \bmod 4$ where $\chi_p$ ramifies. As $\chi_p$ varies there are an infinite number of such places, so $\alpha$ is ramified by Theorem \ref{thm:Harari's_formal_lemma}.
	
	We conclude that Stickelberger's theorem is an obstruction to \textit{strong approximation}
	on $BG$, but not to \textit{weak approximation}.
\end{remark}
\subsection{Grunwald--Wang as a Brauer--Manin obstruction} \label{sec:Grunwald}

\begin{theorem}\label{thm:Grunwald-Wang}
	There is no $\Z/8\Z$-extension $k$ of $\Q$ such that $k \otimes_\Q \Q_2$ is the unramified
	$\Z/8\Z$-extension of $\Q_2$. This is explained by a Brauer--Manin obstruction to 
	weak approximation on $B\Z/8\Z$.
\end{theorem}
\begin{proof}
	Let $\pi: \Spec \Q \to B \Z / 8 \Z$ be the defining $\Z/ 8 \Z$-torsor with class $[\pi] \in \H^1(B \Z / 8 \Z, \Z/ 8 \Z)$. Let $16 \in \Q^{\times}/ \Q^{\times 8} \cong \H^1(\Q, \mu_8)$ and consider the cup product $b := [\pi] \cup 16 \in \H^2(B \Z / 8 \Z, \Gm) = \Br B \Z / 8 \Z$ induced by the standard pairing $\Z / 8 \Z \times \mu_8 \to \Gm$.
	
	Let $(\chi_v)_{v} \in \prod_v B \Z / 8 \Z(\Q_v)$ with $\chi_2$ a character $\Gamma_{\Q_2} \to \Z/8 \Z$ which defines an unramified $\Z / 8\Z$-extension. Note that by construction $[\pi](\chi_2) = \chi_2$. 
	It is classical that $16 \in \Q_v^{\times 8}$ for all $v \neq 2$ thus $b_{\Q_v} =0 \in \Br (B \Z / 8 \Z)_{\Q_v}$ for $v \neq 2$. Hence
	\[\sum_{v} \inv_v(b(\chi_v)) = \inv_2(b(\chi_2)) = \inv_2(\chi_2 \cup 2^4) =  \inv_2(4 \cdot \chi_2 \cup 2).\]
The character $4 \cdot \chi_2$ corresponds to $\Q_2(\sqrt{5})/\Q_2$ so $\inv_2(4 \cdot \chi_2 \cup 2) = (5,2)_2 = \frac{1}{2} \neq 0$.
	
	Thus there is a Brauer--Manin obstruction to the existence of a character $\chi \in B\Z/ 8 \Z(\Q)$ such that $\chi_{\Q_2}$ defines an unramified $\Z/8 \Z$-extension.
\end{proof}

\begin{remark}
We emphasise that the Brauer group element $b$ used in the proof is unramified, contrary to Stickelberger's theorem. Indeed, in the above proof we saw that $b_{\Q_v} = 0$ for all places $v \neq 2$. Theorem \ref{thm:Harari's_formal_lemma} implies that $b$ is unramified. Therefore Grunwald--Wang is an obstruction to \emph{weak approximation} on $B \Z/8\Z$.
\end{remark}

\subsection{A gerbe which fails the Hasse principle}

The above examples all concern failures of weak approximation on stacks with rational points. We give an example of a Brauer--Manin obstruction to the Hasse principle on a gerbe. It arises through considering which fibres of the map $B \Z/16\Z(k) \to B \Z/2\Z(k)$ have a rational point; this is relevant to Malle's conjecture through consideration of the Iitaka fibration. In more classical language, this corresponds to trying to solve an embedding problem in the sense of Galois theory.

\subsubsection{Cohomological argument}
The existence of such a gerbe follows from the following. Poitou--Tate duality \cite[Thm.~8.6.8]{NSW08} yields a perfect pairing
$$\Sha^1(k,\mu_n) \times \Sha^2(k,\Z/n\Z) \to  \Q/\Z$$
for any number field $k$. It is well-known that the group $\Sha^1(k,\mu_8)$, can be non-trivial for suitable $k$ \cite[Thm.~9.1.3]{NSW08}; indeed, this is closely related to the Grunwald--Wang theorem already discussed. In such cases one deduces that $\Sha^2(k,\Z/8\Z)$ is non-trivial. However $\H^2(k,\Z/8\Z)$ classifies $\Z/8\Z$-gerbes, and a non-trivial element of $\Sha^2(k,\Z/8\Z)$ exactly corresponds to a non-neutral $\Z/8\Z$-gerbe which is everywhere locally neutral, i.e.~fails the Hasse principle. But then \cite[Thm.~1.1]{PS22} can be interpreted as the statement that the only obstruction to the Hasse principle for gerbes with constant abelian stabilisers is the Brauer--Manin obstruction, so this failure can be explained by the Brauer--Manin obstruction.

\subsubsection{An explicit example}
The above shows existence; more challenging is to actually construct an explicit such gerbe and to give an explicit Brauer--Manin obstruction. We do this now; the following is more-or-less rephrasing an example of Conrad \cite[Ex.~2.1]{Con11}.

Let $k = \Q(\sqrt{7})$,  $L = k(\sqrt{15 + 4\sqrt{7}})$ and $\chi_L \in B\Z/2 \Z(k)$ the character defining $L/k$. Let $f: \mathcal{X} \to B\Z/16 \Z$ be the fibre product of $\chi_L: \Spec k \to B \Z/2 \Z$ and the map $B \Z/16 \Z \to B \Z/2 \Z$. The stack $\mathcal{X}$ is a $\Z/8 \Z$-gerbe over $k$ which becomes neutral over $L$, as can be checked after base change to $L$.
                                                                            
For any field $K/k$ the image of the map $f: \mathcal{X}[K] \to B \Z/16 \Z[K]$ consist of morphisms $\varphi: \Gamma_K \to \Z/16 \Z$ such that the composition $\Gamma_K \to \Z/16 \Z \to \Z/2 \Z$ is $(\varphi_L)_K: \Gamma_K \to \Gamma_k \to \Z/2 \Z$, i.e.~solutions to the embedding problem posed by $\chi_L$ and $\Z/16\Z \to \Z/2\Z$. (The map $ \mathcal{X}[K] \to B \Z/16 \Z[K]$ is actually injective, but we will not need this.)

Using this description Conrad has shown \cite[Ex.~2.1]{Con11} (in our language) that $\mathcal{X}(k) = \emptyset$ but $\mathcal{X}(k_v) \neq \emptyset$ for all places $v$ of $k$,~i.e.~$\mathcal{X}$ is a gerbe which fails the Hasse principle. We will now show that this is due to a Brauer-Manin obstruction.

\begin{theorem}
	The gerbe $\mathcal{X}$ has a Brauer-Manin obstruction to the Hasse principle.
\end{theorem}
\begin{proof}
	We use the following facts from Conrad \cite[Ex.~2.1]{Con11}.
	Firstly he shows using local class field theory that
	$\chi_L$, viewed as a character with values in $\Q/\Z$,
	is everywhere locally an $8$th power, thus can be lifted everywhere locally to a 
	$\Z/16\Z$-character. This implies that $\mathcal{X}(k_v) \neq \emptyset$ for all $v$.
	Secondly, the character $\chi_L$ is only ramified at the primes of $k$ dividing $113$ and there is a unique place $w$ of $k$ which divides $2$ at which $k_w \cong \Q_2(\sqrt{-1})$ and $L \otimes_k k_w \cong \Q_2(\sqrt{-1}, \sqrt{5})$. 
	
	Let now $\beta \in \Br_{e} B\Z/16 \Z \cong \H^1(k, \mu_{16})$ be the element corresponding to $16 \in k^{\times}/k^{\times 16}$. We will show that $f^*\beta$ induces a Brauer-Manin obstruction to the Hasse principle on $\mathcal{X}$.                                                         
	Let $(x_v)_{v} \in \prod_{v \in \Omega_v} \mathcal{X}[k_v]$ and write $\varphi_v := f(x_v) \in B \Z/16 \Z[k_v]$. By Lemma~\ref{lem:cup_products} we have the equality 
	\[
	f^*(\beta)(x_v) = \beta(\varphi_v) = \varphi_v \cup 16
	\]                                               
	where the relevant cup product is $\H^1(k_v, \Z/16 \Z) \times \H^1(k_v, \mu_{16}) \to \Br k_v$.
	
	We can now compute the local invariants.
	\begin{enumerate}                
		\item If $v$ is an infinite place then $16 = 1 \in \H^1(k_v, \mu_{16})$ so $\text{inv}_v(f^*(\beta)(x_v)) = 0$.                                 
		\item For $v \mid 113$ we check that $4^{16} \equiv 16 \pmod{113}$. Hensel's lemma then implies that $16 \in \Q_{113}^{\times 16}$. Thus $\text{inv}_v(f^*(\beta)(x_v) = 0$.
		\item If $v \nmid 113$ then $16 = \alpha_v^8$ where $\alpha_v \in \{\sqrt{2}, \sqrt{-2}, 1 + \sqrt{-1}\}$. We have $\varphi_v = f(x_v)$ so $8 \varphi_v = \chi_L \in \H^1(k, \Z/2 \Z)$. We then have $\varphi_v \cup 16 = (8\varphi_v) \cup \alpha_v = \chi_L \cup \alpha_v$.
		The  character $\chi_L$ is unramified since $v \nmid 113$.
		\begin{itemize}
			\item If $v$ is coprime to $2$ then $\alpha_v \in \mathcal{O}_v^{\times}$ which implies that $\chi_L \cup \alpha_v = 0$ and thus that $\text{inv}_v(f^*(\beta)(x_v)) = 0$.
			\item If $v = w$ then $\alpha_w = 1 + \sqrt{-1}$. We then have $\text{inv}_w(f^*(\beta)(x_w)) = \text{inv}_w(\chi_L \cup \alpha_w) = (5, 1 + \sqrt{-1})_w = \frac{1}{2}$.
		\end{itemize}                                                                        

	\end{enumerate}
	To summarize $\sum_v \text{inv}_v(f^*(\beta)(x_v)) = \text{inv}_w(f^*(\beta)(x_w)) = \frac{1}{2} \neq 0$. This means that $f^*(\beta)$ induces a Brauer-Manin obstruction to the Hasse principle. 
\end{proof}


\subsection{Brauer--Manin obstruction over global function fields}
In the preceding sections we have seen numerous examples of Brauer--Manin obstructions over number fields. We show now that surprisingly, Brauer--Manin obstructions are a lot rarer over global function fields. 

\begin{theorem} \label{thm:Brauer_function_field}
	Let $K \subset L$ be a field extension in which $K$ is algebraically closed.
	Let $G$ be a finite \'etale tame group scheme over $K$ 
	and $\mathcal{C} \subset \mathcal{C}_G^*$ Galois invariant which generates $G$.
	Then pull-back via $BG_{L} \to BG$ induces an isomorphism 
	$$\Br_{\mathcal{C}} BG/\Br K \cong  \Br_{\mathcal{C}} BG_{L}/ \Br L.$$
\end{theorem}
\begin{proof}
	By the orbifold Kummer sequence (Corollary \ref{cor:orbifold_Kummer}), it suffices to show
	that the maps
	$$\PicOrb_{\mathcal{C}}(BG) \to \PicOrb_{\mathcal{C}}(BG_L),
	\quad \H^{2, \orb}_{\mathcal{C}}(BG, \mu_n) \to 
	\H^{2, \orb}_{\mathcal{C}}(BG_L, \mu_n)$$
	are isomorphisms for all tame $n$.
	The Galois action on $G_L(-1)$ arises exactly from the field extension of $K$
	which gives the Galois action on $G(-1)$.
	Thus the isomorphism of orbifold Picard groups follows immediately from Definition
	\ref{def:partial_orbifold_line_bundle}. For the orbifold cohomology,
	we use the definition in terms of marked central extensions 
	(Definition \ref{def:marked_central_extension}). Injectivity is clear.
	For surjectivity let $1 \to \mu_n \to E \to G_K \to 1$ be
	the corresponding central extension with marking $\mathcal{E}$.
	As $\mathcal{E} \to \mathcal{C}$ is a Galois equivariant
	bijection, it follows that the Galois action
	on $\mathcal{E}$ is uniquely determined by the Galois action of $\mathcal{C}$,
	thus factors uniquely through a finite extension of $K$.
	As $\mathcal{C}$ generates $G$ it follows that $\mathcal{E}$ and $\mu_n$ 
	together generate $E$.
	Thus the Galois action on $E$ factors through a finite 
	field extension of $K$, hence $E$ is actually defined over $K$, as required.
\end{proof}

The map $\Br BG /\Br K \to \Br BG_{L}/ \Br L$ need not be an isomorphism in general (even for $G = \Z/2\Z$). This gives further evidence for the naturality of the (partially) unramified Brauer group. Note that there is no reason to expect an analogue of Theorem \ref{thm:Brauer_function_field} for non-constant group schemes in general. Non-constant $G$ are relevant in Malle's conjecture when counting with unbalanced heights (see Conjecture \ref{conj:non_balanced}). From our results one obtains the following significant strengthening of Theorem \ref{thm:Harari_formal_partially_unramified} for constant $G$.

\begin{corollary} \label{cor:Brauer_function_fields}
	Let $k$ be a global function field and $G$ a finite group of order coprime to the characteristic
	of $k$. Let $\mathcal{C} \subset \mathcal{C}_G^*$ be Galois invariant
	and generate $G$.
	Then every element of $\Br_{\mathcal{C}} BG$ has constant evaluation on
	$BG(\O_v)_\mathcal{C}$ for all places $v$. In particular
	$\prod_vBG(k_v) = (\prod_{v}BG(k_v))^{\Brun BG}$, i.e.~there is no unramified
	Brauer--Manin obstruction.
\end{corollary}
\begin{proof}
	Let $\kappa$ be the field of constants of 
	$k$ and $b \in \Br_{\mathcal{C}} BG$.
	Translating by an element of $\Br k$, we may assume that 
	$b$ evaluates trivially at the identity cocycle.
	By Theorem~\ref{thm:Brauer_function_field} this element arises from
	base change from $\kappa$. Moreover by Lemma~\ref{lem:central_gerbe} 
	it  has everywhere good reduction. 
	Thus by Theorem~\ref{thm:Harari_formal_partially_unramified}
	the element $b$ evaluates trivially on $BG(\O_v)_{\mathcal{C}}$ for all $v$.
	The result follows.
\end{proof}

The reason why Corollary \ref{cor:Brauer_function_fields} holds over function fields, but not number fields, is that over function fields $BG$ has everywhere good tame reduction. We believe that Corollary \ref{cor:Brauer_function_fields} corroborates the observation that the Brauer--Manin obstruction does not appear in the literature on field counting problems over function fields, where often the leading constant is given by a single Euler product. It also highlights the limitations of using heuristics and results over global function fields to give predictions over number fields.

\section{Heights and Tamagawa measures} \label{sec:heights_Tamagawa}
\subsection{Heights}  \label{sec:heights}
We now prepare for considering Malle's conjecture. We first define heights on $BG$. Our definition is inspired by the papers \cite{DYTor,DYBM}, as well as the ``$f$-discriminant'' of Ellenberg and Venkatesh \cite[\S4.2]{EV05} and Wood's ``counting function'' \cite[\S2.1]{Woo10}. Let $G$ be a finite \'etale tame group scheme over a global field $k$.

Let $L=(\chi,w)$ be an orbifold line bundle (Definition \ref{def:orbifold_line_bundle_over_k}). A \textit{local height function} at a place $v$ is simply a map 
$$H_v: BG(k_v) \to \R_{> 0},$$
which is constant on isomorphism classes of objects. An \emph{adelic height} on $L$ is a collection $H=(H_v)_v$ of local heights for all places $v$ of $k$ such that for all but finitely many tame places $v$ we have
\begin{equation} \label{eqn:tame_height}
H_v(\varphi_v) = q_v^{w(\rho_{G,v}(\varphi_v))} \quad \text{ for all } \varphi_v \in BG(k_v),
\end{equation}
where $\rho_{G,v}$ denotes the ramification type at $v$ from \S \ref{sec:ramification_type} and $w$ the weight function of $L$. The height of $\varphi \in BG(k)$ is then defined to be $H(\varphi) := \prod_v H_v(\varphi)$, where we consider the image of $\varphi$ under the (possibly non-essentially injective) map $BG(k) \to \prod_v BG(k_v)$. For simplicity we often call an adelic height simply a height. 

\begin{definition} \label{def:balanced_height}
We call $H$ \emph{balanced} if $L$ is balanced in the sense of Definition \ref{def:rigid}, i.e.~if the minimal weight conjugacy classes $\mathcal{M}(H)$ generate $G$.
\end{definition}

\subsubsection{Examples}

\begin{example}[The discriminant] \label{ex:disc}
	Let $G \subseteq S_n$ be a transitive subgroup. We take the orbifold line bundle
	$\Delta = (1,\ind)$ where 
	$$\ind:  G \to \Z, \quad g \mapsto n - \#\{\text{orbits under the action of } g\}.$$
	The absolute value of the norm of the relative
	discriminant of the associated degree $n$ extension (see Lemma \ref{lem:S_n})
	is a height function corresponding to $\Delta$ 
	(see \cite[\S7]{Mal02}).
\end{example}

\begin{example}[The radical discriminant] \label{ex:rad_disc}
	We take the orbifold line bundle $-K_{BG}^{\text{orb}} := (1, \mathbf{1})$
	where $1$ denotes the trivial character and $\mathbf{1}$ denotes the constant function $\mathcal{C}^*_G \to \{e\}$ (see Definition \ref{def:canonical_class}).
	A choice of associated height is given by the local heights 
	$$H_v(\varphi_v) = \begin{cases}
	1,&  \text{if $\varphi_v$ is unramified}, \\
	q_v,& \text{if $\varphi_v$ is ramified}.
	\end{cases}$$ The corresponding orbifold anticanonical height is thus the 
	radical discriminant.
\end{example}

\begin{remark}
From Lemma \ref{lem:Picorb_torsion_free} the weight function $w$ determines the character $\chi$, hence it is not suprising that $\chi$ does not appear in the definition of the height. It implicitly plays a role as $w$ satisfies some compatibility with $\chi$ via the age pairing (Definition~\ref{def:orbifold_line_bundle}). In particular, for non-trivial $\chi$ the height function can take non-integer values. This happens for example when $G = S_n$ and $L=(\mathrm{sign}, \ind/2)$, with corresponding height function the square root of the absolute value of the norm of the discriminant. This is seen to be an orbifold line bundle using the relation $\mathrm{sign}(g) = (-1)^{\ind(g)}$ for all $g \in S_n$.

Line bundles play a larger role in the theory of heights in \cite{DYBM} for general stacks. It is important moreover to keep track of $\chi$ to make sure that one has the correct integral structure on the orbifold Picard group. This is relevant to the effective cone constant and  the algebraic Brauer group (see Theorem \ref{thm:Br_BG}).
\end{remark}

\subsubsection{$\hat{\Z}^\times$-invariance}
Assume that $L$ is $\hat{\Z}^\times$-invariant in the sense of Definition \ref{def:special}. Then by Remark \ref{rem:special} we may canonically view $w$ as a function on $G(k^{\sep})$ and not just on $G(-1)(k^{\sep})$. This is the case for example for the  discriminant and radical discriminant above.

Here the formula for the height is slightly more explicit, in terms of a generator of tame inertia. Choose a suitably large finite field extension $K/k$ such that $\varphi$ factors through $\Gal(K/k)$ and $K$ contains the $\mu_{|G|}$th roots of unity.

\begin{lemma} \label{lem:special_height}
	Assume that $L$ is $\hat{\Z}^\times$-invariant.
	Let $v$ be a tame place for $G$ and $w$ a place of $K$ above
	$v$. Let $\sigma_v$ be any choice of generator for the inertia subgroup of $w$ over $v$. Then
	$$w(\rho_{G,v}(\varphi_v)) = w(\varphi_v(\sigma_v)).$$
\end{lemma}
\begin{proof}
	Our assumptions imply that $w(\sigma_v)$ is independent of the corresponding
	choice of primitive root of unity
	in \eqref{eqn:tame_inertia_finite} and Lemma \ref{lem:Galois_action_on_G(-1)}.
\end{proof}

We give an example of a non-$\hat{\Z}^\times$-invariant height for completeness.

\begin{example}
	Let $k = \Q(\omega)$ where $\omega = e^{2\pi i /3}$ and $G = \Z/3\Z$. 
	We take the orbifold line bundle $(1,w)$ with weight function
	$$w: \Z/3\Z(-1) \to \Q, \quad \gamma \mapsto \gamma(\omega)$$
	where we view $\gamma(\omega)$ as an element of $\{0,1,2\}$.
	We let $H$ be an associated adelic height whose local heights
	take the constant value $1$ for $v \mid 3 \infty$.
	Let $\varphi: \Gamma_k \to \Z/3\Z$. Then the description in 
	\S\ref{sec:ramification_type} shows that
	$$H(\varphi) = \prod_{\substack{v \nmid 3\infty \\ \varphi_v \textrm{ ramified}}} q_v^{w(\sigma_{\omega,\varphi_v})}$$
	where $\sigma_{\omega,\varphi_v}$ denotes the element of 
	the inertia group of $\varphi_v$ which multiplies a uniformiser
	by $\omega$. 
	
	This height genuinely depends on the choice
	of generator for tame inertia and hence not only
	on the inertia group as a subgroup of the Galois group.
\end{example}

\subsubsection{Pull-back heights} \label{sec:pull_back_heights}
One can pull-back heights. Let $f:G_1 \to G_2$ be a morphism of finite \'etale tame group schemes over $k$ and $H$ an adelic height on $G_2$ with associated orbifold line bundle $L=(\chi,w)$. Then the map
$$f^*H_v: BG_1(k_v) \to \R_{>0}, \quad \varphi_v \mapsto H_v(f \circ \varphi_v),$$
is easily checked to determine an adelic height $f^*H$ on $BG_1$ with weight function
$$f^*w: \mathcal{C}_{G_1} \to \Q, \quad c \mapsto f(w(c)).$$
In particular, given a cocycle $\psi \in Z^1(k,G)$ and a normal subgroup scheme $N \subset G_\psi$, by Lemma \ref{lem:inner_twist_BG} one can pull-back an adelic height on $BG$ via $BN \to BG_\psi=BG$.

\subsection{Local Tamagawa measures}\label{sec:local_Tamagawa}
Let $H$ be a height associated to a big orbifold line bundle $L$. 
In \cite[\S 2.2.1]{Pey95}, Peyre defines a Tamagawa measure on $X(k_v)$ for a Fano variety $X$ by taking an atlas of $X(k_v)$ and gluing measures locally. He introduces a factor coming from a $v$-adic metric to ensure that these measures glue. Measures for other adjoint rigid line bundles are defined in \cite[\S 3.3]{BT98}.

In our case of $BG$, the natural stack-theoretic atlas is simply a point $\Spec k$. Obviously this one atlas does not cover all of $BG(k_v)$, in fact it only gives rise the identity cocycle. Therefore  we need to take multiple atlases to cover all $k_v$-points;  one atlas for each $k_v$-point. The conclusion is that our Tamagawa measure should simply be a sum over all elements of $BG[k_v]$ weighted by the metric corresponding to our height function. Moreover we should naturally use the groupoid cardinality which weights each element by the inverse of its automorphism group. These considerations lead us to the following definition. (This generalises Kedlaya's ``total mass'' from \cite[Def~2.2, (2.3.1)]{Ked07}; note also that the papers \cite{DYTor,DYBM} do not define Tamagawa measures.)

\begin{definition} \label{def:local_Tamagawa_measure}
	Let $v$ be a place of $k$ and $W_v \subseteq BG[k_v]$. We define the \textit{Tamagawa measure}
	$\tau_{H,v}$ associated to our choice of adelic height to be
	$$\tau_{H,v}(W_v) = \sum_{\varphi_v \in [W_v]} \frac{1}{|\Aut(\varphi_v)| H_v(\varphi_v)^{a(L)}}.$$
\end{definition}
The fact that $G$ is tame implies that this sum is finite. This determines a well-defined measure on the set $BG[k_v]$ of isomorphism classes of $k_v$-points of $BG$.

Lemma \ref{lem:groupoid_count} allows one to rewrite this in terms of a count over $1$-cocycles with a different weighting. Namely if $F: Z^1(k_v,G) \to BG(k_v)$ denotes the natural map then
\begin{equation} \label{eqn:tau_alternative}
	\tau_{H,v}(W_v) = \frac{1}{|G|}\sum_{\varphi_v \in F^{-1}(W_v)} \frac{1}{H_v(F(\varphi_v))^{a(L)}}.
\end{equation}

\subsection{Mass formula}

We now prepare for our mass formula (Theorem \ref{thm:mass_formula_intro} from the introduction). The following will be crucial.

\begin{lemma} \label{lem:gerbe_finite_field}
	Let $\mathcal{X}$ be a proper \'etale gerbe over $\F_q$. 
	Then $\#\mathcal{X}(\F_q) = 1$. In particular 
	$\mathcal{X}(\F_q) \neq \emptyset$.
\end{lemma}
\begin{proof}
	We apply the version of the Grothendieck--Lefschetz trace formula for stacks 
	from \cite[Thm.~2.5.2]{Beh93}. This says that
	$$\#\mathcal{X}(\F_q) = 
	\sum_{n = 1}^\infty (-1)^n\mathrm{Tr}(\Frob_q | \H^n( \mathcal{X}_{{\bar \F}_q}, \Q_\ell))$$
	where $\Frob_q$ denotes the arithmetic Frobenius and $\ell$ is a sufficiently large prime. However over $\bar{\F}_q$ the gerbe $\mathcal{X}$
	becomes neutral, thus there exists a finite group $G$ such that $\mathcal{X}_{\bar\F_q} \cong
	BG_{\bar \F_q}$. But then
	$$\H^n( \mathcal{X}_{{\bar \F}_q}, \Q_\ell) = \H^n( BG_{\bar \F_q}, \Q_\ell) = \H^n( G, \Q_\ell), \quad n \geq 0,$$
	as vector spaces,
	where the latter cohomology is group cohomology. For $n > 0$ the group cohomology
	$\H^n( G, \Q_\ell)$ is torsion of order dividing $|G|$, but $\Q_{\ell}$ is torsion-free so
	this cohomology group is trivial providing $\ell \nmid |G|$.
	Hence only $\H^0( \mathcal{X}_{{\bar \F}_q}, \Q_\ell)$ contributes,
	and here $\Frob_q$ acts trivially, thus the trace is $1$, as required.
\end{proof}

We say that $v$ is a \emph{good place} for $G$ and $H$ if $G$ has good reduction as a group scheme, if $v$ is tame for $G$ (i.e~$\gcd(q_v,|G|) = 1$), and $v$ is tame for $H_v$ in the sense that the formula \eqref{eqn:tame_height} holds.

In this case $G$ has a unique finite smooth group scheme model $\mathcal{G}_v$ over $\O_v$; by abuse of notation we denote by $BG(\O_v) := B\mathcal{G}_v(\O_v)$. 
If $G$ is constant, then $BG(\O_v)$ is equivalent to the groupoid of homomorphisms $\Gamma_{\F_v} \to G$. Warning: $BG(\O_v) \neq BG(k_v)$ in general despite $BG$ being smooth and proper at $v$. This is because the valuative criterion for properness for stacks \cite[Tag 0CLK]{stacks-project} only guarantees a lift after a possibly ramified base-change; our replacement for this is Theorem \ref{thm:Hensel}.

\begin{lemma} \label{lem:tau_BG(O_v)}
	Let $v$ be good place for $G$ and $H$. Then the natural map
	$BG(\O_v) \to BG(\F_v)$ is an equivalence of categories and 
	$\tau_{H,v}(BG(\O_v)) = \#BG(\F_v) = 1.$
\end{lemma}
\begin{proof}	
	The equivalence easily follows from the isomorphism $\pi_1(\O_v) \cong \Gal(\bar{\F}_v/\F_v)$
	The first equality follows from the fact that $H_v(\varphi_v) = 1$
	for all $\varphi_v \in BG(\O_v)$. The last equality is
	Lemma \ref{lem:gerbe_finite_field}.	
\end{proof}

We next obtain a generalisation of Bhargava's \cite[Thm.~1.1]{Bha07} and Kedlaya's \cite[Prop.~5.3]{Ked07} mass formulae. Our proof is completely different and uses our stacky Hensel's Lemma (Theorem \ref{thm:Hensel}). We consider a certain zeta function in a complex variable $s$, which can be viewed as an analogue in our setting of an Igusa integral in the sense of \cite[\S4.1]{CT10}. Our formula may be interpreted as an analogue of Denef's formula \cite[Prop.~4.1.6]{CT10} for $BG$, using the philosophy that the boundary divisors of $BG$ should correspond to the non-identity conjugacy classes of $G$, and moreover that they should be disjoint (see Remark \ref{rem:intuition}). There are differences however in our setting. For example Denef's formula for varieties contains the term $(q_v-1)/(q_v^{s+1}-1)$, which arises from a geometric series keeping track of the possible valuations at which a $v$-adic point meets the boundary divisors; in our case only the valuation $1$ occurs. In the statement $\rho_{G,v}$ denotes the ramification type at $v$ from \S\ref{sec:ramification_type}.

\begin{theorem}[Igusa integral formula] \label{thm:Igusa}
	Let $v$ be a good place of $G$ and for a height $H$ with weight function 
	$w: \mathcal{C}_G \to \R$. Let $f: \mathcal{C}_G \to \C$ be any Galois equivariant function.
	Then for any $s \in \C$ we have
	$$
	\sum_{\varphi_v \in BG[k_{v}]} \frac{f(\rho_{G,v}(\varphi_v))}{|\Aut(\varphi_v)| H_v(\varphi_v)^{s}}
	= \sum_{c \in \mathcal{C}_G^{\Gamma_{k_v}}} \frac{f(c)}{q_v^{w(c)s}}.$$
\end{theorem}
\begin{proof}
	By Proposition \ref{prop:cyclotomic_inertia_stack_BG}
	and Theorem~\ref{thm:Hensel} the reduction modulo $\pi_v$ induces an 
	equivalence of groupoids $BG(k_v) \to [G(-1)/G](\F_v)$.
	We can therefore perform our groupoid count on $[G(-1)/G](\F_v)$,
	which gives 
	$$\sum_{\varphi_v \in BG[k_{v}]} \frac{f(\rho_{G,v}(\varphi_v))}{|\Aut(\varphi_v)| H_v(\varphi_v)^{s}}
	= \sum_{c \in \mathcal{C}_G^{\Gamma_{k_v}}} \frac{f(c)\#\mathcal{S}_c(\F_v)}{q_v^{w(c)s}},$$
	where $\mathcal{S}_c$ is the sector corresponding to $c$. 
	For this equality we use that the height function 
	takes constant value $q_v^{w(c)}$ on $\mathcal{S}_c(\F_v)$,
	and that if a sector has an $\F_v$-point then it must be
	Galois invariant. As $\mathcal{S}_c$ is a proper \'etale
	gerbe over $\F_v$, the result
	then follows from Lemma \ref{lem:gerbe_finite_field}.
\end{proof}

We use this to calculate the Tamagawa measure of $BG(k_v)$ and the partial adelic space from Definition~\ref{def:adelic_space}.

\begin{corollary}[Mass formula] \label{cor:mass_formula}
	For good places $v$ we have
 	\begin{align*}
		\tau_{H, v}(BG(k_v))= \sum_{c \in \mathcal{C}_G^{\Gamma_{k_v}}} q_v^{-w(c)a(L)}, \quad
		\tau_{H, v}(BG(\O_{v})_{\mathcal{M}(L)}) = 
		1+ \frac{\#\mathcal{M}(L)^{\Gamma_{k_v}}}{q_v}.		
	\end{align*}
\end{corollary}
\begin{proof}
	Follows from Theorem \ref{thm:Igusa} by taking $s = a(L)$ and $f$ to be the identity
	and indicator function for $\mathcal{M}(L)$, respectively.
\end{proof}

Following Kedlaya \cite[Def.~5.1]{Ked07}, given a set of places $T$ of $k$, we say that there is a \emph{uniform mass formula on $T$} if there is an integer polynomial $P$ such that $\tau_{H, v}(BG(k_v)) = P(q_v^{-1})$ for all $v \in T$. Kedlaya \cite[Cor.~5.5]{Ked07} has shown that for constant $G$, there is a uniform mass formula away from some finite set of places if and only if the character table of $G$ is rational. The following generalises this to non-constant $G$ (new cases include for example $G = \mu_n$).

\begin{corollary}
	$G$ has a uniform mass formula with respect to $H$ away from some finite set of places 
	if and only if the Galois action on $\mathcal{C}_G$ is trivial.
\end{corollary}
\begin{proof}
	Immediate from Corollary \ref{cor:mass_formula} and the Chebotarev density theorem.
\end{proof}

\begin{remark}
More generally, Corollary \ref{cor:mass_formula} shows that there is a partition of the set of places of $k$ into finitely many frobenian sets such that there is a uniform mass formula on each set in the partition. Moreover if $G$ is constant, then the frobenian sets are determined by finitely many congruence conditions (since the Galois action factors through a cyclotomic extension).
\end{remark}

\begin{remark}
	Corollary \ref{cor:mass_formula} recovers Bhargava's mass formula \cite[Thm.~1.1]{Bha07}
	at the tame places. To see this one notes that the character table of $S_n$ is rational
	and that $q(k,n-k)$ equals the number of conjugacy classes of $S_n$ of index
	$k$. It would be interesting to obtain a stack-theoretic explanation at the wild places.
\end{remark}


\begin{lemma} \label{lem:archimedean_mass_formula}
	Let $v$ be archimedean. Then
	$$\tau_{H, v}(BG(k_v)) =
	\begin{cases}
	 \frac{1}{|G|}\sum_{\varphi_v \in Z^1(k_v,G)} \frac{1}{H_v(\varphi_v)^{a(L)}}, & \text{if $v$ is real}, \\
	 \frac{1}{|G|}\cdot \frac{1}{H_v(e)^{a(L)}}, & \text{if $v$ is complex}.
	 \end{cases}$$
\end{lemma}
\begin{proof}
	Immediate from the definition and Lemma \ref{lem:groupoid_count}.
\end{proof}

In particular if $G$ is constant, $v$ is real, and $H_v$ takes constant value $1$, then
\begin{equation} \label{eqn:real_mass_formula}
	\tau_{H, v}(BG(k_v)) = \frac{\#G[2]}{|G|}
\end{equation}
where $G[2]$ denotes the set of $2$-torsion elements of $G$.

\begin{remark}
Peyre shows in \cite[Lem.~2.2.1]{Pey95} that for a Fano variety $X$, we have $\tau_v(X(k_v)) = \#X(\F_v)/q_v^{\dim X}$ for all but finitely many places $v$. This is \textit{not} true in our case; indeed we showed in Lemma \ref{lem:tau_BG(O_v)} that $\# BG(\F_v) = 1$. This difference all
comes from the modified valuative criterion for properness for stacks.
\end{remark}

\subsection{Global Tamagawa measure} \label{sec:global_Tamagawa}
We now wish to take the product of the local Tamagawa measures. To do so we need to introduce convergence factors. In exact analogy with the convergence factors of Peyre, these should come from the Artin $L$-function  of $\PicOrb_{\mathcal{M}(L)} BG_{k^{\sep}}$. However from Lemma \ref{lem:Picorb(BG,C)} we can take the following equivalent $L$-function. The collection $\mathcal{M}(L)$ of minimal weight conjugacy classes is a finite $\Gamma_k$-set. We denote by $\L(\mathcal{M}(L),s) := \L(\C[\mathcal{M}(L)],s)$ the Artin $L$-function of the corresponding permutation representation, with corresponding local Euler factors $\L_v(\mathcal{M}(L),s)$. Then for each place $v$ of $k$ we define
$$ \lambda_v = \left \{
	\begin{array}{ll}
		\L_v(\mathcal{M}(L),1),& \quad v \text{ non-archimedean}, \\
		1,& \quad v \text{ archimedean}. \\
	\end{array}\right.$$	
We emphasise that our convergence factors depend on the orbifold line bundle $L$.
Our global Tamagawa measure is now defined to be
\begin{equation} \label{def:Tamagawa}
	\tau_H = \L^*(\mathcal{M}(L),1)\prod_v \lambda_v^{-1} \tau_{H,v}.
\end{equation}
We let $\L^*(\mathcal{M}(L),1) := \lim_{s \to 1}(s-1)^{b(k,L)} \L(\mathcal{M}(L),s)$, which is non-zero as $b(k,L)$ is the order of pole at $s= 1$. We now show that these $\lambda_v$ are indeed a family of convergence factors on the partial adelic space (see Definition \ref{def:adelic_space}). The following is a slightly more general version of Theorem \ref{thm:Tamagawa_products_intro}

\begin{theorem} \label{thm:Tamagawa_products}
	The infinite product measure $\prod_v \lambda_v^{-1} \tau_{H,v}$ 
	converges absolutely on $BG(\Adele_k)_{\mathcal{M}(L)}$ and $\prod_v BG(k_v)$.
\end{theorem}
\begin{proof}
	By Corollary \ref{cor:mass_formula} there exists an $\varepsilon > 0$ such that
	for all sufficiently large $v$
	$$\tau_{H, v}(BG(k_v)) = 1 +  \sum_{c \in \mathcal{M}(L)^{\Gamma_v}} q_v^{-1} + O(q_v^{-1 - \varepsilon}),$$
	and similarly for $\tau_{H, v}(BG(\O_{v})_{\mathcal{M}(L)})$.
	However we also have
	$$\lambda_v^{-1} = 1-  \sum_{c \in \mathcal{M}(L)^{\Gamma_v}} q_v^{-1} + O(q_v^{-2}).$$
	Thus the resulting product converges absolutely, as required.
\end{proof}

It is important for topological reasons that one considers the measure on the partial adelic space, particularly when considering the Brauer--Manin obstruction (cf.~Lemma \ref{lem:modified_adeles_Brauer-Manin}). However for the total measure one has a simple Euler product.

\begin{lemma} \label{lem:total_adelic_measure}
	$$\tau_H(BG(\Adele_k)_{\mathcal{M}(L)}) = \L^*(\mathcal{M}(L),1)\prod_v \lambda_v^{-1} \tau_{H,v}(BG(k_v)).$$
\end{lemma}
\begin{proof}
	Follows immediately from Definition \ref{def:adelic_space} and absolute convergence.
\end{proof}


Our convergence factors have the nice property that $\prod_v \lambda_v^{-1} \tau_{H,v}$ is absolutely convergent. There are alternative convergence factors which appear in practice, but are only conditionally convergent in general.
\begin{lemma} \label{lem:convergence_factors_alternative}
	We have
	$$\tau_H = (\mathrm{Res}_{s =1} \zeta_k(s))^{b(k,L)} 
	\prod_{v \mid \infty}  \tau_{H,v} 
	\prod_{v \nmid \infty} (1 - 1/q_v)^{b(k,L)}\tau_{H,v}$$
	where the product is convergent.
\end{lemma}
\begin{proof}
	Write $\C[\mathcal{M}(L)]	= \C^{b(k,L)} \oplus V$ where $V$ is an
	Artin representation containing no non-zero trivial subrepresentations. 
	Then $\L(\mathcal{M}(L),s) = \zeta_k(s)^{b(k,L)} \L(V,s)$. 
	Moreover $\L(V,1) = \prod_v \L_v(V,1)$ and the product is conditionally convergent 
	(this follows from analyticity and non-vanishing of $\L(V,s)$ along $\re s = 1$,
	together with Newman's Tauberian theorem \cite{New80} applied to both $\L(V,s)$ and
	$\log \L(V,s)$.)
	The result now follows from \eqref{def:Tamagawa} and Theorem \ref{thm:Tamagawa_products}.
\end{proof}


\subsection{Measure of the Brauer--Manin set}
We now consider the Tamagawa measure of  
$BG(\Adele_k)_{\mathcal{M}(L)}^{\Br}:= BG(\Adele_k)_{\mathcal{M}(L)}^{\Br_{\mathcal{M}(L)} BG}$.
To study this we assume that $H$ is balanced, since this implies that  $\Br_{\mathcal{M}(L)} BG/ \Br k$ is finite (Corollary \ref{cor:Br_finite}).

There is an analogue of Lemma \ref{lem:total_adelic_measure} for the orthogonal complement to the Brauer group, though it is somewhat subtle. For any finite set of places $S$, there is a well-defined Brauer--Manin pairing
$$ \Br BG \times \prod_{v \in S} BG(k_v) \to \Q/\Z$$
given by the sum of local invariants. For a subset $\br \subset \Br BG$, we let $\prod_{v \in S} BG(k_v)^{\br}$ denote the orthogonal complement to $\br$.
\begin{lemma} \label{lem:total_Brauer_measure}
	$$\tau_H( BG(\Adele_k)_{\mathcal{M}(L)}^{\Br}) = \lim_{S} 
	\prod_{v \in S}\lambda_v^{-1} \tau_{H,v}\left(
	\prod_{v \in S} BG(k_v)^{\Br_{\mathcal{M}(L)}BG}\right)$$
	where the limit is over all finite sets of places $S$. Moreover 
	$\tau_H( BG(\Adele_k)_{\mathcal{M}(L)}^{\Br}) \neq 0$.
\end{lemma}
\begin{proof}
	By Corollary \ref{cor:Br_finite} the group $\Br_{\mathcal{M}(L)} BG/ \Br k$ is finite.
	Let $\mathscr{B}$ be a finite group of representatives.
	By Theorem \ref{thm:Harari_formal_partially_unramified} we have that there is a finite set
	of places $S_0$ such that for all finite sets of places $S_0 \subset S$ we have
	$$BG[\Adele_k]_{\mathcal{M}(L)}^{\mathscr{B}} \bigcap\left( \prod_{v \in S} BG[k_v] 
	\prod_{v \notin S} BG[\O_v]_{\mathcal{M}(L)}  \right)
	= \left(\prod_{v \in S} BG[k_v]\right)^{\mathscr{B}} 
	\prod_{v \notin S} BG[\O_v]_{\mathcal{M}(L)}.$$
	The limit now follows from Definition \ref{def:adelic_space} and absolute convergence.
	The non-vanishing is because 
	$\left(\prod_{v \in S} BG[k_v]\right)^{\mathscr{B}} 
	\prod_{v \notin S} BG[\O_v]_{\mathcal{M}(L)}$ has positive measure by Corollary \ref{cor:mass_formula}.
\end{proof}

If the Brauer--Manin pairing restricted to $BG(\Adele_k)_{\mathcal{M}(L)}$ is trivial at all but finitely many places, then one can take the limit inside the product in Lemma \ref{lem:total_Brauer_measure}. However by Theorem \ref{thm:Harari's_formal_lemma}, this only happens if the Brauer group elements are unramified.

In many examples in the literature, the leading constant in Malle's conjecture is not given by a 	measure, but rather a sum of Euler products. This is compatible with our perspective (in practice the following actually seems a more useful formula than Lemma \ref{lem:total_Brauer_measure}). We use the map $\Q/\Z \to S^1: x \to e^{2 \pi i x}$.
	
\begin{lemma} \label{lem:sum_Euler_products}
	For each $b \in \Br_{\mathcal{M}(L)} BG$ consider the Euler product
	\begin{align*}
	\hat{\tau}_H(b) &:= \int_{BG(\Adele_k)_{\mathcal{M}(L)}} e^{2 \pi i \langle b, \varphi \rangle_{\text{BM}}} d\tau_H(\varphi)  \\
	&= \L^*(\mathcal{M}(L),1)
	\prod_v \lambda_v^{-1}\int_{BG(k_v)} e^{2 \pi i \inv_v b(\varphi_v) } d\tau_{H, v}(\varphi_v).
	\end{align*}
	Then 
	$$
		|\Br_{\mathcal{M}(L)} BG/\Br k| \cdot \tau_H( BG(\Adele_k)_{\mathcal{M}(L)}^{\Br}) = \sum_{b \in \Br_{\mathcal{M}(L)} BG/\Br k} \hat{\tau}_H(b)
	$$
	is a finite sum of Euler products.
\end{lemma}
\begin{proof}
	Immediate from character orthogonality, which implies that
	\[\sum_{b \in \Br_{\mathcal{M}(L)} BG/\Br k} e^{2 \pi i \langle b, x \rangle_{\text{BM}}}=
	\begin{cases}
		|\Br_{\mathcal{M}(L)} BG/\Br k|, & \text{ if } 
		\varphi \in BG(\Adele_k)_{\mathcal{M}(L)}^{\Br}, \\
		0, & \text{ otherwise}.
	\end{cases}   \]
\end{proof}

The method of proof for the Igusa integral formula (Theorem \ref{thm:Igusa}) is powerful enough to allow one to calculate the integrals which appear in Lemma \ref{lem:sum_Euler_products} at the tame places. We start with a lemma on exponential sums over gerbes.
\begin{lemma}\label{lem:exponential_sum_gerbe}
	Let $\mathcal{X}$ be a proper \'etale gerbe over $\F_q$ and let $\chi \in \H^1(\mathcal{X}, \Q/\Z)$ be such that $\chi_{{\bar \F}_q} \neq 0$. Then
	\[
	\sum_{x \in \mathcal{X}(\F_q)} \frac{e^{2 \pi i \chi(x)(\Frob_q)} }{|\Aut(x)|}= 0.
	\]
\end{lemma}
\begin{proof}
	Let $n$ be such that $\chi \in \H^1(\mathcal{X}, \Z/n\Z)$. Note that evaluation at $\Frob_q$ defines an isomorphism $\H^1(\F, \Z/n \Z) \cong \Z/n \Z$. For each element $a \in \H^1(\F, \Z/n \Z)$ let $\psi_a: \mathcal{Y}_a \to \mathcal{X}$ be the $\Z/n \Z$-torsor corresponding to $\chi - a \in \H^1(\mathcal{X}, \Z/n \Z)$. We will frequently abuse notation by identifying $a$ with $a(\Frob_q)$
	
	 Note that if $x \in \mathcal{X}(\F_q)$ then $x$ lies in the image of $\psi_a$ if and only if $(\chi - a)(x)$ is trivial. This later statement is equivalent to $\chi(x)(\Frob_q) = a \in \Z/n\Z$. In this case the fibre product of $\Spec \F_q \xrightarrow{x} \mathcal{X} \xleftarrow{\psi_a} \mathcal{Y}_a$ is the trivial $\Z/n \Z$-torsor so in terms of groupoid cardinality the map $\mathcal{Y}_a(\F_q) \to \mathcal{X}(\F_q)$ is $n$ to $1$. It follows that
	 \[
	 	\sum_{x \in \mathcal{X}(\F_q)} e^{2 \pi i \chi(x)(\Frob_q)} = \frac{1}{n}\sum_{a \in \Z/n \Z} e^{2 \pi i a} \#\mathcal{Y}_a(\F_q).
	 \]
	 If a connected component of $\mathcal{Y}_a$ is geometrically connected then it is a gerbe so the $\F_q$-points have cardinality $1$ by Lemma \ref{lem:groupoid_count}. If it is not geometrically connected then it contains no $\F_q$-points. We thus have $\#\mathcal{Y}_a(\F_q) = \pi_0(\mathcal{Y}_{a, {\bar \F}_q})^{\Gamma_{\F_q}}$.
	 
	 Let $\chi: \pi_{1}(\mathcal{X}) \to \Z/n \Z$ be the character corresponding to $\chi$ and let $m\Z/n \Z \subset \Z/n\Z$ be $\chi(\pi_{1}(\mathcal{X}_{\overline{\F_q}}))$. The fact that $\mathcal{X}_{\overline{\F}_q}$ is geometrically connected implies that $\pi_0(\mathcal{Y}_{a, {\bar \F}_q})$ naturally has the structure of a $(\Z/ n \Z)/\chi(\pi_{1}(\mathcal{X}_{\overline{\F_q}})) \cong \Z/m \Z$-torsor over $\Gamma_{\F}$. The assumption that $\chi \neq 0$ implies that $m \neq n$.

	 Let $\pi_0(\chi)$ be the quotient map $\Gamma_F \cong \pi_{1}(\mathcal{X})/\pi_{1}(\mathcal{X}_{\F_q}) \to \Z/m\Z$ induced by $\chi$. The $\Z/m\Z$-torsor $\pi_0(\mathcal{Y}_{a, {\bar \F}_q})$ corresponds to the map $\chi - a: \Gamma_{\F} \to \Z/m \Z$.
	 
	 Since $\Frob_q$ generates $\Gamma_{\F_q}$ it follows that if $\pi_{0}(\chi)(\Frob_q) = a(\Frob_q) \pmod{m}$ then $\# \pi_0(\mathcal{Y}_{a, {\bar \F}_q})^{\Gamma_{\F_q}} = m$ and that $\# \pi_0(\mathcal{Y}_{a, {\bar \F}_q})^{\Gamma_{\F_q}}  = 0$ otherwise. Let $b \in \Z/n \Z$ be a lift of $\pi_{0}(\chi)(\Frob_q) \in \Z/m\Z$. Letting $c = a - b$ we find that
	 \[
	 \sum_{x \in \mathcal{X}(\F_q)} e^{2 \pi i \chi(x)(\Frob_q)}  = \frac{m}{n} \sum_{\substack{a \in \Z/n\Z \\ a \equiv b \pmod{m}}} e^{2 \pi i a} = \frac{m e^{2 \pi i b}}{n} \sum_{c \in m\Z/n \Z} e^{2 \pi i c} = 0. \qedhere
	 \]
\end{proof}

\begin{theorem}\label{thm:local_invariant_integral}
	Let $v$ be a good place for $G$ and for a height $H$ with weight function 
	$w: \mathcal{C}_G \to \Q$. Let $\mathcal{G}$ be the $\mathcal{O}_v$-model of $G$ and
	$b \in \Br B \mathcal{G}$. We define a function $\chi_{v}: \mathcal{C}_G^{\Gamma_{k_v}} \to \C$ as
	follows. For each $c \in \mathcal{C}_G^{\Gamma_{k_v}}$
	let $\mathcal{S}_{c} \subset I_{\mu} B\mathcal{G}$ be the corresponding sector.
	
	\noindent \textbf{Transcendental residue:}  If $\partial_{c}(b)_{\bar k_v} \in \H^1((\mathcal{S}_{c})_{\bar k_v}, \Q/\Z)$ is non-zero then $\chi_{v}(c) = 0$.
	
	\noindent \textbf{Algebraic residue: } 
	Otherwise $\partial_{c}(b) \in \ker(\H^1(\mathcal{S}_{c}, \Q/\Z) \to
	\H^1((\mathcal{S}_{c})_{\bar k_v}, \Q/\Z))$, thus by the Hochshild--Serre spectral 
	sequence we may view $\partial_{c}(b) \in \H^1(k_v,\Q/\Z) = \Hom(\Gamma_{k_v},\Q/\Z)$.
	Moreover as $b \in \Br B \mathcal{G}$ this is unramified. So we 
	define $\chi_{v}(c) = e^{2 \pi i \partial_{c}(b)(\Frob_v)}$ where $\Frob_v$ 
	denotes the Frobenius element in $\Gamma_{k_v}$.
		
	Then for all $s \in \C$ and all Galois equivariant functions $f: \mathcal{C}_G \to \C$ 
	we have
	\[
	\sum_{\varphi_v \in BG[k_v]} \frac{f(\rho_{G,v}(\varphi_v))e^{2 \pi i \cdot \emph{inv}_v(b(\varphi_v))}}{|\Aut{\varphi_v}| H_v(\varphi_v)^s} = \sum_{c \in \mathcal{C}_{G}^{\Gamma_v}} \frac{f(c)\chi_{v}(c)}{q_v^{w(c)s}}.
	\]
\end{theorem}
\begin{proof}
	By Proposition \ref{prop:cyclotomic_inertia_stack_BG}
	and Theorem \ref{thm:Hensel} the modulo $\pi_v$ map defines an equivalence $BG(k_v) \to [G(-1)/G](\F_v)$ of groupoids. On the other hand by Lemma~\ref{lem:local_evaluation_residue_map} and the definition of $\text{inv}_v$ for all $\varphi_v \in BG[k_v]$ we have 
	\[
		\text{inv}_v(b(\varphi_v)) = \left(\res_{\rho_{G,v}(\varphi_v)}(b)\left(\varphi_v \bmod \pi_v\right)\right)(\Frob_v).
	\]
	Combining these two facts, and noting that a sector $\mathcal{S}_{c}$ only has an $\F_v$-point if it corresponds to an element $c \in \mathcal{C}_G^{\Gamma_v}$, we find that 
	\[
		\sum_{\varphi_v \in BG[k_v]} \frac{f(\rho_{G,v}(\varphi_v))e^{2 \pi i \cdot \text{inv}_v(b(\varphi_v))}}{|\Aut{\varphi}| H_v(\varphi)^s} 
		= \sum_{c \in \mathcal{C}_G^{\Gamma_v}} \frac{f(c)}{q_v^{w(c) s}}
		\sum_{x \in \mathcal{S}_c[\F_v]}
		\frac{e^{2 \pi i \left(\res_{c}(b)(x)\right)(\Frob_v)}}{|\Aut(x)|}.
	\]
	If $\chi_{v}(c)= 0$ then Lemma \ref{lem:exponential_sum_gerbe} applied to $\res_{\mathcal{S}_c}(b) \in \H^1(\mathcal{S}_c, \Q/\Z)$ implies that the sum over $\mathcal{S}_c[\F_v]$ is $0$.
	
	If $\chi_{v}(c)\neq 0$ then $e^{2 \pi i \left(\res_{\mathcal{S}_c}(b)(x)\right)(\Frob_v)}$ takes the constant value $\chi_v(c)$ since the residue is constant. In this case the sum over $\mathcal{S}_c[\F_v]$ is $\chi_c(v)$ by Lemma~\ref{lem:gerbe_finite_field}. 
\end{proof}

This theorem allows us for any $b \in \Br BG$ to compute $\hat{\tau}_{H, v}(b)$ for all but finitely many places. For algebraic residues, the formula is in terms of a Galois character, which can be viewed as a Dirichlet character by class field theory. In particular the local integrals can be viewed as character sums. One can then compute $\hat{\tau}_{H, v}(b)$ at the remaining places by hand (or computer). See Lemma~\ref{lem:A4_local_densities_conductor} for a worked example. 

\begin{remark}
 It is presumably possible to give a proof of Lemma \ref{lem:exponential_sum_gerbe} in the style
 of the proof of Lemma \ref{lem:gerbe_finite_field}, by using a suitable version of the Lefschetz 
 trace formula for lisse $\ell$-adic sheaves on stacks.
\end{remark}

\begin{remark}[Hierarchy for the Tamagawa measure] \label{rem:hierarchy}
	Lemma \ref{lem:sum_Euler_products} leads to a natural hierarchy for 
	$\tau_H( BG(\Adele_k)_{\mathcal{M}(L)}^{\Br})$,
	with three cases, providing that $H$ is balanced.
	\begin{enumerate}
		\item If $\Br_{\mathcal{M}(L)} BG \neq \Br_{\mathrm{un}} BG$, then 
		$\tau_H( BG(\Adele_k)_{\mathcal{M}(L)}^{\Br})$ is a sum of finitely many Euler
		products with different Euler factors at infinitely many places. This follows
		from Theorem \ref{thm:Harari's_formal_lemma}, 
		and is the situation in Conjecture \ref{conj:A_4_conductor_intro}.
		\item If $\Br_{\mathcal{M}(L)} BG = \Br_{\mathrm{un}} BG$, then 
		$\tau_H( BG(\Adele_k)_{\mathcal{M}(L)}^{\Br})$ is a sum of finitely many Euler
		products, where the Euler factors may differ at only finitely many places. This
		again follows from Theorem \ref{thm:Harari's_formal_lemma}.
		This is the situation which has occurred so far in the literatue, for example
		in Wood's paper \cite{Woo10} (see \S \ref{sec:balanced_wood} for details).
		\item If $\Br_{\mathrm{un}} BG = \Br k$, then 
		$\tau_H( BG(\Adele_k)_{\mathcal{M}(L)}^{\Br})$ is given by a single Euler product.
		This is the case in Bhargava's paper \cite{Bha07}, for example, and is typical
		in many examples in the literature.
	\end{enumerate}

	Moreover in cases (1) and (2), Theorem \ref{thm:local_invariant_integral} implies
	that the local Euler factors are determined by frobenian conditions.
\end{remark}

\section{Conjectures} \label{sec:conjectures}

We now state our conjectures on the leading constant for Malle's conjecture, inspired by Peyre's conjecture \cite{Pey95} concerning the leading constant in Manin's conjecture, as well as the generalisations of Batyrev and Tschinkel \cite{BT98}. Recall that for a Fano variety $X$ equipped with the anticanonical height function, the asymptotic formula in Manin's conjecture is 
\begin{equation} \label{eqn:Peyre}
	\frac{\alpha^*(X) \beta(X) \tau( X(\Adele_k)^{\Br})}{(\rho(X) - 1)!} B (\log B)^{\rho(X) - 1}
\end{equation}
where $\rho(X) = \rank \Pic X$, Peyre's (renormalised) effective cone constant is denoted by $\alpha^*(X)$, and $\beta(X) = |\H^1(k, \Pic \bar{X})| = |\Br_1 X/ \Br k|$. (The constant $\beta(X)$ first appeared in \cite{BT95, BT98, Sal98}.) Here $\tau$ is Peyre's Tamagawa measure and $X(\Adele_k)^{\Br}$ is the collection of adelic points orthogonal to $\Br X$.
Peyre considered the anticanonical height only; for general line bundles we follow the dichotomy presented in \cite[\S 3.4]{BT98}, as well as the predictions from \cite[\S 3.3]{BT98} (note that what Batyrev-Tschinkel  call ``$\mathcal{L}$-primitive'' is now called ``adjoint rigid'' in the literature; in our paper we use the softer terminology ``balanced'').

Throughout this section, we let  $H$ be a height function associated to a big orbifold line bundle $L = (\chi,w)$ on $BG$, for some non-trivial finite \'etale group scheme $G$ over a number field $k$ (see \S \ref{sec:heights}). We expect analogous conjectures to hold over global function fields  providing $G$ is tame, but one needs to be slightly more careful with the statements due to the nature of asymptotic formulae over function fields, one often needs to take a weighted average over the count for different fixed height values (see e.g. \cite[Thm.~I.3]{Wri89} for example of this).

\subsection{Balanced heights} \label{sec:adj_rigid_conj}
We say that $H$ (and $L$) is \textit{balanced} if the line bundle $L$ is adjoint rigid; equivalently if the collection $\mathcal{M}(L)$ of minimal weight elements generates $G$ (Definition \ref{def:balanced_height}). Recall that we denote by $BG[k]$ the set of isomorphism classes of elements of the groupoid $BG(k)$ (see \S \ref{sec:groupoid} for basic properties of this groupoid).

\begin{conjecture} \label{conj:balanced}
	Assume that $H$ is balanced. Then there exists a thin subset $\Omega \subset BG[k]$ such that
	$$\frac{1}{|Z(G)(k)|}\#\{ \varphi \in BG[k] \setminus \Omega : H(\varphi) \leq B\} \sim c(k,G,H) B^{a(L)} (\log B)^{b(k,L)-1}$$
	where 
	\begin{align*}
	a(L) & = (\min_{c \in \mathcal{C}_G^*}w(c))^{-1},   \quad
	\mathcal{M}(L) = \{ c \in \mathcal{C}_G^* : w(c) = a(L)^{-1}\}, \\
	 b(k,L) & = \#\mathcal{M}(L)/\Gamma_k,
	\end{align*}
	and 
	$$c(k,G,H) = \frac{a(L)^{b(k,L)-1}\cdot |\Br_{\mathcal{M}(L)} BG / \Br k| \cdot
	\tau_H( BG(\Adele_k)_{\mathcal{M}(L)}^{\Br})}{\#\dual{G}(k) (b(k,L) - 1)!}.$$
\end{conjecture}

An equivalent conjecture can be obtained from Lemma \ref{lem:groupoid_count}, which allows one to replace the sum over $\varphi \in BG[k]$ by a sum over $1$-cocycles $\Gamma_k \to G$ and replace the weighting by $1/|G|$ (if $G$ is constant then a $1$-cocycle is simply a homorphism, cf.~Lemma \ref{lem:BG(k)}). We denote by $BG(\Adele_k)_{\mathcal{M}(L)}^{\Br}:= BG(\Adele_k)_{\mathcal{M}(L)}^{\Br_{\mathcal{M}(L)} BG}$.

Let us motivate this expression using \eqref{eqn:Peyre} as a guide. Firstly, since we are performing a groupoid count, we weight each element by its reciprocal  automorphism group, which away from a thin set is $1/|Z(G)(k)|$ by Lemma \ref{lem:automorphism_group_is_inner_twist}(3). The exponents $a(L)$ and $b(k,L)$ come from Lemma \ref{lem:Fujita_minimal}.  The effective cone constant is $\alpha^*(BG, L) = a(L)^{b(k,L)}/|\dual{G}(k)|$ by Lemma \ref{lem:effective_cone_calc}. There is also an additional Fujita invariant $a(L)$ on the denominator coming from the exponent of $B$ (since in Peyre's formula the exponent of $B$ is $1$).
For the Tamagawa measure, we mimic the setting considered in \cite[Def.~3.3.10]{BT98} where the Tamagawa measure is taken on an open subset $U$ of $X$. For the cohomological constant, we use the cardinality of the corresponding Brauer group of the mimicked open subset $U$ (cf.~\cite[Def.~3.4.3]{BT98}); this is exactly the partially unramified Brauer group from Definition \ref{def:partially_ramified_Br} and the partial adelic space from Definition \ref{def:adelic_space}. 

The Brauer group $\Br_{\mathcal{M}(L)} BG$ is indeed finite modulo $\Br k$ by Corollary \ref{cor:Br_finite}. In the case of the anticanonical height function given by the radical discriminant, the corresponding Brauer group is the unramified Brauer group $\Brun BG$. 

The Tamagawa factor $\tau_H( BG(\Adele_k)_{\mathcal{M}(L)}^{\Br})$ which appears is quite complicated in general; it is best viewed via the hierarchy explained in Remark \ref{rem:hierarchy}. It is non-zero by Lemma \ref{lem:total_Brauer_measure}.

Conjecture \ref{conj:balanced} is a more precise version of \cite[Conj.~1.3.1]{DYTor} and \cite[Conj.~3.10]{Alb21}, which gave no prediction for the leading constant.

\begin{remark}[The thin set $\Omega$] \label{rem:breaking}
We have allowed flexibility in the choice of $\Omega$ in Conjecture~\ref{conj:balanced}. However once one has found an $\Omega$ which works, any larger choice of thin set will also do (see Theorem \ref{thm:thin_negligable}). Nevertheless, we also have a conjecture for which $\Omega$ to take. This should be the collection of breaking cocycles from \S \ref{sec:breaking}. In the case where $G$ is constant, one can instead take the larger set from Theorem~\ref{thm:breaking}. As Conjecture \ref{conj:balanced} is phrased, for general group schemes $G$ one should also restrict to $\varphi$ with $\Aut \varphi \cong Z(G)(k)$; this is cothin by Lemma \ref{lem:automorphism_group_is_inner_twist}(3).
\end{remark}

We also give a version of the conjecture for field extensions. In the statement we consider field
extensions up to isomorphism, denote by $\widetilde{K}$ the Galois closure of $K$, and consider $\Gal(\widetilde{K}/k)$ as a permutation group. This is the original setting for Malle's conjecture, and we give a precise prediction for the asymptotic providing the discriminant is balanced.

\begin{conjecture} \label{conj:discriminant}
	Let $G \subseteq S_n$ be a transitive subgroup with normaliser $N$ and centraliser $C$.
	Assume that the discriminant is a balanced height function. Then 
	\begin{align*}
	\frac{|N|}{|C| \cdot |G|}&\#\left\{ [K:k] = n :
	\begin{array}{l}
	  |\Norm_{k/\Q} \Delta_{K/k}| \leq B, \Gal(\widetilde{K}/k) \cong G, \\
	 \widetilde{K} \cap k(\mu_{\exp(G)}) = k
	 \end{array} \right\} \\ 
	 & \quad \sim c(k,G,\Delta) B^{a(G)} (\log B)^{b(k,G) - 1}
	\end{align*}
	where $c(k,G,\Delta)$ is as in Conjecture \ref{conj:balanced} and $(a(G),b(k,G))$ are as in
	\eqref{def:a_b_Malle}.
\end{conjecture}

\begin{lemma} \label{lem:equiv_conjectures}
	Conjecture \ref{conj:balanced} for $\Omega$ as in Theorem \ref{thm:breaking}
	implies both Conjectures \ref{conj:intro}(1) and  \ref{conj:discriminant}.
\end{lemma}
\begin{proof}
	The implication for Conjecture \ref{conj:intro}(1) follows from 
	Lemmas~\ref{lem:groupoid_count} and \ref{lem:convergence_factors_alternative}.
	For Conjecture \ref{conj:discriminant}
	as we are counting field extensions, rather than homomorphisms,
	by Lemma \ref{lem:S_n} the count is actually taking place in the essential image
	of the functor $BG(k) \to BS_n(k)$. The result then follows from Lemma \ref{lem:G_S_n}.
\end{proof}

\begin{remark}[Transcendental Brauer group]
	In Manin's conjecture \eqref{eqn:Peyre},
	the leading constant contains the term $|\Br_1 X / \Br k|$, i.e.~only
	algebraic Brauer group elements are considered. In our setting however, we allow
	transcendental Brauer group elements. Evidence for this 
	comes from Ellenberg--Venkatesh \cite[\S2.4]{EV10}, who heuristically 
	observed that the Schur multiplier
	can play a role in Malle's conjecture. The Schur multiplier arises for us via the transcendental
	Brauer group of $BG$ (see Remark \ref{rem:Schur_multiplier}). Moreover we provide numerical evidence in \S \ref{sec:A_4-transcendental} of an explicit example that the transcendental Brauer group should play a role. This suggests that in \eqref{eqn:Peyre} it is actually the term
	$|\Br X/\Br k|$ which should appear instead of $\beta(X)$.
\end{remark}

\subsection{Unbalanced heights} \label{sec:non-balanced}

Assume now that $H$ is not necessarily balanced. Here one reduces to the balanced case as follows. Let $M:=\langle \mathcal{M}(L) \rangle \subset G$ be the subgroup scheme generated by the minimal weight conjugacy classes with respect to $H$. We consider the associated (adjoint) Iitaka fibration $BG \to B(G/M)$ from \S \ref{sec:Iitaka}. The counting problem is then given as a sum of the counting problems of the fibres of the Iitaka fibration, with the restriction of $H$ to each fibre now being balanced by Lemma~\ref{lem:Iitaka_rigid_fibres}. So the setting from \S \ref{sec:adj_rigid_conj} applies. We thus apply Conjecture \ref{conj:balanced} to each fibre, and keep track of groupoid cardinalities using Lemmas \ref{lem:automorphism_group_is_inner_twist} and \ref{lem:outer_automorphism}(1), with the expectation that the sum over all the fibres converges. This leads to the following.

\begin{conjecture} \label{conj:non_balanced}
	There exists a thin subset $\Omega \subset BG[k]$ such that
	$$\frac{1}{|Z(G)(k)|}\#\{ \varphi \in BG[k] \setminus \Omega : H(\varphi) \leq B\} \sim c(k,G,H,\Omega) B^{a(L)} (\log B)^{b(k,L)-1}$$
	where $a(L)$ and $b(k,L)$ are as in Conjecture \ref{conj:balanced}, and 
	$$c(k,G,H,\Omega) = \frac{1}{|(G/M)(k)|}\sum_{\substack{\psi \in \im(BG[k] \to B(G/M)[k]) \\
	BM_\psi(k) \not \subseteq \Omega}} c(k,M_\psi,H)$$
	with the resulting sum being convergent.
\end{conjecture}

Here $M_\psi$ denotes the inner twist of $M$ by a lift of $\psi$ along $BG(k) \to B(G/M)(k)$, as defined in Definition \ref{def:inner_twist_normaliser}; this is a minor abuse of notation which naturally requires that $\psi$ is in the image of $BG(k)$, and $BM_\psi$ is independent of the choice of lift. The stack $BM_\psi$ is a fibre of the Iitaka fibration by Lemma \ref{lem:fibration_normal_quotient}. The constant $c(k,M_\psi,H)$ is as in Conjecture \ref{conj:balanced}; this is admissible as the pull-back of $H$ to $B M_\psi$ is balanced by Lemma \ref{lem:Iitaka_rigid_fibres}. If $G$ is constant, this fibre is exactly the collection of $G$-extensions of $k$ which realise a given $G/M$-extension.

Implicit in this conjecture is that a positive proportion of $G$-extensions of $k$ contain any given $G/M$-extension, providing the set of such $G$-extensions is non-empty. In particular the leading constant $c(k,G,H,\Omega)$ depends on the choice of $\Omega$, since adding to $\Omega$ a fibre of the Iitaka fibration changes the leading constant. However, as in Remark \ref{rem:breaking} we conjecture that a valid choice of $\Omega$ is given by the collection of breaking cocycles from \S \ref{sec:breaking} which, when $G$ is constant, can be replaced by the larger set from Theorem \ref{thm:breaking}.

Warning: even if $G$ is constant, the group scheme $M_\psi$ need not be constant (this happens when counting $D_4$-quartics of bounded discriminant; see \S \ref{sec:D_4}). So in Malle's conjecture one needs to work with general height functions on finite \'etale group schemes, even if one is only interested in number fields of bounded discriminant. The appearance of finite \'etale group schemes in Malle's conjecture appears to have been first noticed by T\"{u}kelli \cite[Conj.~4.8]{Tur15} (see also \cite[Lem.~1.3]{Alb21} and \cite[\S9.5]{DYBM}).

\subsection{The total count} \label{sec:total_count}

Despite our conjectures requiring the removal of a thin set, it is also possible to use Conjecture \ref{conj:non_balanced} 	to give a formula for the total number of surjective elements of $BG$ of bounded height. By Lemma \ref{lem:thin_classification}, one stratifies the thin set $\Omega$ into a finite union of images of maps of the form $BT \to BG$, then applies Conjecture \ref{conj:non_balanced} and Lemma~\ref{lem:outer_automorphism} to each such $BT$ to get a precise prediction; a further stratification of $BT$ may be required so this gives an inductive description for the total count.

We explain this process in detail  when $G$ is a finite group and $\Omega$ is the thin set from Theorem \ref{thm:breaking}(2). We represent elements of $BG(k)$ by homorphisms $\Gamma_k \to G$ of bounded height $H$. Writing the thin set from Theorem \ref{thm:breaking}(2) as a union of images of $BT \to BG$ corresponds to stratifying all homorphisms $\Gamma_k \to G$ according to whether they lift suitable cyclotomic fields; this latter perspective is also considered in a recent paper of Wang \cite{Wan24}.  The leading constant obtained will be quite complicated, as it can be a mix of both balanced and unbalanced cases, and finite \'etale group schemes (even though $G$ itself is constant). This leads to 
$$\Hom(\Gamma_k,G) = \bigsqcup_{N \subset G} \bigsqcup_{ \psi: \Gal(k(\mu_{\exp(G)})/k) \to G/N} \{ \varphi \in \Hom(\Gamma_k,G) : \varphi \text{ is an exact lift of } \psi \}$$
where $N$ runs over all normal subgroups of $G$. We say that $\varphi$ \textit{lifts} $\psi$ if the diagram
		\[
\xymatrix{ \Gamma_k \ar[r] \ar[d] & G  \ar[d] \\ 
\Gal(k(\mu_{\exp(G)})/k) \ar[r] & G/N}
\]
commutes. An \emph{exact lift} is a lift for which $\varphi(\Gamma_{k(\mu_{\exp(G)})}) = N$; it is easy to see that this is equivalent to $\varphi$ not lifting any other homomorphism $\Gal(k(\mu_{\exp(G)})/k) \to G/M$ with $M \subsetneq N \subset G$.
 This therefore leads to the equality
\begin{align} \label{eqn:decomposition}
&\#\{ \varphi \in \Hom(\Gamma_k,G) : \varphi \text{ surjective}, H(\varphi) \leq B\} \\
 = & \sum_{N \subset G} \sum_{ \substack{ \text{surjections } \\ \psi: \Gal(k(\mu_{\exp(G)})/k) \to G/N}} \#\{ \varphi \in \Hom(\Gamma_k,G) : \varphi \text{ is a surjective exact lift of } \psi \} \nonumber
\end{align}
of counting functions. We consider each inner cardinality separately using our framework as follows. For each $\psi$ fix a choice of lift $\varphi_\psi$, providing it exists. Then by Lemma~\ref{lem:fibration_normal_quotient}, the collection of lifts is the $k$-rational points of the stack $B N_{\varphi_\psi}$, where $N_{\varphi_\psi}$ denotes the inner twist of $N$ with respect to the $\varphi_\psi$. To ensure that breaking cocycles, as in \S \ref{sec:breaking}, will not cause problems, we verify that they give non-exact lifts.

\begin{lemma} \label{lem:breaking_total_count}
	The image of a breaking cocycle for $BN_{\varphi_\psi}$ in $BG$ is a non-exact lift.
\end{lemma}
\begin{proof}
The base change of a breaking cocycle is still breaking. Moreover if the base change of a homomorphism to $k(\mu_{\exp(G)})$ is an exact lift, then the homomorphism itself must be an exact lift. Therefore we may assume that $k(\mu_{\exp(G)}) \subset k$.

So let $\varpi \in B N_{\varphi_\psi}(k)$ be breaking. As $\varphi_{\psi}(\Gamma_{k}) \subset N$, the group scheme $N_{\varphi_\psi}$ is an inner twist of $N$. So let $\varpi'$ be the image of $\varpi$ under the isomorphism $B(N_{\varphi_\psi}) \cong BN$ from Lemma \ref{lem:inner_twist_BG}. As $\varpi':\Gamma_k \to N$ is breaking, Lemma~\ref{lem:Breaking_implies_reducible} implies it is not surjective, so it cannot correspond to an exact lift.
\end{proof}

Therefore applying Conjecture \ref{conj:non_balanced}, we obtain the following prediction for the terms in \eqref{eqn:decomposition}. To emphasise the dependence on the group, we include this in the exponents.

\begin{conjecture}
Assume that $G$ is constant. Let $N \subset G$ be a normal subgroup and $\psi: \Gal(k(\mu_{\exp(G)})/k) \to G/N$ surjective. Assume that a lift $\varphi_\psi$ of $\psi$ exists. Let $\Omega_\psi$ be the collection of lifts of $\psi$ which are either non-surjective or non-exact. Then
\begin{align*}
&\frac{1}{|G|}\#\{ \varphi \in \Hom(\Gamma_k,G)  : \varphi \text{ is a surjective exact lift of } \psi, H(\varphi) \leq B \} \\
\sim  & \frac{|N|}{|G|} c(k,N_{\varphi_\psi},H,\Omega_\psi) B^{a(N_{\varphi_\psi},H)} (\log B)^{b(k,N_{\varphi_\psi},H)-1}
\end{align*}
where $c(k,N_{\varphi_\psi},H,\Omega_\psi)$ is as in Conjecture \ref{conj:non_balanced}.
\end{conjecture}
The factor $|N|/|G|$ comes from groupoid cardinality considerations: Lemma \ref{lem:outer_automorphism}(1) implies that each fibre of $BN_{\varphi_{\psi}}(k) \to BG(k)$ has groupoid cardinality $|(G_{\varphi_\psi}/N_{\varphi_\psi})(k)|$. But $G_{\varphi_\psi}/N_{\varphi_\psi} = (G/N)_{\psi} = G/N$ because $G/N$ is abelian.


\begin{remark}
Kl\"{u}ners \cite{Klu05} was the first to observe that fields which contain a given cyclotomic field can be accumulating in Malle's conjecture. Türkelli \cite[Conj~4.8]{Tur15} proposed a possible fix to Malle's conjecture by taking into account these cyclotomic fields. A recent paper of Wang \cite[Conj.~7]{Wan24} gives a counter-example to Türkelli's original conjecture and suggests a modification for finding the power of $\log B$, with a key observation that one discounts those subfields of $k(\mu_{\exp(G)})$ for which no lift exists.  A priori it could be possible that one should remove further thin sets from each $B N_{\varphi_\psi}$; but  we have shown in Lemma \ref{lem:breaking_total_count} that this is unnecessary. Our work builds on Wang's paper by also putting forward a conjecture for the leading constant.
\end{remark}

\subsection{Equidistribution} \label{sec:equidistribution}
There are numerous papers in the Malle's conjecture literature which concern counting number fields
with ``local specifications'' imposed, and ask whether the quotient of this count with respect to the total
count has the expected local behaviour; see for example \cite{Bha07, Woo10, BSW, ShTh22}. In the Malle conjecture literature this is informally referred to as the \textit{Malle--Bhargava heuristics}.

In the Manin's conjecture literature the property is called \emph{equidistribution}, since it stipulates weak convergence of a sequence of measures, and was first considered by Peyre \cite[\S 3]{Pey95} (see \cite[3.21]{Pey21} for this perspective, as well as the Portmanteau theorem \cite[Thm.~13.16]{Klen14} for various equivalent formulations). With the relevant Tamagawa measure now in hand for $BG$, we are able to state the following equidistribution conjecture.  Recall that a \emph{continuity set} is set whose boundary has measure $0$.

\begin{conjecture}[Equidistribution] \label{conj:equi}
	Assume that $H$ is balanced. Let $W \subseteq \prod_v BG[k_v]$ be a continuity set.
	Then there exists a thin subset $\Omega \subset BG[k]$ such that
	$$\lim_{B \to \infty}\frac{\#\{ \varphi \in BG[k] \setminus \Omega : \varphi \in W, H(\varphi) \leq B\}}
	{\#\{ \varphi \in BG[k] \setminus \Omega : H(\varphi) \leq B\}} = 
	\frac{\tau_H(W \cap BG[\Adele_k]_{\mathcal{M}(L)}^{\Br})}
	{\tau_H( BG[\Adele_k]_{\mathcal{M}(L)}^{\Br})}.$$
\end{conjecture}

In the statement, we abuse notation and the condition $\varphi \in W$ means that the image of $\varphi$ under the map $BG[k] \to \prod_v BG[k_v]$ lies in $W$ (this map need not be injective if $G$ is non-constant). The following shows that this conjecture is equivalent to a more elementary statement regarding approximating at finitely many places (as in the Malle--Bhargava heuristics). In the statement we denote by $\varphi_v$ the image of $\varphi$ in $BG[k_v]$ and $\lambda_v$ the convergence factors from \S \ref{sec:global_Tamagawa}.

\begin{proposition} \label{prop:Malle-Bhargava}
\, \hfill \,

(1) Conjecture \ref{conj:equi} holds for $(G,H)$ if and only if there exists a thin subset $\Omega \subset BG[k]$ such that for all finite sets of places $S$ and all $\psi_v \in BG(k_v)$ we have
\begin{align*}
& \lim_{B \to \infty}\frac{\#\left\{ \varphi \in BG[k] \setminus \Omega : 
\begin{array}{l}
\varphi_v = \psi_v \text{ for all } v \in S, \\
H(\varphi) \leq B
\end{array}\right\}}
	{\#\{ \varphi \in BG[k] \setminus \Omega : H(\varphi) \leq B\}} 
	= \frac{\L^*(\mathcal{M}(L),1)}	{\tau_H( BG(\Adele_k)_{\mathcal{M}(L)}^{\Br})} \times \\
&\frac{1}{|\Br_{\mathcal{M}(L)} BG/\Br k|} \sum_{b}
\prod_{v \in S} \frac{\lambda_v^{-1}e^{2 \pi i \inv_v b(\psi_v)}}{|\Aut(\psi_v)| H_v(\psi_v)^{a(L)} }
\prod_{v \notin S}\lambda_v^{-1}\int_{BG(k_v)} e^{2 \pi i \inv_v b(\varphi_v) } d\tau_{H, v}(\varphi_v)
\end{align*}
where the sum is over $b \in \Br_{\mathcal{M}(L)} BG/\Br k$.

	(2) Assume that $\Br_{\mathcal{M}(L)} BG =\Brun BG$. 	Conjecture \ref{conj:equi} holds for $(G,H)$ if and only if there exists a thin subset $\Omega \subset BG[k]$ such that for all finite sets of places $S$ and all $\psi_v \in BG(k_v)$ with $\prod_{v \in S} \{\psi_v\} \times \prod_{v \notin S} BG(k_v) \subset \prod_v BG(k_v)^{\Brun BG}$ we have
\begin{align*}
& \lim_{B \to \infty}\frac{\#\left\{ \varphi \in BG[k] \setminus \Omega : 
\begin{array}{l}
\varphi_v = \psi_v \text{ for all } v \in S, \\
H(\varphi) \leq B
\end{array}\right\}}
	{\#\{ \varphi \in BG[k] \setminus \Omega : H(\varphi) \leq B\}}  \\
&= \frac{\L^*(\mathcal{M}(L),1)}	{\tau_H( BG(\Adele_k)_{\mathcal{M}(L)}^{\Br})} 
\prod_{v \in S} \frac{\lambda_v^{-1}}{|\Aut(\psi_v)| H_v(\psi_v)^{a(L)} }
\prod_{v \notin S}\lambda_v^{-1}\tau_{H,v}(BG(k_v)).
\end{align*}
\end{proposition}
\begin{proof}
	Let $W$ be as in Conjecture \ref{conj:equi}. We first obtain a lower bound.
	As $W$ and its interior
	have the same measure, we are free to replace $W$ by its interior, so assume
	that $W$ is open. The set $BG[k_v]$ is finite with the discrete topology. 
	Thus $W$ may be written as a	disjoint union of open sets of the form
	\begin{equation} \label{eqn:W_open}
	\prod_{v \in S} \{\psi_v\} \times \prod_{v \notin S} BG[k_v]
	\end{equation} for some varying finite
	set of places $S$ and some $\psi_v$. Moreover for all $\varepsilon > 0$, there exists
	finitely many such sets $W_j$ for $j \in J_\varepsilon$ whose union has measure within $\varepsilon$
	of the measure of $W$. However a minor variant of Lemma \ref{lem:sum_Euler_products}
	shows that $\tau_H(W_j \cap BG[\Adele_k]_{\mathcal{M}(L)}^{\Br})$ equals
	$$\frac{\L^*(\mathcal{M}(L),1)}{|\Br_{\mathcal{M}(L)} BG/\Br k|} \sum_{b}
\prod_{v \in S} \frac{\lambda_v^{-1}e^{2 \pi i \inv_v b(\psi_v)}}{|\Aut(\psi_v)| H_v(\psi_v)^{a(L)} }
\prod_{v \notin S}\lambda_v^{-1}\int_{BG(k_v)} e^{2 \pi i \inv_v b(\varphi_v) } d\tau_{H, v}(\varphi_v)$$
as in the statement. Thus applying our assumptions, altogether we obtain
	\begin{align*}
		& \liminf_{B \to \infty}\frac{\#\{ \varphi \in BG[k] \setminus \Omega : \varphi \in W, H(\varphi) \leq B\}}
	{\#\{ \varphi \in BG[k] \setminus \Omega : H(\varphi) \leq B\}}  \\
	\geq &
	\sum_{j \in J_{\varepsilon}}\lim_{B \to \infty}\frac{\#\{ \varphi \in BG[k] \setminus \Omega : \varphi \in W_j, H(\varphi) \leq B\}}
	{\#\{ \varphi \in BG[k] \setminus \Omega : H(\varphi) \leq B\}} \\
	\geq &\frac{\tau_H(W \cap BG[\Adele_k]_{\mathcal{M}(L)}^{\Br}) - \varepsilon}
	{\tau_H( BG[\Adele_k]_{\mathcal{M}(L)}^{\Br})}.
	\end{align*}
	Thus taking $\varepsilon \to 0$ gives the required lower bound.
	For the upper bound, as $W$ is a continuity set so is its complement. Hence applying the lower bound to the complement
	we obtain the correct upper bound for $W$. Part (1) now follows.
	
For Part (2) let $\mathscr{B}$ be a finite group of representatives of $\Brun BG/\Br k$ with trivial evaluation at the identity cocycle. By Theorem \ref{thm:Harari's_formal_lemma} there is a finite set of places $S_0$ such that the local invariant of each $b \in \mathscr{B}$ is constant away from $S_0$, hence by our choice trivial away from $S_0$.

First assume Conjecture \ref{conj:equi}. Let $\prod_{v \in S} \{\psi_v\} \times \prod_{v \notin S} BG(k_v) \subset \prod_v BG(k_v)^{\Brun BG}$. The only option is that the local invariant of each $b \in \mathscr{B}$ is constant on $BG(k_v)$ for $v\notin S$, thus it is trivial. We deduce $\sum_{v \in S}\inv_v b(\psi_v) = 0$ hence $\prod_{v \in S} e^{2 \pi i \inv_v b(\psi_v)} = 1$ for all $b \in \mathscr{B}$. The stated formula now follows from Part (1).

Next let $W$ as in \eqref{eqn:W_open}. We may assume that $S_0 \subset S$. As $S_0 \subset S$ we have either $W \cap BG[\Adele_k]_{\mathcal{M}(L)}^{\Br} = W$ or $\emptyset$. In the first case, our assumptions in (2) imply the formula in (1) since the sum of all local invariants is $0$. In the second case character orthogonality shows that the both sides in (1) equal $0$. Thus Part (1) shows that Conjecture \ref{conj:equi} holds, as required. 
\end{proof}

The Malle--Bhargava heuristics as they usually appear in the literature are stated in a version which more closely resembles the formula from (2) from Proposition~\ref{prop:Malle-Bhargava}. This is the second case of the hierarchy explained in Remark \ref{rem:hierarchy}. However for general height functions the Brauer group need not be unramified and the more complicated formula in (1) is the correct one. We could not find any cases of this more general formula in the literature.

Peyre shows in \cite[Prop.~3.3]{Pey95} in the setting of Manin's conjecture that the equidistribution property is equivalent to proving his asymptotic with respect to arbitrary choices of adelic metric on the line bundle (see also \cite[Prop.~2.10]{CT10} for a general topological statement). We have an analogous property in our case.

\begin{lemma} \label{lem:equi_any_height} \hfill
\begin{enumerate}
 \item Assume that Conjectures \ref{conj:balanced} and \ref{conj:equi} hold for $(G,H)$. Then Conjecture~\ref{conj:balanced} holds for all choices of height on $L$.
\item Assume that Conjecture \ref{conj:balanced} holds for all choices of height on $L$. Then Conjecture \ref{conj:equi} holds for $(G,H)$.
\end{enumerate}
\end{lemma}
\begin{proof}
	Recall from \S \ref{sec:heights} that we allow our local heights to be arbitrary 
	\textit{positive functions} at finitely many
	places (in particular one cannot take indicator functions for the local heights, since this
	would break the Northcott property).
	
	(1) Let $H'$ be a different choice of height on $L$.
	Then $H'/H: \prod_v BG[k_v] \to \R_{>0}$ is a well-defined continuous
	function with finite image. Thus for any $c \in \R_{>0}$ the inverse
	image, which we denote by $W_c$, is open. We obtain
	\begin{align*}
	&\lim_{B \to \infty}\frac{\#\left\{ \varphi \in BG[k] \setminus \Omega : 
\begin{array}{l}H'(\varphi) \leq B
\end{array}\right\}}
	{\#\{ \varphi \in BG[k] \setminus \Omega : H(\varphi) \leq B\}} \\
	= &\sum_{c \in \mathrm{Im}(H'/H)} 
	\lim_{B \to \infty}\frac{\#\left\{ \varphi \in BG[k] \setminus \Omega : 
\begin{array}{l}
\varphi \in W_c, H(\varphi) \leq cB
\end{array}\right\}}
	{\#\{ \varphi \in BG[k] \setminus \Omega : H(\varphi) \leq B\}}.
	\end{align*}
	Applying Conjectures \ref{conj:balanced} and \ref{conj:equi}
	shows that this limit equals
	$$
	\frac{c^{a(L)}\tau_H(W_c \cap BG[\Adele_k]_{\mathcal{M}(L)}^{\Br})}
	{\tau_H( BG(\Adele_k)_{\mathcal{M}(L)}^{\Br})}	 = 	
	\frac{\tau_{H'}(W_c \cap BG[\Adele_k]_{\mathcal{M}(L)}^{\Br})}
	{\tau_H( BG(\Adele_k)_{\mathcal{M}(L)}^{\Br})}.
	$$
	Summing over all $c$ now shows that Conjecture \ref{conj:balanced} holds for $H'$.

	(2) By Proposition \ref{prop:Malle-Bhargava} it suffices
	to consider the case where 
	$W = \prod_{v \in S} \{\psi_v\} \times \prod_{v \notin S} BG(k_v)$ for some 
	$\psi_v \in BG(k_v)$ and finite set of places $S$.
	We begin with an upper bound.
	Let $\varepsilon > 0$ and let $H_{\varepsilon}$ be the height on $L$
	such that
	$$H_{\varepsilon,v}(\varphi_v)  =
	\begin{cases}
	H_v(\varphi_v), & \text{if $v \notin S$ or $v \in S$ and } 
	\varphi_v = \psi_v, \\
	\varepsilon, & \textit{otherwise},
	\end{cases}.$$
	Thus applying Conjecture \ref{conj:balanced} and
	Lemma \ref{lem:sum_Euler_products} we find that
	\begin{align*}
	& \lim_{B \to \infty}\frac{\#\left\{ \varphi \in BG[k] \setminus \Omega : 
\begin{array}{l}
\varphi_v = \psi_v \text{ for all } v \in S, \\
H(\varphi) \leq B
\end{array}\right\}}
	{\#\{ \varphi \in BG[k] \setminus \Omega : H(\varphi) \leq B\}} \\
	\leq & \lim_{B \to \infty}\frac{\#\left\{ \varphi \in BG[k] \setminus \Omega : 
\begin{array}{l}
H_{\varepsilon}(\varphi) \leq B
\end{array}\right\}}
	{\#\{ \varphi \in BG[k] \setminus \Omega : H(\varphi) \leq B\}} \\
= &\frac{\L^*(\mathcal{M}(L),1)}	{\tau_H( BG(\Adele_k)_{\mathcal{M}(L)}^{\Br})|\Br_{\mathcal{M}(L)} BG/\Br k|} \sum_{b}
\prod_{v}\lambda_v^{-1}\int_{BG(k_v)} e^{2 \pi i \inv_v b(\varphi_v) } d\tau_{H_{\varepsilon},v}(\varphi_v),
\end{align*}
where the sum is over  $b \in \Br_{\mathcal{M}(L)} BG/\Br k$.
Taking $\varepsilon \to 0$,
we obtain the correct upper bound (see Proposition \ref{prop:Malle-Bhargava}).
For the lower bound, $W$ is closed as $BG(k_v)$ is finite. Thus applying the  upper bound to the complement of $W$, which is open, we obtain the correct lower bound for Conjecture \ref{conj:equi}.
\end{proof}

The equidistribution conjecture adds further evidence to the need to remove a thin set, since in general equidistribution can only hold after removal of a thin set (this has been shown for example in the case of cubic extensions of bounded radical discriminant \cite[Thm.~6]{ShTh22}). But more than that, once equidistribution has been obtained, the following result shows that removing an additional thin set does not change the asymptotic formula. This means that actually there is no danger in removing a larger thin set than required in Conjecture \ref{conj:balanced}. This property for Manin's conjecture was first verified in \cite[Thm.~1.2]{BL19}.

\begin{theorem}	\label{thm:thin_negligable}
	Assume that Conjecture \ref{conj:equi} holds for $(G,H)$.
	Let $\Upsilon \subseteq BG[k]$ be thin. Then
	$$\lim_{B \to \infty} 
	\frac{\#\{ \varphi \in \Upsilon \setminus \Omega : H(\varphi) \leq B\}}
	{\#\{ \varphi \in BG[k] \setminus \Omega : H(\varphi) \leq B\}} = 0.$$
\end{theorem}
\begin{proof}
	By Lemma \ref{lem:thin_classification}, it suffices to consider the case of $\Upsilon = f(BT[k])$
	where $f:BT \to BG$ for $T$ a proper subgroup scheme of some inner twist of $G$.
	For any finite set of places $S$ of $k$,
	the equidistribution property implies that
	\begin{align*}
	&\lim_{B \to \infty} \frac{\#\{ \varphi \in f(BT[k]) \setminus \Omega : H(\varphi) \leq B\}}
	{\#\{ \varphi \in BG[k] \setminus \Omega : H(\varphi) \leq B\}} \\
	\leq &\lim_{B \to \infty} 
	\frac{\#\{ \varphi \in BG[k] \setminus \Omega : H(\varphi) \leq B, \varphi \in f(BG[k_v]) \text{ for all }
	v \in S\}}
	{\#\{ \varphi \in BG[k] \setminus \Omega : H(\varphi) \leq B\}} \\
	\leq & \L^*(\mathcal{M}(L),1) \prod_{v \in S}  \lambda_v^{-1} \tau_{H,v}(f(BT(k_v)) 
	\prod_{v \notin S} \lambda_v^{-1}\tau_{H,v}(BG(k_v)).
	\end{align*}
	It thus suffices to show that this product diverges to $0$ as $S \to \Omega_k$.
	
	Let $K/k$ be a splitting field for $G$ and $T$. Let
	$v$ be completely split in $K$. By Lemma \ref{lem:tau_BG(O_v)}
	the map $BT(\O_v) \to BG(\O_v)$ corresponds
	to $\Hom(\Gamma_{\F_v},T_{k_v}) \to \Hom(\Gamma_{\F_v},G_{k_v})$, which is injective 
	as $T_{k_v} \subseteq G_{k_v}$. However, by \eqref{eqn:tau_alternative}, 
	Lemma \ref{lem:tau_BG(O_v)}, and Corollary \ref{cor:mass_formula}, for such $v$ we have
	\begin{align*}
	\tau_{H,v}(f(BT(k_v))) & = \tau_{H,v}(f(BT(\O_v)))
	+ O(\tau_{H,v}(BG(k_v) \setminus BG(\O_v))) \\
	& = \frac{|T|}{|G|} + O(1/q_v).
	\end{align*}
	The claim now follows from the Chebotarev density theorem as $T \neq G$.
\end{proof}

\begin{remark}
	Conjecture \ref{conj:equi} implies that the map
	$$BG[k] \to (\prod_v BG[k_v])^{\Brun BG}$$
	has dense image. This can be interpreted as saying that the Grunwald problem
	(with Brauer--Manin obstruction) always has a solution for $G$.
	(See for example \cite{DG12, DLAN17,BN24} for background on the Grunwald problem).
\end{remark}

\subsection{Strong equidistribution} 
Conjecture \ref{conj:equi} considers counting problems related to imposing finitely many local conditions. However there are results in the Malle's conjecture literature which give results with \emph{infinitely} many local conditions imposed; see e.g.~\cite[Thm.~1.1]{Bha14}, \cite[Thm.~2]{BSW}, \cite[Thm.~1.7]{FLN18}, or \cite[Thm.~3]{ASVW21}. We finish by considering a strengthening of Conjecture \ref{conj:equi} which covers this case. 

\begin{conjecture}[Strong equidistribution] \label{conj:equi_strong}
	Assume that $H$ is balanced. Let $W \subseteq BG[\Adele_k]_{\mathcal{M}(L)}$ be a
	continuity set.
	Then there exists a thin subset $\Omega \subset BG[k]$ such that
	$$\lim_{B \to \infty}\frac{\#\{ \varphi \in BG[k] \setminus \Omega : \varphi \in W, H(\varphi) \leq B\}}
	{\#\{ \varphi \in BG[k] \setminus \Omega : H(\varphi) \leq B\}} = 
	\frac{\tau_H(W \cap BG[\Adele_k]_{\mathcal{M}(L)}^{\Br})}
	{\tau_H( BG[\Adele_k]_{\mathcal{M}(L)}^{\Br})}.$$
\end{conjecture}

We forgive the reader for not appreciating the difference between the two equidistribution conjectures at first glance, which comes from the important difference between the topologies of $BG(\Adele_k)_{\mathcal{M}(L)}$ and $\prod_v BG(k_v)$.

Let $S$ be a sufficiently large finite set of places of $k$. In Conjecture \ref{conj:equi_strong} one may take $W = \prod_{v \in S} BG(k_v) \prod_{v \notin S} BG(\O_v)_{\mathcal{M}(L)}$. When $G$ is constant, the corresponding counting function counts those $G$-extensions of $k$ whose ramification type at all $v \notin S$ is either trivial or lies in the collection $\mathcal{M}(L)$ of minimal weight conjugacy classes. One can thus view this as counting number fields with prescribed ramification imposed. For example, this recovers as a special case the problem of counting $S_n$-extensions of degree $n$ of bounded discriminant whose discriminant is a squarefree; this latter problem was considered by Ellenberg and Venkatesh in \cite[\S 2.3]{EV10}.
We have the following analogue of Proposition \ref{prop:Malle-Bhargava}.

\begin{proposition}\label{prop:Malle-Bhargava_strong}
	Conjecture \ref{conj:equi_strong} holds for $(G,H)$ if and only if 
	there exists a thin subset $\Omega \subset BG[k]$ such that for all finite sets of 
	places $S$ containing the non-good places,  and all $\psi_v \in BG(k_v)$ with 
	$\prod_{v \in S}\{\psi_v\} \times \prod_{v \notin S} BG(\O_k)_{\mathcal{M}(L)} \subseteq BG(\Adele_k)_{\mathcal{M}(L)}^{\Br}$
	we have	 
\begin{align*}
& \lim_{B \to \infty}\frac{\#\left\{ \varphi \in BG[k] \setminus \Omega : 
\begin{array}{l}
\varphi_v = \psi_v \text{ for all } v \in S, 
\varphi_v \in BG[\O_v]_{\mathcal{M}(L)} \text{ for all } v \notin S, \\ H(\varphi) \leq B
\end{array}\right\}}
	{\#\{ \varphi \in BG[k] \setminus \Omega : H(\varphi) \leq B\}} \\
&= \frac{\L^*(\mathcal{M}(L),1)}	{\tau_H( BG(\Adele_k)_{\mathcal{M}(L)}^{\Br})} 
\prod_{v \in S} \frac{\lambda_v^{-1}}{|\Aut(\psi_v)| H_v(\psi_v)^{a(L)} }
\prod_{v \notin S}\lambda_v^{-1}\tau_{H,v}(BG(\O_v)_{\mathcal{M}(L)}).
\end{align*}
\end{proposition}
\begin{proof}
	The proof is similar to the proof of Proposition \ref{prop:Malle-Bhargava},
	so we shall be brief. Firstly $\prod_{v \in S}\{\psi_v\} \times \prod_{v \notin S} BG(\O_k)_{\mathcal{M}(L)}$
	is obviously a continuity set, being both open and compact. Thus Conjecture \ref{conj:equi_strong}
	gives an asymptotic formula, which is easily verified to agree with the formula in the statement.
	
	Next let $W$ be as in Conjecture \ref{conj:equi_strong}. It suffices to prove
	a lower bound in the case where $W$ is open. 
	The set $BG(\Adele_k)_{\mathcal{M}(L)}^{\Br}$ is open and closed by Corollary~\ref{cor:Br_finite}
	and Lemma \ref{lem:modified_adeles_Brauer-Manin}.
	Thus we may assume that $W \subseteq BG(\Adele_k)_{\mathcal{M}(L)}^{\Br}$.
	Recall from	Definition~\ref{def:adelic_space} that
	$$BG(\Adele_k)_{\mathcal{M}(L)} = \lim_{S} \prod_{v \in S} BG(k_v) 
	\prod_{v \notin S} BG(\O_v)_{\mathcal{M}(L)}.$$
	For any finite set of places $S$ we let $W_S = W \cap(\prod_{v \in S} BG[k_v] 
	\prod_{v \notin S} BG[\O_v]_{\mathcal{M}(L)}). $ Then for any such $S$ we have 
	\begin{align*}
	\#\{ \varphi \in BG[k] \setminus \Omega : \varphi \in W, H(\varphi) \leq B\} 
	\geq \#\{ \varphi \in BG[k] \setminus \Omega : \varphi \in W_S, H(\varphi) \leq B\}.
	\end{align*}
	However $W_S$ is open in $\prod_{v \in S} BG[k_v] 
	\prod_{v \notin S} BG[\O_v]_{\mathcal{M}(L)}$, thus may be well approximated
	by a disjoint union of sets as appearing in the statement.
	Hence applying our assumptions we obtain
	$$\liminf_{B \to \infty}
	\frac{\#\{ \varphi \in BG[k] \setminus \Omega : \varphi \in W, H(\varphi) \leq B\}}
	{\#\{ \varphi \in BG[k] \setminus \Omega : H(\varphi) \leq B\}} \geq
	\frac{\tau_H(W_S)}
	{\tau_H( BG(\Adele_k)_{\mathcal{M}(L)}^{\Br})}.$$
	But $\lim_S \tau_H(W_S) = \tau_H(W)$: this follows from the fact
	that the Tamagawa measure is inner regular and $W_S$ is compact. This gives the correct
	lower bound. The upper bound comes from applying the lower bound to the complement of $W$.
\end{proof}

We also have the following version of Lemma \ref{lem:equi_any_height}. The proof is analogous, hence omitted (a similar converse can be formulated).

\begin{lemma}
Assume that for all choices of height $H$ on $L$ and all finite sets of places $S$ we have
\begin{align*}
&\frac{1}{|Z(G)(k)|}\#\{ \varphi \in BG[k] \setminus \Omega : H(\varphi) \leq B,  \varphi \in W_S\} \\
& \sim \frac{a(L)^{b(k,L)-1}\cdot |\Br_{\mathcal{M}(L)} BG / \Br k| \cdot
	\tau_H( W_S \cap BG(\Adele_k)_{\mathcal{M}(L) }^{\Br})}{\#\dual{G}(k) (b(k,L) - 1)!} B^{a(L)} (\log B)^{b(k,L)-1}
\end{align*}
where $W_S = \prod_{v \in S} BG[k_v]  \prod_{v \notin S} BG[\O_v]_{\mathcal{M}(L)}$.
Then Conjecture \ref{conj:equi_strong} holds for all $H$ on $L$.
\end{lemma}

Naturally Conjecture \ref{conj:equi_strong} implies Conjecture \ref{conj:equi}. One can obtain the converse providing one shows a suitable tail estimate, which is reminiscent of the condition in Ekedahl's geometric sieve (see e.g.~\cite[Lem.~3.1]{BBL16}, \cite[Thm.~3.3]{Bha14}), \cite[Thm~17]{BSW}, \cite[Thm.~1.2]{Eke91}, \cite{PS99b}).

\begin{proposition} \label{prop:Ekedahl}
	Assume that Conjecture \ref{conj:equi} holds for $(G,H)$ and  that
	$$\lim_{M \to \infty} \limsup_{B \to \infty} 
\frac{\#\left\{ \varphi \in BG[k] \setminus \Omega : 
\begin{array}{l}
\varphi_v \notin BG(\O_v)_{\mathcal{M}(L)} \text{ for some }q_v \geq M, \\
 H(\varphi) \leq B
\end{array}\right\}}
	{\#\{ \varphi \in BG[k] \setminus \Omega : H(\varphi) \leq B\}} = 0.$$
	Then Conjecture \ref{conj:equi_strong} holds for $(G,H)$.
\end{proposition}
\begin{proof}
	By Proposition \ref{prop:Malle-Bhargava_strong}, it suffices
	to consider the case where $W = 
	\prod_{v \in S}\{\psi_v\} \times \prod_{v \notin S} 
	BG(\O_k)_{\mathcal{M}(L)} \subseteq BG(\Adele_k)_{\mathcal{M}(L)}^{\Br}$ for some finite set 
	of places $S$. We have
	\begin{align*}
	&\#\{ \varphi \in BG[k] \setminus \Omega : \varphi \in W, H(\varphi) \leq B\} \\
	=& 
	\#\left\{ \varphi \in BG[k] \setminus \Omega :
	\begin{array}{l}
	\varphi_v = \psi_v \text{ for all } v \in S , \\
	\varphi_v \in BG(\O_v)_{\mathcal{M}(L)} \text{ for all } v \notin S, q_v < M\\
	H(\varphi) \leq B
	\end{array}\right\} \\
	+ &O\left(\#\left\{ \varphi \in BG[k] \setminus \Omega : 
\begin{array}{l}
\varphi_v \notin BG(\O_v)_{\mathcal{M}(L)} \text{ for some }q_v \geq M, \\
 H(\varphi) \leq B
\end{array}\right\}\right)
	\end{align*}
	for all $M \geq 0$. Conjecture \ref{conj:equi} gives an asymptotic for the first term
	as $B \to \infty$,
	which is absolutely convergent as $M \to \infty$ by Corollary \ref{cor:mass_formula}.
	Our assumption shows that the error term is negligible. Conjecture \ref{conj:equi_strong}
	follows.
\end{proof}

Conjecture \ref{conj:equi_strong} has applications to the Cohen--Lenstra--Martinet heuristics on  class groups of number fields. This will be explored in forthcoming work.

The proof of the following is a minor variant of Lemma \ref{lem:total_Brauer_measure}. However its significance (assuming Conjecture \ref{conj:equi_strong}) is that it suffices to find a single
cocycle in the intersection, which may be non-surjective, and this allows one to deduce the existence
of a surjective cocycle in the intersection (as follows from Theorem~\ref{thm:thin_negligable}). This flexibility can be crucial for applications.

\begin{lemma}
	Let $W \subseteq BG[\Adele_k]_{\mathcal{M}(L)}$ be a open.
	If the intersection of $W$ and $BG[\Adele_k]_{\mathcal{M}(L)}^{\Br}$ is non-empty 
	then it has positive Tamagawa measure.
\end{lemma}

\begin{remark}
	Peyre \cite[\S3]{Pey95} only considered equidistribution for anticanonical heights. We could not find
	any results in the Manin's conjecture literature regarding equidistribution for 
	non-anticanonical heights. However it seems very reasonable to expect that the 
	equidistribution property should hold for height functions associated to big adjoint rigid
	line bundles, providing equidistribution takes place in the adelic
	points of the open subset given by removing the adjoint divisor. Conjecture \ref{conj:equi_strong}
	is exactly modelled on this situation (cf.~Remark \ref{rem:intuition}).
\end{remark}

\subsection{Equidistribution for unbalanced heights} \label{sec:equi_unbalanced}

If $H$ is not balanced then we follow the procedure from \S \ref{sec:non-balanced}. Namely we pass to the Iitaka fibration, and then Conjecture~\ref{conj:equi} predicts that the rational points in each fibre are equidistributed with respect to the induced Tamagawa measure on the fibre. 

This explains the phenomenon, first observed by Wood \cite[Prop.~1.4]{Woo10} when counting abelian extensions of bounded discriminant, that the equidistribution property need not hold when counting all extensions. The fix is very simple: to get equidistribution one should only count extensions which lie in a given fibre of the Iitaka fibration.

\begin{remark}
	Darda and Yasuda define in \cite[\S3.5]{DYTor} a Radon measure in the case of abelian $G$ 
	via an abstract process. It is not clear how this relates to our Tamagawa measure. 
	The fact that they obtain equidistribution in \cite[Thm.~3.5.8]{DYTor} for 
	unbalanced heights suggests it is of a very different nature.
\end{remark}

\subsection{Products}
We next consider compatibility of our conjectures with respect to products. We focus on the balanced case as it is more fundamental. Let $G_1$ and $G_2$ be finite \'etale group schemes over $k$ with balanced heights $H_1$ and $H_2$. We define a height $H:= H_1 \boxtimes H_2$ on $G:=G_1 \times G_2$ by taking the product of height functions. To ensure that a product of balanced heights is balanced, we assume that the Fujita invariants are equal.

\begin{proposition} \label{prop:products}
	Assume that Conjecture \ref{conj:balanced} holds for $(G_1,H_1)$
	and $(G_2,H_2)$ and that $a(H_1) = a(H_2)$.
	Then Conjecture \ref{conj:balanced} holds for $(G,H)$.
\end{proposition}
\begin{proof}
	Firstly, by \cite[\S1.1]{FMT89}, one finds that
	\begin{align*}
	&\#\{ \varphi \in BG(k) \setminus \Omega_1 \times \Omega_2 : 
	H(\varphi) \leq B\} \\
	\sim &
	\frac{a(H_1)\Gamma(b(H_1))\Gamma(b(H_2))c(k,G_1,H_1)c(k,G_2,H_2)}{\Gamma(b(H_1) +b(H_2))} B^{a(H_1)}(\log B)^{b(H_1) + b(H_2) -1}.
	\end{align*}
	(The cited result assumes that $a(H_1) = a(H_2) = 1$, but one obtains the stated formula by rescaling the height.)
	However one easily checks that
	\begin{align*}
	&\dual{G_1 \times G_2} = \dual{G}_1 \times \dual{G}_2, \\
	& \mathcal{M}(H) = \mathcal{M}(H_1) \times \{e\} 
	\, \cup \, \{e\} \times \mathcal{M}(H_2), \\
	&a(H) = a(H_1), \quad b(H_1 \boxtimes H_2) = b(H_1) + b(H_2).
	\end{align*}
	This allows one to match up all factors in the conjecture, except the Brauer group and Tamagawa
	measure. For the Tamagawa measures, the above description of $\mathcal{M}(H)$
	shows that the convergences factors match up. We also  have 
	$\Aut (\varphi_1, \varphi_2) \cong  \Aut(\varphi_1 \times \varphi_2)$, 
	so the groupoid cardinalities match. We obtain the equality
	$\tau_{G,H} = \tau_{G_1,H_1}\tau_{G_1,H_2}$ of measures. The partially unramified Brauer group is compatible with products by Proposition \ref{prop:product_Brauer_group}, which by functoriality implies that $B(G_1 \times G_2)(\Adele_k)_{\mathcal{M}(H)}^{\Br} = BG_1(\Adele_k)_{\mathcal{M}(H_1)}^{\Br} \times BG_2(\Adele_k)_{\mathcal{M}(H_2)}^{\Br}$. The proposition follows.
\end{proof}

\begin{example} \label{ex:products}
	Beware that a product of anticanonical heights need not be an anticanonical height in general
	(contrary to Manin's conjecture).
	Take $G = \Z/2\Z \times \Z/2\Z$. Let $H_1$ (resp.~$H_2$) 
	be the conductor of the quadratic extension corresponding to the first (resp.~second) factor.
	Then each $H_i$ is an anticanonial height on the quadratic field,
	and  Proposition \ref{prop:products} implies an asymptotic formula for $H_1 \boxtimes H_2$. 
	However $H_1 \boxtimes H_2$ is not an anticanonical height: an anticanonical height
	can be given by taking the square root of the product of the discriminants
	of the three quadratic subfields of the corresponding biquadratic extension.
\end{example}
\begin{remark}
	Wang \cite{Wan21} has studied Malle's conjecture for $S_n \times A$ where $A$ is an abelian group. The height is the discriminant coming from the embedding $S_n \times A \subset S_n \times S_{|A|} \subset S_{n|A|}$. On the level of \'etale algebras the inclusion $S_n \times S_{|A|} \subset S_{n|A|}$ corresponds to the tensor product.
	
	This work is not a special case of Proposition \ref{prop:products} as the discriminant of a tensor product of \'etale algebras is not a product (of powers) of the individual discriminants. (This is related to the phenomenon described in Example \ref{ex:products}).
\end{remark}

\begin{remark}[Weil restriction] \label{rem:Weil}
The compatibility of the Weil restriction with Manin's conjecture has been studied in \cite{Lou15}. It would be worthwhile to investigate this in the context of Malle's conjecture, since Weil restrictions of constant group schemes appear in practice as the fibres of Iitaka fibrations (see \S\ref{sec:D4_disc} and \S\ref{sec:Kluners}). 
\end{remark}

\section{Examples} \label{sec:examples}

We now study explicit examples of groups $G$ and compare our conjectures with existing results in the literature, as well as consider some new examples. We focus on some of the more well-known results in the literature which moreover exhibit the wide range of behaviour we wish to encapsulate. A nice survey of the current state of the art in Malle's conjecture can be found in \cite{BFLV}. Unless otherwise stated, we work over a number field $k$.

\subsection{$S_n$-extensions} \label{sec:S_n}
Recall from Lemma \ref{lem:S_n} that the category of $S_n$-torsors is equivalent to the category of degree $n$ \'etale $k$-algebras, with the irreducible torsors corresponding to field extensions with Galois closure $S_n$. Bhargava \cite[Conj.~1.2]{Bha07}  put forward a conjecture on the number of such extensions of bounded discriminant. We explain here how this is a special case of our conjectures.

We take the orbifold line bundle $\Delta$ whose corresponding height is the discriminant, as in Example \ref{ex:disc}.

\begin{lemma}
	The orbifold line bundle $\Delta$ is balanced on $B S_n$.
\end{lemma}
\begin{proof}
	The conjugacy classes of minimal index are exactly the transpositions, which generate $S_n$.
\end{proof}

Therefore in this case the adjoint Iitaka fibration is trivial and the discriminant is a balanced height function, thus we are in the situation of Conjecture \ref{conj:balanced}.

\begin{lemma}
	$\dual{S}_n= \mu_2$ and $\Br_{\Delta} B S_n = \Br k$.
\end{lemma}
\begin{proof}
	There is only one non-trivial $1$-dimensional representation of $S_n$, given
	by the sign representation. This is defined over $\Q$ thus always Galois invariant.
	We have $\Br_{\Delta} BS_n = \Br_{\Delta,1} BS_n$ by 
	Lemma~\ref{lem:unramified_geometric_Brauer_group_computation}(2) and $\Br_{\Delta,1} B S_n = \Br k$ by 
	Lemma~\ref{lem:Br_1_trivial}.
\end{proof}

The minimal index conjugacy class is the transpositions with trivial Galois action, thus the convergence factors for the Tamagawa measure come from the Dedekind zeta function $\zeta_k(s)$. The leading constant \cite[(4.2)]{Bha07} in Bhargava's conjecture is thus seen to agree with Conjecture \ref{conj:balanced} on applying Lemma \ref{lem:S_n}. (Bhargava has a missing factor of $1/2$ in the case $n =2$ which is necesary to account for the centre of $S_2$; the correct factor appears in \cite[Thm.~1]{BSW}).

Bhargava also conjectures a version of equidistribution in \cite[Conj.~5.1]{Bha07}, though only for a single prime. This conjecture is a special case of our Conjecture \ref{conj:equi}. Moreover  \cite[Thm.~2]{BSW} is a version of our Conjecture \ref{conj:equi_strong}.

\begin{remark}
Understanding the factor $1/2$ which appears in Bhargava's heuristic was in fact the initial problem which started this project. Bhargava appears to introduce this factor $1/2$ into his constant in \cite[(4.2)]{Bha07} as an extra archimedean factor coming from sign considerations. In our case it is arises more naturally as the effective cone constant, since $\alpha^*(BS_n,\Delta) = 1/2$ by Lemma \ref{lem:effective_cone_calc}.
\end{remark}

\subsection{$D_4$-quartics} \label{sec:D_4}
\subsubsection{Groupoid cardinalities} \label{sec:D4-groupoid-counts}

$D_4$ admits a unique embedding into $S_4$, up to conjugation.
A $D_4$-\emph{quartic extension} of $k$ is a degree $4$ extension $K/k$ such that its Galois closure $\widetilde{K}$ has Galois group isomorphic to $D_4$ (as a permutation group). In the literature one usually counts isomorphism classes of such extensions. This means counting in the image of the map $BD_4(k) \to BS_4(k)$. To relate to our conjecture we therefore need to keep track of the correct groupoid counts; this is achieved using Lemma \ref{lem:G_S_n}. The conclusion is that we need to multiply counts which appear in the literature by a factor of $1/2$ (the normaliser of $D_4$ in $S_4$ being $D_4$ and the centraliser having order $2$, generated by the double transposition $(1,3)(2,4)$).

\begin{remark}
	A Galois $D_4$ extension $L/k$ contains two non-isomorphic $D_4$-quartics, namely $L^{(1,3)}$ and $L^{(1,2)(3,4)}$. This is due to the existence of non-trivial outer automorphisms of $D_4$, see \cite[\S 2.1]{ASVW21}.
\end{remark}

\subsubsection{Discriminant} \label{sec:D4_disc}
Malle's conjecture is known here when $k = \Q$ by \cite[Thm~1.3]{CDO02}.
The character table of $D_4$ is rational so the Galois action on the conjugacy classes is trivial.
Here is an important difference with the case of $S_n$.

\begin{lemma}
	The subgroup generated by the minimal index elements is $M = \{(), (1,3), (2,4), (1,3)(2,4)\}$.
	The adjoint Iitaka fibration of the orbifold line bundle $\Delta$ is given by $BD_4 \to B C_2$. 
	In particular $\Delta$ is not balanced.
\end{lemma}
\begin{proof}
	The non-identity conjugacy classes of $D_4$ have the following indices:
	$$
	\begin{tabular}{c|cccc}
		c& (1,3) & (1,2)(3,4) & (1,3)(2,4) & (1,2,3,4) \\ \hline
		$\ind(c)$ & 1 & 2 & 2 & 3
	\end{tabular}
	$$
	The minimal index elements are  the reflections fixing two points, which gives $M$.
\end{proof}

The discriminant is not a balanced height function in this case. The adjoint Iitaka fibration is by definition the map which sends a $D_4$-extension $L/k$ to its quadratic subfield $L^{(1,3), (2,4)}$. By the Galois correspondence this is the unique quadratic subfield of the corresponding $D_4$-quartic $L^{(1,3)}$.

We conclude from \S \ref{sec:non-balanced} that the counting function should be sorted by the quadratic subfield which each $D_4$-quartic contains. The leading constant is thus given by a sum of leading constants, as occurs in \cite[Thm~1.3]{CDO02}. To determine this sum we compute the fibres of the Iitaka fibration.
\begin{lemma}
	Let $A/k \in BC_2(k)$ be a quadratic \'etale $k$-algebra. The fibre of this point along $BD_4 \to BC_2$ is the Weil restriction $\Res_{A/k} B C_2 = B \Res_{A/k} C_2$.
	
	Let $\Delta_{A}$ be the orbifold line bundle which is the restriction of $\Delta$ to 
	$B \Res_{A/k} C_2$. Then $a(\Delta_{A}) = 1$ and $b(\Delta_{A}) = 1$ for $A \neq k \times k$.
\end{lemma}
\begin{proof}
	Let $\psi_A: \Gamma_k \to C_2$ be the character defining $A$. Note that $\psi_A$ lifts to a map $\varphi_A:\Gamma_k \to D_4$ whose image is generated by the reflection $(1,2)(3,4)$. It follows from Lemma \ref{lem:fibration_normal_quotient} that the fibre of $A/k$ of $BD_4 \to BC_2$ is $B M_A$, where $M_A := \{1, (1,3), (2,4), (1,3)(2,4)\}$ and $\psi_A$ acts by permuting $(1,3)$ and $(2,4)$. We see that $M_A \cong R_{A/k}C_2$.

	We have $M_A(-1) = M_A$ as group schemes. 
	The conjugacy classes $c$ of $M_A$ for which $\ind(c)$ are minimal are $(1,3)$ and $(2,4)$, where $\ind(c) = 1$. If $A \neq k \times k$ then these lie in the same Galois orbit. The lemma follows.
\end{proof}

For a field $A$ the groupoid $B \Res_{A/k} C_2(k)  = BC_2(A)$ classifies quadratic \'etale $A$-algebras. Let $H_A$ be the restriction of the discriminant height along the map $B \Res_{A/k} C_2 \to BD_4$
\begin{lemma}
	For $K \in B \Res_{A/k} C_2(k)$ a field $H_A(K) =  \Norm_{k}(\Delta_{A/k})^2 |\Norm_{A}(\Delta_{K/A})|$.
\end{lemma}
\begin{proof}
	The discriminant formula for an extension of \'etale algebras gives that $\Delta_{K/k} = \Delta_{A/k}^2 \Norm_{A/k}(\Delta_{K/A})$. The lemma follows by taking absolute values of the norms.
\end{proof}

Conjecture \ref{conj:non_balanced} then predicts that
\[
	\frac{1}{2} \# \{[K: \Q] = 4, |\Delta_{K}| \leq B, \Gal(\tilde{K}/\Q) \cong A_4\} \sim \frac{1}{2} \sum_{[A: \Q] = 2} c(\Q, \Res_{A/\Q} C_2, H_A) B.
\]

Let $A=\Q(\sqrt{D})$ for a fundamental discriminant $D$. We have the equality $B\Res_{\Q(\sqrt{D})/\Q} C_2(\Q) = BC_2(\Q(\sqrt{D}))$ hence an equality  $c(\Q, \Res_{\Q(\sqrt{D})/\Q} C_2, H_{\Q(\sqrt{D})/\Q}) = c(\Q(\sqrt{D}), C_2, D^{-2} \Delta)$ of leading constants. It is simple to verify here that the leading constant agrees with Conjecture \ref{conj:balanced} (we expect a more general result about Weil restrictions, c.f.~Remark \ref{rem:Weil}). But in any case the leading constant for $C_2$ is well-known, which gives us
\[
	\begin{split}
		c(\Q, \Res_{\Q(\sqrt{D})/\Q} C_2, H_{\Q(\sqrt{D})/\Q}) &= \frac{1}{2} \frac{1}{D^2} 2^{-i(D)} \zeta^*_{\Q(\sqrt{D})}(1) 
		\prod_{\mathfrak{p}}(1 - \Norm(\mathfrak{p})^{-1})(1 + \Norm(\mathfrak{p})^{-1}) \\ &= \frac{{2^{-i(D)}}}{2D^2} \zeta^*_{\Q(\sqrt{D})}(1)  \zeta_{\Q(\sqrt{D})}(2)^{-1}.
	\end{split}
\]
Here $2^{-i(D)}$ is the archimedean density where $i(D) = 0$ if $D> 0$ and $i(D) = 1$ if $D < 0$, and  $1/2$ is the effective cone constant. The Euler product is over primes of $\Q(\sqrt{D})$ and the factors $(1 - \Norm(\mathfrak{p})^{-1})$ are the convergence factors. Thus
\[
\sum_{[A: \Q] = 2} c(\Q, \Res_{A/\Q} C_2, H_A) B =  \frac{1}{2} \sum_{D} \frac{2^{-i(D)}}{D^2}\zeta^*_{\Q(\sqrt{D})}(1)  \zeta_{\Q(\sqrt{D})}(2)^{-1}.
\] 
This is exactly the same sum as in \cite[Thm~1.3]{CDO02}.

\subsubsection{An Artin conductor}\label{sec:Artin_D4}
A different height to the discriminant is considered in the paper \cite{ASVW21}. This is given by the conductor of the irreducible $2$-dimensional representation of $D_4$. Let $\chi$ be the character of this representation.

We put this conductor into the height framework on $BD_4$ from \S \ref{sec:heights}. Note that it follows directly from the definition of the Artin conductor that it is a height given by the weight function $C$ such that for $c \in \mathcal{C}_{D_4}$ we have $C(c) = \chi(1) - \chi(\langle c\rangle)$, where $\chi(\langle c\rangle)$ denotes the average value of $\chi$ on the subgroup $\langle c \rangle$ generated by an element of $c$. A computation of these values is given in the following table
$$
\begin{tabular}{c|cccc}
	c& (1,3) & (1,2)(3,4) & (1,3)(2,4) & (1,2,3,4) \\ \hline
	$C(c)$ & 1 & 1 & 2 & 2
\end{tabular}
$$
with each conjugacy class being Galois invariant, as the character table is rational.
(Note that this weight function is $\hat{\Z}^\times$-invariant, as in Remark \ref{rem:special}, hence can be defined on $A_4$ rather than $A_4(-1)$.)
Contrary to the discriminant, we have the following.

\begin{lemma}
	We have $a(C) = 1, b(C) = 2$. The orbifold line bundle $C$ is balanced.
\end{lemma}
\begin{proof}
	The minimal value of $C$ is $a(C) = 1$ and it is achieved by the $b(C) = 2$ conjugacy classes $(1,3)$ and $(1,2)(3,4)$.
	These generate $D_4$.
\end{proof}

We can now compare our leading constant from Conjecture \ref{conj:balanced} with the leading constant of \cite[Thm.~3]{ASVW21}. First the effective cone constant and Brauer group.

\begin{lemma}
	We have $\dual{D}_4(\Q) \cong \Z/2\Z \times \Z/2\Z$ and $\Br_{C} BD_4 = \Br \Q$.
\end{lemma}
\begin{proof}
	There are four $1$-dimensional representations and they are all defined over $\Q$. The structure
	is easily verified.
	We have $\Br_{C} BD_4 = \Br_{C,1} BD_4$ by 
	Lemma~\ref{lem:unramified_geometric_Brauer_group_computation}(3) and 
	$\Br_{C,1} B D_4 = \Br \Q$ by Lemma~\ref{lem:Br_1_trivial}.
\end{proof}

Recalling from \S\ref{sec:D4-groupoid-counts} that one should multiply results in the literature by $1/2$ to compare with Conjecture \ref{conj:balanced}, we obtain from  \cite[Thm.~3]{ASVW21} the correct factor of $1/4$. It remains to compare the products of local densities. Our convergence factors come from $\zeta(s)^2$, which agrees with the convergence factors from  \cite[Thm.~3]{ASVW21}. 
The local densities in their Euler products involve a sum over pairs $(L_p,K_p)$ where $L_p$ is a quartic \'etale algebra and $K_p \subset L_p$ us a quadratic subfield. Let $C_p(L_p, K_p) := \text{Disc}_p(L_p)/\text{Disc}_p(K_p)$ as in \cite[p.~2735]{ASVW21}.
The following lemma shows that these local terms agree with our prediction.
\begin{lemma}
	$$\tau_{p,C_p}(BD_4(\Q_p))  =  \sum_{(L_p,K_p)} \frac{1}{|\Aut(L_p,K_p)| \cdot C_p(L_p,K_p)}.$$
\end{lemma}
\begin{proof}
	Let $k$ be a field. We define $\mathcal{Q}$ to be
	stack over $k$ which classifies pairs $(L,K)$ where $K$ is a degree $4$ \'etale algebra 
	and $K \subset L$ is a quadratic subalgebra. There is a functor
	$BD_4(k) \to \mathcal{Q}(k)$ given by associating to a homomorphism $\Gamma_k \to D_4$
	the corresponding quartic \'etale algebra and its quadratic subalgebra coming from the
	composition with $D_4 \to C_2$. We claim that this is an equivalence of groupoids.
	
	To see this, we note that they are both neutral gerbes. Therefore it suffices to show
	that the automorphism group of the trivial point is $D_4$ in both cases. For $BD_4$
	this is clear. For $\mathcal{Q}$, the trivial point is given by
	$k \subset k^2 \subset k^4$ embedded diagonally,
	and the automorphism group corresponds to the subgroup
	of $S_4$ which preserves a partition of $\{1,2,3,4\}$ into subsets each of size $2$;
	such a group is isomorphic to $D_4$. To finish it suffices to note that 
	the local height function $C_p$ on $\mathcal{Q}$ corresponds to the
	conductor on $BD_4$ under this isomorphism by \cite[Prop.~2.4]{ASVW21}.
\end{proof}


	We have explained how \cite[Thm.~3]{ASVW21} is a special case of Conjecture~\ref{conj:balanced}, which contains the correct factor $1/2$. The reader should beware that the asymptotic in \cite[Thm.~1]{ASVW21} is missing this factor of $1/2$.

\begin{remark}
	The effective cone constant factor $\frac{1}{2} = 2 \alpha^*(B D_4, C)$ is considered in 
	\cite[Assumption 3.2]{ASVW21}, and a heuristic is given to justify it coming from multiplying $\# \Aut(L) = 2$ for a $D_4$-quartic $L$ by $\frac{1}{4}$ which comes from two global obstruction related to Stickelberger's theorem.
	
	We argue that this is not a reasonable explanation since the Brauer element which causes Stickelberger's theorem is ramified. For example Stickelberger's theorem does not cause a global obstruction if $k = \Q(i)$ since the norm of any ideal of $\Q(i)$ is $0, 1 \bmod 4$, but the $\frac{1}{4}$ factor is still expected in this case from the effective cone constant. This can also be seen from the fact that the Brauer element $b$ in the proof of Theorem \ref{thm:Stickelberger} becomes trivial after base change to $\Q(i)$.
\end{remark}

\subsection{Kl\"{u}ners's counter-example} \label{sec:Kluners}
Let $G := C_3 \wr C_2$ be the wreath product of $C_3$ and $C_2$, i.e.~$G$ is the semi-direct product $(C_3 \times C_3) \rtimes C_2$ where $C_2$ acts on $C_3 \times C_3$ by permuting the two factors of $C_3$. This embeds into $S_6$ via $C_3 \times C_3 \subset S_3 \times S_3 \subset S_6$ and $C_2 = \langle (1,4)(2,5)(3,6) \rangle$. Let $\Delta$ be the orbifold line bundle corresponding to the discriminant height.

The minimal index $\ind(g)$ of $1 \neq g \in G$ is $2$ and there are two conjugacy classes $c$ with $\ind(c) = 2$, namely $(1,2,3)$ and $(1,3,2)$. These are Galois conjugate, so $b(D) = 1$. Thus Malle's conjecture predicts order of magnitude $B^{\frac{1}{2}}$.

The group generated by these two conjugacy classes is $C_3 \times C_3$ so the discriminant is not balanced and the Iitaka fibration is $BG \to B C_2$.

Kl\"{u}ners's counter-example comes from the fibre of the Iitaka fibration of the point $\Q(\zeta_3)/\Q \in B C_2(\Q)$. We can lift $C_2$ to $G$ and conjugating by this lift acts by permuting the two factors of $C_3$, so a computation using Lemma \ref{lem:fibration_normal_quotient} shows that this fibre is given by the Weil restriction $R_{\Q(\zeta_3)/\Q} B C_3 = B R_{\Q(\zeta_3)/\Q} C_3$. The restriction of $\Delta$ to this Weil restriction has $b(\Q, \Delta) = 2 > 1$.  So the number of points in this fibre has order of magnitude $B^{\frac{1}{2}} \log B$. These come from breaking cocycles, and by Theorem~\ref{thm:breaking} they are removed when counting, since the corresponding fields are not linearly disjoint with $\Q(\mu_3)$. Thus Kl\"{u}ners's counter-example to Malle's conjecture is compatible with Conjecture \ref{conj:non_balanced}.

\subsection{A new dihedral counter-example} \label{sec:D_n}
In \cite{KP23}, Koymans and Pagano consider the expected generalisation of Malle's conjecture to counting by radical discriminant, as first put forward by Ellenberg and Venkatesh \cite[Ques.~4.3]{EV05}. They obtain in \cite[Thm.~1.3]{KP23} a counter-example to the log-factor for certain nilpotent groups.

We obtain a new counter-example coming instead from the dihedral group over $\Q$.
Let $n \geq 3$ be odd and consider the radical discriminant height $H$ on $D_n \subset S_n$.

The non-identity conjugacy classes of $D_n$ are as follows: there is one corresponding to reflections and there are $\frac{n-1}{2}$ corresponding to rotations. Two rotations lie in the same $\Gal(\Q(\mu_n)/\Q) \cong (\Z/n \Z)^{\times}$ orbit via Definition \ref{def:G(cycl)} if one is equal to the other after multiplication by an element of $(\Z/n \Z)^{\times}$.

There is thus one $\Q$-conjugacy class of rotations for each divisor $1 \neq d \mid n$ and $b(\Q,D_n) = \tau(n)$ is the number of divisors of $n$.

Let $d \mid n$ be square-free. Let $K_d := \Q(\sqrt{(-1)^{d}d}$ be the quadratic subfield of $\Q(\mu_n)$ corresponding to the Dirichlet character $\left(\frac{\cdot}{d}\right): (\Z/n \Z)^{\times} \to \{-1,1\}$; we identify this with the corresponding Galois character $\chi_d: \Gamma_{\Q} \to \{-1,1\}$.

The collection of $D_n$-extensions which contain $K_d$ is given by the fibre of $\chi_d$ for the map $BD_n(\Q) \to BC_2(\Q)$. The character $\chi_d$ lifts to $D_n$ by sending $-1$ to a reflection. Conjugation by a reflection sends a rotation to its inverse so Lemma~\ref{lem:fibration_normal_quotient} shows that the fibre is equivalent to $B G_d(\Q)$ where $G_d$ is the group scheme with underlying group $C_n$ and on which $\sigma \in \Gamma_{\Q}$ acts via the formula $\sigma(g) = g^{\chi_d(\sigma)}$.

\begin{lemma}
	Let $d \mid n$ be such that $d \equiv 3 \bmod{4}$. Then  $$\#\{ \varphi \in BD_d(\Q) : \varphi \in \im(BG_d(\Q) \to BD_n(\Q)), H(\varphi) \leq B\}
	\gg B (\log B)^{\tau(n) + \tau(n/d) - 2}.$$
\end{lemma}
\begin{proof}
An element $r \in \Gal(\Q(\mu_n)/\Q) \cong (\Z/n \Z)^{\times}$ acts on $\gamma \in G_d(-1)(\overline{\Q})$ by $\gamma \mapsto \gamma^{\chi_d(r) r}$. 
Consider the map $(\Z/n \Z)^{\times} \to (\Z/n \Z)^{\times}:r \mapsto \chi_d(r) r$. As $d \equiv 3 \bmod 4$ we have $\chi_d(-1) = -1$, so the kernel has order $2$, hence the image has index $2$. 
Thus there are two Galois orbits in $G_d(-1)(\overline{\Q})$ of elements of order $n$.
 
More generally, let $d \mid m \mid n$ and consider the map $(\Z/n \Z)^{\times} \to (\Z/m \Z)^{\times}:r \mapsto \chi_d(r) r \bmod{m}$. The class $\chi_d(r) \bmod{m}$ only depends on $r \bmod{m}$ as $d \mid m$ and the image again has index $2$. 
Thus there are two  Galois orbits in $G_d(-1)(\overline{\Q})$ of elements of order $m$.

The map $(\Z/n \Z)^{\times} \to (\Z/m \Z)^{\times}:r \mapsto \chi_d(r) r \bmod{m}$ is surjective for $d \nmid m \mid n$ because the function $\chi_d(r)$ attains both $1$ and $-1$ on each fibre. This implies that $G_d(-1)$ only has one Galois orbit of elements of order $m$ in this case.

The group scheme $G_d$ is abelian so we have 
$$b(\Q,G_d) = \sum_{\substack{m \mid n \\ d \mid m}} 2 + \sum_{\substack{1 \neq m \mid n \\ d \nmid m}} 1 = \sum_{\substack{m \mid n \\ d \mid m}} 1 + \sum_{1 \neq m \mid n} 1 = \tau(n/d) + (\tau(n)  - 1).$$ 
To count the number of such cocycles we apply Darda--Yasuda \cite[Thm.~1.3.2]{DYTor}.
\end{proof}

In the case $n = 3$ we obtain $S_3$, and one is counting non-cyclic cubic fields of bounded radical discriminant. The collection of points in the image of $BG_3$ is weakly accumulating: it needs to be removed before equidistribution is obtained. These correspond to exactly the pure cubic extensions; cf.~Example~\ref{ex:S_3_breaking}. In fact this has been proven by Shankar and Thorne \cite[Thm.~4]{ShTh22}.

For $n$ composite with a prime divisor $p \equiv 3 \bmod{4}$, the points in the image of $BG_p$ are strongly accumulating, and need to be removed to obtain the correct exponent of $\log B$ in Conjecture \ref{conj:balanced}.

\subsection{$A_4$-quartics of bounded discriminant} \label{sec:A_4-quartics}

In the following sections we perform a detailed study of the problem of counting $A_4$-quartic fields of bounded height. This is to make clear that, despite the abstract nature of Conjecture \ref{conj:balanced}, in practice the terms which appear can be made completely explicit and lead to explicit and elementary predictions for the counting problem.

Counting $A_4$-quartics of bounded discriminant is a notorious open problem in the Malle's conjecture literature. We make completely explicit what our conjecture says in this case as a challenge to researchers in the community. Here we recall Conjecture \ref{conj:A_4_intro} from the introduction.

\begin{conjecture} \label{conj:A_4}
	$$2\#\left\{ [K:\Q] = 4 : 
	|\Delta_K|  \leq B, \Gal(\widetilde{K}/\Q) \cong A_4 \right\} \sim c(\Q,A_4,\Delta) B^{1/2} \log B,$$
	where $\widetilde{K}$ denotes the Galois closure of $K$ and
	$$c(\Q,A_4,\Delta) = \frac{35}{648}\prod_{p > 3} \left(1 - \frac{1}{p}\right)^2\left(1 + \frac{2 + \left(\frac{-3}{p}\right)}{p}\right).$$
\end{conjecture}

There is in fact already a conjecture regarding this problem in \cite[\S2.7]{CDO02b}, but the expression is more complicated than ours and involves a limit process and sorting $A_4$-quartics by their cubic resolvent. It is not immediate how to compare the two conjectures directly. Nonetheless, in \cite[\S2.7]{CDO02b} the numerical value $0.074\ldots$ is given for the leading constant, and a simple computation shows that our constant satisfies $2c(\Q,A_4,\Delta) = 0.0729\ldots$ after considering the product over $3 < p < 100,000$. (In private communication, Henri Cohen has confirmed that this minor discrepancy should not a problem and an artifact of using a least squares approximation in their calculation). The convergence is very slow, as one would expect, since it involves a conditionally convergent product; a faster converging product can be obtained by using the convergence factors from Theorem \ref{thm:Tamagawa_products}. We count isomorphism classes of fields, whereas \cite{CDO02b} counts subfields of $\bar{\Q}$. To compare the two counts, we note that there are $4$ distinct embeddings of any $A_4$-quartic into $\bar{\Q}$, which matches with our factor $2$ in front and additional factor $1/2$ in the leading constant.

\subsubsection{Deduction of Conjecture \ref{conj:A_4}}
Let us explain how Conjecture \ref{conj:A_4} follows from Conjecture \ref{conj:discriminant}. Firstly the normaliser of $A_4$ in $S_4$ is $S_4$ and the centraliser of $A_4$ in $S_4$ is trivial. This gives the factor of $|S_4|/|A_4| = 2$ outside the front of the counting function.

The non-identity conjugacy classes of $A_4$ have the following indices:
$$
\begin{tabular}{c|cccc}
	c& (1,2)(3,4) & (1,2,3) & (1,3,2) \\ \hline
	$\ind(c)$ & 2 & 2 & 2 
\end{tabular}
$$
with the latter two classes being swapped by $\Gal(\Q(\zeta_3)/\Q)$. 
It follows that $a(\Delta) = 1/2$ and $b(\Q,\Delta) = 2$ (taking into account the Galois action). Since all conjugacy classes have minimal index, it follows that together they generate the group and hence the Iitaka fibration is trivial. Moreover the discriminant is the square of an anticanonical height; this may explain why this problem is so difficult, since the anticanonical height is generally the most difficult height in Manin's conjecture.

We have $\dual{A}_4 = \mu_3$, so there are no non-trivial $1$-dimensional characters defined over $\Q$. The effective cone constant is thus $a(\Delta)^{b(\Q,\Delta)-1} = 1/2$. Since our height is a power of an anticanonical height, the relevant Brauer group is the unramified Brauer group.

\begin{lemma}
	$\Brun B A_4 = \Br \Q$.
\end{lemma}
\begin{proof}
	We have $\Brun BA_4 = \Br_{\mathrm{un},1} BA_4$ by 
	Lemma~\ref{lem:unramified_geometric_Brauer_group_computation}(4).
	The abelianisation map $A_4 \to C_3$ admits a section
	given by the subgroup generated by $(1,2,3)$, hence 
	$\Br_{\mathrm{un},1} B A_4 = \Br_{\mathrm{un},1} B C_3$ by 
	Lemma \ref{lem:Brun_ab}. However $\Br_{\mathrm{un},1} B C_3 = \Br \Q$
	by Lemma~\ref{lem:BrBG-Abelian} since $\Sha^1_\omega(\Q,\mu_3) = 0$ (Grunwald--Wang).
\end{proof}

It thus suffices to calculate the Tamagawa numbers. We use the convergence factors from Lemma \ref{lem:convergence_factors_alternative}, which come from the $\L$-function $\zeta(s)^2$.

\begin{lemma}\label{lem:A4_local_densities}
	$$\tau_p(A_4(\Q_p)) = 
	\begin{cases}
		1/3, & p = \infty, \\
		15/8, & p = 2, \\
		14/9, & p = 3, \\
		1 + \frac{3}{p}, & p \equiv 1 \bmod 3, \\
		1 + \frac{1}{p}, & p \equiv 2 \bmod 3.
	\end{cases} $$
\end{lemma}
\begin{proof}
	The cases $p\neq 2,3$ are tame so the result follows from our mass formula
	Corollary \ref{cor:mass_formula}. 
	For $p = \infty$, we use \eqref{eqn:real_mass_formula} and the fact that $\#A_4[2] = 4$.
	
	For $p=2,3$ we take a computational approach. We have by definition
	$$\tau_p(BA_4(\Q_p)) = \sum_{\varphi_p \in BA_4[\Q_p]} 
	\frac{1}{|\Aut(\varphi_p)| H_v(\varphi_p)^{1/2}}.$$
	We pushforward the groupoid count to $BS_4$.
	By Lemma~\ref{lem:S_n}, this allows us to write the count in terms of quartic \'etale 
	algebras. We then use Lemma \ref{lem:outer_automorphism} and the fact that $A_4 \subset S_4$
	is normal to keep track of the correct normalisations for the groupoid cardinalities.
	This yields
	$$\tau_p(BA_4(\Q_p)) = 2\sum_{\substack{ [K:\Q_p] = 4 \\ \Gal(\widetilde{K}/\Q_p) \subset A_4}}
	\frac{1}{|\Aut(K)| p^{v(\Delta(K))/2}}$$
	where the sum is over isomorphism classes of quartic \'etale algebras $K$ over $\Q_p$
	whose Galois closure has Galois group a subgroup of $A_4$.
	We have the following possibilities with stated order of the automorphism
	group.
	\begin{center}
	\begin{tabular}{|l|l|} \hline
		\'etale algebra & $\# \Aut$ \\ \hline
		$\Q_p^4$ & $24$ \\
		$L \times L$, $L/ \Q_p$ quadratic & 8 \\ 
		$\Q_p \times L$, $L/\Q_p$ cyclic cubic & 3 \\
		$L$, $L/\Q_p$ biquadratic & 4 \\
		$L$, $A_4$-quartic field & 1 \\ \hline
	\end{tabular}
	\end{center}
	Using the LMFDB and Sage, we enumerate all such \'etale algebras over 
	$\Q_2$ and $\Q_3$ weighted by the 
	reciprocal discriminant and the automorphism group.
	This gives the stated formulas.
\end{proof}

To obtain the formula stated in Conjecture \ref{conj:A_4} we use the local densities at $p =\infty,2,3$ as well the convergence factors $(1-1/2)^2(1 - 1/3)^{2} = 1/9$, and the effective cone constant $1/2$. This gives $(1/2)\times(1/3)\times(15/8)\times(14/9)\times(1/9)=35/648$.

Finally Conjecture \ref{conj:discriminant} states that we should remove $\widetilde{K}$ for which $\widetilde{K} \cap \Q(\mu_{6}) \neq \Q$. However there are no such fields; indeed we have $\Q(\mu_{6}) = \Q(\sqrt{-3})$, but $\widetilde{K}$ cannot contain a quadratic subfield as $A_4$ has no subgroup of index $2$.

\subsection{$A_4$-quartics with a transcendental Brauer group} \label{sec:A_4-transcendental}
We now consider a different height function which exhibits new behaviour. Namely, there exists a balanced height function on $A_4$ for which the partially unramified transcendental Brauer group is non-trivial.

Recall that $A_4$ has $4$ irreducible representations. The trivial one, two non-trivial $1$-dimensional representations which factor through $A_4^{\text{ab}} = C_3$, and a $3$-dimensional representation. Let $1, \chi, \overline{\chi}, \psi$ be the respective characters of these representations.
\begin{lemma}
	Let $K/\Q$ be an $A_4$-quartic, $\tilde{K}/\Q$ its Galois closure and $L/\Q$ the cubic resolvent. Let $H(K)$ be the Artin conductor of $\tilde{K}/\Q$ with respect to the virtual character $\psi - (\chi + \overline{\chi})/2$. We then have 
	\[
	H(K) = |\Delta_{K}|\cdot |\Delta_{L}|^{-\frac{1}{2}} = |\Delta_{K}|^{\frac{5}{2}} \cdot|\Delta_{\tilde{K}}|^{-\frac{1}{2}}.
	\]
\end{lemma}
\begin{proof}
	The Artin conductor of the trivial character $1$ is trivial. 	The Artin conductor of a sum is the product of Artin characters. The lemma then follows from the following, which all follow from the conductor-discriminant formula.
	\begin{enumerate}
		\item $|\Delta_{K}|$ equals the Artin conductor of the permutation representation corresponding to $A_4 \subset S_4$. The character of this representation is $1 + \psi$.
		\item $|\Delta_{L}|$ equals the Artin conductor of the permutation representation given by  $A_4 \to C_3 \subset S_3$. The character of this representation is $1 + \chi + \overline{\chi}$.
		\item $|\Delta_{\tilde{K}}|$ equals the Artin conductor of the representation $\C[A_4]$. The character of this representation is $1 + \chi + \overline{\chi} + 3\psi$. \qedhere
	\end{enumerate}
\end{proof}

Let $H$ be the height on $BA_4$ given by one of the above equivalent formulas. We recall Conjecture \ref{conj:A_4_conductor_intro}.
\begin{conjecture} \label{conj:A_4_conductor}
	For an $A_4$-quartic field $K$, let $H(K) := |\Delta_{K}|^{\frac{5}{2}}|\Delta_{\tilde{K}}|^{-\frac{1}{2}}.$ Then
	\[
	\begin{split}
		&2\#\left\{ [K:\Q] = 4 :\Gal(\widetilde{K}/\Q) \cong A_4, H(K) \leq B, K \otimes_{\Q} \R \cong \R^4 \right\} \sim c_{\R^4}(\Q,A_4,H) B, \\
		&2\#\left\{ [K:\Q] = 4 :\Gal(\widetilde{K}/\Q) \cong A_4, H(K) \leq B, K \otimes_{\Q} \R \cong \C^2 \right\} \sim c_{\C^2}(\Q,A_4,H) B,
	\end{split}
	\]
	where
	\[
	\begin{split}
	c_{\R^4}(\Q,A_4,H) = &\frac{145}{3456}\prod_{p > 3} \left(1 - \frac{1}{p}\right)\left(1 + \frac{1 + \left(\frac{-3}{p}\right)}{p} + \frac{1}{p^2}\right) \\ + \,\, &\frac{319}{10368}\prod_{p > 3} \left(1 - \frac{1}{p}\right)\left(1 + \frac{1 + \left(\frac{-3}{p}\right)}{p}\right) =  0.0594...\\
	c_{\C^2}(\Q,A_4,H) = &\frac{145}{1152}\prod_{p > 3} \left(1 - \frac{1}{p}\right)\left(1 + \frac{1 + \left(\frac{-3}{p}\right)}{p} + \frac{1}{p^2}\right) \\ - \,\, &\frac{319}{3456}\prod_{p > 3} \left(1 - \frac{1}{p}\right)\left(1 + \frac{1 + \left(\frac{-3}{p}\right)}{p}\right) = 0.0347....
	\end{split}
	\]
\end{conjecture} 

The numerical values were computed by using the primes $p < 100,000$.
An interesting prediction of the conjecture is that $63\%$ of $A_4$-quartics are totally real when ordered by $H$. On the other hand, when ordering by discriminant Conjecture~\ref{conj:equi} (and experimental verification) predicts that only $25 \%$ of $A_4$-quartics are totally real. This difference in the local behaviour when using a different height function is explained by the existence of a partially unramified Brauer element.

\subsubsection{Numerical experiments}
We test Conjecture \ref{conj:A_4_conductor} experimentally by counting the number of $A_4$-quartics $K$ in the LMFDB such that $H(K) \leq 100,000$ and dividing this number by $100,000$. We also consider the value of the first Euler product, these are the values predicted by applying the Malle--Bhargava heuristics to our setting, equivalently the value predicted by Conjecture \ref{conj:equi} if the Brauer group were trivial. The results of this test are as follows.
\begin{center}
\begin{tabular}{c|c|c|c}
				& $c_{\R^4}(\Q,A_4,H)$	& $c_{\C^2}(\Q,A_4,H)$  &$\%$ totally real \\
	Conjecture \ref{conj:A_4_conductor}  & 0.0595... 			& 0.0348...				&$63\%$			\\
	LMFDB		& 0.0463... 			& 0.0298...				&$61\%$			\\
	Malle--Bhargava& 0.0355...			& 0.1066...         	&$25\%$
\end{tabular}
\end{center}
We make two remarks about this numerical test. Firstly, if Conjecture \ref{conj:A_4_conductor} had logarithmic factors then we would expect the relative error term to be $O(1/\log B)$ and so it would not be surprising if the error term dominated the main term for $B = 100,000$. But Conjecture \ref{conj:A_4_conductor} contains no logarithmic factors so we should expect a power-saving error term and the numerical value should be relatively close to the true value. Assuming a strong error term of the shape $O(\sqrt{B})$, we would therefore expect the data to be correct to within $10^{-5/2} \approx 0.003.$

A second remark is that the LMFDB does not necessarily contain all quartic $A_4$ fields $K$ with $H(K) \leq 100,000$ (it contains $3,808$ such fields). One should therefore expect that the numerical value computed from the LMFDB is smaller than the actual proportion. Moreover, one should expect the LMFDB to be biased against totally real extensions since only $27 \%$ of the $A_4$-quartic fields in the LMFDB are totally real. 

Taking these two remarks into account we see that the numerical data agrees quite well with our conjecture.
On the other hand, the prediction from the Malle--Bhargava heuristics is rather different to the numerical data. This provides strong evidence that the transcendental Brauer group plays a role in the leading constant.

\subsubsection{Deduction of Conjecture \ref{conj:A_4_conductor}}
Let us now explain how Conjecture \ref{conj:A_4_conductor} follows from Conjecture \ref{conj:equi}. The same reasoning as in \S\ref{sec:Artin_D4} shows that the height $H$ is given by the weight function $C$ where
\[
	\begin{tabular}{c|cccc}
		c& (1,2)(3,4) & (1,2,3) & (1,3,2) \\ \hline
		C(c) & 2 & 1 & 1 
	\end{tabular}
\]
Let $\mathcal{C}$ consist of the minimal conjugacy classes $(1,2,3), (1,3,2)$. Note that we have $a(C) = 1$ and $b(\Q, C) = 1$. The height is balanced since $(1,2,3)$ and its conjugate $(1,2,4)$ generate $A_4$. 

We now construct the relevant Brauer group element using the construction from \S \ref{sec:central_extension}.
Consider the double cover $\SL(2,\F_3) \to A_4$. This is a central extension and thus defines a Brauer element which we will denote by $\beta \in \Br BA_4$.
\begin{lemma}\label{lem:A4_conductor_Brauer_class}
	We have $\beta \in \Br_{\mathcal{C}} BA_4$ and $\beta_{\overline{\Q}} \neq 0$.
\end{lemma}
\begin{proof}
For the first statement note that the sector $\mathcal{S}_{\mathcal{C}}$ corresponding to $\mathcal{C}$ has order $3$ so $\res_{\mathcal{S}_{\mathcal{C}}}(\beta) \in \H^1(\mathcal{S}_{\mathcal{C}}, \Z/3\Z)$ is $3$-torsion. But $\beta$ is $2$-torsion so $\res_{\mathcal{S}_{\mathcal{C}}}(\beta) = 0$.

For the second part note that $\Pic (BA_4)_{\overline{\Q}} = \dual{A}_4 \cong \mu_3$ so the long exact sequence in cohomology induced by the exact sequence $0 \to \mu_2 \to \Gm \to \Gm \to 0$ shows that $\H^2((BA_4)_{\overline{\Q}}, \Z/2\Z) \cong \Br (BA_4)_{\overline{\Q}}[2]$.
It thus suffices to show that the $\Z/2\Z$-gerbe $(B\SL(2,\F_3))_{\overline{\Q}} \to (BA_4)_{\overline{\Q}}$ is not neutral. This is true because the map $\SL(2,\F_3) \to A_4$ of groups has no section.
\end{proof}

\begin{lemma}
	We have $\Br_{\mathcal{C}} BA_4/ \Br \Q \cong \Z/2 \Z$  generated by $\beta$.
\end{lemma}
\begin{proof}
	We have $\dual{A}_4 = \mu_3$. The Galois module $\PicOrb_{\mathcal{C}} BA_4$ is by definition the submodule of elements $w \in \Hom(\mathcal{C}, \Q)$ such that $w((1,2,3)) + w((1,3,2)) \in \Z$. Consider the two elements $w_1, w_2 \in \PicOrb_{\mathcal{C}} BA_4$ defined as
	\begin{align*}
		&w_1((1,2,3)) = \frac{1}{3} & w_1((1,3,2)) = \frac{2}{3}; & \quad w_2((1,2,3)) = \frac{2}{3} & w_2((1,3,2)) = \frac{1}{3}.
	\end{align*}
	The two elements $w_1,w_2$ form a basis for $\PicOrb_{\mathcal{C}} BA_4$ and $\Gamma_{\Q}$ acts by permuting them through the quotient $\Gal(\Q(\zeta_3)/\Q)$. It thus follows from Shapiro's lemma that $\H^1(\Q, \PicOrb_{\mathcal{C}} BA_4) = 0$. We deduce from Theorem \ref{thm:Br_BG} that $\Br_{\mathcal{C}, 1} BA_4/ \Br \Q = 0$.
	
	It follows that $\Br_{\mathcal{C}} BA_4/ \Br \Q$ embeds into $\Br_{\mathcal{C}} (BA_4)_{\overline{\Q}}$. This group is a subgroup of $\Br (BA_4)_{\overline{\Q}} \cong \H^2(A_4, \Q/\Z) \cong \Z/2 \Z$ by Lemma \ref{lem:coh_sep_closed} and \cite[Thm.~2.12.5]{Kar87}. Lemma~\ref{lem:A4_conductor_Brauer_class} shows that the embedding has to be surjective.
\end{proof}

We can now compute local densities. We use the notation from Lemma \ref{lem:sum_Euler_products} as well as the fact $BA_4[\R] = \{\R^4, \C^2\}$.
\begin{lemma}\label{lem:A4_local_densities_conductor}
	At $\infty$ we have the local densities 
	$$\tau_{H, \infty}(\R^4) = \frac{1}{12},\,\, \tau_{H, \infty}(\C^2) = \frac{1}{4}, \,\,\hat{\tau}_{H, \infty}(\beta)(\R^4) = \frac{1}{12}, \,\, \hat{\tau}_{H, \infty}(\beta)(\C^2) = -\frac{1}{4}.$$	
	At the primes we have the local densities
	\begin{align*}
		\tau_{H, p}(BA_4(\Q_p)) &= 
		\begin{cases}
			145/128 & p = 2, \\
			4/3 & p = 3, \\
			1 + \frac{2}{p} + \frac{1}{p^2}, & p \equiv 1 \bmod 3, \\
			1 + \frac{1}{p^2} , & p \equiv 2 \bmod 3.
		\end{cases} \\
		\hat{\tau}_{H, p}(\beta)(BA_4(\Q_p)) &= 
		\begin{cases}
			29/32 & p = 2, \\
			11/9 & p = 3, \\
			1 + \frac{2}{p}, & p \equiv 1 \bmod 3, \\
			1, & p \equiv 2 \bmod 3.
		\end{cases}
	\end{align*}
\end{lemma}
\begin{proof}
	We first deal with the large primes. The fact that $\beta \in \Br_{\mathcal{C}} BA_4$ implies that $\res_{\mathcal{C}}(\beta) = 0$. Let $\mathcal{S}$ be the sector corresponding to the conjugacy class of $(1,2)(3,4)$. In Lemma~\ref{lem:unramified_geometric_Brauer_group_computation}(4) it has been shown that $\Br_{\text{un}} \Br (BA_4)_{\overline{\Q}} = 0$. Since $\beta_{\overline{\Q}} \neq 0$ this implies that $\res_{\mathcal{S}}(\beta)_{\overline{\Q}} \neq 0$. The class is defined over $\Z[\frac{1}{6}]$ and the only bad places for $H$ are $2,3$ so the cases $p \neq 2,3, \infty$ follow from Corollary \ref{cor:mass_formula} and Theorem \ref{thm:local_invariant_integral}.
	
	The computation of $\tau_{H, \infty}(\R^4)$ and $\tau_{H, \infty}(\C^2)$ follows from Definition \ref{def:Tamagawa}, recalling that we are considering them in $BA_4$ and not $BS_4$. For  $\tau_{H, p}(BA_4(\Q_p))$ and $p=2,3$ the calculation is exactly as in Lemma \ref{lem:A4_local_densities}, namely from the LMFDB.

	It suffices to calculate $\hat{\tau}_{H, p}(\beta)(BA_4(\Q_p))$ for $p \in \{2,3,\infty\}$. Let $\varphi_p \in BA_4(\Q_p)$. We will describe in the rest of this proof how one determines the value $\text{inv}_p(\beta(\varphi_p))$ in terms of the image of $\varphi_p$ in $BS_4(\Q_p)$. Given this description one can compute $\hat{\tau}_{H, \infty}(\beta)(\C^2), \hat{\tau}_{H, \infty}(\beta)(\R^4)$ and $\hat{\tau}_{H, p}(\beta)(BA_4(\Q_p))$ in a completely analogous way to Lemma \ref{lem:A4_local_densities}.
	
	The local invariant $\text{inv}_p(\beta(\varphi_p)) \in \Z/2\Z \subset \Q/\Z$ is either trivial or equal to $\frac{1}{2}$. It thus suffices to determine when it is trivial. Since $\text{inv}_p$ is injective it is trivial if and only if $\beta(\varphi_p)$ is trivial.  We achieve this using the procedure described in \S \ref{sec:BM_central_extension}. We use the map  $\pi:\SL(2,\F_3) \to A_4$ defining $\beta$.
	
	The class $\beta(\varphi_p) \in \Br \Q_p[2] \cong \H^2(\Q_p, \Z/2\Z)$ is represented by the $\Z/2\Z$-gerbe given by the fibre product of the diagram $\Spec \Q_p \xrightarrow{\varphi_p} BA_4 \leftarrow B\SL(2, \F_3)$. The class $\beta(\varphi_p)$ is thus trivial if and only if this gerbe is neutral, i.e.~if $\varphi_p: \Gamma_{\Q_p} \to A_4$ lifts to a morphism $\tilde{\varphi}_p: \Gamma_{\Q_p} \to \SL(2, \F_3)$.
	
	Consider the image $H := \im(\varphi_p) \subset A_4$ and $\tilde{H} := \pi^{-1}(H)$. The image $\im(\tilde{\varphi}_p)$ has to be contained in $\tilde{H}$. We now consider two cases: first $\tilde{H} \cong H \times \Z/2 \Z$, where a lift $\tilde{\varphi}_p$ clearly exists, or secondly the central extension $1 \to \Z/2 \Z \to \tilde{H} \to H \to 1$ is non-split, where $\im(\tilde{\varphi}_p) = \tilde{H}$. In the latter case we can use the LMFDB to enumerate all Galois $\tilde{H}$ extensions of $\Q_p$ and check explicitly if one of them is a lift of $\varphi_p$ (the LMFDB contains all $\tilde{H}$ extensions as $\SL(2, \F_3) \subset S_8$).
	
	All possibilities are enumerated in the following table.
	\begin{center}
		\begin{tabular}{|l|l|l|l|} \hline
			\'etale algebra & $H$ & $\tilde{H}$ & $\tilde{H} \cong H \times \Z/2 \Z$? \\ \hline
			$\Q_p^4$ & $1$ & $\Z/2\Z$ & Yes\\
			$L \times L$, $L/ \Q_p$ quadratic & $\Z/2\Z$ & $\Z/4\Z$ & No \\ 
			$\Q_p \times L$, $L/\Q_p$ cyclic cubic & $\Z/3\Z$ & $\Z/6\Z$ & Yes \\
			$L$, $L/\Q_p$ biquadratic & $(\Z/2\Z)^2$ & $Q_8$ & No  \\
			$L$, $A_4$-quartic field & $A_4$ & $\SL(2, \F_3)$ & No \\ \hline
		\end{tabular}
	\end{center}
	For example, we find that $\text{inv}_{\infty}\beta(\R^4) = 0$ and $\text{inv}_{\infty} \beta(\C^2) = \frac{1}{2}$.
\end{proof}
To obtain the formula stated in Conjecture \ref{conj:A_4_conductor} we put these local densities into the sum of Euler products using Proposition \ref{prop:Malle-Bhargava}(1). The convergence factors at $2,3$ are $(1 - 1/2)(1 - 1/3) = 1/3$ and the effective cone constant is $1$. As explained at the end of \S \ref{sec:A_4-quartics} there are no accumulating fields to remove.

\subsection{Abelian groups}\label{sec:abelian}

Let now $G$ be a tame finite \'etale abelian group scheme. For $G$ constant, counting by discriminant was obtained in \cite{Mak85,Wri89} (for $k = \Q$ and $k$ general, respectively), and by conductor in \cite{Mak93,Woo10} (for $k = \Q$ and $k$ general, respectively). Moreover Wood's work \cite{Woo10} obtained asymptotics for a suitable class of balanced height functions. Non-constant $G$ have been considered in \cite{Alb21, OA21, DYTor}.

\subsubsection{Brauer group of $BG$}

We begin with $\Br BG$ with $k$ a general field. We recall the following notation. Let $K/k$ be a choice of splitting field for $G$. Then we define
\[
\Sha^1_\omega(k,G) = \ker\left(\H^1(\Gal(K/k),G) \to \prod_{g \in \Gal(K/k)} \H^1(\langle g \rangle, G)\right).
\]
This group is finite and is seen to be independent of the choice of $K$, since $\Sha_\omega^1(K,G) = 0$. Moreover if $k$ is a global field, then the Chebotarev density theorem implies that
$$\Sha^1_\omega(k,G) =\left\{ c \in \H^1(\Gal(K/k),G) :
\begin{array}{l} c_v = 0 \in \H^1(k_v, G) \\
 \text{for all but finitely many } v
 \end{array} \right\}.$$
For these facts see \cite[Lem.~1.1, 1.2]{San81}. We use the map from Lemma \ref{lem:H1_Pic}.

\begin{lemma} \label{lem:BrBG-Abelian}
	Let $\mathcal{C} \subseteq \mathcal{C}_G^*$ generate $G$. Then 
	$\Br_{\mathcal{C}} BG \subseteq \Br_1 BG$. Moreover the image of
	$\Brun BG$ in $\H^1(k,\dual{G})$ equals $\Sha^1_\omega(k,\dual{G})$.
\end{lemma}
\begin{proof}
	The first statement follows from Lemma \ref{lem:unramified_geometric_Brauer_group_computation}(1).
	For the second statement, observe that $\mathcal{C}_G = G(-1)$ as $G$ is abelian and that $\dual{G} = \Hom(G(-1), \Q/\Z)$. Let $\psi: \Gamma_k \to \dual{G} = \Hom(G(-1), \Q/\Z)$ be a cocycle. We will use Theorem \ref{thm:Br_BG} in terms of the age residue to determine when $\psi$ is unramified.
	
	Let $\gamma \in  G(-1)(k^{\sep})$ and $k(\gamma)$ its field of definition, i.e.~$\Gamma_{k(\gamma)} \subset \Gamma_k$ consists of all $\sigma$ such that $\sigma(\gamma) = \gamma$. By Lemma \ref{lem:age_residue_restriction} we have that $\res_{\gamma}(\psi) \in \H^1(\Gamma_{k(\gamma)}, \Q/\Z)$ is given by the cocycle $\sigma \to \psi(\sigma)(\gamma)$.
	
	Applying this description shows that $\psi$ is unramified if and only if for all $\sigma \in \Gamma_k$ and $\gamma \in  G(-1)(k^{\sep})$ with $\sigma(\gamma) = \gamma$ we have $\psi(\sigma)(\gamma) = 0$.
	
	Let $K/k$ be an extension which splits $G(-1)$ and such that $\psi$ factors through $\Gal(K/k)$. The group cohomology of cyclic groups is well-known, in particular the restriction $\psi|_{\langle \sigma \rangle} \in H^1(\langle \sigma \rangle, \dual{G})$ is trivial if and only if it is trivial in the coinvariant group $\dual{G}(K)_{{\langle \sigma \rangle}}$. As coinvariants are dual to invariants this is equivalent to $\psi$ being trivial in $\Hom(G(-1)(K)^{\langle \sigma \rangle}, \Q/\Z)$. More explicitly, the restriction of $\psi$ to $\langle \sigma \rangle$ is trivial if and only if for all $\gamma \in G(-1)(K)$ with $\sigma(\gamma) = \gamma$ we have $\psi(\sigma)(\gamma) = 0$. 
	
	Ranging over $\sigma \in \Gal(K/k)$ we see $\psi \in \Brun BG$ if and only if $\psi \in \Sha^1_\omega(k,\dual{G})$.
\end{proof}

\subsubsection{Conductor} \label{sec:conductor}
We now  assume that $G$ is constant and that $k$ is a number field. The conductor is an orbifold anticanonical height function. It agrees with the radical discriminant at the tame places. We use the expression for the leading constant obtained in \cite[Thm.~3.22]{FLN22}. This reads:
\begin{align*}
	c_{k,G} =& \frac{(\mathrm{Res}_{s=1} \zeta_k(s))^{\varpi(k,G)}}{(\varpi(k,G)! -1)\cdot|\O_k^* \otimes G|\cdot|G|^{|S_{\mathrm{f}}|}}
	\prod_{v \notin S} \left(\sum_{\substack{
	\chi_v \in \Hom(\O_v^*, G)}} \frac{1}{\Phi_v(\chi_v)}\right) 
	\frac{1}{\zeta_{k,v}(1)^{\varpi(k,G)}}	 \\
	 & \,\,\,  \times 
 	 \left(\sum_{ \substack{\chi \in \Hom\bigl(\prod_{v \in S}k_v^*,G\bigr)}} \frac{1}{\prod_{v \in S}\Phi_v(\chi_v)\zeta_{k,v}(1)^{\varpi(k,G)}}\sum_{x \in \mathcal{X}(k,G)} 
\prod_{v \in S} \langle \chi_v, x_v \rangle\right).
\end{align*}
Here $S$ is a large finite set of places and $S_{\mathrm{f}}$ is the non-archimedean places in $S$, we have $\varpi(k,G) = \sum_{\substack{g \in G \\ g \neq 0}} 1/[k(\zeta_{\ord g}) :k]$, and
$$\mathcal{X}(k,G) = \{x \in k^\times \otimes G^\sim : x_v = 1 \in k^\times_v \otimes G^\sim \text{ for all but finitely many places } v\}.$$
Here $G^\sim$ denotes the $\Q/\Z$-dual of $G$ (this canonically isomorphic to the Pontryagin dual via the map $\Q/\Z \to S^1, t \mapsto \exp(2 \pi i t).)$ Here $\langle \cdot, \cdot \rangle$ denotes the Pontryagin duality pairing (see \cite[\S 3.1]{FLN18} for details). We first identify the effective cone constant. 

\begin{lemma}
	$$\frac{1}{|\O_k^* \otimes G|\cdot|G|^{|S_{\mathrm{f}}|}} = 
	\frac{|G|}{|\dual{G}(k)|\cdot|G|^{|S|}}.$$
\end{lemma}
\begin{proof}
	We use Dirichlet's unit theorem to write
	$$|\O_k^* \otimes G| = |\mu_\infty(k) \otimes G| \cdot |G|^{|S_\infty| - 1}.$$
	However a simple calculation shows that $|\mu_\infty(k) \otimes G| = |\dual{G}(k)|.$
\end{proof}

To identify the Brauer group, we use the following version of Kummer theory.

\begin{lemma} \label{lem:Kummer}
	There is a canonical isomorphism 
	$k^\times \otimes G^\sim \cong \H^1(k, \dual{G})$.
\end{lemma}
\begin{proof}
	Firstly we claim that $G^\sim \cong \Ext^1(G,\Z)$. To see this, apply
	$\Hom(G,\cdot)$ to the exact sequence
	$$0 \to \Z \to \Q \to \Q/\Z \to 0$$
	and use $\Hom(G,\Q) = \Ext^1(G,\Q) = 0$. 
	We deduce that each element of 
	$G^\sim$ gives rise to an exact sequence
	\begin{equation} \label{eqn:E}
		0 \to \Z \to E \to G \to 0
	\end{equation}
	for some finitely generated abelian group $E$. We now take Cartier duals to obtain
	$$0 \to \dual{G} \to \dual{E} \to \Gm \to 0.$$
	By Galois cohomology we have a map
	$k^\times \to \H^1(k,\dual{G}).$
	Applying this to each element of $G^\sim$ yields a  map 
	$k^\times \otimes G^\sim \to \H^1(k, \dual{G})$.
	To prove that this is an isomorphism we may choose a presentation of $G$,
	so that  $G = \Z/n\Z$ for some $n \in \N$.	
	In which case, the  sequence \eqref{eqn:E} corresponding to the element 
	$1 \in (\Z/n\Z)^\sim = \Z/n\Z$ is simply
	$$0 \to \Z \to \Z \to \Z/n\Z \to 0.$$
	Unravelling our construction, the corresponding map 
	$k^\times \otimes \Z/n\Z = k^\times/k^{\times n} \to \H^1(k,\mu_n)$ is just the usual
	isomorphism from Kummer theory.	
\end{proof}

Next the Tamagawa measure.
\begin{lemma}
	\begin{align*}
	|\Brun BG/\Br k|& \cdot \tau(\prod_v BG(k_v)^{\Brun}) = \\
	& \frac{(\mathrm{Res}_{s=1} \zeta_k(s))^{\varpi(k,G)}}{|G|^{|S|}}
	\prod_{v \notin S} \left(\sum_{\substack{
	\chi_v \in \Hom(\O_v^*, G)}} \frac{1}{\Phi_v(\chi_v)}\right) \zeta_{k,v}(1)^{-\varpi(k,G)}	 \\
	 & \times 
 	 \left(\sum_{ \substack{\chi \in \Hom\bigl(\prod_{v \in S}k_v^*,G\bigr)}} \frac{1}{\prod_{v \in S}\Phi_v(\chi_v)\zeta_{k,v}(1)^{\varpi(k,G)}}\sum_{x \in \mathcal{X}(k,G)} 
\prod_{v \in S} \langle \chi_v, x_v \rangle\right).
\end{align*}
\end{lemma}
\begin{proof}
For the local Euler factors we have
$$\sum_{\substack{
\chi_v \in \Hom(\O_v^*, G)}} \frac{1}{\Phi_v(\chi_v)} = 
\frac{1}{|G|}\sum_{\substack{
\chi_v \in \Hom(k_v^*, G)}} \frac{1}{\Phi_v(\chi_v)} = 
\frac{1}{|G|}\sum_{\substack{	\chi_v \in \Hom(\Gamma_{k_v}, G)}} \frac{1}{\Phi_v(\chi_v)}.$$
Indeed, the first equality follows from the fact that multiplication by an unramified character does not change the conductor, and the second equality is local class field theory. We deduce that the right hand side in the statement equals
$$(\mathrm{Res}_{s=1} \zeta_k(s))^{\varpi(k,G)}
\sum_{x \in \mathcal{X}(k,G)} 
\prod_v \frac{1}{|G|}\sum_{\substack{\chi_v \in \Hom(\Gamma_{k_v},G)}} \frac{\langle \chi_v, x_v \rangle}{\Phi_v(\chi_v)\zeta_{k,v}(1)^{\varpi(k,G)}}.$$
We relate this to the stated Tamagawa measure with Lemma \ref{lem:sum_Euler_products}, using the convergence factors from Lemma \ref{lem:convergence_factors_alternative}. By Lemmas~\ref{lem:Kummer} and \ref{lem:H1_Pic} we have $k^\times \otimes \dual{G} = \H^1(k,\dual{G}) = \Br_e BG$, and similarly over each $k_v$. Thus via Lemma \ref{lem:BrBG-Abelian}, we can canonically identify $\mathcal{X}(k,G)$ with $\Br_{\mathrm{un},e} BG$. Moreover Lemma~\ref{lem:cup_products} identifies the Brauer--Manin pairing with the cup product, which identifies with the Pontryagin pairing $\langle \cdot, \cdot \rangle$ by local Tate duality \cite[Thm~7.2.6]{NSW08}. The result now easily follows.
\end{proof}

Combining the above lemmas we deduce that $c_{k,G} = |G| c(k,G,H)$. The additional factor of $|G|$ matches with the groupoid cardinality factor $Z(G) = G$ from Conjecture \ref{conj:balanced}, hence the leading constant agrees with Conjecture \ref{conj:balanced}.

\subsubsection{Discriminant} 
Let $G$ be a finite abelian group  and consider $G \subset S_{|G|}$ the obvious embedding.
Let $\Delta$ be as in Example \ref{ex:disc}. For $g \in G$ we have
\begin{equation*}
	\ind(g) = |G|(1 - 1/\ord(g)).
\end{equation*}
Thus the index function is minimised when $g$ has minimal order, in other words when the order is the minimal prime $Q$ dividing $|G|$.

\begin{lemma}
	We have $a(\Delta) = 1/|G|(1-Q^{-1}), b(D) = (|G[Q]|-1)/[k(\mu_Q):k]$. Moreover $\Delta$ is balanced
	if and only if $G = G[Q]$, i.e.~$G$ is an elementary abelian group.
\end{lemma}
\begin{proof}
	The group generated by the minimal index elements is exactly $G(-1)[Q]$. This equals
	$G(-1)$ if and only if $G = G[Q]$.
\end{proof}

Wright \cite{Wri89} was the first to obtain an asymptotic formula in this case, and his method can yield an explicit formula for the leading constant. This is given by a finite sum of Euler products (see \cite[(4.13)]{FLN18} for such an expression). However this is different to the sum that we have in Lemma \ref{lem:sum_Euler_products}, since Wright's sum arises from a M\"{o}bius inversion argument.

It is possible by Poisson summation on the base of the Iitaka fibration and a M\"{o}bius inversion argument to rewrite Wright's constant in the form given in Conjecture \ref{conj:non_balanced}. However this argument is slightly convoluted, and it seems simpler to give a more direct argument. We give a sketch of the details; a full proof, including for general height functions, will appear in forthcoming work of Tavernier \cite{Tav24}. 

Firstly as Conjecture \ref{conj:non_balanced} predicts, we should sort $G$-extensions according to the Iitaka fibration $BG \to B(G/G[Q])$. By Lemma \ref{lem:fibration_normal_quotient} the fibres over rational points are isomorphic to the classifying stack of an inner twist of $G[Q]$. However since $G$ is abelian, any such inner twist is trivial. The discriminant becomes balanced when restricted to a fibre, and the rational points in each fibre can be counted using the harmonic analysis techniques from \cite{FLN18, FLN22}. Conjecture \ref{conj:non_balanced} then predicts that the sum over all the fibres is absolutely convergent. This can be shown using the dominated convergence argument from Step 1 of the proof of \cite[Thm.~1.1]{KR24}. This gives a formula which is in agreement with Conjecture \ref{conj:non_balanced}.

\subsubsection{Other balanced height functions} \label{sec:fair_wood}
For $G$ abelian, in \cite[\S 2.1]{Woo10} Wood calls a height function $H$ \textit{fair} if $\mathcal{M}(H) \cap G(-1)[r]$ generates $G[r]$ for all $r \in \N$. This is strictly stronger than our condition that $H$ be balanced, which states only that $\mathcal{M}(H)$ generates $G$. Wood's condition turns out to have strong consequences for the Brauer group.

\begin{lemma} \label{lem:fair_Wood}
	Assume that $\mathcal{M}(H) \cap G(-1)[r]$ generates $G(-1)[r]$ for all $r \in \N$. Then
	$\Br_{\mathcal{M}(L)} BG =\Brun BG$.
\end{lemma}
\begin{proof}
	Lemma \ref{lem:unramified_geometric_Brauer_group_computation} implies that 
	$\Br_{\mathcal{M}(L)} BG  = \Br_{\mathcal{M}(L), 1} BG$. 	
	Let $\varphi \in \Br_{\mathcal{M}(L), e} BG \subset \H^1(k, \hat{G})$.
	We will show that $\varphi \in \Brun_{,e} BG$.
	To do this we will use Theorem \ref{thm:Br_BG} in terms of the age residue.
	
	We use a similar approach to the proof of Lemma \ref{lem:BrBG-Abelian}.
	Represent $\psi$ by a cocycle $\Gamma_k \to \Hom(G(-1)(k^{\sep}), \Q/\Z)$. Let $\gamma \in G(-1)(k^{\sep})$ and identify the age residue at $\gamma$ with the morphism $\Gamma_{k(\gamma)} \to \Q/\Z: \sigma \to \psi(\sigma)(\gamma)$. 
	
	We now use that $G$ is constant to see that $k(\gamma) = k(\mu_r)$ where $r$ is the order of $\gamma$. The assumption of the lemma implies that $\gamma = \sum_{m \in \mathcal{M}(H) \cap G(-1)[r]} a_m m$ for some $a_m \in \Z$. For all $m \in G(-1)[r]$ we have $k(m) \subset k(\mu_r)$.
	
	As $\psi \in \Br_{\mathcal{M}(L), e} BG$ we have $\psi(\sigma)(m) = 0$ for all $m \in \mathcal{M}(H) \cap G(-1)[r]$ and all $\sigma \in \Gamma_{k(\mu_r)} \subset \Gamma_{k(m)}$. As $\psi(\sigma) \in \Hom(G(-1)(k^{\sep}), \Q/\Z)$ we have $\psi(\sigma)(\gamma) = \sum_{m \in \mathcal{M}(H) \cap G(-1)[r]} a_m \psi(\sigma)(m) = 0$ for all $\sigma \in \Gamma_{k(\mu_r)}$. This implies that $\psi$ is unramified at $\gamma$, and we are done as $\gamma$ was arbitrary.
\end{proof}

Lemma \ref{lem:fair_Wood} explains the fact that Wood was able in \cite[Thm.~3.1]{Woo10} to state and prove a much simpler version of the equidistribution conjecture (Conjecture~\ref{conj:equi}). Namely her conditions imply that she lay in case (2) of Proposition \ref{prop:Malle-Bhargava} and Remark~\ref{rem:hierarchy}. For general balanced heights the equidistribution property is more complicated as the corresponding Brauer group can be larger; the conjecture in this case will be proven in forthcoming work of Tavernier \cite{Tav24}.

\begin{example}
 An example of a height function which is balanced, but not fair in the sense of Wood, would be a height on $B \Z / 4 \Z$ with weight function $1 \mapsto 1, 2 \mapsto 2, 3 \mapsto 1$. This explicit example was provided by Wood and presented in Alberts's paper \cite[\S 7.6]{Alb23} as pathological to cast doubt on whether the definition of fair is correct, since the leading constant is a sum of two Euler products even though there is no Grunwald--Wang phenomenon. But the presence of this second Euler product occurs as the relevant partially unramified Brauer group has two elements (modulo constants): the non-trivial element is given by $-4 \in \Br_1 B\Z/4\Z \cong H^1(k, \mu_4) \cong k^{\times}/k^{\times 4}$ (this can be shown using Theorem \ref{thm:Br_BG}(3)). This leads to a sum of Euler products as explained in Lemma \ref{lem:sum_Euler_products}.
\end{example}
More details and further examples of related phenomena can be found in the work of Tavernier \cite{Tav24}.

\end{document}